\documentclass[12pt]{amsart}
\usepackage{amssymb}    
\usepackage[T1]{fontenc} 
\DeclareMathAlphabet{\mathantt}{OT1}{antt}{li}{it}   
\DeclareMathAlphabet{\mathpzc}{OT1}{pzc}{m}{it} 
\usepackage[normalem]{ulem}
\usepackage{tikz-cd} 
\usepackage{tikz}
\usepackage[colorlinks=true,linkcolor=blue,urlcolor=blue,citecolor=blue, pagebackref]{hyperref}
\usepackage{soul}
\usepackage[margin=1.0in]{geometry} 
\usepackage{xypic, verbatim, amscd,color}
\usepackage{mathrsfs}
\usepackage{tikz-cd}

\usepackage{eucal}

\usepackage{courier}

\usepackage{url} %

\numberwithin{equation}{section}

\newcommand\starr{\bullet}
\newcommand\mci{\mathscr{I}}

\newcommand\mco{\mathscr{O}}

\newcommand\dlog{\operatorname{dlog}}

\newcommand\op{\operatorname}

\newcommand\bull{\sssize{\bullet}}

\newcommand{\vv}{\vec}

\newcommand{\nc}{\newcommand}
\nc{\Z}{\mathbb{Z}}
\nc{\Q}{\mathbb{Q}}
\nc{\C}{\mathbb{C}}
\nc{\R}{\mathbb{R}}
\nc{\A}{\mathbb{A}}
\nc{\G}{\mathbb{G}}
\nc{\N}{\mathbb{N}}
\nc{\V}{\mathbb{V}}
\nc{\D}{\mathbb{D}}
\nc{\dc}{\Delta}

\nc{\ms}{\mathscr}
\nc{\mr}{\mathrm}
\nc{\mc}{\mathcal}
\nc{\mt}{\mathtt}
\nc{\msf}{\mathsf}
\nc{\mf}{\mathfrak}
\nc{\mb}{\mathbb} 

\nc{\wtid}{\widetilde{\mathcal{D}}}
\nc{\wtic}{\widetilde{\mathcal{C}}}
\nc{\wh}{\widehat}
\nc{\ti}{\widetilde}

\nc{\rat}{\mt{rat}}

\nc{\la}{\lambda}

\nc{\msp}{\ms{P}}

\nc{\ve}{\varepsilon}
\nc{\vpi}{\varpi}

\nc{\sm}{\setminus}

\nc{\dbc}{\mr{D}^b_c}
\nc{\dpl}{\mr{D}^+}
\nc{\mrd}{\mr{D}}
\nc{\db}{\mr{D}^b}

\nc{\msm}{\ms{M}}

\nc{\msk}{\msf{k}}

\nc{\h}{\mbox{-}}

\nc{\bs}{\backslash}
\nc{\ov}{\overline}

\nc{\rank}{\mr{rank}}

\newcommand\home{\operatorname{Hom}}

\newcommand\spec{\operatorname{Spec}}

\newcommand\tensor{\otimes}
\newcommand\ml{\mathcal{L}}

\newcommand\ma{\mathcal{A}}

\renewcommand{\P}{\mathbb{P}}

\newcommand\End{\operatorname{End}}
\newcommand\mj{\mathcal{J}}

\newcommand\im{\operatorname{im}}

\newcommand\rk{\operatorname{rk}}

\newcommand\wt{\widetilde}
\newcommand\Res{\operatorname{Res}}
\newcommand \lra {\longrightarrow}
\newcommand{\leto}[1]{\stackrel{#1}{\lra}}

\newtheorem{theorem}[equation]{Theorem}

\newtheorem{corollary}[equation]{Corollary}

\newtheorem{proposition}[equation]{Proposition}

\newtheorem{lemma}[equation]{Lemma}

\newtheorem{conj}[equation]{Conjecture}
\newtheorem{problem}[equation]{Problem}

\theoremstyle{definition}
\newtheorem{definition}[equation]{Definition}
\newtheorem{defi}[equation]{Definition}

\newtheorem*{algo}{Algorithm}
\newtheorem*{expect}{Expectation}
\theoremstyle{remark}

\newtheorem{example}[equation]{Example}  
\newtheorem{remark}[equation]{Remark}
\newtheorem*{claimn}{Claim}

\begin{document} 
\title[Motivic factorisation of KZ local systems]{Motivic
  factorisation of KZ local systems and deformations of representation
  and fusion rings}

 \author{P. Belkale, N. Fakhruddin and S. Mukhopadhyay}
\maketitle

\begin{abstract}
  Let $\mf{g}$ be a simple Lie algebra over $\C$.  The KZ connection
  is a connection on the constant bundle associated to a set of $n$
  finite dimensional irreducible representations of $\mf{g}$ and a
  nonzero $\kappa \in \C$, over the configuration space $\mc{C}_n$ of
  $n$-distinct points on the affine line. Via the work of
  Schechtman--Varchenko and Looijenga, when $\kappa \in \Q$ the
  associated local systems can be seen to be realisations of naturally
  defined motivic local systems. We prove a basic factorisation for
  the nearby cycles of these motivic local systems as some of the $n$
  points coalesce.

  This leads to the construction of a family (parametrised by
  $\kappa$) of deformations over $\Z[t]$ of the representation ring of
  $\mf{g}$---we call these enriched representation rings---which
  allows one to compute the ranks of the Hodge filtration of the
  associated variations of mixed Hodge structure; in turn, this has
  applications to both the local and global monodromy of the KZ
  connection. In the case of $\mf{sl}_n$ we give an explicit algorithm
  for computing all products in the enriched representation rings,
  which we use to prove that if $1/\kappa \in \Z$ then the global
  monodromy is finite and scalar.

  We also prove a similar factorisation result for motivic local
  systems associated to conformal blocks in genus $0$; this leads to
  the construction of a family of deformations of the fusion
  rings. Computations in these rings have potential applications to
  finiteness of global monodromy.

  Several open problems and conjectures are formulated. These include
  questions about motivic BGG-type resolutions and the relationship
  between the Hodge filtration and the filtration by conformal blocks
  at varying levels.
\end{abstract}

\vspace{0.5cm}
\begin{flushright}
  {\it To Madhav Nori, with gratitude and admiration}
  \end{flushright}
\vspace{.5cm}

\section{Introduction}

Let $\mf{g}$ be a split finite dimensional simple Lie algebra over a
field $\msk\subseteq \C$.  Let $P^+$ be the set of dominant integral
weights of $\mf{g}$ corresponding to a fixed Cartan decomposition. See
Section \ref{s:notations} for any unexplained notation that we use
below.

\subsection{Representation rings, multiplicity motives, and deformations}

Let $\lambda_1,\dots, \lambda_n\in P^+$. The tensor product of the
corresponding irreducible representations
$V_{\lambda_1},\dots,V_{\lambda_n}$ of $\mf{g}$ has a decomposition
\begin{equation*}\label{decompose}
  V_{\lambda_1}\tensor \dots\tensor V_{\lambda_n}=\bigoplus_{\nu\in
    P^+} \op{Hom}_{\mf{g}}(V_{\nu}, V_{\lambda_1}\tensor
  \dots\tensor V_{\lambda_n})\tensor V_{\nu} \, .
\end{equation*}

The multiplicity space
$\op{Hom}_{\mf{g}}(V_{\nu}, V_{\lambda_1}\tensor \dots\tensor
V_{\lambda_n})$ is isomorphic to the space of coinvariants
$\mb{A}(\vec{\lambda},\nu^*)$, where
$\vec{\lambda} = (\lambda_1,\dots,\lambda_n)$, which is defined as
follows:
\begin{equation}\label{coinvariants}
  \mb{A}(\vec{\lambda},\nu^*)= (V(\vec{\lambda})\tensor
  V^*_{\nu})_{\mf{g}}=\frac{V(\vec{\lambda})\tensor V^*_{\nu}}{\mf{g}
    (V(\vec{\lambda})\tensor V^*_{\nu})}, \text{ where }
  {V}(\vec{\lambda})= V_{\lambda_1}\tensor\dots\tensor V_{\lambda_n} .
\end{equation}
Note that for a finite dimensional representation $V$ of $\mf{g}$, the
invariants $V^{\mf{g}}$ map isomorphically to the coinvariants
$V_{\mf{g}}$. The ranks of the vector spaces
$\mb{A}(\vec{\lambda},\nu^*)$ are encoded in the representation ring:
It is the free abelian group $\mf{R}(\mf{g})=\mb{Z}^{P^+}$ with
the unital commutative and associative product given by
\begin{equation}
  [\lambda]\cdot [\mu] =\sum_{\nu\in P^+}
  \dim (V_{\lambda}\tensor V_{\mu}\tensor V^*_{\nu})_{\mf{g}}\,[\nu]\, .
\end{equation}
The rank of $\mb{A}(\vec{\lambda},\nu^*)$ is the coefficient of $[\nu]$ in the product
$[\lambda_1]\cdots[\lambda_n]$. 

\smallskip

  Let $\kappa \in \Q^{\times}$ and $\msf{k} = \Q$. Using results of
  Schechtman--Varchenko and Looijenga \cite{SV,L2,BBM} we construct
  canonical Nori motives $M_{\kappa}(\la,\mu,\nu^*)$ over a cyclotomic
  field $\Q(\mu_N)$ with coefficients also in $\Q(\mu_N)$, for a
  suitable integer $N$, which we call
``multiplicity motives'', whose Betti realisations over $\C$ are
naturally isomorphic to the dual of the space of coinvariants
$(V_{\lambda}\tensor V_{\mu}\tensor V_{\nu}^*)_{\mf{g}, \C}$.  Roughly
speaking, $M_{\kappa}(\la,\mu,\nu^*)$ are motives associated to
cohomology with supports of finite covers of affine spaces ramified
along a certain finite set of hyperplanes; for the precise definition
see Section \ref{s:kzdef}.  The reader not familiar with (Nori)
motives will not lose much for now by thinking of
$M_{\kappa}(\la,\mu,\nu^*)$ as being mixed Hodge structures.

\smallskip

Let $P(\msf{M})=\sum_{p} \dim (F^p/F^{p+1} ) t^p\in \Z[t]$ denote the
Hodge polynomial of a mixed Hodge structure $\msf{M}$ (see Definition
\ref{d:hn}). Then we define a deformation of $\mf{R}_{\kappa}(\mf{g})$
with coefficients in $\Z[t]$ as follows:
\begin{definition}\label{kappaRing}
  $\mf{R}_{\kappa}(\mf{g})$ is the $\Z[t]$-module freely generated by
  $\{[\lambda]\}_{\lambda \in P^+}$ with the $\Z[t]$-linear product
  given by
\begin{equation*}
  [\lambda]\star_{\mf{R}} [\mu] =\sum_{\nu\in P^+}
  P(M_{\kappa}(\la,\mu,\nu^*))[\nu] .
\end{equation*}
\end{definition}
Here $P(M_{\kappa}(\la,\mu,\nu^*))$ is the Hodge polynomial of the
Hodge realisation of the Nori motive $M_{\kappa}(\la,\mu,\nu^*)$.

One of the main results of this article is the following:
\begin{theorem}\label{t:assoc}
  The $\star_{\mf{R}}$-product on $\mf{R}_{\kappa}(\mf{g})$ gives it the
  structure of a unital, commutative and associative $\Z[t]$-algebra.
\end{theorem}
Since for a mixed Hodge structure $\msf{M}$, $P(\msf{M})(1) = \rk(\msf{M})$,
it follows that for all $\kappa$
\[
  \mf{R}_{\kappa}(\mf{g}) \otimes_{\Z[t], \,t \mapsto 1} \Z =
  \mf{R}(\mf{g}) .
\]
Thus, our construction gives rise to a family, parametrised by
$\kappa \in \Q^{\times}$, of deformations over $\Z[t]$ of the
classical representation ring $\mf{R}(\mf{g})$. We will
  call these rings enriched representation rings.

\smallskip

\subsection{{The KZ connection and related motives}}

Theorem \ref{t:assoc} is an easy consequence of a fundamental
factorisation theorem, which we prove, for the nearby cycles of
certain mixed local systems (or admissible variations of mixed Hodge
structure) on the configuration spaces $\mc{C}_n$ of $n$ distinct
points on $\Bbb{A}^1_k$. These mixed local systems arise from a
motivic enhancement of the Knizhnik--Zamolodchikov (KZ) equations on
spaces of coinvariants and conformal blocks.

The KZ connection is defined on the constant vector bundle on
$\mathcal{C}_n$ with fibre
$V(\vv{\lambda})=V_{\lambda_1}\tensor V_{\lambda_2}\tensor\dots\tensor
V_{\lambda_n}$. Using variables $z_1,z_2,\dots,z_n$ on $\mc{C}_n$, the
connection equations are:
\begin{equation}\label{e:KZ}
    \kappa\frac{\partial}{\partial z_i}f =\biggl(\sum_{j\neq
    i}\frac{\Omega_{ij}}{z_i-z_j}\biggr)f.
\end{equation}
Here $f$ is any local section of
$V(\vv{\lambda})\otimes\mathcal{O}_{\mathcal{C}_n}$,  $\Omega_{ij}$
is the normalised Casimir element acting on the $i$ and $j$ tensor
factors, and $\kappa\in\msk^{\times}$.

This connection, and the related WZW/Hitchin connections, especially
in genus zero, and their $q$-analogs appear in many areas of
mathematics, e.g., representation theory, enumerative geometry,
algebraic geometry, number theory, and also mathematical physics
\cite{BRigid,BFS,Drin,EFK,Fakh,etrange,sorger,varch}.

The KZ connection is flat and commutes with the diagonal action of
$\mf{g}$, hence it induces a connection on the constant bundle of
coinvariants \eqref{coinvariants}, i.e.,
$\mathbb{A}(\vv{\lambda},\nu^*)$.  We will think of these spaces as
being attached to the representations
$V_{\lambda_1},\dots,V_{\lambda_n}$ at $z_1,\dots,z_n$ respectively,
and the representation $V_{\nu}^*$ at $\infty\in\mathbb{P}_{\msk}^1$.

\smallskip

For an integer $\ell\geq 0$, an element $\lambda \in P^+$ is said to
be of level $\ell$ if $(\lambda,\theta)\leq \ell$. Let
$P_{\ell}\subset P^+$ be the set of weights of level $\ell$. When
$\kappa=\ell + h^{\vee}$, $\ell\in\mathbb{Z}_{>0}$, and
$\lambda_1,\dots,\lambda_n$ and $\nu$ are of level $\ell$, the
connection on
$(V_{\lambda_1}\tensor V_{\lambda_2}\tensor\cdots\tensor
V_{\lambda_n}\tensor V^*_{\nu})_{\mf{g}} \otimes \mc{O}_{\mc{C}_n}$
induces a connection, also called the KZ connection, on the quotient
vector bundle of dual conformal blocks
$\mathbb{V}_{\mf{g},\vv{\lambda},\nu^*,\ell}$ (with an insertion at
infinity) \cite{TUY}, \cite[Chapter 6]{Ueno}.  Each
$\mb{V}_{\mathfrak{g},\vec{\lambda},\nu^*,\ell}$ is a quotient of
$\mb{V}_{\mathfrak{g},\vec{\lambda},\nu^*,\ell+1}$ in a natural way,
(see, e.g., the first theorem on p.~220 of \cite{FSV2}), but the
latter is not in general preserved by the KZ connection with
$\kappa = \ell + h^{\vee}$.

The {mathematical} theory of integral representations of solutions of
KZ equations for generic $\kappa$ \cite{SV}, and the extension to
arbitrary $\kappa$ in \cite{L2,BBM}, allows one to construct (see
Section \ref{s:kzmls}) motivic local systems
$\mathcal{KZ}_{\kappa}(\vec{\lambda},\nu^*)$ on $\mc{C}_n$ whose Betti
realisations give the { duals of} classical KZ local systems over
$\C$. When $\kappa = \ell + h^{\vee}$, we construct (see Section
\ref{s:cbmls}) motivic local systems of conformal blocks
$\mathcal{CB}_{\kappa}(\vec{\lambda},\nu^*)$ on $\mathcal{C}_n$ which
are natural subsystems of
$\mathcal{KZ}_{\kappa}(\vec{\lambda},\nu^*)$.  The fibres of these
local systems over points $\vec{z}\in \mc{C}_n(\msk)$ will be denoted
by $\mathcal{KZ}_{\kappa}(\vec{\lambda},\nu^*)_{\vec{z}}$ and
$\mathcal{CB}_{\kappa}(\vec{\lambda},\nu^*)_{\vec{z}}$. These are Nori
motives over a suitable finite extension of $\msk$ with coefficients
in a cyclotomic field. The same constructions carried out in the
derived category of mixed Hodge modules leads to variations of mixed
Hodge structures which we denote by
$\msf{KZ}_{\kappa}(\vec{\lambda},\nu^*)$ and
$\msf{CB}_{\kappa}(\vec{\lambda},\nu^*)$. These variations were
introduced by Looijenga \cite[p. 221 and 222]{L2}. The conformal block
variations are pure, but the KZ variations are in general mixed.

The motives $M_{\kappa}(\la,\mu,\nu^*)$ used in constructing the rings
$\mf{R}_{\kappa}(\mf{g})$ above
are the fibres of $\mc{KZ}_{\kappa}((\la,\mu),\nu^*)$ over
$(0,1)\in\mathcal{C}_2$. These motives need not be pure or even
  mixed Tate, see Example \ref{e:introduction}. We similarly have
motives $C_{\kappa}(\la,\mu,\nu^*)$ (for $\kappa = \ell + h^{\vee}$)
corresponding to the conformal blocks local systems which are always
pure.

The coefficient of $[\nu]$ in
$[\lambda_1]\star _{\mf{R}}\dots \star_{\mf{R}} [\lambda_n]$ in
$\mf{R}_{\kappa}(\mf{g})$ equals the Hodge polynomial of the MHS
$\msf{KZ}_{\kappa}(\vec{\lambda},\nu^*)_{\vec{z}}$, where $\vec{z}$ is
any point in $\mc{C}_n(\C)$. Therefore the Hodge polynomials of
arbitrary KZ-motives can be expressed in terms of the ring
$\mf{R}_{\kappa}(\mf{g})$ (see Remark \ref{r:arbitraryprod}).
\smallskip

For arbitrary $\mf{g}$ there
exists $m_{\mf{g}} \in \N$ such that for all $\kappa>0$, replacing
$\kappa$ by $\kappa'=\kappa/(1 +m_{\mf{g}}\kappa)$, does not change
the MHS associated to
$\mathcal{KZ}_{\kappa}(\vec{\lambda},\nu^*)_{\vec{z}}$ and hence does
not change its Hodge polynomial. In particular,
$\mf{R}_{\kappa}(\mf{g}) = \mf{R}_{\kappa'}(\mf{g})$. If we allow
$\vec{z}$ to vary, there is no natural map between the corresponding
local systems on $\mathcal{C}_n$. However, if we replace $\kappa$ by
${\kappa}/(1 +i_{\mf{g}}m_{\mf{g}}\kappa)$ the motivic local systems
on $\mathcal{C}_n$ are isomorphic.  When $\kappa\in \mb{C}^{\times}$
is irrational, the invariance of topological local systems follows
from results of Kazhdan--Lusztig \cite[Theorem 8.6.4]{EFK}. For all
this see Section \ref{s:trp}.

\subsection{Galois twists and conformal blocks}\label{s:gt}

If $\kappa = r/s$, then setting $N = i_{\mf{g}}m_{\mf{g}}|r|$, all the
mixed sheaves $\mathcal{KZ}_{\kappa}(\vec{\lambda},\nu^*)$ have
coefficients in $\mb{Q}(\mu_N)$ (see Lemma \ref{formula:N}).  If
$\sigma\in \op{Gal}(\Q(\mu_N)/\mb{Q})$, we have Galois twists
$\mathcal{KZ}^{\sigma}_{\kappa}(\vec{\lambda},\nu^*)$
(resp.~$\mathcal{CB}^{\sigma}_{\kappa}(\vec{\lambda},\nu^*)$) of the
motivic sheaves $\mathcal{KZ}_{\kappa}(\vec{\lambda},\nu^*)$
(resp.~$\mathcal{CB}_{\kappa}(\vec{\lambda},\nu^*)$ for
$\kappa = \ell +h^{\vee}$) given by precomposing the
$\Q(\mu_N)$-structure with $\sigma^{-1}$.  If
$\sigma(\zeta_N)=\zeta_N^a$ for some positive integer $a$ with
$(a,N) = 1$, then
$\mathcal{KZ}^{\sigma}_{\kappa}(\vec{\lambda},\nu^*) \simeq
\mathcal{KZ}_{\kappa/a}(\vec{\lambda},\nu^*)$ (see Remark
\ref{r:twist}).  This justifies the following:
  \begin{defi}\label{d:gtwist} 
    For any integer $\ell > 0$ and $\kappa = (\ell + h^{\vee})/a$,
    with $a$ a positive integer such that $(a,N) = 1$, let
    $\mathcal{CB}_{\kappa}(\vec{\lambda},\nu^*) =
    \mathcal{CB}^{\sigma}_{\ell + h^{\vee}}(\vec{\lambda},\nu^*)$,
    where $\sigma(\zeta_N)=\zeta_N^a$.  In particular, we define the
    motives $C_{\kappa}(\la,\mu,\nu^*)$, for $\kappa$ as above, to be
    $C_{\ell + h^{\vee}}^{\sigma}(\la,\mu,\nu^*)$.
  \end{defi}
  By construction, the motivic local system
  $\mathcal{CB}_{\kappa}(\vec{\lambda},\nu^*)$ is a subquotient of the
  motivic local system
  $\mathcal{KZ}_{\kappa}(\vec{\lambda},\nu^*)$. Using the
  $C_{\kappa}(\la,\mu,\nu^*)$ we define $\Z[t]$-algebras
  $\mf{F}_{\kappa}(\mf{g})$ which are deformations of the classical
  fusion rings at level $\ell$.
\begin{definition}\label{d:fr}
  $\mf{F}_{\kappa}(\mf{g})$ is the free $\Z[t]$-module
  $\Z[t]^{P_{\ell}}$ with the $\Z[t]$-linear product given by
\begin{equation*}
  [\lambda]\star_{\mf{F}} [\mu] =\sum_{\nu\in P_{\ell}}
  P(C_{\kappa}(\la,\mu,\nu^*))[\nu] .
\end{equation*}
\end{definition}
Analogously to Theorem \ref{t:assoc} we have:
\begin{theorem}\label{t:assoc2}
  The $\star_{\mf{F}}$-product on $\mf{F}_{\kappa}(\mf{g})$ gives it the
  structure of a unital, commutative and associative $\Z[t]$-algebra.
\end{theorem}
The ranks of $C_{\kappa}(\la,\mu,\nu^*)$ are equal to the classical
fusion coefficients (see, e.g., \cite[Section 7]{beauville}), and they
do not change upon Galois conjugation (though the Hodge polynomials
do!), therefore
$\mf{F}_{\kappa}(\mf{g}) \otimes_{\Z[t], t \mapsto 1}\Z$ is
canonically isomorphic to the classical level $\ell$ fusion ring. We
will call these rings enriched fusion rings.

\subsection{The motivic factorisation theorem}\label{s:motivicFormula}

We now state a simple form of the motivic factorisation theorem---see
Theorem \ref{t:fact} for the most general formulation---on which all
our other results are based. This is a statement about the nearby
cycles of $\mathcal{KZ}_{\kappa}(\vec{\lambda},\nu^*)$ as ``two of the
marked points come together''. This case is sufficient for the proof
of associativity of the product $\star$ on $\mf{R}_{\kappa}(\mf{g})$
and $\mf{F}_{\kappa}(\mf{g})$.

Recall that we have chosen a natural number $N$ depending only on
$\mf{g}$ and $\kappa$ such that
$\mathcal{KZ}_{\kappa}(\vec{\lambda},\nu^*)$ is a motivic local
system with coefficients in $K=\Q(\mu_N)$ on $\mc{C}_n$.

Let $W$ be a smooth variety over $\msk$ and $Z \subset X$ a smooth
divisor which is the zero locus of a regular function
$f:W\to \Bbb{A}^1_{\msk}$. Let $\mc{A}$ be a motivic (or mixed) local
system with coefficients in $K$ on $W \setminus Z$. Then the nearby
cycles $\Psi_f\, \mc{A}$ is a motivic (or mixed) local system on $Z$
which depends on the choice $f$. (For a brief review of nearby cycles
in the topological setting see Section \ref{s:rnb} and for the
motivic/mixed setting see Sections \ref{s:saito} and \ref{s:nbc}.)

Let $I=\{i_0,i_1\}$ be a two element subset of $[n]$. Let $W_I$ be the
open subset of $\A^n_{\msk}$ given by the complement of the closed set
defined by the equations $z_i= z_j$ for all $i \neq j$ such that
$\{i,j\} \neq I$. Let $Z_I$ be the closed subset of $W_I$ defined by
the equation $z_{i_0} = z_{i_1}$. Then $W_I-Z_I=\mc{C}_n$ and
$Z_I \simeq \mc{C}_{I^c \cup \{i_0\}}$.  Let $f:W_I\to \A^1_{\msk}$ be
the function $z_{i_1}-z_{i_0}$. We then have the following
factorisation theorem in the category of $K$-mixed sheaves.
\begin{theorem}\label{t:motivic}
  The $K$-mixed local system
  $\Psi_f(\mathcal{KZ}_{\kappa}(\vec{\lambda},\nu^*))$ on $Z_I$ has a
  filtration with associated graded isomorphic to
  \begin{equation*}
    \bigoplus_{\mu \in P^+}\mc{KZ}_{\kappa}(\vv{\lambda}'_{\mu}, \nu^*)
    \otimes_K M_{\kappa}(\lambda_{i_0},\lambda_{i_1},\mu^*) ,
  \end{equation*}
  where $\vv{\lambda}'_{\mu}$ is the set of weights labelled by
  $I^c\cup \{i_0\}$, where for $i$ in $I^c$ the weight assigned to $i$
  is $\lambda_i$ and for the point $i_0$ the weight is $\mu$.

  Furthermore, if all the $\lambda_i$ and $\nu$ are of level
  $\ell \geq 0$, then setting $\kappa = \ell + h^{\vee}$,
  we have a direct sum decomposition of the nearby cycles of conformal blocks,
  \begin{equation*}
    \Psi_f(\mathcal{CB}_{\kappa}(\vec{\lambda},\nu^*)) = \bigoplus_{\mu \in P_{\ell}}\mc{CB}_{\kappa}(\vv{\lambda}'_{\mu}, \nu^*)
    \otimes_K C_{\kappa}(\lambda_{i_0},\lambda_{i_1}, \mu^*) .
  \end{equation*} 
\end{theorem}

\begin{remark}
  By applying automorphisms of $K$ we get twisted forms of these
  factorisations.
\end{remark}

Theorem \ref{t:motivic} implies the associativity in Theorem
\ref{t:assoc}, using for the category of mixed sheaves Saito's
category $\mr{MHM}(X_{\C})$ of mixed Hodge modules \cite{saito}. The
key property is that the Hodge filtration on nearby cycles in Saito's
category is a suitable limit of the Hodge filtration along the
degeneration, see Lemma \ref{l:saito} and Section
\ref{s:associativity}.

Theorem \ref{t:motivic} is proved in a stronger form in Theorem
\ref{t:fact}. This includes degenerations where more than two of the
$z_i$ come together. Theorem \ref{t:fact} also includes information
about eigenvalues of local monodromy and sizes of Jordan blocks (and
the corresponding nilpotent operators $\mr{N}$), and recovers the
classically known results about the eigenvalues of the KZ and
conformal block local systems (see e.g., the book of Bakalov and
  Kirillov \cite[Theorem 6.5.3, Corollary 7.8.9]{BK}). It also
recovers the fact that the local monodromy of the conformal block
local systems is semisimple (this follows from \cite[Theorem
6.2.6]{TUY}) and that the KZ local systems have semisimple local
monodromy under a suitable genericity hypothesis as in \cite[Theorem
6.5.3]{BK}.

Factorisation results for the topological local systems with complex
coefficients corresponding to
$\Psi_f(\mathcal{KZ}_{\kappa}(\vec{\lambda},\nu^*))$, for generic
$\kappa$, as a direct sum with the same terms (i.e., tensored with
$\C$) as the graded factors in Theorem \ref{t:motivic} and Theorem
\ref{t:fact} are given in \cite[Theorem 6.5.3]{BK}. Analogous results
for the topological local systems with complex coefficients
corresponding to $\Psi_f(\mathcal{CB}_{\kappa}(\vec{\lambda},\nu^*))$
with the same terms (i.e., tensored with $\C$) as in Theorem
\ref{t:motivic} and Theorem \ref{t:fact} are given in \cite[Theorem
7.8.8]{BK}.
 
While the result of Bakalov and Kirillov in \cite[Theorem 6.5.3]{BK}
is for generic $\kappa$ and at the level of complex local systems (and
without Hodge structures), it nevertheless places possibly non-zero
dominant weights at infinity, as we also do in our motivic
factorisation theorems. We note the importance of allowing weights at
infinity to even formulate the statement of Theorem \ref{t:motivic}:
If we start with $\nu=0$ and let $z_{i_0}$ and $z_{i_1}$ come
together, then the second factors of the factorisations,
$M_{\kappa}(\lambda_{i_0},\lambda_{i_1},\mu^*)$ and
$C_{\kappa}(\lambda_{i_0},\lambda_{i_1}, \mu^*)$, can have non-trivial
weights $\mu^*$ at infinity.

Our motivic factorisation theorem, Theorem \ref{t:fact}, has new
consequences for the (topological) monodromy representations of the KZ
local systems; these are representations of the fundamental group of
$\mathcal{C}_n$, the pure braid group on $n$ strands. When $\kappa$ is
generic or irrational, by the Drinfeld--Kohno theorem and work of
Kazhdan--Lusztig (see e.g., \cite[Theorem 8.6.4]{EFK}), there are
connections to the theory of quantum groups and these results impose
conditions on the local monodromy (also see \cite[Theorem 6.5.3,
Corollary 7.8.9]{BK}). But very little was known about the monodromy
when $\kappa \in \Q^{\times}$ is arbitrary.

Theorem \ref{t:fact} imposes restrictions on sizes of Jordan blocks
for local monodromy for arbitrary $\kappa$, see Remark
\ref{r:topological}. Since we are able to compute the Hodge numbers of
these motivic local systems for $\mf{sl}_n$ (as we explain below), the
theorem also has consequences for the global monodromy of KZ local
systems.  The motivic nature of our constructions potentially have
applications in number theory, via Galois representations and periods.

\subsection{The enriched rings for $\mf{sl}_n$}

We do not know how to compute all the products in the rings
$\mf{R}_{\kappa}(\mf{g})$ and $\mc{F}_{\kappa}(\mf{g})$ for general
$\mf{g}$, but in the case of $\mf{sl}_n$ we have a simple inductive
algorithm---depending on a basic motivic computation---for doing this,
which we describe in this section.  We first need some basic rank one
motives:

Let $\kappa \in \Q^{\times}$ and $a \in \Q^{>0}$. Let $N$ be a
positive integer such that $N\kappa^{-1}$ and $Na\kappa^{-1}$ are both
integers and assume that $\mu_N\subseteq \msk$.

There are basic rank one Nori motives $[a:\kappa]$ over $\msk$ with
coefficients in $\Q(\mu_N)$, see Definition \ref{d:basicmotives}: In brief,
the construction proceeds as follows. Consider first the case
$\kappa>0$. Let $C$ be the closed subvariety of $\Bbb{A}^2_{\msk}$
given by the equation
  $$y^N= x^{N\kappa^{-1}}(1-x)^{Na\kappa^{-1}}.$$
Let $\pi:C\to \A^1_{\msk}$ be the (finite) projection to the first factor. Note that $\mu_N$ acts on $C$ by scaling $y$, and let $\tau$ be  the inverse of the tautological character of $\mu_N$. The motive $[a;\kappa]$ is defined to be the $\tau$-isotypical summand of the Nori motive associated to $H^1(C,\pi^{-1}\{0,1\};\Q(\mu_N))$. 

\smallskip

If $\kappa<0$, we define $[a:\kappa]$ to be the dual of $[a:-\kappa]$
tensored with the Tate motive $\Q(\mu_N)(-1)$.  The Hodge polynomial
of all Galois conjugates of $[a:\kappa]$ can be computed easily: it is
always $1$, $t$ or $t^2$ (see Section \ref{s:vpi1}).

\smallskip

Let $\alpha_1,\alpha_2,\dots,\alpha_{n-1}$ be the simple roots and
$\vpi_1,\vpi_2,\dots,\dots,\vpi_{n-1}$ the fundamental dominant
weights for $\mf{sl}_n$ in the standard ordering.  By the Pieri
formula, for any $\lambda \in P^+$, $V_{\lambda}\tensor V_{\vpi_1}$ is
a sum without multiplicities of $V_{\mu}$, where
$\mu=\lambda +\vpi_1 - \sum n_j\alpha_j$ is such that for some
$M\in \{0,\dots,n-1\}$,
$$\sum n_j\alpha_j =L_1-L_{M+1}=\alpha_1+\dots +\alpha_M .$$
	
The Hodge numbers of the KZ local system for $\kappa<0$ are determined
by the numbers for $\kappa>0$ (see Remark \ref{dual_kappa1}), hence we
restrict to the case $\kappa>0$ below.  Let $N$ be the integer defined
in Section \ref{s:gt} and let $K = \Q(\mu_N)$. The following theorem
(proved in Section \ref{s:vpi1}) is the main geometric input in our
algorithm.
\begin{theorem}\label{genp}
  Assume $\kappa>0$.  If $M>0$, the motive
  $\mathcal{KZ}_{\kappa}((\lambda,\vpi_1),\mu^*)_{(0,1)}$ is equal to
  \begin{equation}\label{e:ring}
    [(\lambda,\alpha_M);\kappa]\tensor_K
    [1+(\lambda,\alpha_M+\alpha_{M-1});\kappa]\tensor_K\dots\tensor_K
    [M-1+ \sum_{i=1}^M(\lambda,\alpha_i);\kappa],
  \end{equation}
  and if $M=0$, the motive is $K(0)$.
\end{theorem}

\begin{remark}
  When $\lambda$ and $\mu$ are level $\ell$, $\kappa=\ell+h^{\vee}$, then
  $\mathcal{KZ}_{\kappa}((\lambda,\vpi_1),\mu^*)=
  \mathcal{CB}_{\kappa}((\lambda,\vpi_1),\mu^*)$.
\end{remark}
We show in Section \ref{s:sln} that for all other fundamental dominant
weights $\vpi_k$, $\lambda\star \vpi_k$ can be computed inductively
(i.e., in terms of $\mf{sl}_m$ for $m<n$) from the products
$\lambda\star\vpi_1$ for both the enriched representation and fusion
rings. The associativity of the enriched rings plays a crucial role in
these computations.  This leads to an explicit algorithm (see Section
\ref{s:algo}) for computing all products in our enriched
representation and fusion rings for all $\mf{sl}_n$.

Conformal block local systems give variations of pure Hodge structures
with coefficients in a cyclotomic field. Properties of the Hodge
numbers of such a geometric VHS together with those of all Galois
conjugates, have implications for the finiteness of the global
monodromy group (see, e.g., \cite[Lemma 4.11]{BF}). In particular, our
work gives sufficient (computable) criteria for the finiteness of
global monodromy of conformal blocks local systems for
$\mf{sl}_n$. For an application of this, see Example
  \ref{e:FiniteGoursat3} wherein we show the finiteness of global
  monodromy for a rank $4$ local system associated to $\mf{sl}_4$.
\smallskip

It is an important problem to compute the enriched rings for general
$\mf{g}$.  An example, discussed in Section \ref{s:genG}, suggests
that we might be reduced, by induction on the rank of $\mf{g}$, and
multiplication by generators, to covering certain ``minimal'' cases
(such as multiplication by $\vpi_1$ for $\mf{sl}_n$).  The case of
enriched fusion rings is simpler and should probably be considered
first.

\subsection{Conjectures}
Numerical computations for $\mf{sl}_n$ using our algorithm suggest
several conjectures for general $\mf{g}$. We describe some of them
below.

\subsubsection{}

Suppose $\vec{\lambda}$ and $\nu$ are at level $\ell$, and
$\kappa=\ell +h^{\vee}$. Assume
$\mathcal{KZ}_{\kappa}(\vec{\lambda},\nu^*)\neq 0$ and let $M \geq 0$
be the number of simple roots, counted with multiplicity, in an
expression of $\sum\lambda_i-\nu$ as a sum of simple roots.  We
conjecture that
$\mathcal{KZ}_{\kappa}(\vec{\lambda},\nu^*)/\mathcal{CB}_{\kappa}(\vec{\lambda},\nu^*)$
is mixed of weights $> M$ and for $k\geq 0$,
$F^{M-k}(\msf{KZ}_{\kappa}(\vec{\lambda},\nu^*)/\msf{CB}_{\kappa}(\vec{\lambda},\nu^*))$
 is a subspace of 
$\mb{V}^*_{\mathfrak{g},\vec{\lambda},\nu^*,\ell+k+1}/
\mb{V}^*_{\mathfrak{g},\vec{\lambda},\nu^*,\ell}$, i.e., the Hodge
filtration is a refinement of the  filtration by conformal blocks at higher
levels, see Conjecture \ref{c:HodgeFiltration}.

\subsubsection{}
We also conjecture that there exists a BGG type resolution of the
mixed local systems $\mathcal{KZ}_{\kappa}(\vec{\lambda},\nu^*)$ which
should reduce to Teleman's BGG resolution on realisations
\cite{TelemanCMP}, see Conjecture \ref{BGGtype}.

\subsubsection{}
When $1/\kappa$ is an integer, we expect $\mf{R}_{\kappa}(\mf{g})$ to
be canonically a trivial deformation of the usual representation ring,
see Conjecture \ref{c:mtate}; this is implied by all the KZ motives
for such $\kappa$ being pure of weight zero (in which case the global
monodromy acts by scalars). In this case, it would be very interesting
to determine the solutions of the KZ equations algebraically.
Corollary \ref{c:sln} shows that this Conjecture \ref{c:mtate} holds
for $\mf{sl}_n$.

\subsubsection{}
Finally, we expect the rings $\mf{R}_{\kappa}(\mf{g})$ and
$\mf{F}_{\kappa}(\mf{g})$ for $\kappa$ as in Definition \ref{d:gtwist}
to be related by an explicit surjective ring homomorphism, as is the
case for the classical representation and fusion rings
\cite{Faltings,TelemanCMP}, see Conjectures \ref{c:maptofusion} and
\ref{c:ec}; we verify these conjectures for $\mf{sl}_2$ in Theorem
\ref{t:sl2}.

\subsection{Methods of proof}

We first explain the main ideas in the proof of Theorem
\ref{t:fact} in the simpler setting of Theorem \ref{t:motivic} and
then discuss the multiplication algorithm for $\mf{sl}_n$.

Recall that the KZ local systems are motivic local systems on
$W\setminus Z=\mc{C}_n$. For notational convenience we drop the
subscript $I$, ignore isotypical components, and also consider the
statements at the topological level here (the formalism of mixed
sheaves naturally upgrades these results to a motivic level).

There is a smooth projective morphism $X^* \to W\setminus Z$ such that
the KZ local system $\mathcal{KZ}_{\kappa}(\vec{\lambda},\nu^*)$ is
obtained, via results of Looijenga---see the pointwise Proposition
\ref{p:kz} (and for conformal blocks, Proposition \ref{FSVCor}), and
the results in families given in Sections \ref{s:kzmls} and
\ref{s:cbmls}---as the image of the pushforward of a map
  $\mc{F}' \to \mc{G}'$ of motivic sheaves on $X^*$. The first step
is to extend the family $X^* \to W \setminus Z$ to a family $g:X\to W$
which gives a semistable degeneration over $Z$. The variety $X$ is
defined using the moduli of stable $N$-pointed curves of genus
$0$. Letting $Y\subset X$ be the inverse image of $Z$, we have
$X^* = X \setminus Y$.  We have function
  $f_X=f \circ g:X\to\A^1_{\msk}$, where $f$ is the function on $W$
  that cuts out $Z$, with $Y=f_X^{-1}(0)$.  The irreducible
components of $Y$ can be described explicitly.

We then consider the nearby cycles $\mc{F} := \Psi_{f_X}\mc{F}'$ and
$\mc{G} = \Psi_{f_X}\mc{G}'$ of these motivic sheaves on $Y$ and
filter them using the natural geometry of $Y$ and the ordering of
residues of a naturally defined rational $1$-form on $X$ associated to
$\vec{\lambda}$, $\nu^*$ and $\kappa$. To express the behaviour of
cohomology (i.e., pushforwards to  $Z$) of these filtered sheaves
in terms of the associate graded, which we identify as giving rise to
the factors in Theorem \ref{t:motivic}, we use some novel techniques
to overcome the problem of only getting limited exactness from the
general formalism.

The most important technique (used in the case of KZ systems and not
for conformal blocks) is the use of a de Rham type complex $A(\eta)$
on the central fibre $Y$ (see Definition \ref{def11} and Section
\ref{s:Aeta}). This complex is inspired by the work of Steenbrink
\cite{steenbrink}, and we use limit mixed Hodge theoretic results due
to Steenbrink \cite{steenbrink}, Steenbrink--Zucker \cite{SZ} and El
Zein \cite{EZ} in obtaining vanishing results for these complexes.

We use systematically techniques of obtaining subquotients which
correspond to irreducible components of $Y$ (see the results in
Section \ref{s:elementary}). We also use the fact that we already know
that the ranks are the same on the two sides of the factorisation
formula, so the problem becomes one of finding a subquotient with
equal rank. The cohomology groups of the complexes $A(\eta)$, as well
Aomoto cohomology groups on the general fibres of $g:X\to W$, have
constant ranks under scaling $\eta$. Although the topological
interpretation changes under this operation, the (generically) scaled
$\eta$ has better behaviour under the filtrations. This observation,
which uses the vanishing results for $A(\eta)$, also plays an
important role in establishing exactness properties for the
filtrations.

As already noted earlier, this motivic factorisation result implies the
associativity in Theorems \ref{t:assoc} and \ref{t:assoc2} via
properties of the Hodge filtration on nearby cycles.

\smallskip
The motivic Pieri formula for $\mf{sl}_n$ for the fundamental weight
$\vpi_1$ (Theorem \ref{genp}) is proved by induction via a ``fibration
by curves'' method. The problem of computing Pieri coefficients for
multiplication with other dominant fundamental weights $[\vpi_r]$
reduces, surprisingly, to the case of $\vpi_1$ using a double
induction; the fact that we know the product is associative, i.e.,
Theorem \ref{t:assoc}, plays a key role in the induction.  Using these
results we can compute arbitrary products in
$\mf{R}_{\kappa}(\mf{\mf{sl}_n})$ and
$\mf{F}_{\kappa}(\mf{\mf{sl}_n})$.

\subsection{Organisation of the paper}

In Section \ref{s:Prelim}, we review nearby cycles and some basic
cohomological results in the complex analytic setting. We also discuss
the local complex analytic form of some natural morphisms that appear
in this work. In Section \ref{s:aomoto}, we recall the definition of a
de Rham type complex, called the Aomoto complex, appearing in the
theory of weighted hyperplane arrangements, and discuss its
topological and motivic interpretations. The KZ and conformal block
motives are also defined in this section. Saito's notion of mixed
sheaves, which is an abstraction of properties expected from motivic
sheaves, is recalled in Section \ref{s:mixte}. We define the notion of
mixed sheaves with coefficients, and show that the local systems of
variations of KZ and conformal blocks give natural mixed sheaves with
coefficients on configuration spaces, and discuss Hodge decompositions
and behaviour with respect to nearby cycles in Saito's theory. These
sections are mostly a review of known results.

We next state and prove the motivic factorisation theorem in Sections
\ref{s:nb}--\ref{kzf}. In Section \ref{s:nb}, we study the nearby
cycles $\Psi_{f_X}$ on the fibre $Y=f_X^{-1}(0)$ of a semistable
degeneration $f_X:X\to \A^1_{\msk}$ of suitable mixed sheaves on
$X^* = X \setminus Y$, and ways of filtering such objects in terms of
irreducible components of $Y$ (Theorem \ref{stillholds}). In Section
\ref{s:prep}, we state the general factorisation theorem and make a
choice of a geometric family $g:X\to W$ and a function
$f:W\to \A^1_{\msk}$ with respect to which we study nearby cycles. We
also study the combinatorics and geometry of the irreducible
components of $Y$. The conformal blocks part of Theorem \ref{t:fact}
is proved first in Section \ref{cbf}, and the part about invariants
is proved later in Section \ref{kzf}.  Section \ref{s:mnb} is
preparatory to Section \ref{kzf} and defines a filtration on nearby
cycles corresponding to a filtration of $Y$ by components ordered by
the residues of the form $\eta$. A de Rham type complex $A(\eta)$ used
to study properties of this filtration is also introduced in Section
\ref{s:mnb}.

In Section \ref{s:tr} we prove the associativity of the enriched
representation and fusion rings. We discuss some basic properties of
these ring and then formulate some (conjectural) relations between the
fusion and representation ring parallel to the classical case, which
we prove for the basic case of $\mf{sl}_2$ in Section \ref{s:sl2}. In
Section \ref{s:sln} we describe our algorithm for computing all
products in the enriched fusion and representation rings for
$\mf{sl}_n$.

In Section \ref{s:level0}, we study $\mf{sl}_2$ KZ motives at level
$0$ and identify the basic cases of the $KZ$ motives with motives
coming from hyperelliptic curves. Using this we give examples of KZ motives
which are mixed with non-split weight filtration.

In Section \ref{s:Cpe} we discuss some conjectures, problems and
examples. In the appendix, we give the proof of Proposition \ref{p:kz}
for arbitrary dominant integral weights $\nu$ extending the arguments
for $\nu=0$ in \cite{L2, BBM}.

\bigskip

\noindent {\bf Acknowledgements.} 
We thank Joseph Ayoub, Patrick Brosnan and Morihiko Saito for helpful
conversations and/or correspondence.

A part of this work was done when P.B. visited TIFR Mumbai in July of
2022 and 2023 as an adjunct professor.  P.B. thanks TIFR for the
invitation and hospitality. P.B. was supported by NSF DMS -
2302288. N.F. thanks Vincent Pilloni for an invitation to visit the
Laboratoire Math\'ematique d'Orsay in April--June 2022. N.F. and S.M.
acknowledge support of the DAE, Government of India, under Project
Identification No. RTI4001.  S.M. thanks MPI, Bonn, for an invitation to visit in 
May--June 2023. 

\subsection{Notation}\label{s:notations}

\subsubsection{Lie theoretic notation}\label{s:lie}

As in the introduction, let $\mf{g}$ be a split finite dimensional
simple Lie algebra over a field $\msk\subseteq \C$ and let
$\mf{h}\subset \mf{g}$ be a Cartan subalgebra. Let
$R=R^+\cup R^{-} \subseteq\mf{h}^*$ be the set of roots, and
$\Delta\subseteq R^+$, the set of simple roots corresponding to a
fixed Cartan decomposition. Let $(\ ,\ )$ be the Killing form on
$\mf{h}^*$ normalised by $(\theta,\theta)=2$, where
$\theta\in \mf{h}^*$ is the highest root. We let $\{X_i\}$, be an
orthonormal basis of $\mf{g}$ with respect to the Killing form. Let
$h^{\vee}$ be the dual Coxeter number of $\mf{g}$; this is half the
scalar by which the normalised Casimir element
$\Omega= \sum_i X_i\tensor X_i\in U(\mf{g})\tensor U(\mf{g})$ acts on
the adjoint representation of $\mf{g}$ by its image under the
multiplication map on $U(\mf{g})$. Note that $h^{\vee}=n$ for
$\mf{g}=\mf{sl}_n$.

Let $P$ be the lattice of integral weights, let $Q \subset P$ be the
root lattice, and let $i_{\mf{g}} = |P/Q|$ be the index of connectivity
of $\mf{g}$. Let $P^+ \subset P$ be the set of dominant integral
weights of $\mf{g}$.  For $\lambda\in P^+$, let $V_{\lambda}$ denote
the corresponding irreducible representation of $\mf{g}$. For
$\lambda\in P^+$, let $c(\lambda)=(\lambda,\lambda+2\rho)$ where
$\rho$ is half the sum of positive roots. 

Let
$$m_{\mf{g}}=\frac{(\alpha_{\ell},\alpha_{\ell})}{(\alpha_s,\alpha_s)}, $$
where $\alpha_{\ell}$ is a long root and $\alpha_{s}$ a short
root. Note that $m_{\mf{g}}\in\{1,2,3\}$ and $m_{\mf{g}}=1$ if
$\mf{g}$ is simply laced. Since $(\alpha_{\ell},\alpha_{\ell})=2$,
$m_{\mf{g}}\cdot(\alpha,\alpha)$ is an even integer for all roots $\alpha$.

\subsubsection{Topological and complex analytic notation}

For any topological space $X$ and a field $K$ we let $\dpl(X,K)$ be
the bounded below derived category of sheaves of $K$-vector spaces on
$X$. For a continous map $f:X \to Y$ we will use the symbols $f_*$,
$f_!$ and $f^*$ to denote the derived functors of the standard
functors on categories of sheaves corresponding to these symbols.  For
example, if $f:X \to Y$ is a continuous map then
$f_*: D^+(X,K) \to D^+(Y,K)$ will always be the derived pushforward.
(We will only use $f_!$ for open or closed embeddings, in which case
the underived functor is also exact but we often use $f_*$ when the
underived functor is not exact.) However, if $\mc{F}$ is a sheaf on
$X$ then $R^if_* \mc{F}$ will have its usual meaning.

For a map $f$ as above we will often refer to the functor $f_*$ as the
``star pushforward'' and $f_!$ as the ``shriek pushforward''; this
terminology will be particularly useful when $f$ is an inclusion which
we do not want to name explicitly. Also, our maps will often have
subscripts, e.g., we have $g_Y$, $k_B$, etc., and for typographical
convenience we will write $g_{Y,*}$ for $(g_Y)_*$, $k_B^!$ for
$(k_B)^!$, etc.

When $X$ is a complex analytic space we let $\dbc(X,K)$ denote the
full subcategory of $\dpl(X,K)$ consisting of complexes of sheaves
with only finitely many nonzero cohomology sheaves which are
constructible with respect to an analytic Whitney stratification. We
will always assume that $X$ is second countable and the irreducible
components of $X$ have bounded dimension.  Almost all the objects of
$\dpl(X,K)$ that we consider will actually lie in $\dbc(X,K)$.

For a subfield $\msk$ of $\C$ and a variety $X$ over
$\msk$, we let $X^{\mr{an}}$ be the associated complex analytic
space and we denote by $\dbc(X,K)$ the subcategory of
$\dbc(X^{\mr{an}},K)$ consisting of objects with algebraically
constructible cohomology.

We also use $\dc$ for the unit disc in $\C$ and
$\dc^* = \dc \setminus \{0\}$ (in contexts where there will be no
confusion with the simple roots).

\subsubsection{Some indexing notation}\label{s:ind}
Sometimes it will be convenient for us to have objects such as
$\A^n_{\msk}$ with coordinates labelled by elements in a finite set
$S$: we will denote this by $\A^S_{\msk}$. Similarly, we have the
configuration space $\mc{C}_S \subset \A^S_{\msk}$, the moduli spaces
of (stable) pointed genus $0$ curves $M_{0,S}$ and $\ov{M}_{0,S}$,
etc. For any finite set $S$, we will denote by $\ov{S}$ the disjoint
union of $S$ and a one point set $\{*\}$. If $A \subset S$ are sets we
let $A^c = S \sm A$.

\subsubsection{Miscellaneous}

Let $\mu_N \subset \C^{\times}$ be the subgroup of $N$-th roots of
unity.  For any positive integer $n$, we set $[n] := \{1,2,\dots,n\}$.

By a filtration of an object $F$ in a triangulated category, we mean a
sequence of distinguished triangles of the form
$F_i\to F_{i+1}\to G_i\leto{[1]}$ for $0\leq i<k$, with $F_k=F$ and
$F_0=0$. In such a situation we will also say that $F$ is a successive
extension of the objects $G_i$, and that $F$ is filtered by the $F_i$.

\section{Preliminaries}\label{s:Prelim}

In Section \ref{s:rnb}, we review some standard facts about nearby
cycles, Verdier duality and tensor products in a complex analytic
setting.  In Section \ref{s:normale} we give a local description of
some key morphisms (for example $f_X:X\to \A^1_{\msk}$ in Theorem \ref{t:fact}),
see Lemma \ref{l:normalform}.

\subsection{Review of nearby cycles}\label{s:rnb}

A reference for the material in this section is \cite{Schurmann}.  We
fix a field $K$ throughout this section and to simplify notation we
write $\db(X)$ for $\db(X,K)$, $\dbc(X)$ for $\dbc(X,K)$, etc.,
throughout this section. All tensor products ($\otimes$ and
$\boxtimes$) are over $K$.

Let $X$ be a complex analytic space and $f:X\to \C$ an analytic
function.  Let $\C^{*}=\C \sm \{0\}$, $X^* = X \times_{\C} {\C}^*$ and
let $\ti{\C}^{*}$ be a universal cover of $\C$. Form the diagram with
cartesian squares:
\begin{equation*}
  \xymatrix{
    \widetilde{X}^*\ar[r]^k\ar[d] & X \ar[d]^f & Y\ar[l]^{\iota}\ar[d]\\
    \widetilde{\C}^*\ar[r] & \C  & \{0\}\ar[l]
  }
\end{equation*}

\begin{defi}\label{d:nearby}
  Let $\mc{F}\in \dpl(X)$.  The nearby cycles for $\mc{F}$ is the
  object $\Psi_f \,\mc{F} := \iota^* k_*k^*\mc{F}$ in $\dpl(Y)$. The
  group $\pi_1(\C^*)$ acts on $\Psi_f\, \mc{F}$ and we denote the action
  of the positive generator by $m_f: \Psi_f\, \mc{F} \to \Psi_f\, \mc{F}$.
  
\end{defi}

By construction, $\Psi_f$ is an exact functor. If $\mc{F} \in \dbc(X)$
then $\Psi_f\, \mc{F} \in \dbc(Y)$ by \cite[Theorem 4.0.2]{Schurmann};
this will always be the case whenever we apply the nearby cycles
functor and we will use this without further mention.  Furthermore,
$\Psi_f$ commutes with Verdier duality up to a shift when restricted
to $\dbc(X)$, i.e.,  there is a natural isomorphism of functors
\[
\Psi_f[-1] \circ \D_X \cong \D_Y \circ \Psi_f[-1] ,
\]
where $\D_X$ and $\D_Y$ denote Verdier duals. See, e.g., \cite[Corollary
3.2]{massey}.

We also note that since the image of the map $k$ is $X^*$, the functor
$\Psi_f$ factors uniquely through a functor $\dpl(X^*) \to \dpl(Y)$
which we also denote by $\Psi_f$.

The next lemma lists some well-known properties of nearby cycles which
we will often use.
\begin{lemma}\label{kunneth}
  Let $f:X \to \C$ be as above and let $\mc{F} \in \dbc(X)$.
  \begin{enumerate}
  \item Let $\mathcal{L}$ be a $K$-local system on $\C^*$ and let $L$ be
    the stalk of $\ml$ at a chosen point of $\C^*$, viewed as a
    representation of $\pi_1(\C^*)$.  Then
    $\Psi_f\,(\mc{F}\tensor_{\C}
    f^*\mathcal{L})\leto{\sim}(\Psi_f\,\mc{F})\tensor L$.
  \item Let $Z$ be another analytic space and $\mc{G}\in \dbc(Z)$.
    Then the nearby cycles $\Psi_g\, (\mathcal{F}\boxtimes \mathcal{G})$
    in $\dbc(Y \times Z)$ for the composite $g:X\times Z\to X\to \C$
    is isomorphic to $(\Psi_f \,\mathcal{F})\boxtimes \mathcal{G}$.
  \item Let $Z$ be another analytic space, $g:Z \to X$ a proper
    analytic map and $\mc{G} \in \dbc(Z)$. Then there is a base change
    isomorphism
    $\Psi_f\, ( g_{*}\mc{G}) \simeq g_{Y,*} (\Psi_{f \circ g}\, \mc{G})$
    where $Y=(f \circ g)^{-1}(0)\subset Z$, and
    $g_Y:Y\to f^{-1}(0)$ is the restriction of $g$ to $Y$.
  \item Suppose $X$ is a complex manifold and $f$ is smooth, i.e., a
    submersion.  If $\mc{F}$ is a local system then $\Psi_f \mc{F}$ is
    isomorphic to $\mc{F}|_Y$ (and the monodromy action is trivial).
  \end{enumerate}
\end{lemma}
All isomorphims in the lemma are compatible with the monodromy action.

\begin{proof}
  Item (1) is a special case of (2) which in turn is a special case of
  \cite[Theorem 1.0.4]{Schurmann}. For (3) and (4) see \cite[Remark
  4.3.7]{Schurmann}.
\end{proof}

\begin{remark}\label{r:bc}
  As a special case of Lemma \ref{kunneth} (3) we get the following:
  Suppose $f:X \to \C$ is proper over an open neighbourhood of
    $0$ and $\mc{F} \in \dbc(X)$. Then given $s \in \C$ sufficiently
  close to $0$ there is an isomorphism
  $${H}^*(Y, \Psi_f \,\mc{F})\cong (f_*\mc{F})_s={H}^*(X_s,\mc{F}|_{X_s}).$$
  Hence $\Psi_f\, \mc{F}$ (which is supported on $Y$) computes the
  cohomology of $\mc{F}$ restricted to the ``nearby'' fibre
  $X_s=f^{-1}(s)$. (Use the fact that all the cohomology sheaves of
  $f_* \mc{F}$ are local systems in a small enough punctured
  neighbourhood of $0 \in \C$.)
\end{remark}

\subsubsection{}
We recall here some basic facts about external tensor products,
Verdier duality, etc., that we will use later in our computations of
nearby cycles.
\begin{lemma}\label{KunnethF}
  Let $g_i:U_i\to X_i$ be continuous maps of locally compact
  topological spaces and let $\mc{F}_i\in \dpl(U_i)$.  Then (with
    $g_{1,!}= (g_1)_!$, etc.)
  \begin{enumerate}
  \item
    $(g_1\times g_2)_!(\mc{F}_1\boxtimes \mc{F}_2) \simeq
    g_{1,!}\mc{F}_1\boxtimes g_{2,!}\mc{F}_2$.
  \end{enumerate}
  Furthermore, if $X_i$ are definable with respect to some
  $o$-minimal structure on $\R^n$, the $g_i$ are definable inclusions,
  and $\mc{F}_i$ are constructible with respect to a definable
  stratification then:
  \begin{enumerate}
  \item[(2)]
    $(g_1\times g_2)_*(\mc{F}_1\boxtimes \mc{F}_2) \simeq
    g_{1,*}\mc{F}_1\boxtimes g_{2,*}\mc{F}_2$.
  \item[(3)]
    $\D_{U_1}\mc{F}_1 \boxtimes \D_{U_2} \mc{F}_2 \simeq \D_{U_1 \times
      U_2} (\mc{F}_1 \boxtimes \mc{F}_2)$.
  \end{enumerate}
\end{lemma}
\begin{proof}
  For (1), which is well-known, see \cite[Equation
  (1.16) on p.~78]{Schurmann}. For the precise meaning of the
  assumptions in (2) and (3) and the proofs of the statements see
  \cite[Corollary 2.0.4]{Schurmann}.
\end{proof}

When we apply this lemma the $g_i$ will be open inclusions and the
stratification will correspond to a normal crossings
divisor. Therefore, we may take the $o$-minimal structure to be the
one corresponding to semi-algebraic sets.

\begin{lemma}\label{l:ext}
  Let $X$ be a complex manifold and $D = \cup_{i=1}^n D_i$ be a
  divisor with simple normal crossings. Let $\mc{L}$ be a rank one
  $K$-local system on $U = X \sm D$. For each $i = 1,\dots,n$, let
  $U_i = X \sm \cup_{j=i}^N D_i$, let $f_i:U_i \to X$ be the
  inclusion, and let $?_i$ be either $!$ or $*$.
  \begin{enumerate}
  \item The formation of
    $\mc{F} := (f_n)_{?_n} (f_{n-1})_{?_{n-1}}\dots (f_1)_{?_1}\mc{L}$
    is canonically independent of the ordering of the
    $D_i$. Equivalently, all the stalks of $\mc{F}$ vanish at each
    point of $D_i$ for all $i$ such that $?_i = \,!$.
  \item Suppose $?_i'$, $1 \leq i \leq n$ is another choice of $!$ or
    $*$ such that if $?_i = \,!$ then $?_i' = \,!$ and let
    $\mc{F}' := (f_n)_{?_n'} (f_{n-1})_{?_{n-1}'}\dots
    (f_1)_{?_1'}\mc{L}$. Then
    $\home(\mc{F}', \mc{F}) = \home(\mc{L},\mc{L})$.
  \end{enumerate}
  
\end{lemma}
\begin{proof}
  Part (1) of the lemma is formulated and proved informally in the
  discussion after Lemma 3.3 of \cite{L1}. For more details the reader
  may consult the proof of Lemma 29 in the arXiv version of
  \cite{BBM}. 

  Part (2) follows by induction on $n$ using the fact that for any map
  $f:U \to X$ of locally compact topological spaces, $f_!$
  (resp.~$f_*$) is the left adjoint (resp.~right adjoint) of $f^*$.
\end{proof}

\begin{remark}\label{r:p}
  All the objects of $\dbc(X)$ occurring in Lemma \ref{l:ext} have
  perverse cohomology only in one degree, i.e., $\dim(X)$, so they are
  shifts of perverse sheaves.
\end{remark} 

\subsection{Normal forms of some mappings}\label{s:normale}

\begin{lemma}\label{l:normalform}
  Let $X\leto{g} W\leto{f} \C$ be maps of complex manifolds, let
  $E\subset X$ be a divisor and let $y\in X$. Assume that
\begin{itemize}
\item[--]  $Z=f^{-1}(0)\subseteq W$ is smooth.
\item[--] Setting $Y=(f \circ g)^{-1}(0)$, we have $Y=\cup_{i=1}^r Y_i$,
  with each $Y_i$ an irreducible divisor and $\cap_{i=1}^r Y_i$ smooth
  over $Z$ .
\item[--] $E\cup Y$ is a divisor with simple normal crossings, and if
  $E_a$, $a=1,\dots,k$, are the components of $E$, then
  $(\cap_{a=1}^k E_a)\cap(\cap_{i=1}^s Y_i)$ is smooth over $Z$.
\item[--] $y\in (\cap_{a=1}^k E_a)\cap(\cap_{i=1}^s Y_i)$.
\end{itemize}
Then we can choose local coordinates on a neighbourhood of $y$ in $X$
making $X$ an open subset of $\C^{r+s+\ell}$ so that $y=\vec{0}$, and
local coordinates on a neighbourhood of $g(y)$ in $W$ making $W$ an
open subset of $\C^{1+s}$ so that $g(y) = \vec{0}$, such that in these
local coordinates
\begin{enumerate} 
\item[(A)] The map $g$ is given by
  $(z_1,\dots,z_r,u_1,\dots,u_s,v_1,\dots,v_{\ell}) \mapsto (z_1\cdots
  z_r,u_1,\dots,u_s)$ and the map $f$ is the projection to the
  first coordinate.
\item[(B)] $E$ is given by $v_1\cdots v_k=0$.

\end{enumerate}
\end{lemma}

\begin{proof}
  First, since $Y$ is a normal crossings divisor, we can choose local
  coordinates on $X$ so that $f \circ g$ is the product $z_1\dots z_r$,
  with the $z_i$ the some of the coordinate functions. Let
  $X_0=\cap_{i=1}^r Y_i\subset X$ be the complex submanifold given by
  the equations $z_i=0$ for all $i$.

  Choose coordinates in a neighbourhood of $g(y)$ in $W$ so that the
  first coordinate is $f$; this is possible since $f^{-1}(0)$ is
  smooth.  Viewing $W$ as a subset of $\C^{1+s}$ via these
  coordinates, we claim that the composite map $X_0\to W\to \C^s$ is
  smooth at $y\in X_0$. This holds since by assumption $X_0$ is smooth over
  $Z$ and $Z$ maps isomorphically onto an open set in $\C^s$.
  
  This shows the existence of coordinates
  with property (A), since the last $s$ coordinates of $W$ restrict to
  coordinates on $X_0$. We now change coordinates on $X$ so that
  property (B) also holds.

  Let $E_a$ be given by the (local) equation $g_a=0$. Let
  $X_{0,0}\subset X_0$ be the submanifold given by the equations
  $z_i=0,u_j=0$ for all $i,j$. We will show that the $g_a$ restrict to
  a subset of a system of coordinates on $X_{0,0}$ near $0$.

  Write $g_a=L(v)_a+L(z)_a+L(u)_a +\text{higher order terms}$, where
  $L(v)_a, L(z)_a$ and $L(w)_a$ are linear in $v,z$ and $u$
  respectively. We claim that $L(v)_a$ are linearly independent for
  $a=1,\dots,k$, which will then imply property (B) (after a change of
  variables, replacing some of the $v_i$'s by $g_a$'s) .

  If a linear combination is zero, then after renumbering, any vector
  in intersection of the tangent spaces of all $Y_a$, and all $E_i$
  for $i>1$ automatically lies in the tangent space of $E_1$. This
  contradicts the assumption that
  $(\cap_{a=1}^k E_a)\cap(\cap_{i=1}^s Y_i)$ is smooth over $Z$. This
  proves the claim, hence also the lemma.

\end{proof}

\section{Aomoto cohomology and KZ motives}\label{s:aomoto} 

Corresponding to an affine space, a collection of hyperplanes and a
log $1$-form on the affine space regular off the hyperplanes, there is
an associated de Rham type complex, called the Aomoto complex and a
local system on the complement of the union of hyperplanes.  We
associate an Aomoto motive (a certain Nori motive) to this data, and
discuss its de Rham and topological realisations. Some basic
properties of these motives are then discussed, including product
behavior. The KZ and conformal block motives are then defined using
the Schechtman--Varchenko arrangement and master functions.

\subsection{Aomoto cohomology and Aomoto motives}
Following \cite[Appendix A]{loo-sl2}, we recall the definition of
Aomoto cohomology, and its topological realisation.

\subsubsection{Aomoto cohomology}\label{s:ac}

Let $\msk$ be a field of characteristic $0$ and let
$W = \A^M_\msk$. Let $\ms{C}$ be a finite collection of affine
hyperplanes $H$ in $W$ given as the zero locus of a degree one
polynomial $f_H$. For each $H \in \ms{C}$ let $a(H) \in \msk$ be
any element. We call the triple $\ms{A} : = (W,\ms{C},a)$ a weighted
hyperplane arrangement. Let $U \subset W$ be the complement of the
union of all $H \in \ms{C}$ and let
$\eta = \sum_{H \in \ms{C}} a(H)\dlog(f_H) \in \Omega^1(U)$.
\begin{defi}\label{d:aomoto}
  The \emph{Aomoto algebra} $A^{\starr}(U)$ of the weighted hyperplane
  arrangment $(W,\ms{C},a)$ is the subalgebra of the de Rham algebra
  $\Omega^{\starr}(U)$ generated by the $\dlog(f_H)$ for
  $H \in \ms{C}$. It is equipped with a differential given by
  $\omega \mapsto \eta \wedge \omega$. We call the pair
  $(A^{\starr}(U), \eta \wedge)$ the \emph{Aomoto complex} of the
  arrangement and define the \emph{Aomoto cohomology} to be the
  cohomology of $(A^{\starr}(U), \eta \wedge)$.
\end{defi}
It is clear that the Aomoto cohomology is a finite dimensional
$\msk$-vector space in each degree and vanishes in degree $> M$. It is
also clear that scaling the weight function $a$ by an element of
$\msk^{\times}$ does not change the Aomoto cohomology.

\subsubsection{Aomoto motives}\label{s:ao2}
Now suppose $\msk$ is a subfield of $\C$.  Let $P$ be a smooth
projective compactification of $U$ with $P \setminus U=\cup_{\alpha} E_{\alpha}$
a divisor with simple normal crossings.  Let
$V=P \setminus\cup'_{\alpha}E_{\alpha}$, where the union is restricted to
$\alpha$ such that $a(\alpha) =\op{Res}_{E_{\alpha}}\eta$ is not a
positive integer. Let $j:U\to V$ and $k:V \to P$ be the inclusions and
let $D = V \sm U$.  Let $\mathcal{L}(\eta)$ be the rank one complex
local system on $U$ corresponding to the $1$-form $\eta$ (the kernel
of $d+\eta \wedge$).

\begin{lemma}\label{l:topo}%
  \cite[Proposition 4.2]{L1} There is a canonical isomorphism
  \begin{equation}\label{e:loo}
    H^*(A^{\starr}(U), \eta\wedge) \otimes_\msk \C \leto{\simeq}
    H^*(V,j_{!}\mathcal{L}(\eta)) .
  \end{equation}
\end{lemma}
We will sometimes use that if we set
$\mc{F} = \mc{F}(\eta) = k_*j_!\mc{L}(\eta)$ then
$H^*(V,j_{!}\mathcal{L}(\eta)) = H^*(P, \mc{F})$. Note that the RHS of
\eqref{e:loo} is canonically independent of the choice of $P$; this
follows from the fact that any two compactifications $P_1$ and $P_2$
as above can be dominated by a common refinement $P$ and then applying
the proper base change theorem.

\smallskip

The lemma implies that Aomoto cohomology has a topological
interpretation. If all $a(H) \in \Q$, then the local system
$\mc{L}(\eta)$ has finite monodromy and the RHS of \eqref{e:loo} can
be viewed as a Nori motive with coefficients in the finite extension
$K$ of $\Q$ over which the monodromy is defined. We explain this in a
relative setting in Section \ref{s:kzmls} using mixed sheaves but
here we give a more elementary definition of the motive. For the
construction of Nori motives and some basic results the reader may
consult \cite[Section 2]{arapura-ms} and \cite[Section 5.3.3]{levine}.

Let $N$ be a positive integer such that $Na(H) \in \Q$ for all
$H \in \ms{C}$ and assume that $\mu_N \subset \msk$. Consider the
closed subvariety of $\A^1_{\msk} \times U$ given by the equation
\begin{equation*}
	y^N = \prod_{H \in \ms{C}}f_H^{Na(H)} .
\end{equation*}
The second projection $\pi_U$ makes $\wh{U}$ a cyclic cover of $U$. We
note that as a variety over $\msk$, $\wh{U}$ depends on the choice
of the function $f:= \prod_{H \in \ms{C}}f_H^{Na(H)}$ which is
well-defined only up to multiplication by an element of
$\msk^{\times}$, but if $\msk$ is closed under taking $N$-th
roots then different choices lead to isomorphic varieties. When we
apply this construction we will have a natural choice for $f$.

Let $\wh{V}$ be a $\mu_N$-equivariant smooth partial compactification
of $\wh{U}$ such that the map $\pi_U$ extends to a proper map
$\pi_V: \wh{V} \to V$. Let $\wh{j}: \wh{U} \to \wh{V}$ be the
inclusion and let $\wh{D} = \wh{V} \sm \wh{U}$ which we may assume is
a divisor with normal crossings.  The local system $\mc{L}(\eta)$ is
the { $\tau$}-isotypical summand of $(\pi_U)_*(\C_U)$, where $\tau$ is
the inverse of the tautological character of $\mu_N$ and so by the
proper base change theorem $j_!(\mc{L}(\eta))$ is the
$\tau$-isotypical summand of $(\pi_V)_*(\wh{j}_!(\C_{\wh{U}}))$ (which
is concentrated in degree $0$).  This implies that the canonical map
$\iota: H^*(V,j_!(\mc{L}(\eta)) \to H^*(\wh{V},\wh{j}_!(\C_{\wh{U}}))$
is injective and its image is the $\tau$-isotypical summand of
$H^*(\wh{V},\wh{j}_!(\C_{\wh{U}}))$. In particular, the cohomology
groups $H^*(V,j_{!}\mathcal{L}(\eta))$ have canonical mixed Hodge
structures with coefficients in $K$. These mixed Hodge structures can
be defined using Saito's theory of mixed Hodge modules \cite{saito} or
Deligne's mixed Hodge structures for the cohomology of a pair
\cite{hodge3}; we note here that the mixed Hodge structures of Deligne
and Saito agree for quasi-projective varieties \cite[p.~328,
Remark]{saito}.

\smallskip

This above discussion suggests the following:
\begin{defi}\label{d:nm}
  The \emph{Aomoto motive} $H^p(\ms{A})$ associated to the Aomoto
  cohomology group $H^p(A^{\starr}(U), \eta \wedge)$ is the
  $\chi$-isotypical summand of the Nori motive
  $H^p(\wh{V}, \wh{D}; \Q(\mu_N))$.
\end{defi}
\begin{remark}\label{r:notsmooth}
  By the usual ``diagonal'' argument and the proper base change
  theorem one sees that the motive is independent of the choice of
  $\wh{V}$. It also follows from the proper base change theorem that
  the smoothness of $\wh{V}$ is not necessary for the construction.
\end{remark}

To further justify this definition, we prove:
\begin{lemma}\label{l:derhamreal}
  The de Rham realisation of the Aomoto motive $H^p(\ms{A})$ is
  naturally isomorphic to the Aomoto cohomology group
  $H^p(A^{\starr}(U), \eta \wedge)$.
\end{lemma}
\begin{proof}
  We will use the construction of the de Rham realisation using the
  $h$-topology as described in \cite[Section 3.2]{huber-ms}; the fact
  that de Rham cohomology of singular varieties can be described in
  terms of the $h$-topology was first proved by \cite{blee}.
	
  Let $\Omega^p_h$ be the sheafification of the presheaf of K\"ahler
  $p$-differentials on the category $\mr{Sch}_{\msk}$ of finite
  type schemes over $\msk$ equipped with the $h$-topology. For any
  such scheme $X$, we let $\Omega^p_{h/X}$ be the retriction of
  $\Omega^p_h$ to the category of schemes over $X$. If $Z$ is a closed
  subscheme of $X$ with inclusion $i$, there is a surjection
  $\Omega^p_{h/X} \to i_*(\Omega_{h/Z})$ and we denote the kernel by
  $\Omega^p_{h/(X,Z)}$. This gives rise to a short exact sequence of
  complexes of sheaves in the $h$-topology on $X$:
  \begin{equation*}%
    0 \lra \Omega^{\starr}_{h/(X,Z)} \lra \Omega^{\starr}_{h/X} \to
    i_*(\Omega^{\starr}_{h/Z}) \lra 0 .
  \end{equation*}
  The de Rham realisation of any Nori motive of the form $H^p(X,Z)$
  with $(X,Z)$ a pair of varieties over $\msk$ is defined to be
  $H^p_h(X,\Omega^{\starr}_{h/(X,Z)})$. If a finite group $G$ acts on
  the pair $(X,Z)$ and $\chi:G \to \mu_N$ is any character
  then we define the de Rham realisation of the $\chi$-isotypical
  component of the Nori motive to be the $\chi$-isotypical component
  of the de Rham realisation.

  \smallskip

  We now construct a map $\tau$ from Aomoto cohomology to the de Rham
  realisation of the Aomoto motive.  Let $\omega$ be any element of
  $A^p(U)$ such that $\eta \wedge \omega = 0$, so $\pi_U^*(\omega)$ is
  a log form on $\wh{U}$. By the definition of $V$ and the function
  $y$, the form $y \pi_U^*(\omega) \in \Gamma(\Omega_{\wh{U}}^p)$ is
  regular on $\wh{V}$. Furthermore, an easy local calculation shows
  that the image of $y \pi_U^*(\omega)$ in $\Gamma(\Omega^p_{\wh{D}})$
  is supported on the singular locus of $\wh{D}$. By the definition of
  $\Omega^p_{h/{\wh{V}}}$, the form $y \pi_U^*(\omega)$ gives an
  element of $\Gamma(\Omega^p_{h/{\wh{V}}}(\wh{V}))$ and the above
  statement about supports together with \cite[Proposition 4.9]{hjo}
  implies that $y \pi_U^*(\omega)$ maps to $0$ in
  $\Gamma(i_*\Omega^p_{h/\wh{D}}(X))$. Thus,
  $y \pi_U^*(\omega) \in
  \Gamma(\Omega^p_{h/(\wh{V},\wh{D})}(\wh{V}))$, hence gives an
  element of $H^p_h(\wh{V},
  \Omega^{\starr}_{h/(\wh{V},\wh{D})})$. Furthermore, $\mu_N$ acts
  trivially on $\pi_U^*(\omega)$ so $y \pi_U^*(\omega)$ lies in the
  $\chi$-isotypical component of
  $H^p_h(\wh{V}, \Omega^{\starr}_{h/(\wh{V},\wh{D})})$ thereby giving
  the map $\tau$.
	
  To show that $\tau$ is an isomorphism we may assume that
  $\msk = \C$. In this case, there is a commutative diagram
  \begin{equation*}
    \xymatrix{
      H^p(A^{\starr}(U), \eta\wedge) \ar[r]^{\simeq}_{\phi_1} \ar[d]_{\tau} &
      H^p(V,j_{!}\mathcal{L}(\eta)) \ar[d]^{\iota}\\
      H^p_h(\wh{V}, \Omega^{\starr}_{h/(\wh{V},\wh{D})})
      \ar[r]^{\simeq}_{\phi_2} & H^p(\wh{V},\wh{j}_!(\C_{\wh{U}}))
    }
  \end{equation*}
  where the map $\phi_1$ is the isomorphism given in Lemma
  \ref{l:topo}, the map $\phi_2$ is the comparison isomorphism of
  algebraic de Rham and Betti cohomology and $\tau$ and $\iota$ are as
  above. By the unnumbered proposition on \cite[p.~10]{L1},
  $j_{!}\mathcal{L}(\eta)$ is represented by the analytic log de Rham
  complex $\Omega^{\starr}_{V, an}(\log D)$ with the differential
  given by $d + \eta \wedge$, where $d$ is the usual de Rham
  differential. The analogous complex
  $\Omega^{\starr}_{\wh{V},an}(\log \wh{D})$ formed using
  $\pi_U^*(\eta)$ represents the object
  $\wh{j}_!(\pi_U^*(\mc{L}(\eta)))$ for the same reason. The function
  $y^{-1}$ is a nonzero section of $\pi_U^*(\mc{L}(\eta))$, so gives
  an isomorphism of $\pi_U^*(\mc{L}(\eta))$ with $\C_{\wh{U}}$, hence
  we may view $\Omega^{\starr}_{\wh{V},an}(\log \wh{D})$ as a
  representative for $\wh{j}_!(\C_{\wh{U}})$.
	
  With these choices of representatives, for $\omega$ a log $p$-form
  on $U$ with $\eta \wedge \omega = 0$, the map $\phi_1$ is given by
  viewing $\omega$ as an element of $\Omega^p_{V, an}(\log D)(V)$ and
  the map $\iota$ is given by $\omega \mapsto \pi_U^*(\omega)$, viewed
  as an element of $H^p(\wh{V},\wh{j}_!(\pi_U^*(\mc{L}(\eta)))$.  The
  comparison isomorphism $\phi_2$ is constructed in \cite[Section
  4.2]{huber-ms} using an analytic analogue of the
  $h$-topology. However, since $\tau(\omega)$ is represented by the
  algebraic differential form $y \pi_U^*(\omega)$ one sees that
  $\phi_2(\tau(\omega)) \in H^p(\wh{V},\wh{j}_!(\C_{\wh{U}}))$ is
  given by simply viewing $y \pi_U^*(\omega)$ as an analytic
  differential form on $\wh{V}$. Since we identify
  $\pi_U^*(\mc{L}(\eta))$ and $\C_{\wh{U}}$ using the function
  $y^{-1}$, it is then clear that the diagram commutes. Since we know
  that $\phi_1$ is an isomorphism, it follows that $\tau$ is an
  injection which identifies the image with the $\chi$-isotypical
  component.
\end{proof}

The lemma below gives a basic constraint on the Hodge numbers of an
Aomoto motive. It can be seen by examples that the bounds in it are optimal.
\begin{lemma}\label{r:eff}
  If $a(H) \in \Q$ for all $H \in \ms{C}$, then the mixed Hodge
  structure $\msf{M}$ associated to the Aomoto motive $H^p(\ms{A})$
  has the property that $h^{r,s}(\msf{M}) = 0$ if $r$ or $s$ is not in
  $[0,p]$.
\end{lemma}
\begin{proof}
  As explained in the paragraph before Definition \ref{d:nm},
  $H^*(V,j_{!}\mathcal{L}(\eta))$ is a direct summand of
  $H^*(\wh{V}, \wh{D}; \C)$ from which it inherits its mixed Hodge
  structure, so the lemma is reduced to the statement that if $X$ is a
  quasi-projective variety and $Y$ is any closed subvariety then
  Deligne's MHS on $H^p(X,Y;\Q)$ has the above property. This
  statement follows from \cite[Th\'eor\`eme 8.2.4]{hodge3} and the
  long exact sequence of the cohomology of a pair \cite[Proposition
  8.3.9]{hodge3}.
\end{proof}

\subsubsection{}\label{s:resol}

In general, Aomoto motives can be complicated as can already be seen
when $M=1$, but there is a case in which they are relatively
simple. To see this one needs to restrict the choice of the normal
crossings compactification $P$ of $U$, so we explain following
\cite{orlik-terao, loo-sl2, varch} an efficient method for constructing such
compactifications which will also be useful later.

We start with a smooth compactification $P'$ of $U$ which is
\emph{locally arrangement-like}. By this we mean that
\begin{itemize}
\item[--] $P' \sm U = \cap_{i \in I}D_i$ with smooth irreducible divisors
  $D_i$ such that  $Z_{I'} := \cup_{i \in I'} D_i$ for all $I' \subset I$ is
  smooth and irreducible.
\item[--] Each point $x \in P'$ has an open neighbourhood $W'$ such that
  $W'\simeq \A^M_{\msk}$ and $W'\cap U$ is isomorphic to the
  complement of a hyperplane arrangement in $\A^M_{\msk}$.
\end{itemize}
For example, we can always take $P'$ to be $\P^M$ or $(\P^1)^M$.  The
strata of $P'$ are defined to be all closed subsets $Z_{I'}$ as above.

We now define what it means for a stratum $Z$ to be \emph{abnormal}
(or $Z$ is called a \emph{dense edge}). This is a local definition so
we may restrict to a hyperplane arrangement in $\A^M_{\msk}$. If
$Z$ is a point, which we may take to be $0 \in \A^M_{\msk}$, then
it is abnormal if there is a subset of $M+1$ hyperplanes
$\{H_{f_i}\}_{i=1}^{M+1}$ in the arrangement passing through $0$ such
that any subset of $M$ elements of $\{f_i\}_{i=1}^{M+1}$ is linearly
independent. If $0 <\dim(Z) = d < M-1$ we say that $Z$ is abnormal if
for a general $(M-d)$-dimensional affine subspace $L \subset \A^M$,
$Z \cap L$ is abnormal for the induced arrangement in $L$ (given by
intersecting all the hyperplanes in the arrangment with $L$).  We then
have:
\begin{theorem}\label{t:abnormal}
  Let $P'$ be a locally arrangement-like compactification of $U$, the
  complement of a hyperplane arrangement in $\A^M_{\msk}$. Let
  $Z_i$, $i = 1,2,\dots, r$, be the sequence of abnormal strata
  ordered so that $\dim(Z_i) \leq \dim(Z_{i+1})$ for
  $i=1,2,\dots,r-1$. Set $P_0 := P'$ and let $P_i$, $1 \leq i \leq r$,
  be the sequence of varieties obtained iteratively by blowing up the
  strict transorm of $Z_i$ in $P_{i-1}$. Then $P := P_r$ is smooth,
  the map $P \to P'$ is an isomorphism over $U$ and $P \sm U$ is a
  divisor with strict normal crossings.
\end{theorem}
\begin{proof}
  See, for example, \cite[Theorem 4.2.4]{orlik-terao}.
\end{proof}

As an easy application of the theorem we have:
\begin{corollary}\label{c:tate}
  If $a(H) \in \Z$ for all $H \in \ms{C}$ then
  $H^p(\ms{A})$ is mixed Tate for all $p$.
\end{corollary}
\begin{proof}
  The assumption implies that $\mc{L}(\eta)$ has no mondromy along all
  $H \in \ms{C}$, so it is a constant local system. Hence we may take
  $\wh{U} = U$ and $\wh{P} = P$ to be any normal crossings
  compactification of $U$. We choose $P$ to be constructed by the
  procedure described in Theorem \ref{t:abnormal}, taking $P' = \P^M$.
  It is clear from the inductive procedure and the formula for the
  cohomology of a smooth blowup that the Nori motives associated to
  $P$ and each closed stratum are Tate.  This implies, by induction on
  the dimension and the long exact sequence for the cohomology of a
  pair, that if $X \subset P$ is any locally closed subvariety which
  is a union of open strata in $P$ and $Y$ is a closed subvariety of
  $X$ which is also such a union, then the Nori motive $H^p(X,Y)$ is
  mixed Tate for all $p$. In particular, $H^p(\ms{A})$ is mixed Tate
  for all $p$.
\end{proof}

\subsubsection{}\label{s:sub}

We continue with the notation from the beginning of Section
\ref{s:ao2}. Let $V' \supset V$ be given by
$V' = P \sm \cup_{\alpha}' E_{\alpha}$ where now the union is
restricted to the set of $\alpha$ such that $a(\alpha)$ is not a
non-negative integer and we let $j':U \to V'$ and $k':V' \to P$ be the
inclusions.  Let $\mc{F}' = k'_* j'_! \mc{L}(\eta)$ so there is a
canonical map $\mc{F}' \to \mc{F}$.

\begin{lemma}\label{imageKZ}
  Consider a hyperplane arrangement $(W, \ms{C},a)$ with open set $U$
  and log form $\eta$ regular on $U$ and let $P$ and
  $\mc{F}' \to \mc{F}$ be as above.  Let $\ms{C}_1 \supset \ms{C}$
  be another finite set of hyperplanes in $W$ and let $a_1(H) = a(H)$
  for $H \in \ms{C}$ and $a_1(H) = 0$ for $H \in \ms{C}_1 \sm \ms{C}$,
  so $U_1 \subset U$ and the log form $\eta_1$ associated to the
  arrangement $(W, \ms{C}_1,a_1)$ equals $\eta$. Let $P_1$ and
  $\mc{F}_1' \to \mc{F}_1$ be the data associated to the new
  arrangement $(W, \ms{C}_1,a_1)$ analogous to the data above.  Then
  the image of $H^*(P_1,\mc{F}_1') \to H^*(P_1,\mc{F}_1)$ is
  canonically a subquotient of the image of
  $H^*(P, \mc{F}') \to H^*(P,\mc{F})$.
\end{lemma}

\begin{proof}
  By induction on $| \ms{C}_1 \sm \ms{C}|$ we may assume that $U_1$ is
  obtained by deleting a single hyperplane from $U$.

  Let $H_1\subset W$ be the hyperplane in $\ms{C}_1 \sm \ms{C}$ and
  let $P$ be a normal crossings compactification of $U$. Then
  $H_1 \cap U$ does not meet any other hyperplane in $\ms{C}_1$ so we
  can construct a normal crossings compactification $P_1$ of $U_1$ by
  blowing up $P$ outside of $U$. Then the map $P_1 \to P$ is an
  isomorphism over $U$, so replacing $P$ by $P_1$ we may assume
  $P = P_1$.

  Let $E_1$ be the strict transform of $H_1$ in $P$. Since
  $\eta = \eta_1$ and the residue of $\eta$ along $E_1$ is zero,
  $\mc{F}_1$ is the star pushforward of $\mc{F}_{P \sm E_1}$.
  Therefore, by adjunction we obtain a canonical map
  $\mc{F}\to \mc{F}_1$, hence a map
  $H^*(P,\mc{F})\to H^*(P,\mc{F}_1)$. To obtain a map
  $H^*(P,\mc{F}_1')\to H^*(P,\mc{F}')$ we note that $\eta = \eta_1$
  also implies that $\mc{F}_1'$ is the shriek pushforward of
  $\mc{F}'|_{P \sm E_1}$.
	
Finally we have to check that the composite
$\mc{F}_1'\to \mc{F}' \to \mc{F}\to \mc{F}_1$ is the natural map
$\mc{F}_1'\to \mc{F}_1$. This holds because both maps agree when
restricted to $P \sm E_1$ and, by adjunction, any such map is determined
by the restriction.
\end{proof}

\begin{remark}\label{r:imageKZ}
  In the setting of Lemma \ref{imageKZ}, if $a(H) \in \Q$ for all
  $H \in \ms{C}$ then the subquotient statement holds in the category
  of Nori motives. To make $H^*(P, \mc{F}')$ into a Nori motive we
  observe that Lemma \ref{l:ext} (1) implies that $\mc{F}'$ is, up to
  a shift, the Verdier dual of $\mc{F}(-\eta)$.
\end{remark}

Let $\ms{A}_i = (W_i,\ms{C}_i,a_i)$, $i=1,2$, be weighted hyperplane
arrangements in $\A^{M_i}_{\msk}$ with weights in $\Q$ and with
associated forms $\eta_i$. Let $\ms{A} = (W,\ms{C},a)$ be the product
arrangement and assume that $Na(H) \in \Z$ for all $H \in
\ms{C}$. This means that $W = \A^{M}_{\msk}$ with $M = M_1 + M_2$ and
the hyperlanes in $\ms{C}$ are of the form
$H_1 \times \A^{M_2}_{\msk}$ for $H_1 \in \ms{C}_1$ with weight
$a_1(H_1)$ and $\A^{M_1}_{\msk} \times H_2$ for $H_2 \in \ms{C}_2$
with weight $a_2(H_2)$.
\begin{lemma}\label{l:kaom}
  $H^*(\ms{A}_1) \otimes_{\Q(\mu_N)} H^*(\ms{A}_2) \simeq H^*(\ms{A}).$
\end{lemma}
\begin{proof}
  Let $P_i$, $V_i$, $\wh{U}_i$, $\wh{V}_i$, $\pi_{U_i}$, $\pi_{V_i}$,
  for $i=1,2$ be the data associated to the two arrangements as in the
  construction of the Nori motive in Section \ref{s:ao2}. We let
  $P = P_1 \times P_2$; this is a normal crossings compactification of
  $U = U_1 \times U_2$, and we let $V = V_1 \times V_2$. The product
  $\mu_N \times \mu_N$ acts on $\wh{U}_1 \times \wh{U}_2$
  (resp.~$\wh{V}_1 \times \wh{V}_2$) and we let $\wh{U}$
  (resp.~$\wh{V}$) be the quotient by the anti-diagonal subgroup
  $G$. One sees easily that $H^*(\ms{A})$ is given by the
  $\chi$-isotypical component of
  $H^*(\wh{V}, \wh{V} \sm \wh{U}; \Q(\mu_N))$, where we identify
  $\mu_N \times \mu_N/G$ with $\mu_N$ using the product of the two
  projections. Since $H^*(\wh{V}, \wh{V} \sm \wh{U}; \Q(\mu_N))$ is
  the motive of $G$-invariants of
  $H^*(\wh{V}_1 \times \wh{V}_2, \wh{V}_1 \times \wh{V}_2 \sm \wh{U}_1
  \times \wh{U}_2; \Q(\mu_N))$, it follows that $H^*(\ms{A})$ is also
  given by the $(\chi \boxtimes_{\Q(\mu_N)}\chi)$-isotypical component of
  $H^*(\wh{V}_1 \times \wh{V}_2, \wh{U}_1 \times \wh{U}_2;
  \Q(\mu_N)$. The lemma then follows from the Kunneth formula.
\end{proof}

\subsubsection{}
Finally, we have the definition of basic one dimensional Aomoto motives:
\begin{defi}\label{d:basicmotives}
  Let $a$ be a positive rational number and $\msk$ a subfield of
  $\C$. Consider the hyperplane arrangement $\ms{A}$ with hyperplanes
  $\{0\}$ and $\{1\}$ in $W=\A^1_{\msk}$, with weights $a/\kappa$ and
  $1/\kappa$ at $\{0\}$, $\{1\}$ respectively. The corresponding
  Aomoto motive $H^1(\ms{A})$ is denoted by $[a;\kappa]$.
\end{defi}

\subsection{KZ motives and CB motives}\label{s:KZCB}

\subsubsection{The Schechtman--Varchenko arrangement and master
  function}\label{s:sv}
Fix a split simple Lie algebra $\mf{g}$, let
$\vv{\lambda} \in (P^+)^n$, $\nu \in P^+$, and assume
 that $\sum\lambda_i-\nu$ is a non-negative integral sum of simple roots (possibly
repeated). This assumption holds if $\mb{A}(\vec{\lambda},\nu^*)\neq 0$ since the weights in
$V(\vec{\lambda})$ lie in $\sum \lambda_i+\mb{Z}_{\leq 0}\Delta$.
We then write
\begin{equation*}%
  \sum\lambda_i-\nu=\sum_{\alpha\in\Delta}n_{\alpha} \alpha  =
  \sum_{a=1}^M\beta(a) .
\end{equation*}
Here $M=\sum_{\alpha\in \Delta}n_{\alpha}$, and we have made a choice
of $\beta:[M]=\{1,\dots,M\} \to \Delta$ so that the second equality
holds.

Fix $\vec{z}=(z_1,\dots,z_n)$ in $\mc{C}_n(\msk)$ and let
$W=\A^M_{\msk}$. The coordinate variables of $W$ will be denoted by
$t_1,\dots,t_M$. We will consider the variable $t_b$ to be colored by
the simple root $\beta(b)$. Let $\kappa \in \msk^{\times}$.

Consider the weighted hyperplane arrangement
$\ms{A} = (W,\mathscr{C},a)$ given by the following collection
$\mathscr{C}$ of hyperplanes, and their attached weights $a$:
\begin{enumerate}
\item For $i\in [1,n]$ and $b\in [1,M]$, the hyperplane $t_b-z_i=0$, with weight $\displaystyle \frac{(\lambda_i,\beta(b))}{\kappa}$.
\item $b,c\in [1,M]$ with $b<c$,  the hyperplane $t_b-t_c=0$, with weight $\displaystyle-\frac{(\beta(b),\beta(c))}{\kappa}$.
\end{enumerate}
Let $U$ be the complement of the above hyperplane arrangement in $W$.
Consider the (multi-valued) master function
\begin{equation} \label{e:master} %
  \ms{R} = \prod_{1 \leq i < j \leq n} (z_i - z_j)^{-\frac{(\lambda_i,
      \lambda_j)}{\kappa}} \prod_{b=1}^M \prod_{j=1}^n (t_b -
  z_j)^{\frac{(\lambda_j, \beta(b))}{\kappa}} \prod_{1 \leq b < c \leq
    M} (t_b - t_c)^{-\frac{(\beta(b), \beta(c))}{\kappa}}
\end{equation}
We get a form in $\Omega^1(U)$ corresponding to the weighted
hyperplane arrangement $(W ,\mathscr{C},a)$ given by
$\eta= d\ms{R}/\ms{R}$.  The differentials $d(z_i-z_j)$ are zero
because $z_j$ are fixed, but we will consider this multivalued
function later in contexts where the $z$'s are allowed to vary. We
have the formula:
\begin{equation}\label{deta}
  \eta=\frac{1}{\kappa}\biggl(\sum_{b=1}^M\sum_{j=1}^n (\lambda_j,\beta(b))\frac{d(t_b-z_j)}{t_b-z_j} - \sum_{1 \leq b < c \leq M}(\beta(b),\beta(c))\frac{d(t_b-t_c)}{t_b-t_c}\biggr) .
\end{equation}

Let $\Sigma = \Sigma_{\beta}$ be the subgroup of the symmetric group
$S_M$ on $[M]=\{1,\dots,M\}$ formed by permutations $\sigma$ that
preserve colour, i.e. $\beta(\sigma(b))=\beta(b)$ for all $b\in[M]$;
$\Sigma_{\beta}$ acts on $U$, the complex
$(A^{\starr}(U), \eta \wedge)$ as well as the local system $\ml(\eta)$
(over the action on $U$).

If $\kappa \in \Q^{\times}$ there is a positive integer $N$ such that
$\ms{R}^N$ is single valued and preserved by $\Sigma_{\beta}$. A
specific choice of $N$ is given in the following:
\begin{lemma}\label{formula:N}
  Let $\kappa=r/s$ where $r$ and $s$ are coprime integers.  Let
  $N= i_{\mf{g}}m_{\mf{g}}|r|$. Then, $\ms{R}^N$ is single valued and
  is preserved by $\Sigma_{\beta}$.
\end{lemma}
\begin{proof}
  The lemma follows from:
\begin{enumerate}
\item[(1)] If $\alpha_a$ is a simple root, and $\lambda$ an integral
  weight, then
  $\lambda(\alpha^{\vee}_a)=
  2(\lambda,\alpha_a)/(\alpha_a,\alpha_a)\in \Bbb{Z}$, hence
  $(\lambda,\alpha_a)=\lambda(\alpha_a^{\vee})(\alpha_a,\alpha_a)/2$
  is equal to $\lambda(\alpha^{\vee}_a)\in \mb{Z}$ if $\alpha_a$ is a
  long root and equals
  $\frac{1}{m_{\mf{g}}}\lambda(\alpha^{\vee}_a)\in
  \frac{1}{m_{\mf{g}}}\mb{Z}$ if $\alpha_a$ is a short root.
\item[(2)] The weight $i_{\mf{g}}\lambda$ is in the root lattice if
  $\lambda$ is integral, hence if $\lambda_i,\lambda_j$ are
  integral, then by (1),
  $(\lambda_i, i_{\mf{g}}\lambda_j)=i_{\mf{g}}(\lambda_i,\lambda_j)
  \in \frac{1}{m_{\mf{g}}}\mb{Z}$.
	\item[(3)] If $\alpha_a$ and $\alpha_b$ are simple roots, then taking $\lambda=\alpha_b$ in (1), we see that
	$(\alpha_b,\alpha_a)  \in  \frac{1}{m_{\mf{g}}}\mb{Z}$.
      \item[(4)] If $\alpha_a$ is a root, then
        $m_{\mf{g}}(\alpha_a,\alpha_a)$ is an \emph{even} integer.
\end{enumerate}
Considering the product representation of $\ms{R}^N$ corresponding to
taking the $N$-th powers of each of the three terms in
\eqref{e:master}, (2) implies that the first term is single valued,
(1) implies that the second term is single valued and (3) implies that
the third term is single valued. The invariance under $\Sigma_{\beta}$
is implied by (4).
\end{proof}

\subsubsection{}\label{s:top}

Let $P$ be any smooth projective $\Sigma_{\beta}$-equivariant
compactification of $U$, with $P \setminus U=\cup_{\alpha} E_{\alpha}$
a divisor with simple normal crossings.  Let
$V=P \setminus \cup'_{\alpha}E_{\alpha}$, where the union is restricted to
$\alpha$ such that $a_{\alpha}=\op{Res}_{E_{\alpha}}\eta$, the residue
of $\eta$ along $E_{\alpha}$ is not a strictly positive integer. Let
$j:U\to V$ be the inclusion.

As in Section \ref{s:sub} let
$V'=P \setminus \cup'_{\alpha}E_{\alpha}$, where the union is
restricted to $\alpha$ such that
$a_{\alpha}=\op{Res}_{E_{\alpha}}\eta$, is not a non-negative integer
and let $q:U\to V'$ be the inclusion.  Let $\chi$ be the sign
character on $S_M$ restricted to $\Sigma_{\beta}$.  Extending results
of Looijenga \cite{L2} (completed in \cite{BBM}, see also \cite{BF})
for $\nu=0$, we have the following result, whose proof for arbitrary
dominant integral $\nu$ is given in Appendix \ref{s:promises}: 

\begin{proposition} \label{p:kz} Suppose $\msf{k} = \C$. Then
  $\mb{A}(\vv{\lambda},\nu^*)^*$ equals (compatibly with the KZ
  connection) the image of the map
  \begin{equation}\label{Sho}
    H^M(V',q_{!}\mathcal{L}(\eta))^{\chi}\to H^{M}(V,j_{!}\mathcal{L}(\eta))^{\chi}.
  \end{equation}
\end{proposition}

\begin{remark}\label{choiceV}
  Note that there are other choices for $V$ and $V'$ which give the
  same result: For $V$ we additionally remove all divisors
  $E_{\alpha}$ with $\alpha\leq 0$ and not an integer. See Section
  \ref{s:trp} for a choice that remains invariant under scaling of
  $\kappa$.
\end{remark}

The following is \cite[Proposition 4.17]{BF}:
\begin{proposition}\label{FSVCor}
	Suppose $\msf{k} = \C$ and $\kappa=\ell+h^{\vee}$, with $\ell>0$ a positive integer. 
	The image of $H^M_c(U,\ml(\eta))^{\chi}\to H^M(U,\ml(\eta))^{\chi}$ can be identified with $\mb{V}^*_{\mathfrak{g},\vec{\lambda},\nu^*,\ell}(\mb{P}^1,\vec{z};\infty)$, the space of  conformal blocks.
\end{proposition}

\subsubsection{}\label{s:kzdef}

When $\kappa \in \Q^{\times}$, we define KZ and CB Nori motives by
replacing the Betti cohomology groups occurring in Propositions
\ref{p:kz} and \ref{FSVCor} by the corresponding Nori motives as in
Definition \ref{d:nm} (see also Remark \ref{r:imageKZ}). We will do
this formally in families using the formalism of mixed sheaves from
Section \ref{s:mixed}, see Remark \ref{r:pts}. The ``multiplicity
motives'' $M_{\kappa}(\lambda_1,\lambda_2,\nu^*)$ of the introduction
are the KZ motives associated to the data $\kappa$,
$\vec{\lambda} = (\lambda_1,\lambda_2)$, and $\nu^*$ with $z_0 =0$ and
$z_1 = 1$.

\section{Mixed sheaves and mixed KZ/CB local systems}\label{s:mixte}

In this section we enhance the KZ and conformal block local systems to
a motivic level using the formalism of mixed sheaves with
coefficients. The notions of nearby cycles and local systems coming
from roots of a function are extended to the level of mixed sheaves
with coefficients. We also recall a basic fact about the Hodge
filtration on nearby cycles.
  
\subsection{Mixed sheaves after Saito}\label{s:mixed}

Let $X$ be an algebraic variety over a field $\msk \subset \C$ and let
$\msp(X)$ be the abelian category of perverse sheaves on $X$ with
$\Q$-coefficients\footnote{We will use $\msp(X,K)$ to denote perverse
  sheaves with coefficients in a field extension $K/\Q$.}; this is the
heart of the perverse $t$-structure on $\dbc(X, \Q)$ introduced by
{ Beilinson--Bernstein--Deligne--Gabber} in \cite{bbd}. We recall
that by a theorem of Beilinson \cite{bei}, the natural functor
$\db(\msp(X)) \to \dbc(X,\Q)$ induced by the embedding of $\msp(X)$
in $\dbc(X,\Q)$ is an equivalence of triangulated categories.

\subsubsection{}\label{s:saito}
A theory of mixed sheaves \cite[Section 1]{saito-form} consists of
$\Q$-linear abelian categories $\msm(X)$ together with exact faithful
functors $\rat_X:\msm(X) \to \msp(X)$,\footnote{Saito uses
  ``$\mr{For}$'' instead of ``$\rat$'', which we use following
  \cite{im}.} which we call the Betti realisation functors, satisfying
a long list of axioms. We not list all of them here, but recall the
axioms and results which will be relevant for us.

Denote the functors $\db(\msm(X)) \to \db(\msp(X))$ induced by the
exact functors $\rat_X$ also by $\rat_X$. 

\begin{enumerate}
\item For $f:X \to Y$ a morphism of varieties over $\msk$ we have
  functors: $f^*, f^!: \db(\msm(Y)) \to \db(\msm(X))$ such that
  $f^* \circ \rat_Y = \rat_X \circ f^*$ and the same for $f^!$ and
  functors $f_*, f_!: \db(\msm(X) \to \db(\msm(Y))$ such that
  $f_* \circ \rat_X = \rat_Y \circ f_*$ and the same for $f_!$.  These
  functors satisfy the usual adjointness properties of the six-functor
  formalism.
\item For all $X$, there is a contravariant duality functor
  $\D_X: \msm(X) \to \msm(X)$ commuting with Verdier duality on
  $\msp(X)$ via $\rat_X$. Its derived functor is denoted by the same
  symbol and satisfies the usual compatibility properties with the
  functors in the previous item.
\item For varieties $X$ and $Y$ there are exact bifunctors
\[
  \boxtimes: \msm(X) \times \msm(Y) \to \msm(X \times Y)
\]
which also commute via the $\rat$ functors with the corresponding
functors on perverse sheaves. We will denote the derived functors of
these functors by the same symbol. We define the tensor product
$\mc{F} \otimes \mc{G}$ of two objects $\mc{F}$ and $\mc{G}$ in
$\db(\msm(X))$ to be $\delta^*(\mc{F} \boxtimes \mc{G})$ where
$\delta:X \to X \times X$ is the diagonal map.
\item Each object $\mc{F} \in \msm(X)$ has a finite increasing
  filtration $W$, called the \emph{weight filtration} which is
  strictly compatible with any morphism in $\msm(X)$ and the
  associated graded is a semisimple object of $\msm(X)$.
\item There is an object $\Q^{\msm}$ of $\msm(\spec \msk)$ such that
  $\rat(\Q^{\msm}) = \Q$  and $W_{-1}\Q^{\msm} = 0$, $W_0(\Q^{\msm}) =
  \Q^{\msm}$.
\item There exists an object $\Q^{\msm}(-1)$ in $\msm(\spec \msk)$
  called the Tate object. We let $\Q^{\msm}(1) = \D
  (\Q^{\msm}(-1))$. For any $\mc{F} \in \msm(X)$ we let
  $\mc{F}(n) = \mc{F} \otimes \Q^{\msm}(1)^{\otimes n}$ if $n\geq 0$
  and $\mc{F}(n) = \mc{F} \otimes \Q^{\msm}(-1)^{\otimes (-n)}$ if
  $n<0$.
\item For any morphism $f:X \to \A^1_{\msk}$ and $Y = f^{-1}(0)$ there
  is a nearby cycles functor $\psi_f: \msm(X) \to \msm(Y)$ which also
  commutes via the $\rat$ functors with the corresponding functor on
  perverse sheaves. There is a morphism of functors
  $\mr{N}: \psi_f \to \psi_f(-1)$ which after applying $\rat$ gives
  the logarithm of the unipotent part of the topological monodromy. At
  the level of topology, the functor $\psi_f$ corresponds to
  $\Psi_f[-1]$, where $\Psi_f$ is as in Definition \ref{d:nearby} and
  we will also use $\Psi_F$ in the context of mixed sheaves to denote
  $\psi_f[-1]$.
\end{enumerate}

\begin{example}\label{e:M}
  $ $
  \begin{enumerate}
  \item For the construction of the rings $\mf{R}_{\kappa}(\mf{g})$
    and $\mf{F}_{\kappa}(\mf{g})$ the key example of a theory of mixed
    sheaves is Saito's category $\mr{MHM}(X_{\C})$ of mixed Hodge
    modules \cite{saito}.
  \item The optimal category for us is the category of perverse Nori
    motives of Ivorra and Morel \cite{im}; when the base is a point
    this is the category of Nori motives. The external tensor product
    is not constructed in \cite{im}, but this has been done very
    recently in \cite{terenzi}. See also \cite{tubach} for the
    construction of functorial Hodge realisations for this category.
  \item We can also consider any of the categories of arithmetic mixed
    sheaves constructed by Saito in \cite{saito-arith}.
  \end{enumerate}
\end{example}

\smallskip

A basic fact that we shall need is that the analogue of Lemma
\ref{kunneth} (3) holds in the setting of mixed sheaves
(\cite[Proposition 5.10]{saito-form}).
\begin{lemma}\label{l:pbc}
  Let $f: W \to \A^1_{\msk}$ be any morphism and $g:X \to W$
  a proper morphism. There is a natural equivalence of functors
  $\psi_f \circ g_* \simeq g_{Y,*} \circ \psi_{f \circ g}$.  Here
  $Y=(f\circ g)^{-1}(0)\subseteq X$ and $g_Y:Y\to f^{-1}(0) \subset W$
  the restriction of $g$.
  
\end{lemma}

\subsubsection{Cohomology functors}

Let $X$ be a smooth irreducible variety. We have the usual functors
for derived categories given by:
${}^{\msm}H^k:\db(\msm(X))\to \msm(X)$. We define
$H^k:\db(\msm(X))\to \msm(X)$, by
$$H^k(\mathcal{K})= {}^{\msm}H^{k+\dim X}(\mathcal{K})[-\dim X] .$$
We also have functors ${}^{\msp}H^k:\db(\msp(X))\to \msp(X)$
compatible via $\rat$ with ${}^{\msm}H^k$.

Consider a local system $\ml$ placed in degree $-\dim X$, i.e.,
$\ml[\dim X]$ as an element of $\dbc(X,\Q)$. This is perverse, and
hence as an object in $\db(\msp(X))=\dbc(X,\Q)$ lives in perverse
degree $0$. Therefore a local system $\ml$ placed in topological
degree $j$ in $\dbc(X,\Q)$ is an object in $\db(\msp(X))$ in perverse
degree $j+\dim X$.

 \begin{lemma}
   Let $\mathcal{K}$ be any object in $\dbc(X,\Q)= \db(\msp(X))$ whose
   cohomology in any degree is a local system. Then
   ${}^{\msp}H^{k+\dim X}(\mathcal{K})[-\dim X]$ is a local system
   placed in degree $0$ (as an element of $\dbc(X,\Q)$), is naturally
   isomorphic to the degree $k$ topological cohomology of 
   $\mathcal{K}$.
\end{lemma}
\begin{proof}
Using translations, we may assume $i=0$.
We first verify the given statement if $\mathcal{L}$ is a local system in degree $j$, then the degree $0$ topological  unless $j=0$, and the right-hand side is zero unless $\dim X= j+\dim X$ and then ${}^{\msp}{H}^{\dim X}(\mathcal{K})$ is an object in topological degree $-\dim X$, as an element of   $\dbc(X,\Q)$, and the statement follows. 

We now prove the general statement. First, if $\mathcal{K}$ is
concentrated in degrees $\leq 0$, then
${}^{\msp}{H}^{j+\dim X}(\mathcal{K})[-\dim X]=0$ for $i> 0$, and
coincides with $H^0(\mathcal{K})$ if $i=0$. This follows from
induction and distinguished triangles
$\tau_{<0}(\mathcal{K})\to \mathcal{K}\to H^0(\mathcal{K})\leto{[1]}$.

Similarly we check that if $\mathcal{K}$ is concentrated in degrees
$> 0$ then ${}^{\msp}{H}^{k+\dim X}(\mathcal{K})[-\dim X] =0$ for
$k\leq 0$. Now write $\mathcal{K}$ as a middle term of a triangle in
$\dbc(X,\Q)$ with one side in degrees degree $\leq 0$ and the other side
$>0$ to obtain the desired isomorphism. The functorial nature of the
isomorphism is because the above assertions also hold with
${}^{\msp}{H}^{k+\dim X}(\mathcal{K})[-\dim X]$ replaced by
topological cohomology in degree $k$.
\end{proof}

\subsubsection{Mixed sheaves with coefficients}

Let $\ms{A}$ be a $\Q$-linear additive category. For $K/\Q$ a finite
extension we define the category $\ms{A}_K$ to be the category of
pairs $(A,\iota_A)$ where $A$ is an object of $\ms{A}$ and
$\iota_A:K \to \End(A)$ is a ring homomorphism.  A morphism $\phi$ in
$\ms{A}_K$ from $(A,\iota_A)$ to $(B,\iota_B)$ is a morphism
$\phi:A \to B$ in $\ms{A}$ such that $\phi \circ \iota_A = \iota_B$.
There is a forgetful functor $\mr{For}: \ms{A}_K \to \ms{A}$ given by
forgetting $\iota_A$ and a base change functor
$\otimes K: \ms{A} \to \ms{A}_K$, where for $A \in \ms{A}$ the
$K$-action on $A \otimes_{\Q} K$ is the natural action on the
right\footnote{For the definition of the object $A \otimes K$, see,
  e.g., \cite[p.~321]{deligne-v}.}. It is easy to see that
$-\otimes_{\Q} K$ is the left adjoint of $\mr{For}$.  If $\ms{A}$ is an
abelian category then $\ms{A}_K$ is also abelian, but we do not know
whether the corresponding result is true for triangulated categories.

The basic objects under consideration in this article will be objects
of $\msm(X)_K$ for $K$ a suitable cyclotomic field and $\msm$ a theory
of mixed sheaves as above.  The functors $f^*,f^!, f_*, f_!$ extend
automatically to functors in this category and so do the functors $\D$
and $\psi_f$.  We also have a constant object $K^{\msm}$ and the Tate
object $K^{\msm}(-1)$. To define external tensor products over $K$,
which we will denote by $\boxtimes_K$, we proceed as follows.

  First, we note that since $K/\Q$ is separable, the multiplication
  map $K \otimes_{\Q} K \to K$ gives rise to a canonical idempotent
  $e$ in $K \otimes_{\Q}K$ such that $e\cdot (K \otimes_{\Q} K)$ maps
  isomorphically onto $K$. We will call this the ``diagonal
  idempotent''. We then have:
\begin{definition}\label{d:ext}
  Let $\mc{F}$ be an object of $\db(\msm(X))_K$ and $\mc{G}$ of
  $ \db(\msm(Y))_K$. Then $\mc{F} \boxtimes \mc{G}$ in
  $ \db(\msm(X \times Y))_K$ has an action of $K \otimes_{\Q}K$ and
  we let $\mc{F}\; \boxtimes_K \mc{G}$ in $ \db(\msm(X \times Y))_K$ be
  the summand corresponding to the diagonal idempotent.
\end{definition}
In the above definition, and the lemma below, we use the fact that the
bounded derived category of an abelian category is idempotent complete
(\cite[Corollary 2.10]{bs}).

\begin{lemma}\label{l:perv}
  There are natural equivalences of categories
  $\msp(X,K) \simeq \msp(X)_K$, $\dbc(X,K) \simeq\dbc(X,\Q)_K$, and
  for any $\Q$-linear abelian category $\ms{A}$,
  $\db(\ms{A}_K) \simeq \db(\ms{A})_K$.
\end{lemma}
\begin{proof}
  There is a canonical faithful exact functor
  $F: \msp(X,K) \to \msp(X)_K$ induced by the functor
  $\dbc(X,K) \to \dbc(X,\Q)_K$ which forgets the $K$-structure on the
  sheaves in a representing complex but, remembers it on the maps in
  the derived category.  To construct an inverse, given an object
  $\mc{F} \in \msp(X)$ we can consider the object
  $\mc{F}' \in \msp(X,K)$ induced by the functor
  $\dbc(X,\Q) \to \dbc(X,K)$ which tensors each sheaf in a
  representing complex with $K$. If $\mc{F}$ is in $\msp(X)_K$ then
  $\mc{F}'$ has two commuting actions of $K$. The first action is just
  the natural $K$-action on any object of $\msp(X,K)$ induced from the
  $K$-action on all objects of $\dbc(X,K)$. The second action comes
  from the $K$-structure on $\mc{F}$.  The fact that the two actions
  commute means that we have a ring homomorphism
  $K \otimes_{\Q} K \to \mr{End}_{\msp({X,K})}(\mc{F}')$. We let
  $G(\mc{F})$ be the summand of $\mc{F}'$ corresponding to the
  diagonal idempotent, so we get a functor
  $G: \msp(X)_K \to \msp(X,K)$. It is easy to see that $G \circ F
  \simeq \mr{Id}_{\msp(X,K)}$.

  To check that $F \circ G \simeq \mr{Id}_{\msp(X)_K}$ we use the Yoneda
  lemma. For $\mc{F}_1, \mc{F}_2$ objects of $\msp(X)_K$, there is a
  functorial isomorphism
  \[
    \home_{\msp(X)_K}(\mc{F}_1,F(\mc{F}_2')) \simeq
    \home_{\msp(X)_K}(\mc{F}_1,\mc{F}_2) \otimes_{\Q} K ,
  \]
  where we use the $K$-action on $F(\mc{F}_2')$ induced by the
  $K$-action on $\mc{F}_2$. Under this isomorphism, the $K$-action on
  the LHS induced by the $K$-action on $\mc{F}_2'$ where we forget the
  $K$-structure on $\mc{F}_2$ (i.e., the ``first'' action above)
  corresponds to the $K$ action on the RHS given by multiplying by $K$
  on the right in the tensor product.

   Now suppose $\mc{G} \in \msp(X)$ is any object with two
    commuting actions of $K$, which we call the $A$ and $B$
    actions. Then $\home_{\msp(X)_K}(\mc{F}_1,\mc{G})$, where we use
    the $A$ action on $\mc{G}$, has a commuting action of $K$ coming
    from the $B$ action on $\mc{G}$, so
    $\home_{\msp(X)_K}(\mc{F}_1,\mc{G})$ becomes a
    $K \otimes_{\Q} K$-module. From the way this module structure is
    defined, it is clear that
    \[
      e\cdot \home_{\msp(X)_K}(\mc{F}_1,\mc{G}) =
      \home_{\msp(X)_K}(\mc{F}_1,e \cdot\mc{G}) ,
    \]
    where $e \cdot\mc{G}$ denotes the summand of $\mc{G}$
    corresponding to $e$. Applying this to $\mc{G} = F(\mc{F}_2')$, 
    it follows from the discussion of the previous paragraph that
\[
  \home_{\msp(X)_K}(\mc{F}_1, F (G(\mc{F}_2))) \simeq
\home_{\msp(X)_K}(\mc{F}_1,\mc{F}_2) \otimes_K K =
\home_{\msp(X)_K}(\mc{F}_1,\mc{F}_2) .
\]

The second statement follows from the third, since
$\dbc(X,K) \simeq \db(\msp(X,K))$ .

The argument for proving the third part is similar to that for the
first: The functor $\mr{For}: \ms{A}_K \to \ms{A}$ is exact, so
induces an additive functor $\db(\ms{A}_K) \to \db(\ms{A})_K$. To
construct an inverse we use the functor
$\db(\ms{A}) \to \db(\ms{A}_K)$ induced by the exact functor
$-\otimes_{\Q} K: \ms{A} \to \ms{A}_K$ and use the diagonal idempotent
as above.
\end{proof}

\subsubsection{}
We will need a version of the classical K\"unneth isomorphism in the
setting of $K$-mixed sheaves.
\begin{lemma}\label{l:kf}
  Let $X_i,Y_i$, $i=1,2$, be varieties over a base variety $S$ and let
  $f_i:X_i \to Y_i$, $i=1,2$ be morphisms of varieties over $S$.
  Let $\mc{F}_i$, $i=1,2$ be objects of $ \db(\msm(X_i))$. Then
  there is a natural isomorphism
  \[
    (f_1)_! \mc{F}_1 \boxtimes_{S,K} (f_2)_! \mc{F}_2 \leto{\simeq} (f_1 \times_S
    f_2)_!(p_1^*(\mc{F}_1) \boxtimes_{S,K} p_2^*(\mc{F}_2)) .
  \]
  in $ \db(\msm(Y_1\times_S Y_2)_K)$.
\end{lemma}
\begin{proof}
  The map is constructed using formal properties of the functors
  involved as in the case of the classical K\"unneth isomorphism (see,
  e.g., \cite[(1.16)]{Schurmann}). The fact that it is an isomorphism
  follows from the fact that it is an isomorphism after applying the
  $\rat$ functor: we first prove the formula with $\boxtimes_S$
  instead of $\boxtimes_{S,K}$ and then deduce the formula as stated
  by applying the idempotent $e$ used to define $\boxtimes_{S,K}$ on
  both sides.
\end{proof}

\subsubsection{}
To simplify notation we will use $\msm(\msk)$ to mean
$\msm(\spec \msk))$. For the first example in Example \ref{e:M},
$\msm(\msk)$ is the category $\mr{MHS}$ of polarisable mixed Hodge
structures. For the second example, it is the category of Nori motives
over $\msk$. It follows that in both cases the functor $\rat$ on
$\msm(\msk)$ factors through a faithful exact functor to $\mr{MHS}$
which we call the Hodge realisation.

For $K$ a finite extension of $\Q$,  we call $\mr{HS}_K$ (resp.
$\mr{MHS}_K$) the category of pure (resp.~mixed) Hodge structures with
coefficients in $K$, or $K$-Hodge (resp.~$K$-mixed Hodge)
structures. For both choices of $\msm$ above, we have a faithful
exact functor from $ \msm(\msk)_K$ to $\mr{MHS}_K$.

Let $\msf{H}$ be a $K$-Hodge structure of weight $n$.  If $H$
is the underlying $\Q$-vector space of $\msf{H}$, then we have a
decomposition $H\otimes_{\Q}\C = \oplus_{p+q = n} H^{p,q}$ with
$H^{p,q}$ a $\C$-vector space and $\overline{H^{p,q}} = H^{q,p}$; here
the conjugation is with respect to the real structure of
$H \otimes_{\Q}\C$ induced by its $\Q$-structure. If $\iota:K \to \C$
is any embedding, we may consider the space
$H^{p,q}(\iota) :=H^{p,q} \otimes_{K\otimes_{\Q}\C} \C$, where the
$K$-module structure of $H^{p,q}$ is that coming from the action of
$K$ via endomorphisms of $\msf{H}$ and the map $\iota$ is used to
define the map $K \otimes_\Q \C \to \C$. We have
$H^{p,q} = \oplus_{\iota:K \to \C} H^{p,q}(\iota)$ and we let
$h^{p,q}(\iota) = \dim_{\C}(H^{p,q}(\iota))$.

If $\msf{M}$ is a $K$-mixed Hodge structure with
underlying $\Q$-vector space $M$, the Hodge filtration $F^{\bullet}$
on $M\otimes \C$ is preserved by the action of $K$. We thus have
$ F^p(M \otimes_{\Q} \C) = \oplus_{\iota:K \to \C}F^p(M\otimes_{\Q}\C)(\iota)$
and we set $f^p(\iota) := \dim_{\C}F^p(M\otimes_{\Q}\C)(\iota)$.
\begin{definition}\label{d:hn}
  $ $
  \begin{enumerate}
  \item If $(\msf{H}, \iota_{\msf{H}})$ is an object of $ HS_K$ and
    $\iota:K \to \C$ is an embedding then we call the numbers
    $h^{p,q}(\iota)$ the Hodge numbers of $(\msf{H},\iota_A)$ with
    respect to $\iota$. We call the polynomial
    $\sum_{p,q} h^{p,q,}(\iota)u^pv^q \in \Z[u,v]$ the $(u,v)$-Hodge
    polynomial of $(\msf{H}, \iota_{\msf{H}})$. We will denote this
    polynomial by $p(\msf{H},\iota_M)$ or just $p(\msf{H})$ if
    $\iota_{\msf{H}}$ is implicit.
  \item If $(\msf{M}, \iota_{\msf{M}})$ is an object of $ \mr{MHS}_K$
    and $\iota:K \to \C$ is an embedding then we call the polynomial
    $\sum_p (f^p(\iota) - f^{p+1}(\iota))t^p \in \Z[t]$ the $t$-Hodge
    (or often simply Hodge) polynomial of $(\msf{M}, \iota_{\msf{M}})$
    with respect to $\iota$. We will denote this polynomial by
    $P(\msf{M},\iota_{\msf{M}}$ or just $P(\msf{M})$ if $\iota_{\msf{M}}$ is
    implicit.
  \end{enumerate}
  Both polynomials are additive in short exact sequences and
  multiplicative in tensor products.
\end{definition}
If $K$ is given as a subfield of $\C$, then we usually refer to the
Hodge numbers/polynoimials of $(A,\iota_A)$ with respect to the
inclusion of $K$ in $\C$ simply as the Hodge numbers/polynomials of
$(A,\iota_A)$.  Also, if $\msf{V}$ is a $K$-VMHS over a complex
variety $S$, then we denote by $P(\msf{V})$ the polynomial
$P(\msf{V}_s)$ for any $s \in S(\C)$.

\subsubsection{}

The following property of the nearby cycles functor in $\mr{MHM}_K$ is
the key to constructing our enriched representation and fusion
rings. We formulate it as a lemma for ease of reference. (The
statement is well-known, but we do not know any explicit reference.)

\begin{lemma} \label{l:saito}%
  Let $\msf{V}$ be a $K$-admissible variation of mixed Hodge structure
  over an open subset $i:U \to \A^1_{\C}$ which we may view as an
  object of $ \mr{MHM}(U)_K$. Let $\Psi_i(\msf{V})$ be the
  $ \mr{MHS}_K$ given by the (shifted) nearby cycles functor
  corresponding to the inclusion of $i:U \to \A^1_{\C}$. Then for any
  $x \in U$, the Hodge polynomials of $\msf{V}_x$ and
  $\Psi_i(\msf{V})$ are equal for any embedding $\iota:K \to \C$.
\end{lemma}

Saito's definition of the nearby cycles functor is normalised so as to
preserve mixed Hodge modules. We use $\Psi_i$ to denote the twist and
shift of Saito's functor, normalised so that if $0 \in U$ then
$\Psi_i(\msf{V})$ is equal to $\msf{V}_0$.

\begin{proof}
  The action of $K$ on $\msf{V}$ decomposes $\V$, the filtered complex
  local system associated to $\msf{V}$, as a direct sum of filtered
  complex local systems, i.e.,
  $\V = \oplus_{\iota:K \to \C} \V_{\iota}$.  Since the nearby cycles
  functor is additive for direct sums of local systems and Saito's
  construction of the Hodge filtration on nearby cycles is also
  additive we are reduced to proving a statement about nearby cycles
  for filtered $\ms{D}$-modules over $U$. The lemma then follows from
  the construction of this filtration---see, e.g., \cite[Section
  9]{schnell} for a self-contained description---and Lemma
  \ref{e:elem} below.
\end{proof}

\begin{lemma}\label{e:elem}
  Let $(R,\mf{m},k)$ be a DVR and $M$ a free module of finite rank over
  $R$. Let $M'$ be a saturated (free) submodule of $M$ and let
  $0=N_0 \subset N_1 \subset \dots \subset N_i\subset \dots \subset
  N_r = M/\mf{m}M$ be a finite filtration of $M/\mf{m}M$ by $k$-submodules.  For
  $p:M \to M/\mf{m}M$ the natural quotient map, let $M_i =
  p^{-1}(N_i)$. Then
  \[
    \sum_{i=1}^r \dim_k \left ( \frac{M_i \cap M'}{M_{i-1} \cap M'} \right )
    = \rank(M') \,.
  \]
\end{lemma}

\begin{proof}
  One only needs to note that $M_0 \cap M' = \mf{m}M \cap M' = \mf{m}M'$, so the sum
  is equal to $\dim_k(M'/\mf{m}M') = \rank(M')$.
\end{proof}

\subsection{Some mixed local systems} \label{s:mls}

\begin{definition}
  Let $X$ be a smooth irreducible variety over $\msk$ and let $K$
  be a finite extension of $\Q$. A \emph{$K$-mixed local system} over
  $X$ is an object of $ \db(\msm(X))_K$ of the form $\mc{L}[i]$,
  where $\mc{L}$ is in $ \msm(X)_K$ and $\rat_X(\mc{L})$ in
  $\msp(X,K)$ is a $K$-local system. The rank of such a $K$-mixed local
  system is defined to be the rank of $\rat_X(\mc{L})$ as a $K$-local
  system.
\end{definition}

\subsubsection{} \label{s:mls1}

For any smooth variety $X$ we have the constant mixed local system
$\Q^{\msm}_X = \pi^*(\Q^{\msm})$, where $\pi: X \to \spec k$ is the
structure map.  Let $N$ be a positive integer and $m_N:\G_m \to \G_m$
the map given by $t \mapsto t^N$. Then we have the mixed local
system $(m_N)_* \Q^{\msm}_{\G_m}$ on $\G_m$ which is a rank $N$ local
system. If $\mu_N$ is contained in $\msk$, then $m_N$ is a Galois cover
with Galois group $\mu_N$, so we get a map from $\Q[\mu_N]$, the group
ring of $\mu_N$, to $\End((m_N)_* \Q^{\msm}_{\G_m})$. The ring
$\Q[\mu_N]$ is semisimple, and the field $\Q(\mu_N)$ is
canonically a direct factor, hence corresponds to an idempotent. We
define the $\Q(\mu_N)$-mixed local system $\Q(\mu_N)t^{1/N}$ on $\G_m$ to be the
corresponding summand of $(m_N)_* \Q^{\msm}_{\G_m}$. If $X$ is any
smooth variety and $g$ is a nowhere vanishing regular function on $X$,
we define $\Q(\mu_N)g^{1/N}$ to be $g^* (\Q(\mu_N) t^{1/N})$. For any nonzero
rational number $a= b/N$ we define $\Q(\mu_N)g^a$ to be the $\Q(\mu_N)$-local
system $\Q(\mu_N) (g^b)^{1/N}$. If $G$ is a finite group acting on $X$ and
preserving $g$ then $\Q(\mu_N)g^a$ acquires a natural $G$-linearisation.
\begin{lemma}\label{l:twist}
  Let $g:X \to \G_m$, and $N, a$ positive integers such that
  $(a,N) = 1$. Assume $\mu_N \subset \msf{k}$, $K = \Q(\mu_N)$ and let
  $\sigma:K \to K$ be the automorphism given by $\zeta \mapsto
  \zeta^a$ for $\zeta \in \mu_N$. Then $Kg^{a/N}$ is isomorphic to the
  $K$-mixed local system obtained by precomposing the $K$-structure of
  $Kg^{1/N}$ by $\sigma^{-1}$.  
\end{lemma}
\begin{proof}
  We may assume that $X = \G_m$ and $g = t$. The assumption on $a$
  implies that we have a fibre product square
  \begin{equation*}
    \xymatrix{
      \G_m \ar[r]^{m_N} \ar[d]^{m_a} & \G_m \ar[d]^{m_a} \\
      \G_m \ar[r]^{m_N} & \G_m
    }
  \end{equation*}
  which gives rise to a base change isomorphism
  $(m_a)^* (m_N)_* \Q_{\G_m}^{\msm} \simeq
  (m_N)_*\Q_{\G_m}^{\msm}$. Under this isomorphism the $\mu_N$-action
  on the RHS corresponds to precomposing the $\mu_N$-action on the LHS
  with the map given by $\zeta \mapsto \zeta^a$. The lemma follows
  from this and the definition of $Kt^{a/N}$.
\end{proof}

\begin{lemma}\label{l:prod}
  Let $X_i$, $i=1,2$, be smooth varieties over $\msk$ and let $g_i$,
  $i=1,2$ be nowhere vanishing regular functions on $X_i$. Let
  $X = X_1 \times X_2$ and $g = p_1^*(g_1)\cdot p_2^*(g_2)$. If
  $\mu_N \subset \msk$ and $K = \Q(\mu_N)$, then there is a natural
  isomorphism
  \[
    Kg_1^{1/N} \boxtimes_K Kg_2^{1/N} \simeq Kg^{1/N}
  \]
  of $K$-mixed local systems on $X$.
\end{lemma}
\begin{proof}
  By functoriality of tensor products this reduces to the case
  $X_1=X_2= \G_m$ and $g_1 = g_2 =t$. The lemma then follows from the
  definition of $Kt^{1/N}$ and Lemma \ref{l:kf}.
\end{proof}

\subsubsection{Nearby cycles with coefficients}\label{s:nbc}

In the topological setting of Section \ref{s:rnb}, the nearby cycles
$\Psi_f \mc{F}$, for $\mc{F} \in \dbc(X, \C)$, decomposes as a direct
sum $\oplus_{\chi\in \C^{\times}} \Psi_{f,\chi}\,\mc{F}$, where the
semisimple part of the local monodromy acts by $\chi$ on
$\Psi_{f,\chi}\,\mc{F}$. In the setting of mixed sheaves one only has
a decomposition
$\psi_{f} \mc{F} = \psi_{f,1}\, \mc{F} \oplus \psi_{f, \neq 1}\,\mc{F}$
since we work with $\Q$-coefficients. However, one can define more
refined decompositions when we work with $K$-mixed sheaves if $K$
contains some non-trivial roots of unity and we explain this now (for
lack of a suitable reference).

\smallskip
  
Let $f:X \to \A^1_{\msk}$ be a morphism, $Y = f^{-1}(0)$, and
$\psi_f: \msm(X) \to \msm(Y)$ the nearby cycles functor. By exactness
of this functor, it extends to an exact functor
$\db(\msm(X)) \to \db(\msm(Y))$ which we also denote by $\psi_f$. By
functoriality this gives rise to a functor
$ \msm(X)_K \to  \msm(Y)_K$. If $K$ contains some
non-trivial roots of unity then we can refine these functors as
follows:

Let $M \gg 0$ be a sufficiently large integer and let
$m_M:\G_m \to \G_m$ be the map given by $t\mapsto t^M$. The functor
$\psi_f$ is defined (\cite[Section 5]{saito-form}) by first defining
the unipotent nearby cycles functor $\psi_{f,1}$ (the details of whose
definition will not be important for us here) and then defining
\begin{equation}\label{e:psi}
  \psi_f\,\mc{F} : = \psi_{f,1} (\mc{F} \otimes (m_M)_* \Q^{\msm}_{\G_m})
\end{equation}
for any $\mc{F} \in \msm(X)$. This is independent of $M \gg 0$
(depending on $\mc{F}$).

Now suppose $K = \Q(\mu_N)$, so we have the $K$-local systems $Kt^{a}$ for
any $a = b/N$ and $b \in \Z$. These local systems only depend on the
image of $a$ in $\Q/\Z$, so we define local systems $Kt^{\bar{a}}$,
for $\bar{a} \in \Q/\Z$, to be $Kt^a$ for any representative
$a \in \Q$ of $\bar{a}$. We define $Kf^{\bar{a}}$ to be
$f^*(Kt^{\bar{a}})$
\begin{definition}\label{d:nba}
  For $\mc{F} \in  \msm(X)_K$ and $\bar{a} \in (1/N)\Z/\Z$, we
  let
  \[
    \psi_{f,\bar{a}}\,\mc{F} :=\psi_{f,1}\,(\mc{F} \otimes_K Kf^{-\bar{a}})
    \in  \msm(Y)_K .
  \]
\end{definition}

\begin{remark}
  We note that $\psi_{f,1}$ as defined in \cite[Section 5]{saito-form}
  is equal to $\psi_{f,\bar{0}}$ as defined above.
\end{remark}

\begin{lemma}\label{l:ds}
  The object
  $\bigoplus_{\bar{a} \in (1/N)\Z/\Z} \psi_{f,\bar{a}}\,\mc{F}$ is a
  direct summand of $\psi_{f} \,\mc{F}$.
\end{lemma}
\begin{proof}
  We suppose that $M \gg 0$ is such that $N \mid M$. Then we have a
  factorisation $m_M = m_N \circ m_{M/N}: \G_m \to \G_m$, which using
  trace maps as in the proof of \cite[Proposition 5.7]{saito-form}
  gives a direct sum decomposition
  \[
    (m_M)_* \Q^{\msm}_{\G_m} \simeq (m_N)_* \Q^{\msm}_{\G_m} \oplus
    \mc{L}
  \]
  where the nature of $\mc{L}$ is not relevant here. Given the
  definition of $\psi_f$ from \eqref{e:psi}, the lemma follows from
  the claim that
  \[
     (m_N)_* \Q^{\msm}_{\G_m} \otimes K \simeq \bigoplus_{\bar{a} \in
       (1/N)\Z/\Z} Kt^{\bar{a}} .
   \]
   On the LHS we have an action of $\Q[\mu_N]\otimes_{\Q} K$ and one
   easily checks that the direct sum decomposition on the RHS is
   exactly the one given by the primitive idempotents in this algebra.
\end{proof}

\begin{lemma}\label{l:ls}
  Let $f:X \to \A^1_{\msk}$ be a smooth morphism, $Y= f^{-1}(0)$ and
  $\mc{L}$ a $K$-mixed local system on $X$. Then
  $\Psi_f(\mc{L}) = \Psi_{f,1}(\mc{L}) = \mc{L}|_Y$ and
  $\mr{N}: \Psi_f(\mc{L})\to \Psi_f(\mc{L})(-1)$ is equal to zero.
\end{lemma}
\begin{proof}
  There is a canonical specialisation map
  $\mr{sp}: \mc{L}|_{Y}\to \Psi_{f,1}(\mc{L})$ (see
  \cite[(5.2.2)]{saito-form} and it follows from Lemma \ref{kunneth}
  (4) that this map becomes an isomorphism after applying $\rat_Y$.
\end{proof}

\begin{remark}\label{r:ext}
  The analog of Lemma \ref{l:ext} holds in the setting of $K$-mixed
  sheaves. This follows by applying Lemma \ref{l:ext} to Betti
  realisations.
\end{remark}

\subsubsection{} \label{s:mls2}

Let $f:{X} \to S$ be a smooth proper morphism between smooth
irreducible quasi-projective varieties over $\msk$. Let
${E} = \cup_{\alpha \in A}{E}_{\alpha}$ be a normal crossings divisor
in ${X}$ such that ${E}$ is a relative strict normal crossings divisor
with respect to $f$, i.e., ${E} \cap f^{-1}(s)$ is a strict normal
crossings divisor for every geometric point $s$ of $S$. We also assume
that the intersection of each ${E}_{\alpha}$ with every geometric
fibre is irreducible. Let ${U} = {X} \sm {E}$ and for any subset $B$
of $A$ we let ${V} = {X} \sm \cup_{\alpha \in B} {E}_{\alpha}$ and
denote by $j:{U} \to {V}$ the inclusion.

Let $\mc{L}$ be a $K$-mixed local system on ${U}$.  Then $j_!(\mc{L})$
is an object of $ \db(\msm({V}))_K$ and $(f|_{V})_*(j_!(\mc{L}))$ is an
object of $ \db(\msm(S))_K$. The strict normal crossings assumption
on ${E}$ (and the fact that $\mc{L}$ is a $K$-mixed local system)
implies that $H^i((f|_{{V}})_*(j_!(\mc{L}))$ is a $K$-mixed local
system on $S$ for all $i$ since this is only a condition on the Betti
realisation and all the mixed functors are compatible with the Betti
realisation functors.

For $B' \subset B$ let
${{V}}' = {X} \sm \cup_{\alpha \in B'} {E}_{\alpha}$ and let
$q: {U} \to {{V}}'$ and $q':{{V}} \to {{V}}'$ be the inclusions, so
$j = q'\circ q$. Then we have an object $q_!(\mc{L})$ in
$ \db(\msm({{V}}'))_K$ and by adjunction we have a canonical map
$q_!(\mc{L}) \to (q')_*(j_!(\mc{L}))$. Applying $(f|_{{{V}}'})_*$ to
this map and then the cohomology functors $H^i$, we get canonical maps
of $K$-mixed local systems
\begin{equation} \label{eq:1}
	H^i((f|_{{{V}}'})_*(q_!(\mc{L})) \to
	H^i((f|_{{V}})_*(j_!(\mc{L})) .
\end{equation}
Here we use that $(f|_{{{V}}'})_* \circ (q')_*  =
(f|_{{{V}}})_*$.

Let $G$ be a finite group acting on ${X}$ preserving ${E}$. Then $G$ acts
on the index set $A$ and we assume that this action preserves both $B$
and $B'$. If $\mc{L}$ has a $G$-linearisation, then $G$ acts on both
the source and target in \eqref{eq:1} and the map is
$G$-equivariant. If $\chi$ is the projector in the group ring $K[G]$
corresponding to an irreducible representation of $G$ over $K$, we get
a map of ``$\chi$-isotypical components''
\begin{equation*}%
	H^i((f|_{{{V}}'})_*(q_!(\mc{L}))^{\chi} \to
	H^i((f|_{{V}})_*(j_!(\mc{L}))^{\chi} .
\end{equation*}

\subsection{The KZ local systems as mixed local systems}\label{s:kzmls}

We now put ourselves in the setting of Section \ref{s:sv}. In
particular, we have fixed a simple Lie algebra $\mf{g}$ and dominant
integral weights $\lambda_1,\dots,\lambda_n,$ and $\nu$ and we assume
$\kappa \in \Q^{\times}$. Let $M$ and $\beta$ also be as in the
Section \ref{s:sv}. Let $\mc{C}_n$ be the configuration space of
$n$-ordered points in $\A^1_{\msk}$.

Now we view the multi-valued master function
\[
\ms{R} = \prod_{1 \leq i < j \leq n} (z_i -
z_j)^{-\frac{(\lambda_i, \lambda_j)}{\kappa}} \prod_{b=1}^M
\prod_{j=1}^n (t_b - z_j)^{\frac{(\lambda_j, \beta(b))}{\kappa}}
\prod_{1 \leq b < c \leq M} (t_b - t_c)^{-\frac{(\beta(b),
		\beta(c))}{\kappa}}
\]
as a function on the open set ${U}$ of $\mc{C}_n \times \A^M_\msk$
which is the complement of the divisors given by the equations
$t_b = z_j$ for all $b$, $j$ and $t_b = t_c$ for all $b < c$. We
denote by $h: {U} \to \mc{C}_n$ the map induced by the first
projection.

We view $\mc{C}_n \times \A^M_\msk$ as an open subset of
$\mc{C}_n \times \P^M_\msk$. Using the desingularisation procedure
described in Theorem \ref{t:abnormal} we may find a smooth irreducible
variety ${P}$ with an open embedding of ${U}$ with complement
${{E}} = \cup_{\alpha\in A} {{E}}_{\alpha}$ a normal crossings
divisor. Furthermore, we may assume that there is a birational
projective morphism $b:{P} \to \mc{C}_n \times \P^M_\msk$ such
that $f = p_1 \circ b: {P} \to \mc{C}_n$ is a smooth projective
morphism with respect to which ${{E}}$ is a relative strict normal
crossings divisor. We may also assume that the group $\Sigma_{\beta}$
defined in Section \ref{s:sv} acts on ${P}$ inducing the action
defined there on each fibre of $f$. Note that all the varieties and
maps defined above can be taken to be defined over $\Q$. Let $N$ be
integer defined in Lemma \ref{formula:N}. Then the function
$g = \ms{R}^N$ is a regular function on ${U}$ and is preserved by the
action of $\Sigma_{\beta}$.
\smallskip

We now assume that $\msk$ is a subfield of $\C$ containing $\mu_N$
and we set $K = \Q(\mu_N)$. Applying the construction of Section
\ref{s:mls1} to $U$, the pair $(g,N)$ and the group
$G = \Sigma_{\beta}$, we obtain a $K$-mixed local system
$\mc{L}(\eta):=Kg^{1/N}$ of rank one on ${U}$ with a
$\Sigma_{\beta}$-linearisation on
${U}$. We Note that this is uniquely determined by the data fixed at
the beginning of this section.

Now let $B \subset A$ (resp.~$B' \subset A$) be the set of $\alpha$
such that the residue of the form $\eta$ (defined in \eqref{deta})
along ${{E}}_{\alpha}$ is not a strictly positive integer (resp.~is
not a non-negative integer). We now want to apply the procedure of
Section \ref{s:mls2} to the local system $\mc{L}(\eta)$. So we take
${X} = {P}$, $S = \mc{C}_n$, $f = p_1 \circ b: {P} \to \mc{C}_n$ and
${U}$, $B$, $B'$, $g$ and $N$ as above. We let $\chi$ be the sign
character of $S_M$ restricted to $\Sigma_{\beta}$. Then applying the
construction of Section \ref{s:mls1} we obtain a map
\begin{equation}\label{eq:3}
	H^i((f|_{{{V}}'})_*(q_!(\mc{L}(\eta)))^{\chi} \to
	H^i((f|_{{V}})_*(j_!(\mc{L}(\eta)))^{\chi}
\end{equation}
on $\mc{C}_n$.
\begin{definition} \label{d:kzm}%
  The mixed local system $\mathcal{KZ}_{\kappa}(\vec{\lambda},\nu^*)$
  is the $K$-mixed local system on $\mc{C}_n$ (over $\msk$) given by
  the image of the map in \eqref{eq:3} with $i = M$.
\end{definition}
The compatibility of the mixed functors with the usual functors on
constructible sheaves via $\mt{rat}$ and the results recalled
in Section \ref{s:top} imply that
$\mt{rat}_{\mc{C}_n}(\mathcal{KZ}_{\kappa}(\vec{\lambda},\nu^*))
\otimes_K \C$ is the classical KZ local system associated to the
same choice of group theoretic data. We note that
$\mathcal{KZ}_{\kappa}(\vec{\lambda},\nu^*)$ does not depend on the
choice of compactification ${P}$; this can be deduced from
Proposition 35 of the arXiv version of \cite{BBM} as the Betti
realisation functor is conservative on $\msm(X)$. We also note that
the mixed local systems do not depend on the choice of $\beta$
because of the projection to the $\chi$-isotypical component.
\begin{remark}\label{r:field}
  The KZ local systems can often be defined over proper subfields of
  $K = \Q(\mu_N)$, but we always consider them over this fixed field.
\end{remark}
\begin{remark}\label{r:pts}
  For $\vec{z} \in \mc{C}_n(\msf{k})$, by restricting
  $\mathcal{KZ}_{\kappa}(\vec{\lambda},\nu^*)$ constructed using the
  Ivorra--Morel category of motivic sheaves to $\vec{z}$ we obtain a
  Nori motive over $\msk$ with coefficients in $\Q(\mu_N)$. This motive
  equals the motive constructed in Section \ref{s:kzdef}.
\end{remark}
\begin{remark}\label{dual_kappa1}
  For all $\vv{\lambda}$, $\nu$ and $\kappa \neq 0$, we have an
  equality
  $\mathcal{KZ}_{-\kappa}(\vec{\lambda},\nu^*)=
  (\mathcal{KZ}_{\kappa}(\vec{\lambda},\nu^*))^*(-M)$ (where $(-M)$
  denotes a Tate twist.) Here $M$ is the number of simple roots,
  counted with multiplicity, in an expression of $\sum\lambda_i-\nu$
  as a sum of simple roots.
\end{remark}
\begin{remark}\label{r:twist}
  The statement about the Galois twists in Section \ref{s:gt} follows
  from Lemma \ref{l:twist} and the definition of the mixed KZ local
  systems: the assumption that $(a,N) = 1$ implies that the divisors
  $B$ and $B'$ used in the construction do not change when we replace $\kappa$ with
  $\kappa/a$.
\end{remark}
\subsection{Conformal blocks mixed local systems}\label{s:cbmls}

Mixed versions of the conformal blocks local systems can be defined in
a similar, in fact more elementary, way. Let $\ell \geq 0$ be an
integer and assume that all the dominant integral weights
$\lambda_1,\dots,\lambda_n,$ and $\nu$ are of level $\ell$. We let
$\kappa = \ell + h^{\vee}$. Then we have ${U}$, $h:{U} \to \mc{C}_n$,
and a $\Q(\mu_N)$-mixed local system $\mc{L}(\eta)$ on ${U}$ which is
$\Sigma_{\beta}$-equivariant. By applying $h_!$ and $h_*$ to
$\mc{L}(\eta)$, adjunction, and taking cohomology, we obtain maps
\begin{equation*}
	H^i(h_!(\mc{L}(\eta))) \lra H^i(h_*(\mc{L}(\eta))) .
\end{equation*}
Using the $\Sigma_{\beta}$-equivariance and the character $\chi$ we then get
maps
\begin{equation}\label{e:eq4} %
	H^i(h_!(\mc{L}(\eta)))^{\chi} \lra H^i(h_*(\mc{L}(\eta)))^{\chi} .
\end{equation}
\begin{definition}\label{d:cbm}%
  The mixed local system $\mathcal{CB}_{\kappa}(\vec{\lambda},\nu^*)$
  is the $\Q(\mu_N)$-mixed local system on $\mc{C}_n$ (over $\msk$) given by
  the image of the map in \eqref{e:eq4} with $i=M$.
\end{definition}

  \begin{remark}\label{r:cb}
    $ $
    \begin{enumerate}
    \item As in the case of KZ local systems, one sees
      $\rat(\mathcal{CB}_{\kappa}(\vec{\lambda},\nu^*))\otimes_{\Q(\mu_N)}
      \C)$ is the classical conformal blocks local system over $\C$
      associated to the same choice of group-theoretic data.
    \item We note that $\mathcal{CB}_{\kappa}(\vec{\lambda},\nu^*)$ is
      actually pure of (fibrewise) weight $M$, since the weights of
      (the fibres of) $H^i(h_!(\mc{L}(\eta)))^{\chi}$ are $\leq M$ and
      those of $H^i(h_*(\mc{L}(\eta)))^{\chi}$ are $\geq
      M$. Furthermore, it follows from \cite{TRR} (see also
      \cite{loo-sl2}) when $\mf{g} = \mf{sl}_2$ and \cite[Theorem
      3.5]{b-unitarity}, and its extension to arbitrary $\nu$ in
      \cite[Remark 4.5]{BF} for general $\mf{g}$, that
      $P(\msf{CB}_{\kappa}(\vec{\lambda},\nu^*)) = rt^M$ where $r$ is
      the rank of the local system.
  \item These mixed local systems do not depend on the choice of
    $\beta$ and a remark analogous to Remark \ref{r:pts} also applies
    in the conformal blocks setting.  We also note that it follows
    from \cite[Corollary 4.18]{BF} that
    $\mathcal{CB}_{\kappa}(\vec{\lambda},\nu^*)$ is canonically a
    subsystem of $\mathcal{KZ}_{\kappa}(\vec{\lambda},\nu^*)$.
  \end{enumerate}
\end{remark}

\section{Nearby cycles for semistable degenerations}\label{s:nb}

As noted in the sketch of the proof of motivic factorisation in the
introduction, we need to analyse the nearby cycle complexes $\Psi_f$
on the fibre $Y=f^{-1}(0)$ of a semistable degeneration
$f:X\to \A^1_{\msk}$ of suitable complexes on $X^*$. This section lays
down the motivic groundwork for this study: The main result is Theorem
\ref{stillholds} which gives a method of ``decomposing'' these nearby
cycles complexes in terms of the geometry of $Y$.
  
\subsection{Nearby cycles for mixed sheaves}\label{s:nbms}
In this section $\msk$ is a subfield of $\C$ and $K$ is a finite
extension of $\Q$.

Let $X$ be a smooth variety over $\msk$ and let $f: X \to \A^1_{\msk}$
be a morphism. Assume that $f$ is smooth over $\A^1_{\msk} \sm \{0\}$ and
$Y = f^{-1}(0)$ is a reduced divisor with simple normal crossings with
irreducible components $Y_{\alpha}$, $\alpha = 1,2,\dots,n$. Let
$X^* = X \sm Y$.
\begin{defi}\label{basicD}
  For $Z$ any union of irreducible components of $Y$, let
  $Z_0=Z \sm \cup_{Y_{\alpha}\nsubseteq Z} Y_{\alpha}$.  Clearly,
  $Z_0$ is an open subset of $Y$.
\end{defi}
Let $Z$ be a union of some irreducible components of $Y$ and $Z'$ the
union of the remaining irreducible components of $Y$.  Let
$$j:Z_0\lra Y, \ \ j':Z'_0\lra Y,\ \ i:Z\lra Y,\ \ i':Z'\lra Y$$ be the
inclusions maps. Since $j'$ is an open inclusion with complementary
closed inclusion $i$, for any $\mc{H}\in  \msm(Y)_K$ there is a
distinguished triangle
$$j'_{!}j'^*\mc{H}\lra \mc{H}\lra i_*i^*\mc{H}\leto{[1]} .$$
{ We note that the maps are compatible with the $K$-structure because
they are constructed using adjunction.}
Since $j^*i_*i^*\mc{H}=j^*\mc{H}$, there is also a map
\begin{equation}\label{e:adj}
  i_*i^*\mc{H}\lra j_*j^*\mc{H}
\end{equation}
given by adjunction. If this map is an isomorphism, then we get a
distinguished triangle
\begin{equation}\label{e:d1}
  j'_{!}j'^*\mc{H}\lra \mc{H}\lra j_*j^*\mc{H}\leto{[1]} 
\end{equation}
in $\msm(Y)_K$.
\begin{remark}\label{structure} %
  The property that the map in \eqref{e:adj} is an isomorphism, and so
  there is a distinguished triangle as in \eqref{e:d1}, is compatible
  with distinguished triangles in $ \db(\msm(Y))_K$: If two terms in
  a distinguished triangle have this property, then so does the third:
  this follows from \cite[Corollary 1.5.5]{ks} (applied in the
  triangulated category $\db(\msm(Y))$. This property for a fixed
  $\mc{H}$ and for all choices of $Z$ is also closed under Verdier
  duality.
\end{remark}

Let $\mc{K}=\Psi_f \, \mc{F}$, where $\mc{F}$ is any object of
$\msm(X^*)_K$.  The main goal of this section is to prove that for
suitable $\mc{F}$ the maps in \eqref{e:adj} are isomorphisms when
$\mc{H} = \mc{K}$, so we will have distinguished triangles as in
\eqref{e:d1}. To formulate the result we need to introduce some more
notation:

With $f:X \to \A^1_{\msk}$ as above, let
$E=\cup_{\beta \in J} E_{\beta}\subset X$ be a divisor such that
$D = Y \cup E$ is a divisor with simple normal crossings. { We also
  assume that the restriction of $f$ to each $E_{\beta}$ is dominant
  and smooth over $\A^1_{\msk} \sm \{0\}$.} Let
$V=X \sm (Y \cup (\cup_{\beta \in J'}E_{\beta}))$ where $J'$ is an
arbitrary subset of $J$.  Let $\tilde{j}:X \sm D\to V$, and
$\tilde{k}:V\to X \sm Y$ be the inclusions.
\begin{theorem}\label{stillholds} 
  Let $\ml$ be a $K$-mixed local system in $ \msm(X \sm D)_K$, let
  $\mc{F} = \tilde{k}_*\tilde{j}_{!}\ml \in \db(\msm(X^*))_K$ and
  $\mc{K} = \Psi_f \mc{F} \in \db(\msm(Y))_K$. Then the map
  $i_*i^*\mc{K}\to j_*j^*\mc{K}$ as in \eqref{e:adj} is an
  isomorphism and we have a distinguished triangle
  \begin{equation}\label{e:d2}
    j'_{!}j'^*\mc{K}\lra \mc{K}\lra j_*j^*\mc{K}\leto{[1]}
  \end{equation}
  in $\db(\msm(Y))_K$.  Furthermore, there exists a sequence of objects
  $\mc{H}_{\alpha} \in \db(\msm(Y))_K$, $i =1,\dots,n$, with each $\mc{H}_{\alpha}$
  supported on $Y_{\alpha}$, and distinguished triangles
\[
  \mc{K}_{\alpha-1} \lra \mc{K}_{\alpha} \lra \mc{H}_{\alpha} \leto{[1]} 
\]
for $1 < \alpha \leq n$, with $\mc{K}_n = \mc{K}$ and
$\mc{K}_0 = \mc{H}_1$.  Each $\mc{H}_{\alpha}$ is given by applying a
sequence of star and shriek pushforwards to $\mc{K}|_{(Y_{\alpha})_0}$ along
suitable strata, with there being at least one shriek if and only if
$\alpha <n$.
\end{theorem}

We will prove the theorem at the end of this section after proving
some preliminary lemmas.

\begin{remark}\label{still}
  The Verdier dual of $\mc{K}$ as in Theorem \ref{stillholds} is again
  of the same form (for a different collection of components of $E$).
\end{remark}

\begin{remark}\label{devissage}
  Theoreom \ref{stillholds} allows us to carry out a complete
  d\'evissage of $\mc{K}$ as an extension of objects supported on
  $Y_{\alpha}$ which are given by applying a sequence of shriek and star
  pushforwards to the local systems on $(Y_{\alpha})_0$ given by
  $\mc{K}_{(Y_{\alpha})_0}$ (the independence of order of these pushforwards
  follows from Lemma \ref{l:ext}).
\end{remark}

\subsection{Nearby cycles in the analytic setting}

In this section we will prove Theorem \ref{stillholds} by applying the
realisation functor $\rat$. This gives us a map in $\db(\msp(Y))_K$
which we need to prove is an isomorphism. To do this we will simply
forget the $K$-structure and then by the conservativity of $\rat$ we
can replace $\db(\msp(Y))$ by $\dbc(Y,\Q)$, which we then replace by
$\dbc(Y,\C)$.

Thus we now assume that $X$ is a complex manifold and $f:X \to S$ is a
morphism (with $S \subset \C$) satisying properties analogous to those
stated before the statement of Theorem \ref{stillholds}. In the
statement of the theorem we assume that $\mc{L}$ is a $\C$-local
system (in the analytic topology) on $X \sm D$ and then define all the
data in the theorem analogously. We summarize this as follows:
\begin{proposition}\label{p:analytic}
  The analogue of Theorem \ref{stillholds} in the setting of complex
  manifolds holds for all $\C$-local systems $\mc{L}$ on $X \sm D$.
\end{proposition}

The example below shows that arbitrary elements $\mc{F}$ of
$\dbc(X^*,\C)$ may not satisfy the property in Theorem
\ref{stillholds}.
\begin{example}
  Let $X=\mb{C}^2$, $S=\mb{C}$ and $f:X\to S$ be the map
  $f=z_1z_2$. Let $W\subset \C^2$ be given by $z_1=z_2$ and let
  $\mc{K}=\Psi_f (i_*\mb{C}_{W^*})$ where
  $i:W^*=W \sm \{0\}\to \mb{C}^2$ is the inclusion. Note that
  $W\cap f^{-1}(0)=\{(0,0)\}$, and therefore if $Z:=\{z_1=0\}$ in the
  setting of Theorem \ref{stillholds}, then both $j^*\mc{K}$ and
  $j'^*\mc{K}$ are zero. However the stalk of $\mc{K}$ at $0$ is
  nonzero by a Milnor fibre computation, hence there is no
  distinguished triangle of the form \eqref{e:d2}.
\end{example}

\subsubsection{}

We first prove the special case of Proposition \ref{p:analytic} when
$\mc{F}$ is the constant local system $\C_{X^*}$.  Ayoub has proved
various related results in \cite[\S 3.3]{ayoub} (in particular, the
local Theorem 3.3.10 and global Theorem 3.3.44) in a general motivic
setting. We will give a proof using de Rham complexes since it
illustrates methods that we will use later in a more elaborate
setting.
\begin{lemma}\label{clarity}
  Proposition \ref{p:analytic} holds when $E = \emptyset$ and $\mc{F}$
  is the constant local system $\C_{X^*}$.
\end{lemma}

\begin{proof}
  First suppose that $Z$ is a single component $Y_1$ of $Y$. By
  results of Steenbrink \cite[\S 2]{steenbrink}, given the choice of a
  parameter $t$ of $S$, $\mc{K}$ is isomorphic to
  $(\Omega^{\bull}_{X/S}(\log Y)\tensor_{\mc{O}_X} \mathcal{O}_Y, d)$.
  Here $\Omega^1_X(\log Y)$ is the sheaf of logarithmic $1$-forms,
  $\Omega^1_{X/S}(\log Y)$ is its quotient by the $\mc{O}_X$-submodule
  generated by $df/f$, and the complex $\Omega^{\bull}_{X/S}(\log Y)$
  is the exterior algebra on $\Omega^1_{X/S}(\log Y)$ with
  differential induced by the usual de Rham differential $d$. 
    Precomposing the residue map
    $\op{Res}_{Y_1}:\Omega^{\bull}_X(\log Y) \to
    \Omega^{\bull}_{Y_1}(\log (Y_1 \sm (Y_1)_0))[-1]$ (\cite[\S
    3.1]{hodge2}) with the map $\wedge df/f$ induces an isomorphism of
    $(\Omega^{\bull}_{X/S}(\log Y)\otimes_{\mc{O}_X} \mathcal{O}_Y)
    \otimes \mc{O}_{Y_1}$ with
    $\Omega^{\bull}_{Y_1}(\log (Y_1 \sm (Y_1)_0))$ compatible with the
    differentials: the key point is that $df/f \wedge df/f = 0$ and if
    $g$ is any (local) holomorphic function vanishing on $Y_1$ then
    $\op{Res}_{Y_1}(gdf/f) = 0$.  By \cite[Proposition
  3.1.8]{hodge2}, the inclusion of
  $\Omega^{\bull}_{Y_1}(\log (Y_1 \sm (Y_1)_0))$ in the pushforward to
  $Y_1$ of $\Omega^{\bull}_{(Y_1)_0}$ is a quasi-isomorphism, so we
  are done with this case.
  
  We now prove that $i_*i^*\mc{K}\to j_*j^*\mc{K}$ is an isomorphism
  for general $Z$ at all points $P$ of $Y$ by induction on the number
  of branches of $Y$ passing through $P$.  Using Lemma \ref{kunneth}
  (2) we may assume $X = \dc^n$ with coordinates
  $(z_1,z_2,\dots,z_n)$, $P=0$ and $f=z_1\cdots z_n$.  The number of
  branches through $P$ is $n$, $Y_{\alpha}$ is the divisor given by
  $z_{\alpha} = 0$ for $\alpha=1,2,\dots,n$, $Z$ is the divisor given
  by $z_1\cdots z_r=0$ for $1 < r < n$ (the statement is vacuous when
  $Z = Y$) and $Z'$ is the divisor given by $z_{r+1}\cdots z_{n}=0$.
  It suffices to show that $i_*i^*\mc{K}\to j_*j^*\mc{K}$ is an
  isomorphism on stalks at $P$.  Note that
  $(i_*i^*\mc{K})_P=\mc{K}_P$, where for any object of $\dbc(X,\Q)$ we
  use the subscript $P$ to denote the stalk at $P$ (i.e., the pullback
  via the inclusion of $P$ in $X$).

\smallskip

We will apply the induction hypothesis to the restriction $j^*\mc{K}$
of $\mc{K}$ to $Z_0$: we may do so since fewer than $n$ branches of
$Y$ pass through any point of $Z_0$.

Since $Z_0$ does not intersect any $Y_{\alpha}$ with $i>r$, induction gives
a sequence of objects $\mc{H}_{\alpha} \in \dbc(Z_0,\C)$, $\alpha =1,\dots,r$, with
each $\mc{H}_{\alpha}$ supported on $Y_{\alpha}$, and a sequence of distinguished
triangles
\[
  \mc{K}_{\alpha-1} \lra \mc{K}_{\alpha} \lra \mc{H}_{\alpha} \leto{[1]} 
\]
for $1 < \alpha \leq n$, with $\mc{K}_n = j^*\mc{K}$ and
$\mc{K}_0 = \mc{H}_1$.  Furthermore, each $\mc{H}_{\alpha}$ is given
by applying a sequence of star and shriek pushforwards to
$\C_{(Y_{\alpha})_0}$ along various strata, and there is at least one
shriek if and only if $\alpha<r$. Applying $j_*$ to all these
triangles and using Lemma \ref{l:ext} (1), it follows that the map
$(j_*j^*\mc{K})_P \to \mc{H}_{n,P}$ is an isomorphism. We may
therefore apply the case $n=1$ to complete the proof of this step.

Having proved that the map $i_*i^*\mc{K}\to j_*j^*\mc{K}$ is an
isomorphism for all $Z$ we prove the second part of the lemma. For
this let $Z = Y_n$ so $Z' = \cup_{\alpha=1}^n Y_{\alpha}$. Applying
the first part of the theorem, we get a distinguished triangle
\[
  j'_{!}j'^*\mc{K}\lra \mc{K}\lra j_*j^*\mc{K}\leto{[1]}.
\]
Since $Z_0 = (Y_n)_0$, we have $j_*j^*\mc{K} = \C_{(Y_n)_0}$ and we
take it to be $\mc{H}_n$. By induction we have a d\'evissage of
$j'^*\mc{K}$ as in the theorem given by objects
$\mc{H}_{\alpha}' \in \dbc(Z'_0)$ supported on $Y_{\alpha} \sm Y_n$
for $\alpha =1,\dots,n-1$. It follows that we may take
$\mc{H}_{\alpha}$ to be $j'_! \mc{H}_{\alpha}'$ for $\alpha=1,\dots,n-1$.
\end{proof}

Let $\eta$ be a section of $\Omega^1_X(\log D)$.
\begin{defi}
  For $a\in \C$, let $Y(a)$ be the union of irreducible
  components $Y_{\alpha}$ such that $\op{Res}_{Y_{\alpha}}(\eta)$ is in
  $a+\mb{Z}$.
\end{defi}

\begin{lemma}\label{extend2}
  For any $a \in \C$, the local system
  $\ml{(\eta)}\tensor_{\C} f^*(\mb{C} t^{a})$ extends as a local
  system to a neighbourhood of $Y(a)_0 \subset Y(a)$ (with
  $Y(a)_0$ as in Definition \ref{basicD}).
\end{lemma}
\begin{proof}
  By the assumption,
  $\op{Res}_{Y_{\alpha}}(\eta - f^*(a dt/t)) \in \Z$ for any
  $Y_{\alpha} \subset Y(a)$, so the local monodromy of
  $\ml{(\eta)}\tensor_{\C} f^*(\mb{C} t^{a})$ is trivial at any point of
  $Y(a)_0$, therefore the local system extends as claimed.
\end{proof}

\begin{proof}[Proof of Proposition \ref{p:analytic}]
  The theorem is local on $X$, therefore we may let $X= \dc^{n+s+t}$
  with coordinates $x_1,\dots,x_n, u_1,\dots,u_s,v_1,\dots, v_t$, let
  $S=\dc$, and let $f:X\to S$ be the function $x_1 x_2 \cdots x_n$.
  Let $E\subseteq X$ be the divisor given by $u_1u_2\cdots u_s=0$,
  $Y=f^{-1}(0)\subseteq X$, $D=E\cup Y$ and $U = X \sm D$.  Let
  $V=X \sm (Y \cup (\cup'_{\beta\in J}E_{\beta}))$ where the union is
  restricted to an arbitrary set of irreducible components of $E$.
  Let $\tilde{j}:U \to V$, and $\tilde{k}:V\to X \sm Y$ be the
  inclusions.  Since $\pi_1(U)$ is abelian, every $\C$-local system
  $\mc{L}$ on $U$ has a rank one sub-local system, so by the analogue
  of Remark \ref{structure} in the complex analytic setting, we may
  assume that $\mc{L}$ is of rank $1$ of the form $\mc{L}(\eta)$ for
  $\eta$ a section of $\Omega^1_X(\log D)$. Thus, we may assume that
  $\mc{F} = \tilde{k}_*\tilde{j}_{!}\ml \in \dbc(X^*,\C)$ with
  $\mc{L}$ of rank $1$.

The first step in the proof is to show that we may assume that
$E =\emptyset$. To do this we will construct a product decomposition
of all the data.  Let $X' = \dc^{n}$ with coordinates $x_1,\dots,x_n$
and $W = \dc^{s+t}$ with coordinates $u_1,\dots,u_s, v_1,\dots, v_t$.
Let $f':X'\to S$ be given by $x_1x_2\cdots x_n$ and let
$Y' = (f')^{-1}(0) \subset X'$. Then $X = X' \times W$, $Y=Y'\times W$
and $f = f' \circ p_{X'}$, where $p_{X'}: X \to X'$ is the projection.
Then
$$X \sm D= (X' \sm Y')\times W^* ,$$ where
$W^*= (\dc^*)^s \times (\dc)^t \subset \dc^{s+t} = W$.  The rank one
local system $\ml$ on $X \sm D$ is also a corresponding external
product $\ml'\boxtimes\ml''$ of local systems $\mc{L}'$ on $X' \sm Y'$
and $\mc{L}''$ on $W^*$. Let $E_W$ be the divisor in $W$ given by
$u_1u_2\cdots u_s=0$ with irreducible components $E_{W,j}$ and let
$U_W = W \sm E_W$, and let $V_W = W \sm \cup'_{j \in J}E_{W,j}$, where
the union is over the same subset of $J$ used to define $V$, so
$V = X' \times V_W$. Let $\tilde{j}_W:U_W \to V_W$, and
$\tilde{k}_W:V_W\to W $ be the inclusions. We then have:

\begin{claimn}
  $\Psi_f(\mc{F}) \simeq \Psi_{f'}\mc{L}' \boxtimes \tilde{k}_{W,*}
  \tilde{j}_{W,!} \mc{L}''$.
\end{claimn}
To prove the claim we note that
$\tilde{j} = \mr{id}_{X'} \times \tilde{j}_W$ and
$\tilde{k} = \mr{id}_{X'} \times \tilde{k}_W$. The claim then follows
by applying part (2) of Lemma \ref{kunneth} (with $X$ (resp.~$Z$)
there being $X'$ (resp.~$W$) here) and parts (1) and (2) of Lemma
\ref{KunnethF}.

Let $Z' \subset Y'$ correspond to $Z \subset Y$, i.e.,
$Z = Z' \times W$, and let $i'$ (resp.~$j'$) be the analogues of $i$
(resp.~$j$). Appying part (2) of Lemma \ref{KunnethF} again we see
that
\[
  i_*i^* \mc{K} = i'^*i'_* \Psi_{f'} \mc{L}' \boxtimes \tilde{k}_{W,*}
  \tilde{j}_{W,!} \mc{L}'' \mbox{ and }
  j_*j^* \mc{K} = j'^*j'_* \Psi_{f'} \mc{L}' \boxtimes \tilde{k}_{W,*}
  \tilde{j}_{W,!} \mc{L}''.
\]
We are thus reduced to proving the theorem for $X'$ and $f'$ with
$\mc{F} = \mc{L'}$, i.e., we may assume that $E = \emptyset$.

We now prove that the map $(i_*i^*\mc{K})_P \to (j_*j^*\mc{K})_P$ is
an isomorphism at all points $P \in Y$. By induction on the number of
components of $Y$ passing through $P$ we may assume that $P= 0 \in X$.
There are then two cases:
\begin{enumerate}
\item[(1)] Assume that there exist $\alpha$ and $\alpha'$ such that
  $\op{Res}_{Y_{\alpha}}\eta - \op{Res}_{Y_{\alpha'}} \eta \notin
  \Z$. In this case we claim that both stalks are zero.  For
  $i_*i^*\mc{K}$ this follows from a topological computation
  (\cite[Lemma 4.18.5, in Expos\'e XIV]{SGA7II}). For $j_*j^* \mc{K}$,
  we need to understand the restriction of $\mc{K}$ to
  $(Y_{\alpha})_0$ for each $Y_{\alpha}$. By Lemma \ref{extend2} we
  see that $\mc{L} \otimes_{\C} \C t^{a_{\alpha}}$, where
  $a_{\alpha} = \op{Res}_{Y_{\alpha}}\eta$, extends to a neighbourhood of
  $(Y_{\alpha})_0$ as a local system and by Lemma \ref{kunneth} we see
  that $\mc{K}_{(Y_{\alpha})_0}$ is isomorphic to the restriction of
  this local system. This local system corresponds to the log form
  $\eta - a_{\alpha} dt/t$ which has residue $0$ along $Y_{\alpha}$,
  so the restricted local system corresponds to the restriction of
  this form. The assumption on residues at the beginning implies that
  this restricted $1$-form has a non-integral residue along some
  divisor, so $\mc{K}_{(Y_{\alpha})_0}$ has non-trivial monodromy
  around a boundary divisor.  We now completely d\'evissage
  $j^*\mc{K}$ using induction in terms of objects supported on the
  $Y_{\alpha}$ (see Remark \ref{devissage}). Each object is the
  extension to $Y_{\alpha}$ of $\mc{K}_{(Y_{\alpha})_0}$ using
  iterated shriek and star pushforwards.  If there is any shriek
  pushforward the stalk at $P$ is $0$ by Lemma \ref{l:ext}. If all the
  pushforwards are star, we see that the stalk at $P$ is also $0$ from
  the fact that the cohomology of a non-trivial local system on any
  $(\dc^*)^{n-1}$ vanishes in all degrees (an easy consequence of the
  case $n=2$ and the K\"unneth formula).
\item[(2)] Suppose the residues of $\eta$ at all divisors $Y_{\alpha}$ are
  congruent modulo $\mb{Z}$.  By Lemma \ref{extend2}, for $a$ any
  one of these residues the local system
  $\ml\tensor_{\C} f^*(\C t^{a}) $ extends to a local system in a
  neighborhood of $P$ in $X$ and is hence trivial in a neighbourhood
  of $P$.  The claim then follows by using (1) Lemma of \ref{kunneth}
  and the special case of the theorem proved in Lemma \ref{clarity}.
\end{enumerate}
Having proved the first part of the theorem, the second part is proved
in the same way as in the proof of the second part of Lemma
\ref{clarity}.
\end{proof}

\begin{proof}[Proof of Theorem \ref{stillholds}]
  By the conservativity of the Betti realisation functors, the fact
  that the map $i_*i^*\mc{K}\to j_*j^*\mc{K}$ in the theorem is an
  isomorphism follows from Proposition \ref{p:analytic}. This implies
  the existence of the distinguished triangle in \eqref{e:d2}. The
  rest of the theorem is a formal consequence, and follows in the same
  as in the proof of the second part of Lemma \ref{clarity}.
\end{proof}

\section{The statement of the factorisation theorem and some geometric
  preparation} \label{s:prep}

From now till the end of Section \ref{kzf} we fix a theory of mixed
sheaves as in Section \ref{s:mixte}. All our results will apply to any
such theory, but the key ones are (1) and (2) of Example \ref{e:M}. In
particular, (2) justifies the appellation ``motivic'' for our results.

\smallskip

In this section, we formulate a general motivic factorisation theorem
(Theorem \ref{t:fact}) which generalizes Theorem \ref{t:motivic} in
which we allow an arbitrary number $m$ of marked points to come
together. To make use the context of nearby cycles, we blow up the
locus in $\A^n$ where $m$ of the coordinates coincide. In this way we
obtain a space $W$ and a function $f:W\to \A^1_{\msk}$ such that the
KZ and conformal blocks mixed systems are defined on $W \setminus Z$
where $Z$ is the zero locus of $f$. We then state the motivic
factorisation theorem for KZ and conformal blocks mixed local
systems. In Section \ref{s:gprep} we make a choice of a family
$g:X\to W$, with good properties over $Z$, and note geometric
properties (irreducible components, how they meet, etc.) of the
central fibre $Y=f_X^{-1}(0)\subseteq X$, where $f_X=f\circ g$. We
also determine how the local systems defining the KZ systems on
$X^*=X \setminus Y$ behave near generic points of irreducible
components of $Y$.

\subsection{Statement of the factorisation theorem}\label{s:statement}

Recall that $\mc{C}_n \subset \A^n_{\msk}$ is the configuration space
of $n$ points on $\A^1_{\msk}$ and that we use $z_1,z_2,\dots,z_n$ for
the coordinate functions. Let $n \geq 3$, let
$A \subset [n]:= \{1,2,\dots,n\}$ with $|A| =m \geq 2$, and let
$D_A \subset \A^n_{\msk}$ be the codimension $n + 1 -|A|$ linear
subvariety defined by the equations $z_i = z_j$ for $i,j \in A$. Let
$\mr{Bl}_A$ be the blowup of $\A^n_{\msk}$ at $D_A$. The tangent
bundles of $\A^n_{\msk}$ and $D_A$ are trivialised using the
coordinates, so the exceptional divisor of the blowup $E_A$ is
isomorphic to $D_A \times \P(\A^n_{\msk}/D_A)$. Let $\mr{Bl}_A'$ be
the complement in $\mr{Bl}_A$ of the strict transforms of all $D_B$
with $B \subset [n]$, $|B| =2$ and $B \neq A$ and let
$E_A' = E_A \cap \mr{Bl}_A'$. It is easy to see that
$\mr{Bl}_A' \sm E_A' = \mc{C}_n$.

To formulate the factorisation theorem it will be important to choose
explicit local coordinates on $\mr{Bl}_A'$ so we proceed as follows.
Fix $i_0,i_1 \in A$ such $i_0 \neq i_1$ and let $I =
\{i_0,i_1\}$. Define a map $p:\A^n_{\msk} \to \A^n_{\msk}$ by $p(z_i) = z_i$ for
$i \in A^c \cup I$ and
$p(z_i) = (z_i - z_{i_0})\cdot(z_{i_1} - z_{i_0}) + z_{i_0}$ for
$i \in A \sm I$. The map $p$ gives a local chart for the blowup and
(an open subset of) the exceptional divisor $E_A$ is identified with
the locus given by $z_{i_0} = z_{i_1}$ and $E_A'$ with the open subset
thereof given by $z_i \neq z_j$ for all $\{i,j\} \neq I$ and
$z_i-z_{i_0}\neq 0$ for $i\not\in I$.

On the divisor $D_I$ given by $z_{i_0} = z_{i_1}$ we have coordinates
$z_{i_0}$ and $z_i$ for $i \notin A$ and $u_i:=z_i - z_{i_0}$ for
$i \in A \sm I$. These coordinates give an isomorphism of $D_I$ with
$\A^{A^c \cup \{i_0\}}_{\msk} \times \A^{A \sm I}_{\msk}$ which contains the product
of configuration spaces
$\mc{C}_{A^c \cup \{i_0\}} \times \mc{C}_{A \sm I}$ as an open
subset. Let $\mc{C}_{A \sm I}'$ be the open subset of
$\mc{C}_{A \sm I}$ given by $u_i \neq 0,1$ for all $i \in A \sm I$.

Let $W_A$ be the open subset of $\A^n_{\msk}$ given by the complement of the
closed set defined by the equations $z_i= z_j$ for all $i, j$ such
that $\{i,j\} \neq I$ and $z_i - z_{i_0} = 1$ for all $i \in A \sm
I$. Let $Z_A$ be the closed subset of $W_A$ defined by the equation
$z_{i_0} = z_{i_1}$. Then $W_A \sm Z_A$ is an open subset of
$\mc{C}_n$ and
$Z_A \simeq \mc{C}_{A^c \cup \{i_0\}} \times \mc{C}_{A \sm I}'$.

We can now state one of our main results:

\begin{theorem}\label{t:fact}
  Let $\mf{g}$ be a split simple Lie algebra over $\Q$, let
  $\kappa = r/s \in \Q^{\times}$ with $r,s \in \Z$ and $(r,s) =1$, let
  $N = i_{\mf{g}}m_{\mf{g}}|r|$ (see Lemma \ref{formula:N}) and let
  $K = \Q(\mu_N)$.  Let
  $\vv{\lambda} = (\lambda_1,\lambda_2,\dots,\lambda_n) \in (P^+)^n$,
  let $\nu \in P^+$ and let
  $\mathcal{KZ}_{\kappa}(\vec{\lambda},\nu^*)$ be the corresponding
  $K$-mixed local system on $\mc{C}_n$. Let $A \subset [n]$ be such
  that $|A| = m \geq 2$ and let $i_0,i_1 \in A$ be distinct
  elements. Let $(W_A,Z_A)$ as above be the local model for the blowup
  in $\A^n_{\msk}$ of the diagonal corresponding to $A$ and let
  $\mathcal{KZ}_{\kappa}(\vec{\lambda},\nu^*)'$ be the restriction of
  $\mathcal{KZ}_{\kappa}(\vec{\lambda},\nu^*)$ to $W_A \sm Z_A$. For
  any $\mu \in P^+$
  \begin{itemize}
  \item[--] Let $\vv{\lambda}'_{\mu}$ be the set of weight labelled by
    $A^c \cup \{i_0\}$, where for $i$ in $A^c$ the weight assigned to
    $i$ is $\lambda_i$ and for $i_0$ the weight is $\mu$.
  \item[--] Let $\vv{\lambda}''$ be the set of weights labelled by points
    in $A$ given by the restriction of $\vv{\lambda}$ to $A$.
  \end{itemize}
  Let $\iota: \mc{C}_{A \sm I}' \to \mc{C}_A$ be the inclusion given
  by using $0$ (resp.~$1$) for the $i_0$ (resp.~$i_1$) coordinate and
  let {$f: W_A \to \A^1_{\msk}$} be the function $z_{i_1} - z_{i_0}$.  Then we 
  have an equality
  \[
    \Psi_f\,(\mathcal{KZ}_{\kappa}(\vec{\lambda},\nu^*)') =
    \bigoplus_{\bar{a} \in (1/N)\Z/\Z} \Psi_{f, \bar{a}}
      \,(\mathcal{KZ}_{\kappa}(\vec{\lambda},\nu^*)')
    \]
    of $K$-mixed local systems on $W_A$. Furthermore, each summand
    $\Psi_{f, \bar{a}} \
    (\mathcal{KZ}_{\kappa}(\vec{\lambda},\nu^*)')$ on the RHS has an
    increasing filtration $\mr{G}^{\bullet}$ indexed by
    $a \in (1/N)\Z$ and preserved by the monodromy operator $\mr{N}$
    with
  \begin{equation}\label{e:kzfact}
    \mr{Gr}_a^{\mr{G}}(\Psi_{f, \bar{a}} (\mathcal{KZ}_{\kappa}(\vec{\lambda},\nu^*)')) =
    \bigoplus_{\mu \in P^+,\,a(\vv{\lambda},\mu) =a }\mc{KZ}_{\kappa}(\vv{\lambda}'_{\mu}, \nu^*)
    \boxtimes_K \iota^*(\mc{KZ}_{\kappa}(\vv{\lambda}'', \mu^*) ) ,
  \end{equation}
      where
  \begin{equation}\label{formula:a}
    a(\vv{\lambda}, \mu) = \frac{1}{2\kappa}\bigl (-c(\mu) + \sum_{i \in A}
    c(\lambda_i) \bigr ) 
  \end{equation}
  with $c(-)$ is as in Section \ref{s:lie}.  The monodromy
  operator $\mr{N}$ acts trivially on both sides of \eqref{e:kzfact}.

    Furthermore, if all the $\lambda_i$ and $\nu$ are of level
    $\ell \geq 0$, then for $\kappa = \ell + h^{\vee}$ and
    $\mathcal{CB}_{\kappa}(\vec{\lambda},\nu^*)'$ the restriction of
    $\mathcal{CB}_{\kappa}(\vec{\lambda},\nu^*)$ to $W_A \sm Z_A$, we
    have
 \[
    \Psi_f\, (\mathcal{CB}_{\kappa}(\vec{\lambda},\nu^*)') =
    \bigoplus_{\bar{a} \in (1/N)\Z/\Z} \Psi_{f, \bar{a}}
    \, (\mathcal{CB}_{\kappa}(\vec{\lambda},\nu^*)')
  \]
  and
  \begin{equation}\label{e:cbfact}
    \Psi_{f,\bar{a}}(\mathcal{CB}_{\kappa}(\vec{\lambda},\nu^*)') = 
    \bigoplus_{\mu \in P_{\ell}, \,\overline{a(\vv{\lambda}, \mu)} = \bar{a}}\mc{CB}_{\kappa}(\vv{\lambda}'_{\mu}, \nu^*)
  \boxtimes_K \iota^*(\mc{CB}_{\kappa}(\vv{\lambda}'', \mu^*)) .
  \end{equation}
  
\end{theorem}
Note that CB local systems are all pure, so the operator $\mr{N}$ is
zero on $\Psi_f\,(\mathcal{CB}_{\kappa}(\vec{\lambda},\nu^*)')$.
\begin{remark}\label{r:topological}
Theorem \ref{t:fact} has the following consequence in the topological
setting:

Consider the complex local system $\ml$ which is the Betti realisation
of the mixed local system
$\mathcal{KZ}_{\kappa}(\vec{\lambda},\nu^*)'$ on $W_A \sm Z_A$.  There
is an action of local monodromy (``around $Z_A$'') on
$\Psi_f (\ml) = \bigoplus_{\bar{a} \in (1/N)\Z/\Z} \Psi_{f, \bar{a}}
(\ml) $, which gives the conjugacy class of the local monodromy of the
local system $\ml$ on $W_A \sm Z_A$.  The semisimple part of the local
monodromy acts by the scalar $\exp(2\pi i \bar{a})$ on
$\Psi_{f, \bar{a}} (\ml)$ and the unipotent part acts trivially on the
corresponding topological graded $\rat(\mr{Gr}_a^{\mr{G}})$, where
$\mr{G}^{\bullet}$ is as in the statement of the theorem. This gives
upper bounds on the sizes of Jordan blocks of the local
monodromy. Note that the residue matrix of a logarithmic connection on
a punctured disc does not determine the monodromy of the corresponding
local system when there exist eigenvalues of the residue matrix differ
which by nonzero integers.

For the case of conformal blocks, the local monodromy is known to be
semisimple and it acts by multiplication by $\exp(2\pi i \bar{a})$ on
$\Psi_{f, \bar{a}}$.
\end{remark}

We will first prove the second part of the theorem, which is much
easier, in Section \ref{cbf} and then prove the first part of the
theorem in Section \ref{kzf}.

\subsection{Geometric preparation}\label{s:gprep}

We have a morphism $\mc{C}_n \to \ov{M}_{0,n+1}$ given by sending a
tuple $\vv{z} =(z_1,z_2,\dots,z_n)$ to the marked curve
$(\P^1; \vv{z}, \infty)$ (where we have fixed a coordinate $z$ on
$\P^1$). This induces a rational map $\mr{Bl}_A \to \ov{M}_{0,n+1}$
which by the properness of $\ov{M}_{0,n+1}$ is defined at the generic
point of $E_A$. We now describe this map explicitly.

Consider the constant curve $\pi: \mr{Bl}'_A \times \P^1 \to \mr{Bl}'_A$
with sections $\{\sigma_i\}_{i=1}^{n+1}$ given by the coordinate
functions $z_1,z_2,\dots,z_n$ and the constant section $\infty$. The
sections are disjoint over $\mc{C}_n \subset \mr{Bl}'_A$. Over
$E_A'$ the sections $\sigma_i$ for $i \in A$ are all equal and
disjoint from the other sections. Let $T_A$ be the image of
$\sigma_i|_{E_A'}$ for any $i \in A$; this is a smooth
codimension two closed subvariety of $ \P^1 \times \mr{Bl}'_A$. Let
$\pi_A:C_A \to \mr{Bl}_A'$ be the composite of the blowup of
$\P^1 \times \mr{Bl}'_A $ at $T_A$ with $\pi$; the map $\pi_A$ is
flat since $C_A$ is smooth and all the fibres are one
dimensional. Thus, $\pi_A^{-1}(\mc{C}_n) \simeq \P^1 \times \mc{C}_n$
whereas $\pi^{-1}(E_A')$ is naturally the union of two
irreducible components both isomorphic to $ E_A' \times \P^1$,
with $\sigma_{i_0}(E_A')$ glued fibrewise to
$ E_A' \times \{\infty\} $.

We now construct a locus in $\mr{Bl}_A'$ over which the sections
$\sigma_i$ induce sections $\tau_i$ of $\pi_A$ which disjoint images.
For our purposes, it will suffice to do this locally in $\mr{Bl}_A'$
so we now use the local coordinates that we have already introduced in
Section \ref{s:statement}. The sections $\sigma_i$ are given in these
coordinates by the functions $z_i$ for $i \notin A \sm I$ and
$(z_i - z_{i_0})\cdot (z_{i_1}- z_{i_0}) + z_{i_0}$ for
$i \in A \sm I$. Therefore, the subvariety $T_A$ corresponds to the
subvariety given by the equations $z - z_{i_0} = 0$ and
$z_{i_0} - z_{i_1} = 0$ in $\A^n_{\msk} \times \P^1_{\msk}$.  Its blowup can be
viewed as the subvariety of $\A^n_{\msk} \times \P^1_{\msk} \times \P^1_{\msk}$ given by
the closure of the rational map to $\P^1$ given by
\[
  (z_1,z_2,\dots,z_n, z) \mapsto [ z - z_{i_0}:z_{i_1} - z_{i_0}] .
\]
The exceptional divisor of the blowup is the locus
$\sigma_i(D_I) \times \P^1$ for any $i \in A$. This implies that the
fibre of the projection from the blowup to $\A^n_{\msk}$ maps isomorphically
to $\P^1$ via the first projection over the locus
$z_{i_0} \neq z_{i_1}$, whereas over the locus $z_{i_0} = z_{i_1}$ the
fibre has two components, both isomorphic to $\P^1$: the second $\P^1$
is the curve give by the image of
\[
  z' \mapsto (\vv{z}, z_{i_0}, [z',1])
\]
for any $\vv{z} = (z_1,z_2,\dots,z_n) \in D_I$.  We identify the new
component with $\P^1$ using the projection to the second factor. With
this identification, the point of intersection on the second component
is the point $[1:0] \in \P^1$.

One then sees that the sections $\sigma_i$ all extend to sections
$\tau_i$ of the blowup with disjoint images over all of $W_A$. We have
\begin{equation}\label{e:sections}
  \tau_i(\vv{z}) =
  \begin{cases}
    (\vv{z}, [z_i,1], [z_i - z_{i_0}: z_{i_1}
    - z_{i_0}]) & \mbox{ if } i \notin A \\
    (\vv{z}, [(z_i - z_{i_0})(z_{i_1}- z_{i_0}) +
    z_{i_0} ,1], [z_i - z_{i_0}: 1]) & \mbox{ if }
    i \in A \sm I \\
    (\vv{z}, [z_{i_0},1], [0: 1]) & \mbox{ if }
    i = i_0 \\
    (\vv{z}, [z_{i_1},1], [1: 1]) & \mbox{ if } i = i_1
  \end{cases}
\end{equation}
For $i$ in $A$, the sections $\tau_i|_{W_A}$ have image contained in
the new component in each fibre.

We have thus constructed an $(n+1)$-pointed genus zero stable curve $C_A$ over
$W_A$ which gives rise to a classifying morphism $c_A: W_A \to \ov{M}_{0,n+1}$.
\begin{lemma}\label{l:smooth}
  The morphism $c_A$ is smooth.
\end{lemma}
\begin{proof}
  It is easy to see that $c_A$ is dominant and smooth at all points in
  $W_A \sm Z_A$, so we prove smoothness at points in $Z_A$. Let $D_A$
  be the divisor in $\ov{M}_{0,n+1}$ corresponding to the partition
  $\{A, A^c\}$ of $[n+1]$. It is clear that $c_A(Z_A) \subset D_A$ and
  using the product decomposition of $D_A$ in terms of the above
  partition, one easily sees that $c_A|_{Z_A}: Z_A \to D_A$ is
  smooth. Now $c_A$ cannot be ramified along $Z_A$ since otherwise
  $C_A \simeq W_A \times_{\ov{M}_{0,n+1}}\ov{C}_{0,n+1}$ would not be
  smooth, where $\ov{C}_{0,n+1}$ is the universal curve over
  $\ov{M}_{0,n+1}$. It follows that $c_A$ is smooth.
\end{proof}

\subsection{Geometry of forgetful maps at the boundary}\label{fibers}

\subsubsection{}
For basic facts about the geometry of $\ov{M}_{0,N}$, the moduli space
of stable $N$-pointed curves of genus zero, that we use in this
section we refer the reader to \cite{keel}. It will be convenient for
us to allow markings by elements in arbitrary finite sets, so for any
finite set $S$ with $|S| \geq 3$ we have the moduli spaces
${M}_{0,S} \subset \ov{M}_{0,S}$ as in Section \ref{s:notations}.

Let $m,N$ be positive integers with $N \geq 4$ and $N-2 \geq m \geq 2$
and let $S$ be a finite set with $|S| = N$. Corresponding to
$A \subset S$ with $|A| = m$ we have a smooth irreducible divisor
$D_A \subset \ov{M}_{0,S}$ whose general point is a curve with two
irreducible components with all points marked by elements of $A$ lying
in one component and the points marked by elements in $A^c$, the
complement of $A$ in $S$, in the other component; $D_A = D_{A^c}$, and
it is canonically isomorphic to
$\ov{M}_{0,\ov{A}} \times \ov{M}_{0,\ov{A^c}}$ (see Section
  \ref{s:ind} for the notation).  The boundary
  $\ov{M}_{0,S} \sm {M}_{0,S}$ is a divisor with normal crossings,
  each irreducible component of which is a divisor $D_A$ as above.

For any integer $M \geq 0$, let $T$ be a finite set, disjoint from
$S$, with $|T| = M$. There is a flat forgetful map
$\pi: \overline{M}_{0,S \cup T}\to \overline{M}_{0,S}$ given by
forgetting the marked points corresponding to elements of $T$. The
inverse image $\pi^{-1}(D_A)$ is parametrised by the set of all
$B \subset T$: we have $\pi^{-1}(D_A) = \cup_{B}D_{A \cup B}$. The map
\begin{equation}\label{e:proj}
  D_{A \cup B} = \ov{M}_{0, \ov{A \cup B}} \times \ov{M}_{0, \ov{A^c
      \cup B^c}} \to \ov{M}_{0,\ov{A}} \times \ov{M}_{0,\ov{A^c}} =
  D_A
\end{equation}
is given by forgetting the marked points in $B$ and $B^c$; it is
smooth over the open subset $D_A^0 \subset D_A$ corresponding to curves
with exactly two irreducible components.

Let $W_A' = \ov{M}_{0,S} \sm \cup_{A' \neq A,A^c}D_{A'}$. Then
\[
  X_A' := \pi^{-1}(W_A') = \ov{M}_{0,S \cup T} \sm \cup_{C}' D_{C}
\]
where the restricted union is over all $C \subset S \cup T$ such that
$A \nsubseteq C$, $A \nsubseteq C^c$, and $1 < | C \cap S| < |S| - 1$.
We denote by $g_A':X_A' \to W_A'$ the map induced by $\pi$.  Let
$Z_A' = W_A' \cap D_A$, $Y_A' = (g_A')^{-1}(Z_A')$ and let
$U_A' = M_{0,S \cup T} \subset X_A'$. Let
$E_A' = W_A' \cap (\cup_{C}' D_C)$, where now the union runs over all
$C \subset S \cup T$ with $|C \cap S| \leq 1$.

Using the map $c_A: W_A \to \ov{M}_{0,n+1}$ (where $N = n+1$)  whose image
lies in $W_A'$, we define ``unprimed'' versions of all the data above
by fibre product with $W_A$,  e.g., we define
$X_A : = W_A \times_{W_A'}X_A'$, $Y_A ;= W_A \times_{W_A'} Y_A'$, etc.

The properties of these constructions are summarised in the following:
\begin{lemma}\label{l:map2}
  $ $
  \begin{enumerate}
  \item The varieties $X_A$, $W_A$ and $Z_A$ are all smooth.
  \item The divisor $Y_A \cup E_A$ is a reduced divisor with simple
    normal crossings and the intersection of any number of irreducible
    components of $E_A$ is flat over $W_A$ and smooth over
    $W_A \sm Z_A$.
  \item The irreducible components of $Y_A$ are non-empty open
    subsets of $D_{C}$ where $A\subset C$. Equivalently, they are
    parametrised by arbitrary subsets $B$ of $T$, the correspondence
    being given by $C \leftrightarrow A \cup B$. We denote these
    divisors by $Y_{A,B}$.
  \item The irreducible components of $E_A$ are non-empty open
    subsets of $D_C$ where $|C \cap S| \geq |S|-1$. We denote these
    divisors by $E_{A,C}$.
  \item The intersection of any number of irreducible components of
    $E_A$ with any $Y_{A,B}$ is smooth over $Z_A$.
  \end{enumerate}
\end{lemma}
\begin{proof}
  Using the smoothness of the map $c_A$, all the above statements can
  be deduced from exactly analogous statements for the ``primed''
  version of all varieties involved. Then parts (1) to (4) of the
  lemma are immediate from basic properties of $\ov{M}_{0,n+1}$
  (\cite{keel}).

  We now prove (5).  The condition $|C \cap S| \geq |S| -1$ implies
  that for all $k=1,2,\dots, m$ we have either $A \subset C_k$ or
  $A^c \subset C_k$.  This implies that there is a set $A_1$
  containing $\ov{A}$ and $A_2$ containing $\ov{A^c}$ and varieties
  $M_1$ and $M_2$, each of which is a product of a finite number of
  $\ov{M}_{0,r}$ for various $r$ such that that the map
  $D_{A \cup B} \cap (\cap_{k=1}^m D_{C_k}) \to D_A$ induced by
  \eqref{e:proj} can be written as a map
  \[
    (\ov{M}_{0,A_1} \times M_1) \times ( \ov{M}_{0,A_2} \times M_2)
    \lra \ov{M}_{0,A_1} \times \ov{M}_{0,A_2} \lra \ov{M}_{0,\ov{A}}
    \times \ov{M}_{0,\ov{A^c}},
  \]
  where the first map is a projection and the second map is given by
  the product of the forgetful maps induced by the inclusions $\ov{A}
  \subset A_1$ and $\ov{A^c} \subset A_2$. Since the forgetful maps
  are smooth over $M_{0,A}$ and $M_{0,A^c}$ and $M_1$ and $M_2$ are
  also smooth (5) follows.
\end{proof}
We note that the idea of using the moduli spaces $\overline{M}_{0,n}$
as the projective compactifications $P$ in Proposition \ref{p:kz}
already appears in \cite{L2}.

\subsubsection{}

It will be important for us to describe the data in Lemma \ref{l:map2}
explicitly using coordinates. To do this we recall that we already
have coordinates $z_i$, $i=1,\dots,n$ on $W_A$ and each tuple $\vv{z}$
for $\vv{z} \notin W_A$ corresponds to the $(n+1)$-pointed marked
curve $\P^1$ with marked points the $z_i$ and $\infty$. { Let $g: X_A
\to W_A$} be the morphism corresponding to forgetting the points marked
by $T$: a point in $g^{-1}(\vec{z}) \cap U_A$ corresponds to the curve
$\P^1$ with the marked points $z_1,\dots,z_n, \infty$ and additional
distinct marked points $t_1,t_2,\dots, t_M$ such that the sets
$\{z_i\} \cup \{\infty\}$ and $\{t_j\}$ are disjoint. It follows that
$g^{-1}(\vec{z}) \cap U_A$ is naturally isomorphic to the
complement of a hyperplane arrangement in $\A^M_{\msk}$ and $U_A$ itself can
be viewed as a family of such complements over $W_A \sm Z_A$ contained
in $\A^n_{\msk} \times \A^M_{\msk}$ (with coordinates the $z_i$'s and $t_j$'s). 

We now want to extend these coordinates to the boundary, i.e., over
$W_A$. For this we first describe in more detail the $Y_{A,B}$,
beginning with their open subsets $(Y_{A,B})_0 \sm E_A$ (where
$(Y_{A,B})_0$ is in the sense of Definition \ref{basicD}). To do this,
we recall that points of $Z_A$ correspond to stable $(n+1)$-pointed
rational curves with two irreducible components, both of which are
identified with $\P^1$. On one component the marked points are
$\{0,1,\infty\}$ and the $u_i = z_i - z_{i_0}$ for
$i \in A \sm \{i_0,i_1\}$. On the other component the marked points
are the $z_i$ for $i \in A^c$ together with $\infty$. These components
are glued together, with the point $\infty$ on the first component
glued to the point $z_0$ on the second. The points of
$(Y_{A,B})_0 \sm E_A$ mapping to a point $\vec{z}$ in $Z_A$ consist of
a curve as above together with extra markings: there are markings
corresponding to $\{t_i\}$ for $i \in B$ on the first component and
markings corresponding to $\{t_i\}$ for $i \notin B$ on the second
with all the markings being disjoint and disjoint from the point of
intersection. Using the identification
$W_A \simeq \mc{C}_{A^c \cup \{i_0\}} \times \mc{C}_{A \sm I}'$ used
earlier and the above discussion, and replacing $t_j$ by $s_j$ for
$i \in B$, we may identify $(Y_{A,B}) \sm E_A$ with the subset of
$\mc{C}_{A \sm I}' \times \mc{C}_{B}' \times \mc{C}_{A^c \cup \{i_0\}}
\times \mc{C}_{B^c}$ given as follows:
\begin{equation}\label{e:decomp}
  (Y_{A,B})_0 \sm E_A =  \bigl \{ ( \vv{u}, \vv{s}, \vv{z}, \vv{t}\;)
  \mid
  \{u_i\}_{i \in A \sm \{i_0,i_1\}} \cap \{s_j\}_{j
    \in B} = \emptyset
  ,
  \{z_i\}_{i \in A^c \{i_0\}} \cap \{t_j\}_{j \in B^c} = \emptyset \bigr \}.
\end{equation}
We may therefore view $(Y_{A,B})_0 \sm E_A$ as corresponding to a product of
two families of hyperplane arrangements, the first in $\A^{B}_{\msk}$ and the
second in $\A^{B^c}_{\msk}$, the first family parametrised by
$\mc{C}_{A \sm I}'$ and the second by $\mc{C}_{A^c \cup \{i_0\}}$.

\subsection{Factorisation of the master function}\label{s:fmf}

Fix a split simple Lie algebra $\mf{g}$ and $\kappa$ in $\Q^{\times}$
and let $\vv{\lambda}$, $\nu$, $\beta$, $M$, $N$ etc., be as in
Section \ref{s:sv} and let $K = \Q(\mu_N)$. Given our choice of coordinates
on $U_A$, the master function
\[
\ms{R} = \prod_{1 \leq i < j \leq n} (z_i -
z_j)^{\frac{(\lambda_i, \lambda_j)}{\kappa}} \prod_{b=1}^M
\prod_{j=1}^n (t_b - z_j)^{\frac{(\lambda_j, \beta(b))}{\kappa}}
\prod_{1 \leq b < c \leq M} (t_b - t_c)^{-\frac{(\beta(b),
		\beta(c))}{\kappa}}
\]
from \eqref{e:master} can be viewed as a multivalued function on $U_A$
and we let $\eta$ be the log form $d\ms{R}/\ms{R}$ (note that the
$z_i$'s are now variables). Let $\mc{L}(\eta)$ be the $K$-mixed local
system in $\msm(U_A)_K$ defined as in Section \ref{s:kzmls}.

The group $\Sigma_{\beta}$ permuting the $t_b$ variables of the same
color acts on $X_A$ and the map $g_A:X_A\to W_A$ is
$\Sigma_{\beta}$-equivariant, where the action on $W_A$ is
trivial. The action preserves $D_A$ and the form $\eta$ and
$\mc{L}(\eta)$ has a natural $\Sigma_{\beta}$-linearisation.

\begin{defi}
  For $a\in \mb{R}$, let $Y^a \subseteq Y_A$ be the union of
  irreducible components $Y_{A,B}$ of $Y_A$ such that the residue of
  $\eta$ along $Y_{A,B}$ is equal to $a$.  We also have open subsets
  $Y^a_0$ of $Y_A$ (as in Definition \ref{basicD}).
\end{defi}
The group $\Sigma_{\beta}$ preserves $Y^a$ (since it preserves
$\eta$), but it may permute its irreducible components.

\begin{defi}\label{whyi}
  $ $
  \begin{enumerate}
  \item Let $\mathscr{I}$ be the indexing set for the irreducible
    components of $Y_A$. By Lemma \ref{l:map2} it is the set of all
    subsets of $[M]$
  \item Let $\mathscr{O}\subset \mathscr{I}$ be a set of orbit
    representatives for the action of $\Sigma_{\beta}$.
  \item For $B \in \ms{I}$:
    \begin{enumerate}
    \item Let $j_B:(Y_{A,B})_0\to Y_A$ be the (open) inclusion of
      $(Y_{A,B})_0$ in $Y_A$, $j'_B$ the (open) inclusion of $(Y_{A,B})_0$ in
      $Y_{A,B}$, and $\iota_B: Y_{A,B} \to Y_A$ the (closed) inclusion.
    \item Let $\Sigma_B$ be the subgroup of $\Sigma_M$ preserving $Y_{A,B}$,
      and $\chi_B$ the restriction of $\chi$ to $\Sigma_B$.
    \end{enumerate}
	\end{enumerate}
	Note that $B, B' \in \ms{I}$ have the same orbit under
        $\Sigma_{\beta}$ if and only we have
        $\sum_{b\in B}\beta(b)=\sum_{b\in B'}\beta(b)$.
\end{defi}

\subsubsection{}\label{productlocal}

Let $Y_{A,B}$ be an irreducible component of $Y$ and let
\begin{equation*}
	\mu= \sum_{i\in A}\lambda_i-\sum_{j\in B}\beta(j)\, .
\end{equation*} 
Note that $\mu$ is an integral weight which is not necessarily
dominant.  Let $a=a(B)$ be the residue of $\eta$ along $Y_{A,B}$.
This is equal to $a(\vv{\lambda}, \mu)$ as in \ref{formula:a} with
$\mu$ as above. Then $\ml(\eta)\tensor_K K f^{a}$ extends as a local
system $\ml'$ in a neighbourhood of $U_B=(Y_{A,B})_0-E_A$: this
follows from the construction of $\mc{L}(\eta)$ since $(\ms{R}f^a)^N$
is a regular function in a neighbourhood of $U_B$. Now $U_B$ is a
product, and our aim in this section is to write $\ml'$ as an external
tensor product of $K$-mixed local systems.

The scheme $U_B$ is a product $U_B^1\times U_B^2$, where $U_B^1$ corresponds to the configuration space of the $t$ variables labelled by $B$ and the $z$ variables labelled by $A$, and $U_B^2$ corresponds to the configuration space of the remaining variables.
 \begin{itemize}
 \item[--] Let $\vv{\lambda}'_{\mu}$ be the set of weights labelled by
   ${A^c} \cup \{i_0\}$ where for $i$ in $A^c$ the weight assigned to
   $i$ is $\lambda_i$ and for $i_0$ the weight is $\mu$.
  \item[--] Let $\vv{\lambda}''$ be the set of weights labelled by points
    in $A$ given by restriction of $\vv{\lambda}$ to $A$.
  \end{itemize}
We get a $K$-mixed local system on $U_B^1$ denoted by
$\ml_1$ corresponding to the master function with with $z_{i_0}=0$ and
$z_{i_1}=1$,
\begin{equation*}
	\ms{R}_1 = \prod_{i<j, i,j\in A} (z_i -
	z_j)^{-\frac{(\lambda_i, \lambda_j)}{\kappa}} \prod_{b\in B}
	\prod_{j\in A} (t_b - z_j)^{\frac{(\lambda_j, \beta(b))}{\kappa}}
	\prod_{b<c, b,c\in B} (t_b - t_c)^{-\frac{(\beta(b),
			\beta(c))}{\kappa}}
\end{equation*}
(If $\mu\in P^+$ is dominant integral, then this corresponds to (after reindexing of the $t$ variables)  the master function for $\mc{KZ}_{\kappa}(\vv{\lambda}'', \mu^*)$.

We also get a $K$-mixed local system on $U_B^2$ denoted by $\ml_2$
corresponding to the master function
\begin{equation*}%
  \ms{R}_2 = \prod_{ i < j, i,j
    \in\ov{A^c}} (z_i - z_j)^{\frac{-(\lambda'_i,
      \lambda'_j)}{\kappa}} \prod_{b\in B^c} \prod_{j\in \ov{A^c}}
  (t_b - z_j)^{\frac{(\lambda'_j, \beta(b))}{\kappa}} \prod_{b < c,
    b,c\in B^c } (t_b - t_c)^{-\frac{(\beta(b), \beta(c))}{\kappa}}
\end{equation*}
(If $\mu$ is dominant integral, then this corresponds to (after reindexing of the $t$ variables)  the master function for $\mc{KZ}_{\kappa}(\vv{\lambda}'_{\mu}, \nu^*)$.)

The following is related to \cite[Theorem 17.3]{BFS}.
\begin{lemma}\label{product11}
  The restriction of $\ml'$ to $U_B$ equals $\ml_1\boxtimes_K
  \ml_2$. Note that $\Sigma_B$ is a corresponding product of groups
  (and $\chi_B$ a product of characters).
\end{lemma}
\begin{proof}
  We need to show that the master function \eqref{e:master} times
  $f^{-a}$ extends to a function on $U_B$, and when we set $f=0$, the
  extended multi-valued function equals the product of master
  functions for $\ml_1$ and $\ml_2$.

  We use the local coordinates for $U_B$ for this computation. Let
  $(t_1,\dots,t_M,z_1,\dots,z_n)$ be the coordinates on $U_B$.  The
  rational map to the configuration space (when $f\neq 0$), takes the
  point $(t_1,\dots,t_M,z_1,\dots,z_n)$ to
  $(t'_1,\dots,t'_M,z'_1,\dots,z'_n)$, where
  $t'_a=(t_a-z_{i_0})f +z_{i_0}$ if $a\in A$, and $t'_a=t_a$
  otherwise. Similarly $z'_{i}=z_{i}$, if $i\in A^c\cup \{i_0,i_1\}$
  and $z'_i=(z_i-z_{i_0})f +z_{i_0}$ otherwise.

  Note that $t_a-t'_b=t_a-t_b$ if $a,b\not\in B$.  If $a,b\in B$ then
  $t'_a-t'_b=f(t_a-t_b)$. Similarly
  $t'_a-t'_b =f(t_a-z_{i_0}) +z_{i_0}-t_b$ if $a\in A$ and
  $b\not\in B$, which restricts to $z_{i_0}-t_b$ when $f=0$. We handle
  the other difference $z'_i-t'_b$ and $z'_i-z'_j$ in a similar
  manner. The power of $f$ that pulls out has to equal the residue by
  general considerations of residues of logarithmic derivatives. Using
  the above formulas, the proposition follows easily from Lemma
  \ref{l:prod}.
\end{proof}

\section{Factorisation for conformal block local systems} \label{cbf}
We continue with the setup of Section \ref{s:fmf}, but now assume
$\kappa=\ell+h^{\vee}$ where $\ell\in \Bbb{Z}_{>0}$. We also assume
that $\lambda_1,\dots,\lambda_n$, and $\nu$ are of level $\ell$. Our
aim in this section is to prove the conformal blocks part of Theorem
\ref{t:fact}. 

\subsection{A direct sum decomposition for nearby cycles}\label{s:dsd}

We fix $A \subset [n]$ with $|A| \geq 2$ and for the rest of this
section we drop the subscript $A$ from all the notation of Section
\ref{s:fmf}; in particular, $Y_B$ in this section will denote the
divisor $Y_{A,B}$ of Section \ref{s:fmf}. We also write $\mc{L}$ for
$\mc{L}(\eta)$.

Let $j:U\to X$ be the inclusion. We have a map $g:X\to W$ and a
regular function $f:W\to \A^1_{\msk}$.  We let $f_X = f \circ g$.

By the constructions in Sections \ref{s:KZCB} and \ref{s:cbmls}, the
$(\Sigma,\chi)$ isotypical component of the image of
$H^M(g_*j_!\ml) \to H^M(g_* j_*\ml)$ restricted to $W \sm Z$ is the
$K$-mixed local system of conformal blocks
$\mathcal{CB}_{\kappa}(\vec{\lambda},\nu^*)$ restricted to $W\sm Z$.

Let
$$\mc{F}=\Psi_{f_X}\, (j_{!}\ml), \ \ \mc{G}=\Psi_{f_X}\, (j_*\ml)\in \db(\msm(Y))_K;$$
by functoriality of $\Psi_{f_X}$ we have an induced map
$\mc{F} \to \mc{G}$.

Let $g_Y:Y\to Z$ be the restriction of $g$ to $Y$.  The image of
\begin{equation}\label{eq:image}
  H^M(g_{Y,*}\mc{F})^{\Sigma,\chi}\lra  H^M(g_{Y,*}\mc{G})^{\Sigma,\chi} ,
\end{equation}
which is a $K$-mixed local system on $Z$, is the local system of
nearby cycles of the conformal blocks, i.e.,
$\Psi_f(\mathcal{CB}_{\kappa}(\vec{\lambda},\nu^*))$.
\smallskip

{ For $\bar{c} \in (1/N)\Z/\Z$, we define 
$$\mc{F}_{\bar{c}}=\Psi_{f_X,\bar{c}} (j_{!}\ml), \ \ \mc{G}_{\bar{c}}=\Psi_{f_X,\bar{c}} (j_*\ml)\in \db(\msm(Y))_K;$$
by functoriality of $\Psi_{f_X}$ we have an induced map
$\mc{F}_{\bar{c}} \to \mc{G}_{\bar{c}}$. The image of
\begin{equation}\label{eq:imagea}
  H^M(g_{Y,*}\mc{F}_{\bar{c}})^{\Sigma,\chi}\lra  H^M(g_{Y,*}\mc{G}_{\bar{c}})^{\Sigma,\chi},
\end{equation}
which is a $K$-mixed mixed local system on $Z$, is equal to
$\Psi_{f,\bar{c}}(\mathcal{CB}_{\kappa}(\vec{\lambda},\nu^*))$ (since
$\Psi_{f,\bar{c}}$ is an exact functor, $g$ is proper and
\cite[Proposition 5.10]{saito-form}), a direct summand of
$\Psi_f(\mathcal{CB}_{\kappa}(\vec{\lambda},\nu^*))$.  For
$B\in \mathscr{I}$ (see Definition \ref{whyi}), let
$\mc{F}_{B,\bar{c}}:=j_B^*\mc{F}_{\bar{c}}$ and
$\mc{G}_{B,\bar{c}}:=j_B^*\mc{G}_{\bar{c}}$ in $\db(\msm((Y_B)_0))_K$.
Let $g_B:Y_B\to Z$ be the restriction of $g:X\to W$. Also recall from
Definition \ref{whyi} (in which we suppress all occurrences of $A$)
that $j_B:(Y_{B})_0\to Y$, $j'_B:(Y_{B})_0 \to Y_B$ and
$\iota_B: Y_{B} \to Y$ are the inclusion maps.

Using Lemma \ref{quotientie} below we obtain the following:
\begin{lemma}\label{l:comm} The diagram below commutes:
  \begin{equation}\label{diagramcb}
    \xymatrix{ 
      \bigoplus_{B\in\mci} H^M(g_{B,*} j'_{B,!} \mc{F}_{B,\bar{c}})\ar[r]\ar[d] & H^M(g_{Y,*}\mc{F}_{\bar{c}}\ar[d])\\
      \bigoplus_{B\in\mci} H^M (g_{B,*} j'_{B,*}{\mc{G}}_{B,\bar{c}}) & H^M(g_{Y,*}\mc{G}_{\bar{c}})\ar[l]
    } 
  \end{equation}
  The top horizontal map is induced by
  $\iota_{B,!} j_{B,!}{\mc{F}}_{B,\bar{c}}\to \mc{F}_{\bar{c}}$ (since
  $\iota_B\circ j_B'$ is $j_B$ and $i_{B,!} = i_{B,*}$), and
  analogously for the bottom horizontal map.
	
  Consequently, the image of
  $\oplus_{B \in \mci} H^M (g_{B,*} j'_{B,!} {\mc{F}}_{B,\bar{c}})\to
  \oplus_{B\in\mci} H^M (g_{B,*} j'_{B,*}{\mc{G}}_{B,\bar{c}})$ is a
  subquotient of the image of
  $H^M(g_{Y,*}\mc{F}_{\bar{c}})\to H^M(g_{Y,*}\mc{G}_{\bar{c}})$.
\end{lemma}

We may take isotypical components for the action of
$\Sigma=\Sigma_{\beta}$ on the diagram \eqref{diagramcb}. Using
equalities of the form
$$\bigl(\bigoplus_{B\in \mci} H^M(g_{B,*} j'_{B,!} \mc{F}_{B,\bar{c}}) \bigr)^{\Sigma,\chi}=\bigoplus_{B\in \mco} H^M(g_{B,*} j'_{B,!} \mc{F}_{B,\bar{c}})^{\Sigma_B,\chi_B}$$
(also for direct sum of the cohomologies of $j'_{B,*}$), we get a
commutative diagram  of $K$-mixed local systems on $W$:
\begin{equation}\label{diagramcb2}
  \xymatrix{ 
    \bigoplus_{B\in\mci} H^M(g_{B,*} j'_{B,!} \mc{F}_{B,\bar{c}})^{\Sigma_B,\chi_B}\ar[r]\ar[d] & H^M(g_{Y,*}\mc{F}_{\bar{c}})^{\Sigma,\chi}\ar[d]\\
    \bigoplus_{B\in\mci} H^M( g_{B,*} j'_{B,*}{\mc{G}}_{B,\bar{c}})^{\Sigma_B,\chi_B} & H^M(g_{Y,*}\mc{G}_{\bar{c}})^{\Sigma,\chi}\ar[l]
  } 
\end{equation}
The image of the vertical map on the RHS of \eqref{diagramcb2} summed
over all $\bar{c}$ {is a direct summand} of
$\Psi_f(\mathcal{CB}_{\kappa}(\vec{\lambda},\nu^*))$.  We will show
that the image of each direct summand (summed over all $\bar{c}$ on
the vertical map on the LHS of \eqref{diagramcb2} includes all terms
in the numerical factorisation of conformal blocks.  By applying the
$\rat_W$ functor, this will show (see Lemma \ref{subquotient}) that
the subquotient is actually full, and prove the factorisation for
conformal block local systems \eqref{e:cbfact}.

\subsection{Factorisation, and conclusion of the proof of Theorem \ref{t:fact} for conformal blocks}\label{parameteri}

By the classical factorisation formula for conformal blocks in
\cite{TUY}, the rank of $\mathcal{CB}_{\kappa}(\vv{\lambda},\nu^*)$
equals (using notation from Theorem \ref{t:fact})
\begin{equation}\label{summand0}
  \sum_{\mu\in P_{\ell}}\rk \bigl(  \mc{CB}_{\kappa}(\vv{\lambda}'', \mu^*)
  \tensor_K 
  \mc{CB}_{\kappa}(\vv{\lambda}'_{\mu}, \nu^*)\big) .
\end{equation}

The rank of $\mathcal{CB}_{\kappa}(\vv{\lambda},\nu^*)$ equals the rank
of the image of \eqref{eq:image}. Assume that the summand in
\eqref{summand0} for $\mu$ is nonzero. This implies that
$\mu+\sum_{i\in A^c}\lambda_i -\nu$ and $\sum_{i\in A}\lambda_i -\mu$
are non-negative integer combinations of simple roots: In fact it is
easy to see that there exists a unique $B\in\mathscr{O}$ such that
\begin{equation}\label{e111}
  \mu+\sum_{i\in A^c}\lambda_i -\nu =\sum_{b\in B^c}\beta(b),
\end{equation}
\begin{equation}\label{e222}
	\sum_{i\in A}\lambda_i -\mu=\sum_{b\in B}\beta(b) .
\end{equation} 
This gives a bijection $\mu\mapsto B$ from the set of dominant weights
contributing nonzero terms in \eqref{summand0} to a subset of
$\mathscr{O}$.

Let $c=a(B)$ be the residue of $\eta$ along $Y_B$. We claim that  $\Psi_{f,\bar{c}}$
of the image of
\begin{equation}\label{newS}
  H^M(g_{B,*} j'_{B,!} \mc{F}_B)^{\Sigma_B,\chi_B}\lra   H^M( g_{B,*} Rj'_{B,*}{\mc{G}}_B)^{\Sigma_B,\chi_B}
\end{equation}
for $B$ corresponding to $\mu$, which equals the image of the map
\eqref{eq:imagea}, is the same as the corresponding term in
\eqref{e:cbfact}, and hence prove the conformal blocks part of Theorem  \ref{t:fact}
by dimension counting and the K\"unneth isomorphism (Lemma
\ref{l:kf}).

By the same argument as in Section
\ref{productlocal},
$\ml\tensor_K K f^{-a(\vv{\lambda},\mu)}$ extends as a local system
$\ml'$ in a neighbourhood of $U_B=U_B^1\times U_B^2$. Note that $Y_B$
is also a corresponding product of smooth varieties {
  $Y_B = P^1_B \times P^2_B$, as follows from \eqref{e:proj}.}  Lemma
\ref{product11} shows that the restriction of $\ml'$ to $U_B$ equals
$\ml_1\boxtimes_K \ml_2$. Note that $\Sigma_B$ is a corresponding
product of groups (and $\chi_B$ a product of characters).

By Lemma \ref{l:ls} the objects
$\mc{F}_{B,\bar{c}},\mc{G}_{B,\bar{c}} \in \db(\msm((Y_B)_0))_K$
restricted to $U_B$ are isomorphic to { the restriction of} the
$K$-mixed local system $\ml'$. We claim that
$ j'_{B,!} \mc{F}_{B,\bar{c}}$ on $Y_B$ is the lower shriek push
forward of $\mc{F}_{B,\bar{c}}$ from $U_B$ where it is isomorphic to
$\ml'$.  To see this ({which can be checked after taking the Betti
  realisation}), we only need to { show that the stalks at all points
  of $(Y_B)_0\cap E$ are zero.} The claim follows from Lemmas
\ref{kunneth} and \ref{KunnethF}
by the local form of the map $f:X\to Z$ together with the fact that
$\mc{F}$ is the nearby cycles of an object defined by lower shriek at
points of $E$. Similarly $j_{B,*}\mc{G}_{B,\bar{c}}$ on $Y_B$ is the
lower star push forward from $U_B$.  Therefore the image of
\eqref{eq:imagea} is the same by Lemmas \ref{product11} and \ref{l:kf}
as a tensor product of images (since $Y_B$ is a product), and
coincides with the corresponding term in \eqref{e:cbfact}.

As noted in the statement of Theorem \ref{t:fact}, conformal block
local systems are all pure, so the operator $\mr{N}$ is zero on
$\Psi_f\,(\mathcal{CB}_{\kappa}(\vec{\lambda},\nu^*)')$. We give a
different proof here which has a generalisation to the case of
invariants. The key point is to use \eqref{diagramcb} and show that
the action of $\mr{N}$ on $ j'_{B,!} \mc{F}_{B,\bar{c}}$ and
$j_{B,*}\mc{G}_{B,\bar{c}}$ is zero. Now as observed above,
$ j'_{B,!} \mc{F}_{B,\bar{c}}$ on $Y_B$ is the shriek push forward of
$\mc{F}_{B,\bar{c}}$ from $U_B$. By adjunction $\mr{N}$ is determined
by its restriction to $U_B$ where it is zero because the map
$f_X:X\to \A^1_{\msk}$ is smooth on $U_B$ and the nearby cycles sheaf
$\mc{F}_{B,\bar{c}}$ on $U_B$ is the nearby cycles sheaf of a local
system (see Lemma \ref{l:ls}). The argument for the action of $\mr{N}$
on $j_{B,*}\mc{G}_{B,\bar{c}}$ is similar.

This concludes the proof of the conformal blocks part of Theorem
\ref{t:fact}.  \qed

\begin{remark}
  The intersections of the components $Y_B\cap Y_{B'}$ played no role
  in the proof of factorisation for conformal blocks.
\end{remark}

\subsection{Some elementary lemmas}\label{s:elementary}

\subsubsection{}
Let $K$ be any field.
\begin{lemma}\label{LA1}
  If $A\to B\to C$ is exact at $B$, where $A$, $B$ and $C$ are
  finite dimensional $K$-vector spaces, then
  $\rk B \leq \rk A +\rk C$. If equality holds then
  $$0\lra  A\lra  B\lra   C\lra  0$$ is an exact sequence.
\end{lemma}

\begin{lemma}\label{LA2}
  Suppose
  \begin{equation*}
    \xymatrix{
      0\ar[r] & A'\ar[r]\ar[d]^f & B'\ar[r]\ar[d]^g & C'\ar[r]\ar[d]^h & 0\\
      0\ar[r] & A\ar[r] & B\ar[r] & C\ar[r] & 0 
    } 
  \end{equation*}
  is a commutative diagram of finite dimensional $K$-vector spaces
  with exact rows.  Then
  \[
    \rk(\im g)\geq \rk(\im f)+\rk(\im h)
  \]
  and if equality holds we have an exact sequence
  $0\to \im(f)\to \im(g)\to \im(h)\to 0$.
\end{lemma}

\begin{proof} 
  Clearly $\im(f)$ is a subspace of $\im(g)$, and the quotient
  $\im(g)/\im(f)$ surjects onto $\im(h)$.
\end{proof}

\begin{lemma}\label{subquotient}
  Consider a commutative diagram of finite dimensional $K$-vector
  spaces of the form
  $$
  \xymatrix{A'\ar[r]\ar[d]^{f'} &  A\ar[d]^f\\
    B' & B\ar[l]
  }
  $$
  Then $\im(f')$ is a subquotient of $im(f)$. More precisely $\im(f)$
  has a subspace $S\subseteq \im(f)$ given by the image of $A'$.
  There is a surjective map $S\to \im(f')$. Therefore, if $\im(f)$
  and $im(f')$ are vector spaces of the same rank, then $S=\im(f)$ and
  $S$ maps injectively, hence isomorphically onto $\im(f')$.
\end{lemma}

\subsubsection{}
Let $\msk$ be a subfield of $\C$ and $K/\Q$ a finite extension.

Let $Y$ be a variety over $\msk$ and $U_i\subset Y$, $i\in I$, a
collection of disjoint open subsets of $Y$. Let $Y_i$ be closed
subsets in $Y$ such that $Y\supset U_i$ (we will have equalities
$Y_i=\overline{U_i}$).  Let $\mc{F}, \mc{G}\in \db(\msm(Y))_K$ with a
map $\mc{F}\to \mc{G}$.  Let $k_i:U_i\to Y_i$ and $k_i':Y_i\to Y$
be the inclusions and let $h_i=k'_i\circ k_i: U_i\to Y$. Let
$g:Y\to Z$ be a morphism of varieties over $k$ and let $g_i:Y_i\to Z$
denote the restrictions.  Let
$\mc{F}_i= k_{i,!}(h_i^*\mc{F})\in \db(\msm(Y_i))_K$ and
$\mc{G}_i= k_{i,*}(h_i^*\mc{G})\in \db(\msm(Y_i))_K$. There are
natural morphisms $\mc{F}_i\to \mc{G}_i$.
\begin{lemma}\label{quotientie}
  For any integer $M$, there is a commutative diagram (where the
  vertical arrow on the left is a direct sum of maps)
  \begin{equation}\label{subbie}
    \xymatrix{ 
      \bigoplus_{i} H^M(g_*\mc{F}_i)\ar[r]\ar[d] & H^M(g_*\mc{F})\ar[d]\\
      \bigoplus_{i} H^M(g_* \mc{G}_i) & H^M(g_*\mc{G})\ar[l]  
    } 
  \end{equation}
\end{lemma}

\begin{proof}
  First note that $k'_{i,*} =k'_{i,!}$ . Now there are natural maps
  (using adjunction) $k'_{i,*}\mc{F}_i\to \mc{F}$ as well as
  $\mc{G}\to k'_{j,*}\mc{G}_j$.  This sets up the morphisms in the
  diagram. For the commutativity of \eqref{subbie}, we need to show
  that the composites $k'_{i,*}\mc{G}_i\to k'_{j,*}\mc{F}_j$ are
  zero if $i\neq j$ and are the tautological maps if $i=j$. Now we use
  $k'_{i,*}\mc{G}_i= h_{i,!}h_i^*\mc{G}$ and
  $k'_{j,*}\mc{G}_j= h_{j,*}h_j^*\mc{G}$, and adjunction
  $$\home(h_{i,!}h_i^*\mc{G}, h_{j,*}h_j^*\mc{F})= \home(h_j^*h_{i,!}h_i^*\mc{G}, h_j^*\mc{F}).$$
  The claims then follow from the fact that $h_j^*h_{i,!}$ is the zero
  functor if $i\neq j$ and is the identity functor if $i=j$: this can
  be checked after applying $\rat$, in which case both statements are
  elementary (and well-known).
\end{proof}

\section{Filtrations on nearby cycles and de Rham
  complexes} \label{s:mnb} Let $f: X\to S$ be a map of complex
manifolds with $0 \in S \subset \C$. Let $Y=f^{-1}(0)$ and let
$E\subset X$ be a divisor all of whose components are flat over $S$
and smooth over $S^*=S \setminus \{0\}$ and. Assume that
$D = E \cup Y$ is a reduced divisor with normal crossings.  Let
$\eta\in \Omega^1_X(\log D)$ where $D=E\cup Y$.  In preparation for
the proof of Theorem \ref{t:fact} in Section \ref{kzf}, we define, in
Section \ref{s:A}, a filtration on nearby cycles of complexes defined
by $\eta$ on $X^*=X\setminus Y$. This corresponds to a filtration of
$Y$ by unions of its irreducible components ordered by the residues of
the form $\eta$ along them.

The graded quotients for this filtration have a de Rham description by
complexes $A(\eta)^a$ (Lemma \ref{comparison}). The de Rham type
complexes $A(\eta)$ are defined and then we prove the decomposition of
$A(\eta)$ as a direct sum of pushforwards of $A(\eta)^a$ in Section
\ref{s:Aeta}.

We also study some natural maps in this setting induced by Verdier
duality and replacing $\eta$ by $-\eta$.

\subsection{Relative log de Rham complexes}\label{s:Aeta}

\begin{defi} \label{EZee} (See \cite[p.~155]{EZ})
  $\Omega^{1}_{X/S}(\log D)$ is defined to be the locally free
  quotient $\Omega^1_X(\log D)/f^*(\Omega^1_S(\log(0)))$. Let
  $\Omega^p_{X/S}(\log D)=\wedge^p \,\Omega^{1}_{X/S}(\log D)$ for
  $p\geq 0$. The differential $d$ in the log complex
  $(\Omega^{\bull}_{X}(\log D),d)$ induces a complex
  $(\Omega^{\bull}_{X/S}(\log D),d)$.
\end{defi}

\begin{defi}\label{def11}
  Let
  $A(\eta)=(\Omega^{\bull}_{X/S}(\log D)\tensor_{\mc{O}_X} \mathcal{O}_Y, d+\eta
  \wedge)$.
\end{defi}

For $a\in \mb{R}$, let $Y^a$ be the union of components $Y_i$ of $Y$
such that the residue of $\eta$ along $Y_i$ is $a$.  The complexes
$A(\eta)$ can be ``restricted'' to the components $Y^a$:
\begin{defi}\label{d22}
	Let $A(\eta)^a$ be the complexes $(\Omega^{\bull}_{X/S}(\log D)\tensor_{\mc{O}_X} \mathcal{O}_{Y^a}, d+\eta)\in D^b(Y^a)$.
\end{defi}
\begin{remark}
  Note that the differentials in
  $A(\eta)=(\Omega^{\bull}_{X/S}(\log D), d+\eta\wedge)$ are not function
  linear, but $(d+\eta\wedge)$ takes
  $\Omega^{p}_{X/S}(\log D)\otimes_{\mc{O}_X}\mathcal{I}$ to
  $\Omega^{p+1}_{X/S}(\log D)\otimes_{\mc{O}_X}\mc{I}$ for $\mathcal{I}$ the
  ideal sheaf of any union of irreducible components of $Y$ (using
  identities of the form $dy=y dy/y$)
\end{remark}

There are natural morphisms of complexes
$A(\eta)\to  i^a_{*}A(\eta)^a$, where $i^a:Y^a\to  Y$ is the inclusion.
\begin{theorem}\label{locale}
  The natural morphism
  $$A(\eta)\lra  \bigoplus_{a\in\mb{C}} i^a_{*} A(\eta)^a$$ is an
  isomorphism in $\dpl(Y)$.
\end{theorem}
The theorem follows from a computation of stalks of $A(\eta)$ at any
$P\in Y$ (Proposition \ref{prop:local1}) as well as the stalks of
$A(\eta)^a$ when $P\in Y^a$ (Proposition \ref{prop:local2}) that we
carry out in Section \ref{localo}. The stalks of $A(\eta)$ when
$\eta=0$ are computed in \cite{steenbrink}.

\subsection{Local cohomology of $A(\eta)$ and $A(\eta)^a$
}\label{localo}

Let $P\in Y \subset X$ with a coordinate system of the form
$(z_1,\dots,z_k; u_1,\dots;u_{\ell};v_1,\dots,v_m)$, so that locally
$X=\dc^{k+\ell+m}$ and $S=\dc$.  The map
$f:\dc^{k+\ell+m} \rightarrow \dc$ is given by the formula
$$f(z_1,\dots,z_k; u_1,\dots;u_{\ell};v_1,\dots,v_m)=z_1\cdots z_{k} .$$
Then $Y = f^{-1}(0)$ and let $P = 0 \in Y$.
\begin{enumerate}
\item Let $a_i=\op{Res}_{z_i=0}\eta$ for $1\leq i\leq k$.
\item Let $E$ be the union of the coordinate hyperplanes $u_i=0$ and
  let $b_j=\op{Res}_{u_j=0}\eta$ for $1\leq j\leq \ell$.
\end{enumerate}
We may also assume
$$\eta=\sum a_i dz_i/z_i +\sum b_j du_j/u_j.$$ It
is useful to note that if $h\in\mathcal{O}_Y$ is a function such that
$dh+h\eta=0\in \Omega^1_{X/S}(\log D)\tensor_{\mc{O}_X} \mathcal{O}_Y$
and $\omega\in\Omega^p_{X/S}(\log D)\tensor_{\mc{O}_X} \mathcal{O}_Y$
with $d\omega=0$, then $h\omega$ is killed by the operator
$d+\eta\wedge$.
\begin{proposition}\label{prop:local1}
  The stalks $\mathscr{H}^q(A(\eta))_{P}$ of the cohomology sheaves at
  the point $P$ are trivial unless there exists $i_0$ such that
  $a_{i_0} - a_i \in \mb{Z}_{\geq 0}$ and $b_j\in \mb{Z}_{\leq 0}$ for
  all $i,j$.  In this case, assuming without loss of generality that
  $i_0 = 1$, we have:
  \begin{enumerate}
  \item
    $\mathscr{H}^0(A(\eta))_P$ is generated by
    $h=z_2^{a_1-a_2}\dots z_k^{a_1-a_k}u_1^{-b_1}\dots
    u_{\ell}^{-b_\ell}$.
  \item
    $\mathscr{H}^q(A(\eta))_P$ { is the $\C$-vector space
      generated by $h$ times the $q$th exterior power of the vector
      space generated by $dz_i/z_i$ (with sum zero) and $du_j/u_j$}.
  \end{enumerate}
\end{proposition}

\begin{proof}
  By a Koszul argument, we reduce to the case $m=0$. Let
  $$A=\mb{C}\{z_1,\dots, z_k; u_1,\dots, u_\ell\}/(z_1z_2\cdots
  z_k); $$ the ring $\C\{z_1,\dots, z_k; u_1,\dots, u_\ell\}$ is the
  analytic local ring of $P\in X$ and the quotient ring $A$ the analytic
  local ring of $P\in Y$.  Note that $A$ has a monomial basis in each
  degree.  Let
  $$\xi_1=dz_1/z_1,\  \dots,\  \xi_k=dz_k/z_k;\ \ \xi_{k+1}=du_1/u_1,\
  \dots,\ \xi_{k+\ell}=du_{\ell}/u_{\ell} . $$ It follows that
  $\Omega^1_{X/S}(\log D)\tensor_{\mc{O}_X} \mathcal{O}_Y$ is generated by
  $\xi_1,\dots,\xi_{k+\ell}$ with the relation
  $\sum_{i=1}^k\xi_i=0$. We will use the basis formed by
  $\xi_2,\dots,\xi_{k+\ell}$.

  Then for $i>1$, and $g\in \mathcal{O}_Y$,
  \begin{equation}\label{formule}
    (d+\eta\wedge)(g\xi_i)= \sum_{j=2}^{k+\ell} \mathscr{D}_j(g)\xi_j\wedge \xi_i
  \end{equation}
  where
  $\mathscr{D}_2,\dots, \mathscr{D}_{k+\ell}$ are differential operators defined by
  \begin{equation*}
    \mathscr{D}_{i+j}=
    \begin{cases}
      z_i\frac{\partial}{\partial z_i}-z_1\frac{\partial}{\partial z_1}+a_i-a_1, \  2\leq i\leq k \ \mbox{and} \ j=0\\
      u_j\frac{\partial}{\partial u_j}+b_j, \  1\leq j\leq \ell \ \mbox{and} \ i=k.
    \end{cases}
  \end{equation*}
  (The differential operators above are a priori defined on
  $\mb{C}\{z_1,\dots, z_k; u_1,\dots, u_\ell\}$ but pass to any
  quotient by the ideal generated by $z_1\dots z_m$ for some
  $m\leq k$. They act by scalars on monomial expressions of the form
  $z_1^{\alpha_1}\dots z_k^{\alpha_k}u_1^{\beta_1}\dots
  u_{\ell}^{\beta_{\ell}}$.)
  
  To justify \eqref{formule}, note that for $g\in\mathcal{O}_Y$,
  $dg=\sum_{j=1}^{k+\ell} (z_j\frac{\partial g}{\partial z_j})\cdot
  \xi_j$.  Hence $A(\eta)_P$ is quasi-isomorphic to
  \begin{equation}\label{complex}
    C^{\starr} := A\leto{d_0} V\tensor_{\C} A\leto {d_1}\wedge^2
    V\tensor_{\C} A \lra \dots\lra  \wedge^{k+\ell-1}V \tensor_{\C}A
  \end{equation}
  where $V=\mb{C}\xi_2\oplus \dots \mb{C}\xi_{k+\ell}$, and the
  differentials are given by
  $$d_{\ell}(\xi_{i_1}\wedge\dots \wedge \xi_{i_{\ell}}\tensor g)=
  \sum_{j=2}^n \xi_j\wedge \xi_{i_1}\wedge\dots\wedge
  \xi_{i_{\ell}}\tensor \mathscr{D}_j(g) .$$ Here $\mathscr{D}_i$
  multiplies
  $z_1^{\alpha_1}\cdots z_k^{\alpha_k}u_1^{\beta_1}\cdots
  u_{\ell}^{\beta_{\ell}}$ by the scalar $\alpha_i+a_i-(\alpha_1+a_1)$
  if $i\in[2,k]$ and by $\beta_j+b_j$ if $i=k+j$ for
  $1\leq j\leq \ell$.
	
  If in the above we replace the convergent power series ring
  $\C\{z_1,\dots, z_k; u_1,\dots, u_\ell\}$ by the ring
  $\C[z_1,\dots, z_k; u_1,\dots, u_\ell]$ and the ring $A$ by
  $B:= \C[z_1,\dots, z_k; u_1,\dots, u_\ell]/(z_1\cdots z_k)$, then
  the complex analogous to \eqref{complex} is a direct sum of
  complexes indexed by the monomials in $B$. Therefore its cohomology
  is a direct sum of the cohomology of these complexes. Each of the
  complexes is the Koszul complex of $V$ corresponding to an element
  of $V$ of the form $\sum_{i>2}c_i\xi_i$, where $c_i$ are the scalar
  multiples above, namely $c_i= \alpha_i+a_i-(\alpha_1+a_1)$ if
  $i\in[2,k]$ and $c_i=\beta_j+b_j$ if $i=k+j$ for $1\leq j\leq \ell$.
  To get a nonzero contribution the $c_i$ must be zero (see Lemma
  \ref{koszul} below).
	
  The summand corresponding to
  $z_1^{\alpha_1}\cdots z_k^{\alpha_k} u_1^{\beta_1}\cdots
  u_{\ell}^{\beta_{\ell}}$ (with all exponents non-negative integers)
  produces a nonzero contribution in homology (and then the
  contribution is the entire exterior algebra of $V$) if and only if
  \begin{equation*}
    \alpha_i+a_i=\alpha_1+a_1, \ i=2,\dots,k\ \mbox{ and }\ \beta_j+b_i=0, j=1,\dots,\ell.
  \end{equation*}
  
  Hence we must have pairwise differences $a_i-a_j\in \mb{Z}$ for all
  $i,j$, so there exists an $i_0$ as in the statement of the
  proposition. Assuming $i_0=1$ we will then have
  $a_1-a_i\in \Bbb{Z}_{\geq 0}$ for all $i$, and
  $b_j\in \mb{Z}_{\leq 0}$ for all $j$, and the monomial must be of
  the form
  $t^{\alpha_1}z_2^{a_1-a_2}\dots z_{k}^{a_1-a_k}u_1^{-b_1}\dots
  u_{\ell}^{-b_{\ell}}$.  This is nonzero iff $\alpha_1=0$.  This
  completes the proof when $A$ is replaced by $B$.  To prove the
  proposition as stated, i.e., with convergent power series, we use the
  explicit homotopies retracting the Koszul complex (when it is zero)
  given by Lemma \ref{koszul} (2) to construct a homotopy equivalence
  from $C^{\starr}$ to its subcomplex corresponding to the monomial
  $h$.
\end{proof}

The monomial
$h=z_2^{a_1-a_2}\dots z_{k}^{a_1-a_k}u_1^{-b_1}\dots
u_{\ell}^{-b_{\ell}}$ in the above proof satisfies $dh+h\eta=0$, and
vanishes on $Y^a$ unless $a=a_1$. For the local cohomology of
$A(\eta)^a$, we then have the following by a similar Koszul argument:
\begin{proposition}\label{prop:local2}
  Suppose $P\in Y^a$.  The stalks $\mathscr{H}^q(A(\eta)^a)_{P}$ of
  the cohomology sheaves at the point $P \in X$ are trivial unless
  $a-a_i\in \Bbb{Z}_{\geq 0}$ for all $i$.  In this case, the natural
  maps $\mathscr{H}^q(Y,A(\eta))_P\to \mathscr{H}^q(Y^a,A(\eta)^a)_P$
  are isomorphisms for all $q$.
\end{proposition}

\subsubsection{Some Koszul algebra}

Let $A$ be a commutative ring. Given a finitely generated free $A$
module $N$ and $x\in N$, recall that the Koszul complex (see, e.g.,
\cite[Section 17.2]{eisenbud}) is the complex (with $A$ in degree $0$)
$$K(x): 0\lra  A\lra  N\lra  \wedge^2 N\lra \dots\lra 
\wedge^i N \leto{d_x} \wedge^{i+1} N\lra \dots .$$ The differentials are
given by wedge products on the left with $x$.

If $x=0$, then the cohomology of this complex in degree $i$ is
$\wedge^i N$.  If $N=N'\oplus N''$, $x=x'\oplus x''$ then
$K(x)=K(x')\tensor K(x'')$,

\begin{lemma}\label{koszul}
  $ $
  \begin{enumerate}
  \item
    If $A=k$ is a field then $K(x)$ has zero homology iff $x\neq 0$.
  \item If $N = k^r$, $x = \sum_i c_ie_i$ with $e_i$ the standard
    basis and $c_1 \neq 0$, then the maps
    $\alpha_p: \wedge^p N \to \wedge^{p-1}N$ for $p>0$ defined
    by $$\alpha(e_{i_1}\wedge e_{i_2}\wedge\dots\wedge e_{i_p})=0$$
    for all $i_1<i_2<\dots<i_p$ with $i_1\neq 1$ and
    $$\alpha(e_{i_1}\wedge e_{i_2}\wedge\dots\wedge
    e_{i_p})=\frac{1}{c_1}e_{i_2}\wedge\dots\wedge e_{i_p}$$ if
    $i_1=1$ give a homotopy from the identity map to the zero map of
    $K(x)$.
  \end{enumerate}
\end{lemma}
\begin{proof}
  Part (1) is well-known and part (2) is a routine verification.
\end{proof}

\subsection{Filtrations on nearby cycles} \label{s:A}

We now return to the algebro-geometric setting of Section
\ref{s:nbms}: Let $\msk \subset \C$. Let $X$ be a smooth variety over
$\msk$ and $f:X \to \A^1_{\msk}$ a morphism which is smooth over
$\A^1_{\msk} \sm \{0\}$ and let $Y = f^{-1}(0)$ and let $Y_{\alpha}$,
$\alpha =1,2\dots,n$ be the irreducible components of
$Y$. Furthermore, let $E\subset X$ be a divisor such $D = Y \cup E$
has simple normal crossings. Let $U = X \sm D$.

Let $g$ be a nowhere vanishing regular function on $U$, let $N\geq 1$
be an integer, let $a = b/N$ for some $b\in \Z$ and let $K =
\Q(\mu_N)$. Let $\eta = a \dlog g\in \Omega_X^1(\log D)$ and denote by
$\mc{L}(\eta)$ the $K$-mixed local system $Kg^{a}$ on $U$ defined in
Section \ref{s:mls1}. Let $V=X \sm (Y \cup (\cup'_{b}E_{b}))$ where
the union is restricted to the set of irreducible components $E_b$ of
$E$ such that the residue of $\eta$ along $E_{b}$ is not a strictly
positive integer.  Let $\tilde{j}:X \sm D\to V$, and
$\tilde{k}:V\to X \sm Y$ be the inclusion.

Similarly let $V'=X \sm (Y \cup (\cup'_{b}E_{b}))$ where the union is
restricted to the set of irreducible components $E_b$ of $E$ such that
such that the residue of $\eta$ along $E_{b}$ is not a non-negative
integer.  Let $\tilde{j}':X \sm D\to V$, and
$\tilde{k}':V\to X \sm Y$ be the inclusions. Note that $V'\supset V$.
\begin{defi}
  Let
  $\mc{K}(\eta) = \Psi_f (\tilde{k}_*\tilde{j}_{!}\ml(\eta))\in
  \db(\msm(Y))_K.$
\end{defi}
\begin{remark}\label{r:perv}
  The object $\rat_X(\tilde{k}_*\tilde{j}_{!}\ml(\eta))$ has perverse
  cohomology in only one degree by Remark \ref{r:p}, so
  $\tilde{k}_*\tilde{j}_{!}\ml(\eta)$ is a shift of an object of
  $\msm(X)_K = \msm(X,K)$. Since $\Psi_f$ maps $\msm(X)$ to
  $\msm(Y)[1]$, it follows that $\Psi_f
  (\tilde{k}_*\tilde{j}_{!}\ml(\eta))$ is the shift of an object of
  $\msm(Y)_K$.
\end{remark}

Replacing the function $g$ by $g^{-1}$ replaces $\eta$ by $-\eta$, so
we also have a $K$-mixed local system
$\mc{L}(-\eta) \simeq \mc{L}(\eta)^{-1}$ to which we may apply the
same construction as above (but note that $V$ and $V'$ will change).
For ease of notation in what follows we let
$$\mc{G}=\mc{K}(\eta),\ \mc{F}=\D_Y(\mc{K}(-\eta))[-2\dim Y] .$$
\begin{lemma} \label{l:vd}%
  $\mc{F}= \Psi_f(\tilde{k}'_*\tilde{j}'_{!}\ml(\eta))$.
\end{lemma}
\begin{proof} We may verify the desired statement after taking Verdier
  duals, i.e., it suffices to show that
  $\mc{K}(-\eta)[2\dim
  Y]=\D_Y(\Psi_f(\tilde{k}'_*\tilde{j}'_{!}\ml(\eta)))$.  Now, since
  $\Psi_f[-1]$ commutes with Verdier duality,
   $$\D_Y(\Psi_f(\tilde{k}'_*\tilde{j}'_{!}\ml(\eta)))=\Psi_f\D_{X^*}( \tilde{k}'_*\tilde{j}'_{!}\ml(\eta))[-2]$$
    
    But, $\D_{X^*}(\tilde{k}'_*\tilde{j}'_{!}\ml(\eta))= \tilde{k}'_!\tilde{j}'_* \D_{U}\ml(\eta)$, and 
    $\D_U(\ml(\eta))=\ml(-\eta)[2\dim X]$. Therefore 
    $$\D_Y(\Psi_f(\tilde{k}'_*\tilde{j}'_{!}\ml(\eta)))=\Psi_f\tilde{k}'_!\tilde{j}'_{*}\ml(-\eta)[2\dim
    X-2]$$
    
    Now $\dim Y=\dim X-1$, and it therefore suffices to show that
    $\tilde{k}'_!\tilde{j}'_{*}\ml(-\eta)$ is isomorphic to the object
    whose nearby cycles gives $\mc{K}(-\eta)$. Since the residues for
    $-\eta$ are negatives of the residues for $\eta$, the desired
    identification holds by examination of the definitions of $V$ and
    $V'$ and \cite[Proposition 2.2]{saito-form}. Note that along a
    divisor $E_i$ where the residue of $\eta$ is zero,
    $\tilde{k}'_!\tilde{j}'_{*}\ml(-\eta)$ is a lower star extension.
  \end{proof}

Since $V \subset V'$, we have a canonical map
$\tilde{k}'_*\tilde{j}'_{!}\ml(\eta) \to 
\tilde{k}_*\tilde{j}_{!}\ml(\eta)$
which, by Lemma \ref{l:vd}, gives rise to a canonical map
\begin{equation}\label{e:kk'}
	\mc{F}\lra \mc{G} .
\end{equation}

\begin{remark}\label{determined}
  By two applications of adjunction, maps
  $\tilde{k}'_*\tilde{j}'_{!}\ml(\eta) \to
  \tilde{k}_*\tilde{j}_{!}\ml(\eta)$ are determined by their
  restrictions to $U$, i.e, maps $\ml(\eta)\to \ml(\eta)$.
\end{remark}

Let $a_1,a_2,\dots,a_s \in \Q$ be the distinct residues of $\eta$
along the irreducible components of $Y$.  We assume that the $a_i$ are
ordered so that $$a_1>a_2>\dots>\dots>a_s.$$ For $a \in \Q$, let $Y^a$
be the union of the $Y_{\alpha}$ such that
$\Res_{Y_{\alpha}}(\eta) = a$. We then have
$$Y^{\leq a_s}\subseteq\dots\subseteq Y^{\leq a_{r+1}}\subseteq
Y^{\leq a_{r}}\subseteq \dots \subseteq Y^{\leq a_1}=Y_0^{\leq
  a_1}=Y, $$ where $Y^{\leq a}= \cup_{b\leq a} Y^b$. Let
$Y^{\geq a} =\cup_{b\geq a} Y^b$.  For convenience set
$a_{s+1}=-\infty$, $a_0=\infty$. Hence
$Y^{\leq a_{s+1}}=Y^{\geq a_0}=\emptyset$, and
$Y^{\geq a_s}=Y^{\leq a_{0}}=X$.
\begin{defi}
  For $a \in \Q$ let $\mc{G}^{\leq a}$ be the restriction of $\mc{G}$
  to $Y^{\leq a}_0$ (see Definition \ref{basicD}) and let
  $\mj^{\leq a}$ be the extension by zero of $\mc{G}^{\leq a}$ to
  $Y$. Note that $\mc{G}^{\leq a_1}= \mj^{\leq a_1}=\mc{G}$.
\end{defi}

Let $j: Y^{a_r}_0\to Y^{\leq a_r}_0$, $j':Y^{\leq a_{r+1}}_0\to Y_0^{\leq a_{r}}$ and
$\tilde{j}: Y^{\leq a_r}_0\to Y$
be the inclusions. 
By Lemma \ref{stillholds}, we have distinguished triangles
$$ j'_!j'^*\mc{G}^{\leq a_r}\lra  \mc{G}^{\leq a_r}\lra   j_*j^* \mc{G}^{\leq
	a_r}\leto{[1]}
$$ for $r=1,\dots,s$. By applying
$\tilde{j}_!$ to each term we then get distinguished triangles
\begin{equation}\label{e:dt}
  \mj^{\leq a_{r+1}}\lra \mj^{\leq
    a_r}\lra \tilde{j}_!j_*j^* \mc{G}^{\leq a_r}\leto{[1]} .
\end{equation}
The object $\tilde{j}_!j_*j^* \mc{G}^{\leq
  a_r}$ is the pushforward to $Y$ of an object
$\ma^{a_r}$ supported on the closed subvariety $Y^{a_r}$.

We therefore rewrite \eqref{e:dt} as
\begin{equation}\label{exacto}
	\mj^{\leq a_{r+1}}\lra  \mj^{\leq a_r}\lra   \ma^{a_r}\leto{[1]} .
\end{equation}
Thus  $\mc{G}$ is a successive extension of objects
$\ma^{a_r}$ supported on the closed
$Y^{a_r}$.  We can carry out the same constructions for $\mc{F}$ (again using Lemma \ref{stillholds}), to obtain that
$\mc{F}$ is a successive extension of objects
$\ma'^{a_r}$ supported on $Y^{a_r}$
\begin{equation}\label{exactoprime}
	\mj'^{\leq a_{r+1}}\lra  \mj'^{\leq a_r}\lra   \ma'^{a_r}\leto{[1]} .
\end{equation}

\begin{remark}\label{descriptive}
  Here is a description of $\ma^{a_r}$ that we will use later: We
  restrict $\mc{G}$ to $Y^{a_r}_0$, star push it forward to
  $Y^{a_r}\cap Y^{\leq a_r}_0$, and then shriek push forward the
  resulting object to $Y^{a_r}$.  The description of $\ma'^{a_r}$ is
  the same with $\mc{G}$ replaced by $\mc{F}$. By Remark \ref{r:perv}
  and the fact that both star and shriek pushforwards for affine open
  embeddings preserve perverse sheaves (\cite[Corollaire 4.1.10]{bbd})
  it follows that all these objects in $\db(\msm(Y))_K$ are shifts of
  objects in $\msm(Y)_K$.
  
  Note that as will be shown in the proof of Lemma \ref{dualite}
  below, we would get the same object if we restricted $\mc{G}$ to
  $Y^{a_r}_0$, shriek pushed it forward to
  $Y^{a_r}\cap Y^{\geq a_r}_0$, and star pushed it forward to
  $Y^{a_r}$ (similarly for $\ma'^{a_r}$).
\end{remark}

We can replace $\eta$ by $-\eta$ in all the above constructions to
produce analogous objects and maps.  When necessary we distinguish
these objects by adding $(\eta)$ and $(-\eta)$ to the
notation. Objects without any extra decoration always refer to the
original constructions for $\eta$.

\begin{lemma} \label{dualite}%
  For any $r$, $\ma'^{a_r}(\eta)= \D_Y(\ma^{-a_r}(-\eta))[-2\dim Y]$.
\end{lemma}
\begin{proof}
  We will show $\ma'^{-a_r}(-\eta)= \D_Y(\ma^{a_r}(\eta))[-2\dim Y]$.
	
  On $Y^{a_r}_0$, $\ma'^{-a_r}(-\eta)$ is the pullback of
  $\mc{F}(-\eta)= D(\mc{K}(\eta))[-2\dim Y]$.  Also, on this open set
  $\ma^{a_r}(\eta)$ is the pull back of $\mc{K}(\eta)$. Therefore we
  have an identification on $Y^{a_r}_0$.
	
  Let us recall the definition of $\ma^{a_r}(\eta)$: we pull back
  $\mc{G}(\eta) = \Psi_f (\tilde{k}_*\tilde{j}_{!}\ml(\eta))\in
  \db(\msm(Y))_K$ to $Y^{a_r}_0$ and then lower star push forward to
  $Y^{\leq a_r}_0\cap Y^{a_r}$ and lower shriek push forward to
  $Y^{a_r}$.  Therefore $\D_Y(\ma^{a_r}(\eta))[-2\dim Y]$ is the pull
  back of
  $$\D_Y(\Psi_f (\tilde{k}_*\tilde{j}_{!}\ml(\eta)))[-2\dim Y]= \Psi_f (\tilde{k}_!\tilde{j}_{*}\ml(-\eta))$$ to 
  $Y^{a_r}_0$, which we lower shriek push forward to
  $Y^{\leq a_r}_0(\eta)\cap Y^{a_r}$ and then lower star push forward
  to $Y^{a_r}$.  Clearly $Y^{a_r}_0(\eta)= Y_0^{-a_r}(-\eta)$.  This
  shows the desired equality, up to choices of orders of extensions
  which we analyze below.
	
  For simplicity let $U= Y^{a_r}_0$,
  $V = Y^{\leq a_r}_0(\eta)\cap Y^{a_r}$, $W= Y^{a_r}$ and let
  $V'= Y^{\geq a_r}_0(\eta)\cap Y^{a_r}$.  Let $\mc{F}$ be the
  restriction of $\Psi_f (\tilde{k}_!\tilde{j}_{*}\ml(-\eta))$ to $U$.
  Using Lemma \ref{l:vd}, we have
  \begin{equation}\label{formules}
    \D_Y(\ma^{a_r}(\eta))[-2\dim Y]= j_{V,W,*}jj_{U,V,!}\mc{F},  \ \ \ma'^{-a_r}(-\eta)= j_{V',W,!}jj_{U,V',*}\mc{F} 
  \end{equation}
  (the notation $j_{U,V}$, etc., is self-explanatory).
	
  To show the independence from choices of extensions, we just need to
  show that both objects have zero stalks at points of the closed
  $W \sm V'\subseteq W$. This is sufficient because the restrictions
  of both objects to $V'$ are isomorphic. In fact, we need to show the
  vanishing of stalks only at points of $Y^{a_r}\cap Y_i\cap Y_j$ with
  $a(i)>a_r$, $a(j)<a_r$.
	
  The vanishing of stalks for $\ma'^{-a_r}(-\eta)$ follows easily from
  the second formula in \eqref{formules}.  For
  $\D_Y(\ma^{a_r}(\eta))[-2\dim Y]$, using Remark \ref{devissage}, we
  can write $\mc{F}$ as a successive extension of objects on $Y_i$
  (which are themselves pushed forward from $(Y_i)_0$ by a sequence of
  lower star and shriek functors), and the first formula in
  \eqref{formules}.  The verification reduces to Lemma \ref{l:ext}
  after applying $\rat$.
\end{proof}

\subsection{A de Rham description of $\rat(\ma^a)$.}

The following lemma, which gives a de Rham description of
$\rat_Y(\mc{A}^a)$, is a key ingredient in our proof of motivic
factorisation for invariants.

\begin{lemma}\label{comparison}
  There exists an isomorphism
  $A(\eta)^a\cong \rat_Y(\ma^{a})\tensor_K \mb{C}$ in $\dbc(Y^a, \C)$.
\end{lemma}
To make sense of the statement we note that by Lemma \ref{l:perv}
$\rat_Y(\ma^a)$ may be viewed as an object of $\dbc(Y^a,K)$, so we may
base change it using an embedding $K \to \C$ to get an object of
$\dbc(Y,\C)$.
	
\begin{proof}
  Since the $\rat$ functors commute with star and shriek pushforwards,
  nearby cycles and restriction to open sets, we are reduced to
  proving a purely topological statement. I.e., we may assume
  $\msk =\C$ and $K = \C$ and that $\mc{L}(\eta)$ is the $\C$-local
  system associated to a log form $\eta$ on $U$. We then define
  analogs of all objects in this section in $\dbc(X,\C)$ and
  $\db(Y,\C)$ for which we continue to use the same notation and we
  then need to prove that
  \begin{equation}\label{e:2}
    A(\eta)^a\cong \ma^{a} \mbox{ in } \dbc(Y^a, \C) .
  \end{equation}
  We then proceed as follows:
  
  Recall that Remark \ref{descriptive} gives a description of
  $\ma^{a}$ in terms of its restriction to $Y^a_0$.
		
  The objects appearing in \eqref{e:2} lie in $\dbc(Y^a,\C)$ and we
  first show that they are isomorphic over $Y_0^a-E$.
  The case $a\neq 0$ follows from $a=0$ by considering
  $\eta - f^*(a dt/t)$ and applying (1) of Lemma \ref{kunneth}. Note
  that $f^*(dt/t)=0$ in $A(\eta)^a$.  Therefore, without loss of
  generality, we assume that $a=0$. Then $\eta$ defines a local system
  on $Y^a_0-E$, and the computations of Steenbrink \cite[Section
  2]{steenbrink} together with Theorem \ref{locale} imply the lemma on
  $Y_0^a-E$; we recall that Steenbrink shows that when $\eta=0$,
  nearby cycles have a de Rham description as $A(\eta)$. The proof
  will be completed by showing that both objects are extended from
  $Y_0^a -E$ to $Y^a$ in the same way.
		
  To do this let us mark out a divisor $D^a\subset Y^a$. The
  irreducible components of $D^a$ are of two types. The first type are
  horizontal divisors $E_j$ such that the residue of $\eta$ along
  $E_j$ are not in $\Bbb{Z}_{\leq 0}$. The second type is
  $Y_{\alpha} \cap Y^a$ such that
  $a-a(\alpha)\not\in \Bbb{Z}_{\geq 0}$, where
  $a(\alpha) = \Res_{Y_{\alpha}}(\eta)$ is the residue of $\eta$ along
  $Y_{\alpha}$.
		
  We claim
  \begin{enumerate}
  \item[(A)] $\ma^a$ and $A(\eta)^a$ have zero stalks at points of
    $D^a$.
  \item[(B)] Restricted to $Y^a-D^a$, both sheaves are star push
    forwarded from $Y_0^a -(E\cup D^a) = Y_0^a - E$ where they are
    isomorphic as shown above.
  \end{enumerate}
  It is easy to see that the lemma follows from these claims.
		
  Claim (A) for $A(\eta)^a$ follows from Propositions
  \ref{prop:local2} and \ref{prop:local1}. For $\ma^a$, the vanishing
  at points of the second type follows from Remark
  \ref{descriptive}. At points of the first type, we use a Kunneth
  argument since the divisors $E_j$ contribute to a Kunneth component
  which does not see the singularities, and we have a corresponding
  vanishing along $E_j$ at points of the general fibre over $S^*$ (of
  the object whose nearby cycles is $\mc{K}$).
		
  Claim (B) for $\ma^a$ follows from Remark \ref{descriptive}. For
  $A(\eta)^a$, we have generators of the stalk at $P\in Y^a-D^a$.  We
  just have to chase these generators to a local de Rham cohomology
  group of $Y^a_0- E$. The restriction of $A(\eta)^a$ admits a
  devissage since it is isomorphic there to $K(\eta)$. The computation
  can thus be made in a $Y_{\alpha}^0-E$, with $Y_{\alpha}$ in $Y^a$ using the
  remark below.
\end{proof}
	
\begin{remark}\label{locality}%
  If $h$ is an invertible function on an open subset $U$ of $Y^a_0$,
  then the complexes $A(\eta+dh/h)$ and $A(\eta)$ are isomorphic over
  $U$. The map is by multiplication by $h$. This allows us to reduce
  to $\eta=0$ in some local situations.
\end{remark}
	
\subsection{Degenerations of hyperplane arrangements}\label{s:hyp}

We continue with the notation from the previous section but make some
further assumptions:
\begin{enumerate}
\item $f:X \to \A^1_{\msk}$ is proper over some $S \subset \A^1_{\msk}$ containing $0$.
\item For all $s \in S^*$, $f^{-1}(s) - E$ is the complement of a
  hyperplane arrangement in $\A^M_{\msk}$ for some $M$ and the form $\eta$ is
  the form corresponding to the arrangement (as in Section
  \ref{s:ac}).
\end{enumerate}
Under these assumptions we have:
\begin{lemma}\label{l:nearby2}
The sum $\sum_{a\in \R} h^k(Y,A(\eta)^a)$ remains constant
	under scaling of $\eta$.
\end{lemma}
\begin{proof} 
  By Theorem \ref{locale},
  $\sum _{a\in \mb{R}} h^k(Y^{a}, A(\eta)^a)= h^k(Y,
  A(\eta))$. Therefore we need to show that $h^k(Y,A(\eta))$ remains
  constant under scaling of $\eta$.  The higher direct images for the
  map $f|_{X^*}: X^*\to S^*$ of $\Omega^k_{X^*/S^*}(\log D)$ are zero
  since on each fibre we have such vanishing \cite[Section 4.1]{L1}.
  This implies the vanishing $R^if_*(\Omega^k_{X/S}(\log D)) = 0$ for
  $i>0$ (see \cite[(5.7)]{SZ}), and
  $\Omega^{\bull}_{X/S}(\log D)\tensor_{\mc{O}_X} \mathcal{O}_Y$ is a complex of
  sheaves whose individual terms have vanishing higher cohomology. The
  global sections
  $H^0(Y,\Omega^{\bull}_{X/S}(\log D)\tensor_{\mc{O}_X} \mathcal{O}_Y)$ are
  $d$-closed (by \cite[ Section 5.7]{SZ})). Therefore
  \begin{equation}\label{inspect}
    H^k(Y, A(\eta))\ =\frac{\ker {\eta\wedge: H^0(Y, \Omega^{k}_{X/S}(\log D)\tensor_{\mc{O}_X}
        \mathcal{O}_Y) \lra  H^0(Y, \Omega^{k+1}_{X/S}(\log D)\tensor_{\mc{O}_X}
        \mathcal{O}_Y)}}{\eta\wedge H^0(Y, \Omega^{k-1}_{X/S}(\log D)\tensor_{\mc{O}_X}
      \mathcal{O}_Y)} ,
  \end{equation}
  hence the ranks of $H^k(Y,A(\eta))$ remain invariant under scaling
  of $\eta$.
\end{proof}

\section{Motivic factorisation for invariants}\label{kzf}

\subsection{A rank inequality}

We continue with the notation of Section \ref{s:A}, but now assume
that $X \leto{g} W\leto{f} \A^1_{\msk}$, $f_X = f \circ g$, $\eta$, etc., are
as in Section \ref{s:dsd} (so we have fixed $A \subset [n]$ but
dropped it from the notation) except that now we have no assumptions
on the levels of the $\lambda_i$ and $\nu$; we will refer to these
assumptions as the ``KZ setting''. Recall that $g_Y:Y\to Z$ denotes
the restriction of $g$ to $Y$.

Applying the constructions of Section \ref{s:A} to this data we get
objects $\tilde{k}'_*\tilde{j}'_{!}\ml(\eta)$ and
$\tilde{k}_*\tilde{j}_{!}\ml(\eta)$ in $\db(\msm(X^*))$, their nearby
cycles $\mc{F}$, $\mc{G}$ in $\db(\msm(Y))$, etc.  The fibres of
$\tilde{k}'_*\tilde{j}'_{!}\ml(\eta)$ and
$\tilde{k}_*\tilde{j}_{!}\ml(\eta)$ over any point of $\A^1_{\msk}$
correspond to the objects on the two sides of the map in Lemma
\ref{p:kz}. We therefore get:
\begin{lemma}\label{l:image}
  The image of the map
  \begin{equation}\label{imago1}
    H^M(g_{Y,*}\mc{F})^{\Sigma,\chi}\lra H^M(g_{Y,*} \mc{G})^{\Sigma,\chi}
  \end{equation}
  is the nearby cycles for the KZ local system on the dual of the
  space of invariants (as in Theorem \ref{t:fact}).
\end{lemma}

The KZ local system may not be semisimple, or have semisimple local
monodromies, so we do not expect a direct sum decomposition as in the
case of conformal blocks (Section \ref{cbf}).

In the KZ setting, the distinguished triangles \eqref{exacto} and
\eqref{exactoprime}, give  the following commutative diagram, equivariant for the action of $\Sigma$ with exact rows:
\begin{equation}\label{d:small}
  \xymatrix{
    (H^k(g_{Y,*}\mj'^{\leq a_{r+1}}))\ar[d]\ar[r] & (H^k(g_{Y,*} \mj'^{\leq a_r}))\ar[r]\ar[d] & (H^k(g_{Y,*} \ma'^{a_r}))\ar[d] \\
    (H^k(g_{Y,*}\mj^{\leq a_{r+1}}))\ar[r] & (H^k(g_{Y,*} \mj^{\leq a_r}))\ar[r] & (H^k(g_{Y,*}\ma^{a_r}))
  }
\end{equation}
The rows are in fact short exact sequences by the following lemma which is proved in Section \ref{p:exact} (so that the first map in each row  is an injection and the second map a surjection):
\begin{lemma}\label{l:exact}
The following are exact for all $k$:
\begin{equation}\label{e:newexact11prime}
  0\lra  H^k (g_{Y,*}\mj'^{\leq a_{r+1}})\lra  H^k(g_{Y,*}\mj'^{\leq a_r})\lra   H^k(g_{Y,*}
  \ma'^{a_B})\lra  0 
\end{equation}
and
\begin{equation}\label{e:newexact11}
  0\lra  H^k (g_{Y,*} \mj^{\leq a_{r+1}})\lra  H^k (g_{Y,*}\mj^{\leq a_r})\lra
  H^k(g_{Y,*}\ma^{a_r})\lra  0 .
\end{equation}
\end{lemma}
We give the proof at the end of this subsection.

\smallskip

The exact sequences \eqref{e:newexact11prime} and \eqref{e:newexact11}
are equivariant for the action of $\Sigma$. Therefore, taking
isotypical components in these sequences for $k=M$, we get a
commutative diagram with exact rows,
\begin{equation}\label{bigdiagram}
	\xymatrix{
		0\ar[r] & (H^M(g_{Y,*}\mj'^{\leq a_{r+1}}))^{\Sigma,\chi}\ar[r]\ar[d] & (H^M(g_{Y,*} \mj'^{\leq a_r}))^{\Sigma,\chi}\ar[r]\ar[d] & (H^M(g_{Y,*} \ma'^{a_r}))^{\Sigma,\chi}\ar[r]\ar[d] & 0\\
		0\ar[r] & (H^M(g_{Y,*}\mj^{\leq
                  a_{r+1}}))^{\Sigma,\chi}\ar[r] & (H^M(g_{Y,*}
                \mj'^{\leq a_r}))^{\Sigma,\chi}\ar[r] &
                (H^M(g_{Y,*}\ma'^{a_r}))^{\Sigma,\chi}\ar[r] & 0 \,.
	}
\end{equation}
where the vertical maps are induced from \eqref{e:kk'} and functoriality of
the various derived functors associated to the maps $j,j'$ and $\tilde{j}$ above. 

\smallskip

Combining \eqref{bigdiagram} and Lemma \ref{LA2}, we get a lower bound
for the rank of the map in \eqref{imago1}.
\begin{corollary}\label{acarol}
  \begin{multline*} 
    \rk\bigl(H^M(g_{Y,*}\mc{F}))^{\Sigma,\chi}\lra (H^M(g_{Y,*} \mc{G}))^{\Sigma,\chi}\bigr) \geq\\
    \sum_{a\in \Q} \rk \bigl(\im (H^M(g_{Y,*},
    \ma'^a))^{\Sigma,\chi}\lra (H^M(g_{Y,*} \ma^a))^{\Sigma,\chi})\bigr).
  \end{multline*}
\end{corollary}

We now find subquotients of the images of
$H^M(g_{Y,*}\ma'^a)^{\Sigma,\chi}\to H^M(g_{Y,*}\ma^a)^{\Sigma,\chi}$ by the
same method as in the case of conformal blocks. These subquotients
include all terms that appear in the factorisation of invariants, so
we will get equalities throughout.

\smallskip

Recall from Definition \ref{whyi} the indexing set $\ms{I}$ for the
irreducible components of $Y$ and the indexing set $\ms{O}$ for the
orbit representatives.
\begin{defi}
  $ $
  \begin{enumerate}
  \item Let $\mci_a\subset \mci$ be the set of $B$, such that
    $Y_B\in Y^a$, i.e., $Y^a=\cup_{i\in \mci ^a} Y_B$. The set
    $\mci_a$ is preserved by the action of $\Sigma_{\beta}$. Let
    $\mco_a=\mci_a\cap \mco \subseteq\mco$ be the set of orbit
    representatives.
  \item For $i\in \mci^a$, let
    $$Y_B^*=Y_B \setminus \Bigl (\bigcup_{j\in \mci^a, \,
      B'\neq B } Y_{B'} \Bigr ) \ ;$$ this is an open subset of $Y^a$
    (but possibly not open in $Y$). Note that
    $Y_B^*\supseteq (Y_B)_0$.  The difference
    $Y_B^* \setminus (Y_B)_0$ is contained in the union of the
    intersections $Y_B\cap Y_{B'}$ with $a(B)\neq a(B')$, where for
    any $B \in \mci$, $a(B)$ is the residue of $\eta$ along $Y_B$. Let
    $k_B:Y_B^*\to Y_B$ and $k'_B: Y_B\to Y$ be the inclusions.
  \end{enumerate}
\end{defi}

\begin{lemma}
  \begin{multline}\label{inegalite111}
    \rk\bigl(\im (H^M(g_{Y,*}\mc{F})^{\Sigma,\chi}\lra H^M(g_{Y,*} \mc{G})^{\Sigma,\chi})\bigr)\geq\\
    \sum_{B\in \mco_a} \rk  \bigl(\im(H^M(g_{B,*} k_{B,!} (k'_B\circ
    k_B)^*\ma'^a))^{\Sigma_B,\chi_B}\lra H^M(g_{B,*} k_{B,_*}(k'_B\circ
    k'_B)^*\ma^a)^{\Sigma_B,\chi_B})\bigr).
	\end{multline}
\end{lemma}  
\begin{proof}
  By applying Lemma \ref{quotientie} we get an inequality
  \begin{multline*}%
    \rk\bigl(\im (H^M(g_{Y,*}\ma'^{a}))^{\Sigma,\chi}\lra (H^M (g_{Y,*}\ma^a))^{\Sigma,\chi})\bigr)\geq \\
    \sum_{B\in \mco_a} \rk\bigl(\im H^M(g_{B,*}(k_B)_! (k'_B\circ
    k_B)^*\ma'^a)^{\Sigma_B,\chi_B} \lra H^M(g_{B,*} Rk_{B,*}(k'_B\circ
    k_B)^*\ma^a)^{\Sigma_B,\chi_B})\bigr) .
  \end{multline*}
  The lemma follows by combining this with Corollary \ref{acarol}.
\end{proof}

\subsubsection{Proof of Lemma \ref{l:exact}}\label{p:exact}

This reduces to a statement about ranks by breaking up the three term
sequence into two exact sequences (using that if $\rat$ applied to a
$K$-mixed local system is $0$ then the $K$-mixed local system is
zero). Therefore for the verification, we may apply $\rat$ to all
objects under consideration and work in the topological category.
Thus, as in the proof of Lemma \ref{comparison}, we may assume that
that $\msk = \C$, $K = \C$ and we can consider exact analogoues of all
$K$-mixed local sheaves considered above in $\dbc(X,\C)$, etc.  By
duality and Lemma \ref{dualite}, the exactness of \eqref{e:newexact11}
follows from the exactness of \eqref{e:newexact11prime} for
$-\eta$. We therefore only need to prove the exactness of
\eqref{e:newexact11prime}. Lemma \ref{l:exact} then follows from:
\begin{proposition}\label{bigeq1}
  The following equality of ranks holds:
  \begin{equation*}
    \rk (H^M (g_{Y,*}\mc{G}))=  \sum _{a\in \Q}\rk (H^M(g_{Y,*} \ma^{a}))\,.
  \end{equation*}
\end{proposition}

\begin{proof}
  Let $P$ be a general point of $W$. By the proper base change
  theorem, the rank of $H^M(g_{Y,*} \ma^{a})_P$ is the dimension of
  $H^M(Y_P, \ma^a_P)$.
    
  We want to apply Lemma \ref{l:nearby2}, so we pass to a family
  parametrised by an open subset of $\A^1_{\msk}$ as follows. Let
  $S\to W$ be a section of $f:W\to \Bbb{A}^1_{\C}$, where $S$ is an
  open subset of $\A^1_{\C}$ containing $0$, and assume that that the
  image of $S$ meets $Z$ transversally at $P$ corresponding to
  $0\in S$; the fact that such a section exists follows from the
    explicit description of $W$ and $f$ in Section \ref{s:statement}.
  By taking the fibre product, we obtain a map ${X}_S\to S$ with the
  fibre over $0$ being $Y_P$.  By Lemma \ref{l:normalform}, $X_S\to S$
  (together with the restriction of $\eta$ to $X_S$) has all the
  properties needed to define corresponding objects ${\ma}_S^a$, etc.,
  in $\dbc(Y_P,\C)$. We now note that $\ma^a_P=({\ma}^a_S)_P$ as
  objects in $\dbc(Y_P,\C)$ (since they are both extended from
  $(Y_P)_0 \sm E$ in the same way, see Remark \ref{descriptive} and
  Lemma \ref{comparison}).
  
  Note that if we replace $\eta$ by $c\eta$ for some
  $c \in \R^{\times}$ and consider the analogues of the objects
  defined above for $c\eta$, then $h^k(Y_P,\mc{G})$ does not change
  since it is equal to the dimension of the cohomology of an Aomoto
  complex on a nearby fibre (Remark \ref{r:bc}) and this is
  independent of scaling (Section \ref{s:ac}). Also, if $0 < c \ll 1$,
  then $0<|ca_r-ca_{r'}| <1$ for all $r\neq r'$.  In this case the
  distinguished triangles \eqref{exacto} (for $c\eta$) split since
  stalks of $\mc{G}$ at intersections of the $Y^a$ are zero by
  Proposition \ref{prop:local1}, so we will have the desired
  equality. Thus, we only need to show that
  $\sum _{a\in \Q}\rk H^k(\ma^{a}_P)$ is invariant under scaling
  $\eta$, which follows from Lemma \ref{comparison} and Lemma
  \ref{l:nearby2}.
\end{proof}

\subsection{Factorisation for invariants}
Using associativity in the representation ring, we have
$$\rk (V_{\lambda_1}\tensor \dots\tensor V_{\lambda_n}\tensor V_{\nu}^*)_{\mf{g}}= \sum_{\mu\in P^+}\rk(\tensor_{i\in A}V_{\lambda_i}\tensor V_{\mu}^*)_{\mf{g}} \rk(\tensor_{i\in [n]\setminus A}V_{\lambda_i}\tensor V_{\mu}\ \tensor V_{\nu}^*)_{\mf{g}},$$ and hence the rank of 
$\mathcal{KZ}_{\kappa}(\vv{\lambda},\nu^*)$ equals (using notation
from Theorem \ref{t:fact})
\begin{equation}\label{summandKZ0}
  \sum_{\mu\in P^+}\rk\bigl( \mc{KZ}_{\kappa}(\vv{\lambda}'', \mu^*)
  \tensor_K 
  \mc{KZ}_{\kappa}(\vv{\lambda}'_{\mu}, \nu^*)\big) ,
\end{equation}

By Lemma \ref{l:image}, $\mathcal{KZ}_{\kappa}(\vv{\lambda},\nu^*)$
equals the rank of the image of
$H^M(g_{Y,*}\mc{G})^{\Sigma,\chi}\to H^M(g_{Y,*}\mc{K})^{\Sigma,\chi}$.  We
now give a bijection $\lambda\mapsto i$ from the set of dominant
integral weights contributing nonzero terms in \eqref{summandKZ0} to
a subset of $\mathscr{O}$. This extends the map in Section
\ref{parameteri} (in which $\lambda$ was a subset of level $\ell$
weights).

Assume that the summand in \eqref{summandKZ0} corresponding to $\mu$
is nonzero.  This implies that $\mu+\sum_{a\in A^c}\lambda_a -\nu$
and $\sum_{a\in A}\lambda_a -\mu$ are non-negative integer
combinations of simple roots (since the corresponding spaces of
invarinats are nonzero).  In fact it is easy to see that there exists
a unique $B\in\mathscr{O}$ so that equations \eqref{e111} and
\eqref{e222} hold. Let $a=a(B)\in (1/N)\Z/\Z$.

\smallskip { We now note that the constructions above can be
  carried out with $\mc{F}$ and $\mc{G}$ which are nearby cycles
  $\Psi_{f_X}$ of objects in $\db(\msm(X^*))_K$, replaced by
  $\Psi_{f_X,\bar{c}}$ of the corresponding objects for any
  $\bar{c} \in (1/N)\Z/\Z$. This is because \eqref{e:adj} is an
  isomorphism for a direct sum $\mc{H}$ if and only if it is an
  isomorphism for each summand. In particular, we get direct summands
  $\ma^a(\bar{c})$, and $\ma'^a(\bar{c})$ of $\ma^a$ and $\ma'^a$
  respectively, and maps $\ma'^a(\bar{c})\to\ma^a(\bar{c})$.

  In the same way,
  $\Psi_{f,\bar{c}}\,(\mathcal{KZ}_{\kappa}(\vv{\lambda},\nu^*))$ has
  a subquotient given by the image of
  \begin{equation}\label{newSKZ}
    H^M(g_{B,*} k_{B,!} (k'_B\circ k_B)^*\ma'^a(\bar{c}))^{\Sigma_B,\chi_B}\lra  H^M(g_{B,*} k_{B,_*}(k'_B\circ k'_B)^*\ma^a(\bar{c}))^{\Sigma_B,\chi_B}
  \end{equation}
} 
\begin{lemma}\label{prefinal}
  The image of \eqref{newSKZ} for $B$ corresponding to $\mu$, and
  $\bar{c}=\bar{a}$ ($=\overline{a(B)}$) is the same as the
  corresponding term in Theorem \ref{t:fact}, i.e.,
\begin{equation}\label{newS2}
(\mc{KZ}_{\kappa}(\vv{\lambda}'_{\mu}, \nu^*)
    \boxtimes \iota^*(\mc{KZ}_{\kappa}(\vv{\lambda}'', \mu^*) )) .
\end{equation} 
\end{lemma}
\begin{proof}
  Note that it follows from \eqref{e:proj} that
  $Y_B=P_B^1\times P_B^2$, a product of smooth varieties with
  $U^i_B \subset P^i_B$, $i=1,2$.  There is a corresponding
  decomposition of $Z$ as a product $Z_1\times Z_2$ and let
  $g_i:Y_i\to Z_i,\ i=1,2$ be the natural maps.
  
  We will show that the $K$-mixed
  local systems on $Y_B$ are also external products. To see this we
  use notation and results from Section \ref{productlocal}.

  The data of $\mc{KZ}_{\kappa}(\vv{\lambda}'', \mu^*)$ gives a
  $K$-mixed local system on $U_B^1$ denoted by $\ml_1$. Similarly, the
  data of $\mc{KZ}_{\kappa}(\vv{\lambda}'_{\mu}, \nu^*)$ gives a
  $K$-mixed local system on $U_B^2$, denoted by $\ml_2$. The $K$-mixed
  local systems $\ml_1$ and $\ml_2$ are as in Section
  \ref{productlocal}.

  We know that $\mc{KZ}_{\kappa}(\vv{\lambda}'', \mu^*)$ can be
  expressed as a suitable isotypical component of the image of a map
  (by \eqref{Sho} in Proposition \ref{p:kz})
  $$H^{|B|}(g_{1,*} \mathcal{T}_1') \lra H^{|B|}(g_{1,*}
  \mathcal{T}_1)$$ for suitable objects $\mc{T}_1$, $\mc{T}_1'$ in
  $\db(\msm( {P_B^1}))_K$.  Similarly,
  $\mc{KZ}_{\kappa}(\vv{\lambda}'_{\mu}, \nu^*)$ is the image of a map
  $$H^{|B|}(g_{1,*} \mathcal{T}_2') \lra H^{|B|}(g_{1,*}
  \mathcal{T}_2)$$for suitable objects $\mc{T}_2$, $\mc{T}_2'$ in
  $\db(\msm({P_B^2}))_K$.  The following identifications {together
    with the Kunneth isomorphism (Lemma \ref{l:kf})} show that the image
  of \eqref{newSKZ} equals \eqref{newS2}, and therefore complete the
  proof of the lemma.
  \begin{claimn}
    \begin{equation}\label{extension1}
      \mathcal{T}_1'\boxtimes \mathcal{T}_2' \cong k_{B,!} (k'_B\circ
      k_B)^*\ma'^a(\ov{a(B)}),  
    \end{equation}
    \begin{equation}\label{extension2}
      \mathcal{T}_0\boxtimes \mathcal{T}_1 \cong k_{B,*}(k'_B\circ
      k'_B)^*\ma^a(\ov{a(B)}).
    \end{equation}
  \end{claimn}
  Note that the left hand sides of \eqref{extension1} and 
  \eqref{extension2} are objects on $Y_B$ which arise from the product
  arrangement and master functions, as in Section \ref{productlocal}
  with corresponding form $\eta-adt/t$:  The LHS of \eqref{extension2} is the
  object corresponding to Aomoto cohomology for $\eta-adt/t$ and the LHS of 
  \eqref{extension1}, the dual of the object for Aomoto cohomology for
  $-(\eta-adt/t)$.
  
  We now prove the claims: The desired isomorphisms have been
  constructed over $U_B=(Y_B)_0-E$ in Section \ref{productlocal}.  We
  mark out a divisor $D_B\subset Y_B$ as follows. The irreducible
  components of $D_B$ are of two types.  The first type are the
  divisors $E_j\cap Y_B$ with $\Res_{E_j}\eta\not\in\Bbb{Z}_{\leq
    0}$. The second type are the $Y_{B'} \cap Y_B$ with
  $a-a(B')\not\in \Bbb{Z}_{\geq 0}$, where
  $a(B') = \Res_{Y_{B'}}(\eta)$ is the residue of $\eta$ along
  $Y_{B'}$.
    
  We will verify the following assertions:
  \begin{enumerate}
  \item[(A)] Both sides of \eqref{extension2} have zero stalks at points of
    $D_B$.
  \item[(B)] Restricted to $Y_B- D_B$, both sides are star push
    forwarded from $U_B=(Y_B) -(E\cup D_B) = (Y_B)_0 - E$, where they are
    isomorphic as shown above.
  \end{enumerate}
  It is easy to see that the claim \eqref{extension2} follows from (A)
  and (B) by adjunction (first prove an isomorphism after restricting
  to $Y_B- D_B$, then use the canonical map from lower shriek
  extension to the two sides). Further, (A) and (B) are
  topological. The proof of the isomorphism \eqref{extension1} is
  similar, we only need to change the inequalities $\leq 0$ and
  $\geq 0$ in the definition of $D_B$ to strict inequalities $<0$ and
  $>0$ respectively.
  
  We have a topological description in Lemma \ref{comparison} of
  $\ma^a$ in terms of its restriction to $Y^a_0-E$. This implies a
  description of the RHS of \eqref{extension2} from its restriction to
  $U_B$. The LHS of \eqref{extension2} is controlled from its
  restriction to $U_B$ from the product weighted hyperplane
  arrangement, with form given by {the restriction of} $\eta-a
  dt/t$. The assertions (A), and (B), follow immediately using the
  following:
 \begin{enumerate}
 \item The residue of $\eta$ equals the residue of $\eta-adt/t$ along
   each $E_j$.
 \item The rational $1$-form $\eta-adt/t$ on $X$ is regular along
   $U_B\subseteq Y_B$.  Let $\eta_B$ be its restriction to $U_B$. The
   residue of $\eta_B$ along $Y_B\cap Y_{B'}\subseteq Y_B$ is the same
   as the residue of $\eta-a dt/t$ (considered as a rational $1$-form
   on $X$) along $Y_{B'}$ which is $a(B')-a(B)$.
  \end{enumerate}
 \end{proof}

\subsection{Conclusion of the proof of Theorem \ref{t:fact} for invariants}
Lemma \ref{prefinal} implies that equality holds in
\eqref{inegalite111}, since it gives the opposite inequality using the
fact that the rank of $\mathcal{KZ}_{\kappa}(\vv{\lambda},\nu^*)$
equals \eqref{summandKZ0}.

The desired increasing filtration
$\mr{G}^a (\Psi_{f, \bar{a}}\,
(\mathcal{KZ}_{\kappa}(\vec{\lambda},\nu^*)'))$ of
$\Psi_{f, \bar{a}} \, (\mathcal{KZ}_{\kappa}(\vec{\lambda},\nu^*)')$
indexed by $a \in (1/N)\Z$ is obtained as the image of
$\im((H^M(g_{Y,*} \mj'^{\leq a}(\bar{a}))^{\Sigma,\chi} \to (H^M(g_{Y,*}
\mj'^{\leq a}(\bar{a})))^{\Sigma,\chi})$ inside
$\im(H^M(g_{Y,*}\mc{F}(\bar{a}))^{\Sigma,\chi} \to H^M(g_{Y,*}
\mc{G}(\bar{a}))^{\Sigma,\chi})$, which is an injection by Lemma 
\ref{LA2}.  By Lemma \ref{LA2}, this filtration has graded quotients
equal to the image of \eqref{newSKZ}, which by Lemma \ref{prefinal} is
equal to the RHS of \eqref{e:kzfact}. 
   
The monodromy operator $\mr{N}$ on nearby cycles induces compatible
actions on all the objects occurring in the distinguished triangles
\eqref{exacto} and \eqref{exactoprime} since they are constructed
using functorial operations in $\db(\msm(Y))_K$; we never construct
any object as a cone. Since the map $\mc{F} \to \mc{G}$ is also
compatible with $N$, it follows that the filtration $\mr{G}^a$ is
preserved by $\mr{N}$. We now show that $\mr{N}$ acts by zero on the
graded quotients. It suffices to show that it is zero on
$k_{B,!} (k'_B\circ k_B)^*\ma'^a(\bar{a})$ and on
$k_{B,_*}(k'_B\circ k'_B)^*\ma^a(\bar{a})$ in \eqref{newSKZ} with
$\bar{c}=\bar{a}$. Again, as in the conformal blocks case, the actions
on these objects are determined by adjunction from their restrictions
to the open sets $U_B$, which are zero since the map
$f_X:X\to \A^1_{\msk}$ is smooth along $U_B$ and the nearby cycles
sheaf $\mc{F}_{B,\bar{a}}$ on $U_B$ is the nearby cycles sheaf of a
local system (see Lemma \ref{l:ls}). This concludes the proof of
Theorem \ref{t:fact} for invariants.

\section{Enriched representation and fusion rings} \label{s:tr}

Recall that we have defined (see Definition \ref{kappaRing}) enriched
representation rings $(\mf{R}_{\kappa}(\mf{g}),\star_{\mf{R}})$
(resp. fusion rings $(\mf{F}_{\kappa}(\mf{g}),\star_{\mf{F}})$) for
$\kappa \neq 0$ (resp.~for $\kappa$ as in Definition \ref{d:gtwist}).
We prove the basic associativity properties of these rings, complex
conjugation and its relationship with possible weights of the KZ
motive, as well as abstract properties of these enriched rings
(dependence on the parameter $\kappa$, the action of complex
conjugation on the parameters $\kappa$), as well as (conjectural)
relations between the enriched representation rings and fusion rings
parallel to what is known in the classical (unenriched) case.

\subsection{Associativity of enriched rings}\label{s:associativity}

The conformal blocks motives, as well as their nearby cycles, are pure
so we can define a more refined enriched fusion ring which also keeps
track of the weight.
\begin{definition}\label{d:fruv}
  $\mf{F}'_{\kappa}(\mf{g})$ is the free $\Z[u,v]$-module
  $\Z[u,v]^{P_{\ell}}$ with the $\Z[u,v]$-linear product given
  by
  \begin{equation*}
    [\lambda]\star_{\mf{F}'} [\mu] =\sum_{\nu\in
      P_{\ell}} p(\wt{C}_{\kappa}(\la,\mu,\nu^*))[\nu]
  \end{equation*}
  with $p(\wt{C}_{\kappa}(\la,\mu,\nu^*))$ as in Definition \ref{d:hn}.
\end{definition}

Analogously to Theorems \ref{t:assoc} and \ref{t:assoc2} we have:
\begin{theorem}\label{t:assoc2prime}
  The $\star_{\mf{F}'}$-product on $\mf{F}'_{\kappa}(\mf{g})$ gives it the
  structure of a unital, commutative and associative $\Z[t]$-algebra.
\end{theorem}
There is a isomorphism $\mf{F}'_{\kappa}(\mf{g})/(v-1)\leto{\sim} \mf{F}_{\kappa}(\mf{g})$ which sends $u\mapsto t$ and $v\mapsto 1$, and $[\lambda]\mapsto[\lambda]$ for $\lambda\in P_{\ell}$.

\begin{remark}\label{r:param}
  Recall that $\mf{F}_{\kappa}(\mf{g})$ and $\mf{F}'_{\kappa}(\mf{g})$
  are defined for $\kappa$ of the form $(\ell + h^{\vee})/a$, where
  $a$ is a positive integer such that $(a, N) = 1$ (and
  $ N = i_{\mf{g}}m_{\mf{g}}(\ell + h^{\vee})$). By construction, these
  rings depend only on the class of $a$ modulo $N$, however, it
  follows as in Lemma \ref{l:one} that
  $p(C_{(\ell + h^{\vee})/a}(\la,\mu;\nu)) = p(C_{(\ell +
    h^{\vee})/a'}(\la,\mu;\nu))$ if
  $a \equiv a' \mod m_{\mf{g}}(\ell + h^{\vee})$. Thus, the enriched
  rings for a fixed level $\ell$ are naturally parametrised by
  elements of $(\Z/ m_{\mf{g}}(\ell + h^{\vee})\Z)^{\times}$.
\end{remark}

\subsection{Proofs of Theorems \ref{t:assoc}, \ref{t:assoc2} and \ref{t:assoc2prime}}

The commutativity of all the rings $\mf{R}_{\kappa}(\mf{g})$,
$\mf{F}_{\kappa}(\mf{g})$ and $\mf{F}'_{\kappa}(\mf{g})$ holds because
the relevant motives do not change if we permute $\lambda$ and $\mu$;
this is clear from the definition of the master function in
\eqref{e:master}.  We now show that associativity follows from Theorem
\ref{t:fact}.

By the properties noted in Definition \ref{d:hn}, the
polynomials $P(\msf{M})$ (resp. $p(\msf{H})$) are additive in exact
sequences of MHS $\msf{M}$ (resp.~pure Hodge structures $\msf{H}$),
and take tensor products to ordinary products of polynomials.  Now
consider the $P$ polynomial of the nearby cycles in Theorem
\ref{t:fact} of the pull back of
$\mathcal{KZ}_{\kappa}(\vv{\lambda},\nu^*)$ to $S^*$. By Lemma
\ref{l:saito}, this $P$ polynomial is the same as the $P$ polynomial
of $\mathcal{KZ}_{\kappa}(\vv{\lambda},\nu^*)$. On the other hand, we
can calculate the $P$-polynomial of each of the graded pieces of the
filtration of the nearby cycles of
$\mathcal{KZ}_{\kappa}(\vv{\lambda},\nu^*)$ given in Theorem
\ref{t:fact}. This gives
\begin{equation}\label{e:Passoc}
P(\mathcal{KZ}_{\kappa}(\vv{\lambda},\nu^*))=\sum_{\lambda\in
	\Lambda}
P(\mathcal{KZ}_{\kappa}(\lambda_1,\dots,\lambda_m,\lambda^*))\cdot
P(\mathcal{KZ}_{\kappa}(\lambda,\lambda_{m+1}\dots,\lambda_n,\nu^*)) .
\end{equation} 
Taking $m=2$, we get the associativity relations in Theorem
\ref{t:assoc} for $\mf{R}_{\kappa}({ \kappa})$. The proof of
associativity for the fusion rings $\mf{F}_{\kappa}(\mf{g})$ is
similar. For $\mf{F}'_{\kappa}(\mf{g})$ we note that by
\eqref{e:cbfact} the nearby cycles of the conformal block variations
are pure of the same weight as the variation itself. Knowing the
weight of a pure Hodge structure one can determine the $p$ polynomial
from the $P$ polynomial (since the $p$ polynomial is homogenous of
degree equal to the weight). Therefore associativity of the product of
$\mf{F}_{\kappa}(\mf{g})$ implies that of $\mf{F}_{\kappa}'(\mf{g})$.

\begin{remark}\label{r:arbitraryprod}
By induction on $n$, it follows from  \eqref{e:Passoc} that the coefficient of $[\nu]$ in $[\lambda_1]\star_{\mf{R}} \dots \star_{\mf{R}} [\lambda_n]\in \mf{R}_{\kappa}(\mf{g})$  equals 
$P(\msf{KZ}_{\kappa}(\vec{\lambda},\nu^*)_{\vec{z}})$ for any  $\vec{z}\in \mc{C}_n$.  Similarly, when $\lambda_1,\dots,\lambda_n,\nu\in P_{\ell}$, the coefficient of  $[\nu]$ in $[\lambda_1]\star_{\mf{F}'} \dots \star_{\mf{F}'} [\lambda_n]\in \mf{F}'_{\kappa}(\mf{g})$    equals $p(\mathcal{CB}_{\kappa}(\vec{\lambda},\nu^*)_{\vec{z}})$. Therefore 
one can compute all Hodge polynomials of KZ and CB motives using computations in the rings $(\mf{R}_{\kappa}({ \kappa}),\star_{\mf{R}})$ and  $(\mf{F}'_{\kappa}(\mf{g}),\star_{\mf{F}'})$.
\end{remark}

\subsection{On the structure of $\mf{R}_{\kappa}(\mf{g})$,
  $\mf{F}_{\kappa}(\mf{g})$ and
  $\mf{F}'_{\kappa}(\mf{g})$} \label{s:structure}

By construction, the structure ``constants'' for both the rings
$\mf{R}_{\kappa}(\mf{g})$ and $\mf{F}'_{\kappa}(\mf{g})$ are integral
polynomials with non-negative coefficients and a fundamental problem
is to compute these polynomials. The degrees of
$P(\mathcal{KZ}_{\kappa}(\lambda,\mu;\gamma^*))$ and
$p(\mathcal{CB}_{\kappa}(\lambda,\mu;\gamma^*))$ are bounded by the
dimension $M$ of the corresponding hyperplane arrangement and these
rings specialize to the classical representation and fusion rings. In
this section we will prove some basic results and formulate some
conjectures about these rings for general $\mf{g}$.

\subsubsection{}\label{s:integral}
When the master function associated to a weighted hyperplane
arrangement in $\A^M_{\msk}$ is single valued, then the associated Aomoto
motive is mixed Tate by Corollary \ref{c:tate}. However, even in this case
the weights of the motive are not determined by $M$ since the signs of
the residues along various divisors need not be constant, so the Hodge
polynomial can be complicated.
  \begin{example}\label{e:aomoto}
    Let $M=1$ and let $\eta = \sum_{i=1}^n a_i\frac{dt}{t-z_i}$, where
    $a_i \in \Z$ and $\vec{z} = (z_1,z_2,\dots,z_n) \in \mc{C}_n$. If
    at least two $a_i$ are negative and two $a_i$ are positive, then
    the corresponding Aomoto $H^1$ is not pure. Furthermore, the image
    of the monodromy action on the Aomoto $H^1$ is infinite as
    $\vec{z}$ varies in $\mc{C}_n$.
  \end{example}

Nevertheless, we have the following:
\begin{conj}\label{c:mtate}
  If $1/\kappa$ is a positive integral multiple of $m_{\mf{g}}$ the
  canonical equality
  $\mf{R}_{\kappa}(\mf{g}) = \mf{R}(\mf{g}) \otimes \Z[t]$ as
  $\Z[t]$-modules is an equality of $\Z[t]$-algebras.
\end{conj}
This conjecture is equivalent to the apparently stronger assertion
that the mixed local systems
$\mathcal{KZ}_{\kappa}(\vv{\lambda},\nu^*)$ are of weight zero for
such $\kappa$ and all $\vv{\lambda}$ and $\nu$. We prove this
statement in the case of $\mf{sl}_n$ as Corollary \ref{c:sln}.
This implies that the monodromy of the KZ equation is finite for such
$\kappa$. 

It is not the case that KZ motives are always pure of weight zero when
$\kappa > 0$ and the master function $\ms{R}$ is single valued, see
Example \ref{e:mtate}.

\subsubsection{The case $\kappa = \ell + h^{\vee}$ and the action of
  complex conjugation}
      
When $\kappa = \ell + h^{\vee}$,  it follows from Remark
  \ref{r:cb} that the polynomials
$p(\mathcal{CB}_{\kappa}(\lambda,\mu;\gamma^*))$ are equal to $ru^M$,
where $r$ is the rank of the motive
$\mathcal{CB}_{\kappa}(\lambda,\mu;\gamma^*)$ and $M$ (the integer
defined in Section \ref{s:sv}) is its weight. Since $r$ is determined
by the classical fusion ring and $M$ is determined in an elementary
way from $\lambda$, $\mu$ and $\gamma^*$, it follows that
$\mf{F}'_{\kappa}(\mf{g})$ can be determined in an elementary way from
$\mf{F}(\mf{g})$.

For $\kappa = (\ell + h^{\vee})/a$ as in Definition \ref{d:gtwist},
let $\bar{\kappa} = (\ell + h^{\vee})/(\ell + h^{\vee}-a)$ and let
$\tau: \Z[u,v] \to  \Z[u,v]$ be the involution given by interchanging
$u$ and $v$. Then $\mf{F}'_{\bar{\kappa}}(\mf{g})$ equals
$\mf{F}'_{\kappa}(\mf{g}) \otimes_{\Z[u,v], \tau}\Z[u,v]$ as a
$\Z[u,v]$-algebra. This is clear since replacing $\kappa$ by
$\bar{\kappa}$ corresponds to the action of complex conjugation on the CB
motives and $h^{p,q}(\msf{H})$ = $h^{q,p}(\ov{\msf{H}})$ for any pure
Hodge structure $\msf{H}$ with coefficients in any subfield of $\C$
with $\ov{\msf{H}}$ the conjugate Hodge structure.

However, the motives $\mathcal{KZ}_{\kappa}(\lambda,\mu;\gamma^*)$
for $\kappa = \ell + h^{\vee}$ need not be pure of weight $M$ and the
polynomials $P(\mathcal{KZ}_{\kappa}(\lambda,\mu;\gamma^*))$ are,
in general, more complicated, e.g., they are not always a scalar times
a monomial. Therefore, the product in the rings
$\mf{R}_{\kappa}(\mf{g})$ in the basis given by dominant integral weights
depends in a non-trivial way on $\kappa$ and cannot in general be
determined from the corresponding product in $\mf{R}(\mf{g})$. For the
same reason, for general $\kappa$ the products in the rings
$\mf{R}_{\kappa}(\mf{g})$ and
$\mf{R}_{\bar{\kappa}}(\mf{g})$ do not determine each other in
a simple way. Information about the Hodge polynomial of the complex conjugate motive
gives more constraints on the weights:
\begin{lemma}\label{l:weights}
  Let $f$ be the Hodge polynomial of a
  complex MHS $\msf{M}$ (assumed to be effective) and $\bar{f}$ the
  Hodge polynomial of the complex conjugate MHS
  $\overline{\msf{M}}$. Let $d_{min}$ (resp. $d_{max}$) be the minimal
  (resp. maximal) degree of a nonzero monomial in $f$ and define
  $\bar{d}_{min}$ and $\bar{d}_{max}$ similarly. Then the minimal
  weight occurring in $\msf{M}$ is at least $d_{min} + \bar{d}_{min}$
  and the maximal weight is at most $d_{max} + \bar{d}_{max}$.
\end{lemma}

\begin{proof}
  Given the nature of the claimed inequalities we can assume that
  $\msf{M}$ is pure of a fixed weight. The lemma then follows from the
  fact that $\overline{\msf{M}}^{p,q} = \overline{\msf{M}^{q,p}}$ and
  the definition of the Hodge polynomial.
\end{proof}

In particular, if the rank of $V$ is one, or more generally if the
Hodge polynomial of $V$ is a monomial, and we define $d$ and
$\bar{d}$ in the obvious way, then the weight of $V$ is $d + \bar{d}$.

\subsubsection{$\mf{R}_{\kappa}(\mf{g})$ is a polynomial ring}
By construction, the rings $\mf{R}_{\kappa}(\mf{g})$ are free as
$\Z[t]$-modules and the rings $\mf{F}'_{\kappa}(\mf{g})$ are free as
$\Z[u,v]$-modules. We can in fact describe all
$\mf{R}_{\kappa}(\mf{g})$ as abstract $\Z[t]$-algebras.

\begin{lemma}\label{l:poly}
  For any $\mf{g}$ and any $\kappa \neq 0$, the ring
  $\mf{R}_{\kappa}(\mf{g})$ is a polynomial ring over $\Z[t]$ with
  free generators the classes of the fundamental dominant weights.
\end{lemma}
\begin{proof}
  By the construction of $\mf{R}_{\kappa}(\mf{g})$ it follows that
  $\mf{R}_{\kappa}(\mf{g})/(t-1)$ is isomorphic to $\mf{R}(\mf{g})$
  which is well-known to be a polynomial ring over $\Z$ in the the
  classes of the fundamental dominant weights.  For $\lambda$, $\mu$
  dominant integral weights for $\mf{g}$, the coefficient of
  $[\lambda + \mu]$ in $[\lambda]\star[\mu]$ is equal to $1$ (as the
  motive is the motive of a point) and any $\nu$ such that $[\nu]$ has
  a nonzero coefficient in $[\lambda]\star[\mu]$ must satisfy
  $\nu < \lambda + \mu$ (since the structure ``constants'' of
  $\mf{R}_{\kappa}(\mf{g})$ are polynomials with non-negative
  coefficients). This implies that $\mf{R}_{\kappa}(\mf{g})$ is
  generated as a $\Z[t]$-algebra by the classes of the fundamental
  dominant weights.  Thus, we have a surjection
  \begin{equation}\label{e:rings}
    \Z[t][\vpi_1,\vpi_2,\dots,\vpi_r] \lra 
    \mf{R}_{\kappa}(\mf{g}) ,
  \end{equation}
  where $r$ is the rank of $\mf{g}$ and the $\vpi_i$ are the
  fundamental weights.
  
  By construction, $\mf{R}_{\kappa}(\mf{g})$ is free as a
  $\Z[t]$-module so $t-1$ is a nonzero divisor.  It follows that the
  Krull dimension of both rings in \eqref{e:rings} is the same. Since
  the source of the map is a (finite dimensional) integral domain, the
  map must be an isomorphism.
\end{proof}

Although, as shown by Lemma \ref{l:poly}, the structure of
$\mf{R}_{\kappa}(\mf{g})$ as an abstract $\Z[t]$-algebra is
quite simple, this does not help one much in computing all products of
the form $[\lambda]\star [\mu]$ which are the primary objects of
interest. In Section \ref{s:sln} we will give an effective algorithm
for carrying out such computations (also for
$\mf{F}_{\kappa}(\mf{g})$) for all $\mf{sl}_n$.

\subsubsection{A map from  $\mf{R}_{\kappa}(\mf{g})$ to
	$\mf{F}_{\kappa}(\mf{g})$}\label{s:mapRF}

Let $\ell \geq 0$ and let $\kappa= (\ell + h^{\vee})/a$ be as in Definition
\ref{d:gtwist}. In the classical case, natural additive maps $\pi$
from the representation ring to the fusion rings given by an explicit
formula were constructed by Faltings \cite[Appendix]{Faltings} and he
proved that they were ring homomorphisms for classical $\mf{g}$  and $G_2$. A uniform proof of this fact was given later by Teleman
\cite{TelemanCMP}. We do not have a generalisation of the explicit
formula for our enriched rings, however since $\mf{R}_{\kappa}(\mf{g})$
is a polynomial ring by Lemma \ref{l:poly}, one can define a map of
$\Z[t]$-algebras by assigning arbitrary values to the fundamental
dominant integral weights.

\begin{defi}\label{d:pi}
  { Let $\mf{g}$ and $\ell$ be such that all fundamental dominant
    weights are of level $\ell+1$.}  We let
	\begin{equation*}
		\pi_{\kappa}: \mf{R}_{\kappa}(\mf{g}) \lra  \mf{F}_{\kappa}(\mf{g})
	\end{equation*}
	be the $\Z[t]$-algebra homomorphism defined by
	\begin{equation*}
          \pi_{\kappa}([\vpi]) = \begin{cases}
            [\vpi] & \mbox{ if } \vpi \mbox{ is of level } \ell, \\
            0 & \mbox{ if } \vpi \mbox{ is of {exact} level } \ell + 1,
          \end{cases}
        \end{equation*}
        for all fundamental dominant weights $\vpi$.
\end{defi}

\begin{remark}\label{r:map}
  $ $
  \begin{enumerate}
  \item For classical $\mf{g}$ all fundamental dominant weights
    are of level $2$.
  \item If we set $t=1$ in both the rings we recover the maps
    defined by Faltings; this implies that $\pi_{\kappa}$ is
    always surjective.
  \end{enumerate}
\end{remark}

\subsubsection{}\label{s:affine}
Although the maps $\pi_{\kappa}$ are defined in a somewhat ad hoc
manner, we expect them to have an intrinsic meaning. To explain this
we first recall the definition of $\pi$: Let $P$ be the weight lattice
of $\mf{g}$, $P^+$ the subset of dominant integral weights and
$P_{\ell} \subset P^+$ the set of weights of level $\ell$.  The Weyl
group $W$ acts linearly on and this action extends to an (affine)
action of the \emph{affine Weyl group} $W_{\ell}$: this is the group
generated by $W$ and the translation given by
$x \mapsto x + (\ell + h^{\vee})\theta$. The group $W_{\ell}$ is a
Coxeter group and the signature homomorphism $\ve: W \to \{\pm 1\}$
extends to a homomorphism $\ve: W_{\ell}: \to \{\pm 1\}$ by defining
the signature of any translation to be $1$. The \emph{affine walls} of
$P \otimes \R$ are the subsets given by
$(\lambda,\alpha)=m(\ell+h^{\vee})$ for each root $\alpha$ and
$m \in \Z$. Note that $h^{\vee}=n$ for $\mf{g}=\mf{sl}_n$.  The
connected components of the complement of the union of all affine
walls are called \emph{alcoves}.  The \emph{fundamental alcove} is the
alcove defined by the inequalities $(\lambda,\alpha) > 0$ for all
simple roots $\alpha$ and $(\lambda,\theta) < \ell +h^{\vee}$.

We let $l:W_{\ell} \to \N$ be the \emph{length function} corresponding to the
fundamental alcove: $W_{\ell}$ acts simply transitively on the set of
alcoves, so it can be identified with the set of alcoves using the
fundamental alcove as the base point. Then the length of any element
of the Weyl group is the minimal number of reflections (in the affine
walls) needed to move the alcove to the fundamental alcove.

\smallskip
Let $\rho$ be the half sum of the positive roots. Then for any
dominant integral weight $\lambda$, by \cite[Proposition 8.2]{beauville} we
have:
\begin{itemize}
\item[--] $\pi([\lambda]) = 0$ iff there does not any exist
  $w \in W_{\ell}$ such that $w(\lambda + \rho) - \rho \in P_{\ell}$.
\item[--] $\pi([\lambda]) = \ve(w)[\mu]$ if
  $w(\lambda + \rho) - \rho = \mu \in P_{\ell}$ for some
  $w \in W_{\ell}$.
\end{itemize}
We note that in the second item $\mu$ does not depend on $w$ (which is
not unique).

\smallskip

We formulate two conjectures giving a more precise description of
$\pi_{\kappa}$ in terms of the basis $[\lambda]$ of
$\mf{R}_{\kappa}(\mf{g})$. The first, for general $\kappa$, is:
\begin{conj}\label{c:maptofusion}
  For the map $\pi_{\kappa}$ defined above and any dominant integral weight
  $\lambda$ we have:
  \begin{enumerate}
  \item $\pi_{\kappa}([\lambda]) = 0$ iff there does not any exist
    $w \in W_{\ell}$ such that
    $w(\lambda + \rho) - \rho \in P_{\ell}$.
  \item $\pi_{\kappa}([\lambda]) = \ve(w) p_{\kappa}(\lambda)[\mu]$ if
    $w(\lambda + \rho) - \rho = \mu \in P_{\ell}$ for some
    $w \in W_{\ell}$ and $p_{\kappa}(\lambda)$ is a monomial in $t$.
  \end{enumerate}
\end{conj}

We now make a precise conjecture for the coefficients
$p_{\kappa}(\lambda)$ from Conjecture \ref{c:maptofusion} in two
special cases.
\begin{defi}\label{d:C}
  Given any weight $\lambda$ write it as $\sum a_i \alpha_i$, where
  $\alpha_i$ are simple roots. We define the function
  $\mf{s}: P \to \Q$ by setting $\mf{s}(\lambda) = \sum_i a_i$.
\end{defi}

\begin{conj}\label{c:ec}
  Let $\lambda$ be a dominant integral weight and assume that
  $w(\lambda + \rho) - \rho = \mu \in P_{\ell}$ for some
  $w \in W_{\ell}$. If $\kappa = \ell + h^{\vee}$ then
  \begin{enumerate}
  \item
    $\pi_{\kappa}([\lambda]) = \ve(w) t^{\mf{s}(\lambda - \mu)}[\mu]$.
  \item $\pi_{\bar{\kappa}}([\lambda]) = \ve(w) t^{l(w)}[\mu]$ when
    $l(w)$ is minimal (among all $w$ as above).
  \end{enumerate}
\end{conj}
We prove both the above conjectures when $\mf{g} = \mf{sl}_2$ in Theorem
\ref{t:sl2} and have also checked them computationally for $\mf{sl}_n$
in many cases using the results of Section \ref{s:sln}.

\begin{remark}\label{r:tele}
	It would be very interesting to have a more conceptual definition of
	the maps $\pi_{\kappa}$, perhaps an extension of the description of
	$\pi$ in \cite{TelemanCMP}. But note that in our setting the source
	itself is not well understood.
\end{remark}

\subsubsection{}
For the classical maps $\pi$ explicit generators for $\ker(\pi)$ (see
\cite{douglas} and the references therein) are known. We do not know
explicit generators for $\ker(\pi_{\kappa})$ for general $\mf{g}$ but we
have:
\begin{lemma}
  If $\mf{g}$ is $\mf{sl}_n$ or $\mf{sp}_{2n}$, the dominant integral
  weights of exact level $\ell + 1$ generate $\ker(\pi_{\kappa})$.
\end{lemma}
\begin{proof}
  The proof is essentially the same as in the case of the classical
  fusion ring (\cite[p.~216]{bouw-ridout}.
	
  When $\mf{g}$ is $\mf{sl}_n$ or $\mf{sp}_{2n}$ all fundamental
  weights are of level $1$ so any dominant integral weight $\nu$ of
  exact level $\ell + 1$ can be written as a sum
  $\nu = \lambda + \vpi$ where $\lambda$ is of level $\ell$ and $\vpi$
  is a fundamental dominant integral weight. Given the definition of
  the $\star$ products in $\mf{R}_{\kappa}(\mf{g})$ and
  $\mf{F}_{\kappa}(\mf{g})$ this implies that $[\nu]$ is in
  $\ker(\pi_{\kappa})$.
	
  For the converse, we use that any $\nu\in P^+$ of exact level
  $> \ell$ can be written as a sum $\nu = \lambda + \mu$ where
  $\lambda$ is of exact level $\ell + 1$ and $\mu\in P^+$. Any summand
  $V_{\nu'}$ of $V_{\lambda} \otimes V_{\mu}$ is such that
  $(\nu',\theta) \leq (\nu, \theta)$ and $(\nu',\rho) <
  (\nu,\rho)$. This implies by induction on the level and height that
  any element of $\mf{R}_{\kappa}(\mf{g})$ is equal, modulo the ideal
  generated by all dominant integral weights of exact level
  $\ell + 1$, to a $\Z[t]$-linear combination of dominant integral
  weights of level $\leq \ell$.
\end{proof}

\begin{remark}
  If $\mf{g}$ is $\mf{sl}_n$ or $\mf{sp}_{2n}$, $\ker(\pi)$ is known
  to be generated by the set
  \begin{equation*}
    \{\ell \vpi_1 + \vpi_i \ | \ i=1,2,\dots,n\},
  \end{equation*}
  where the $\vpi_i$ are the fundamental dominant integral weights
  (\cite[Section 3.2]{bouw-ridout}). Since all these weights are of
  level $\ell + 1$ they are also in the kernel of
  $\ker(\pi_{\kappa})$, but they do not seem to generate it even in
  type $A_2$.
\end{remark}

\subsection{Replacing $1/\kappa$ by $ m_{\mf{g}}+1/\kappa$} \label{s:trp}

Suppose $\kappa>0$. Let $\kappa'=\frac{\kappa}{m_{\mf{g}}\kappa +1}$, so that $1/\kappa'= m_{\mf{g}}+ 1/\kappa$. Let $\lambda_1,\dots,\lambda_n,\nu$ be dominant integral weights for $\mf{g}$. 

\begin{lemma} \label{l:one}
The Hodge polynomials of
$\mathcal{KZ}_{\kappa}(\vec{\lambda},\nu^*)_{\vec{z}})$ and $\mathcal{KZ}_{\kappa'}(\vec{\lambda},\nu^*)_{\vec{z}})$
are the same for any $\vec{z}\in \mc{C}_n(\C)$. Thus, $\mf{R}_{\kappa}(\mf{g})=\mf{R}_{\kappa'}(\mf{g})$. 
\end{lemma}

\begin{proof}
For arbitrary $\mf{g}$ and real $\kappa>0$, we note that in Lemma
\ref{l:topo}, we may take $P$ to be $V=P-\cup'_{\alpha}E_{\alpha}$,
where the union is restricted to $\alpha$ such that
$a_{\alpha}=\op{Res}_{E_{\alpha}}\eta$, the residue of $\eta$ is a
real number $\leq 0$, and $V'=P-\cup'_{\alpha}E_{\alpha}$, where the
union is restricted to $\alpha$ such that
$a_{\alpha}=\op{Res}_{E_{\alpha}}\eta$, is a real number $<0$.  Then
under addition of $m_{\mf{g}}$ (see Section \ref{s:notations}) to
$1/\kappa$ (indeed when $\frac{1}{\kappa}$ is scaled by any positive
real number), the varieties $V$ and $V'$ do not change. The
restriction of the corresponding local systems to the fibre over a
fixed $\vec{z} \in \mc{C}_n(\C)$ also do not change because we may
ignore the $z_i-z_j$ factors in the definition of the master function.
Therefore we may add $m_{\mf{g}}$ to $\frac{1}{\kappa}$ for a fixed
$\vec{z}$, and not change the Hodge polynomial of
$\mathcal{KZ}_{\kappa}(\vec{\lambda},\nu^*)_{\vec{z}})$.
\end{proof}

Note that the local systems  $\mathcal{KZ}_{\kappa}(\vec{\lambda},\nu^*)$ and $\mathcal{KZ}_{\kappa'}(\vec{\lambda},\nu^*)$ on $\mathcal{C}_n$ corresponding to may not be isomorphic (for example local monodromies could be different). However, we have the following:
\begin{lemma}\label{l:two}
Suppose $1/\kappa''=\frac{1}{\kappa}+m_{\mf{g}}i_{\mf{g}}$. Then the mixed  local systems  $\mathcal{KZ}_{\kappa}(\vec{\lambda},\nu^*)$ and $\mathcal{KZ}_{\kappa''}(\vec{\lambda},\nu^*)$ on $\mathcal{C}_n$  are isomorphic. 
\end{lemma}

\begin{proof}
By easy inspection, the master functions for $\kappa$ and $\kappa''$  differ by a rational function on $\mathcal{C}_n$, and hence the mixed local systems $\mc{L}(\eta)$ for $\kappa$ and $\kappa''$  in Definition \ref{d:kzm} coincide. The rest of the proof is similar to the proof of Lemma \ref{l:one}.
\end{proof}

It is interesting to compare Lemma \ref{l:two} with results of
Kazhdan--Lusztig when $\kappa$ is irrational. It follows from their
results (see \cite[Theorem 8.6.4]{EFK}) that the monodromy of the KZ
local system is controlled by representations of a quantum group
$\mathfrak{U}_q(\mf{g})$ where $q=\exp(\pi i/m_{\mf{g}}\kappa)$ (with
$m_{\mf{g}}$ as above). The $R$-matrix operators in
$\mathfrak{U}_q(\mf{g})$ use coefficients in
$\mb{C}(q^{\frac{1}{i_{\mf{g}}}})$ (see \cite[Section 1.3]{BK}). So
adding $m_{\mf{g}}i_{\mf{g}}$ to $\frac{1}{\kappa}$ will not change
the monodromy of the KZ equations. This agrees with the geometric
picture above when $\kappa>0\in \mb{R}^{\times}$. The quantum group
picture is valid for all irrational $\kappa\in \mb{C}^{\times}$. To
treat the case of $\kappa\in\mb{C}-(\mb{R}^+\cup \mb{Q})$
geometrically, we need to make sure that we can choose the same $V$
and $V'$ as we perform the operation of adding $m_{\mf{g}}i_{\mf{g}}$
to $\frac{1}{\kappa}$. It is easy to see that the residues
$a_{\alpha}$ are either zero (in which case they remain invariant
under the operation) or irrational (and remain so after the operation
which adds a rational number to these residues). Since irrational
strata (i.e., those strata with residue of $\eta$ irrational) can be
removed for both $V$ and $V'$, we get the desired statement
geometrically.

We do not know if one can see this invariance of monodromy directly
from the KZ equations \eqref{e:KZ}.

\section{The enriched rings for $\mf{sl}_n$}\label{s:sln}

Throughout this section, $\mf{g} = \mf{sl}_n$. Our main goal is to
describe an explicit algorithm for computing all products in the rings
$\mf{R}_{\kappa}(\mf{sl}_n)$ and $\mf{F}_{\kappa}(\mf{sl}_n)$. Using
this algorithm we prove Conjecture \ref{c:mtate} for all $n$ and
Conjectures \ref{c:maptofusion} and \ref{c:ec} for $n=2$.

  From now on, when the ring under consideration is clear from the
  context, we will often drop $\mf{R}$ and $\mf{F}$ from
  $\star_{\mf{R}}$ and $\star_{\mf{F}}$ and simply use $\star$ to
  denote the product.

\subsection{The multiplication algorithm for
  $\mf{sl}_n$}\label{s:algo}

The algorithm is based on the classical Pieri formula for the tensor
product with a fundamental representation, so we first recall its
statement.

Recall that we may write the simple roots for $\mf{sl}_n$ as
$\alpha_i=L_i-L_{i+1}$ for $1\leq i<n$, and $L_1+\dots+L_n=0$. Let
$\varpi_1,\dots, \varpi_{n-1}$ be the fundamental dominant weights for
$\mf{sl}_n$. We have $\varpi_k = L_1 + L_2 +\dots + L_k$.

\begin{lemma}[Pieri's formula]\label{l:pieri}
  Let $\lambda$ be a dominant integral weight for $\mf{sl}_n$. Then
  $V_{\lambda}\tensor V_{\vpi_k}$ is a sum without multiplicities of
  all $V_{\mu}$ such that $\mu=\lambda +\vpi_k - \sum n_i\alpha_i$ for
  some $n_i\geq 0$, where for some $K\subset\{1,\dots,n\}$ with
  $|K|=k$, $$\sum n_i\alpha_i =L_1+\dots+L_{k}- \sum_{i\in K}L_i \ .$$
  If $\lambda$ is a weight of level $\ell>0$, then the fusion product
  of $\lambda$ and $\vpi_{k}$ corresponds to the $\mu$ above (also
  with multiplicity $1$) which are of level $\ell$.
\end{lemma}
See \cite[Proposition 15.2.5 (ii)]{fulton-harris} for the classical
case and \cite[Examples 7.3 a), b)]{beauville} for the fusion ring.

\begin{algo}\label{a:sln}
  Let $\lambda$, $\mu \in P^+$ be arbitrary. The following three steps
  allow one to compute $[\lambda]\star[\mu]$ in
  $\mf{R}_{\kappa}(\mf{sl}_n)$ and $\mf{F}_{\kappa}(\mf{sl}_n)$.
\begin{enumerate}
\item Computation of the coefficients that occur in all Pieri products
  $[\lambda]\star[\vpi_1]$. This is implied by Theorem \ref{genp},
  which is a motivic version of this computation, and we prove this in
  Section \ref{s:vpi1}. By duality this implies that we may also
  compute all products $[\lambda]\star[\vpi_{n-1}]$.
\item Inductive (in $n$) computation of all Pieri products
  $[\lambda]\star[\vpi_k]$. We carry this out in Section \ref{s:pieri}.
\item Computation of all products $[\lambda]\star [\mu]$. This follows
  from (2) since $\mu$ can be written as a polynomial in the
  fundamental dominant weights, as we have seen in Lemma \ref{l:poly}.
  This can be done algorithmically, as in the classical case, by using
  the partial order on the dominant weights: If
  $\mu = \sum_{i=1}^{n-1} a_i \varpi_i$ with $a_i \in \N$, if $[\nu]$
  occurs in
  $[\varpi_1]^{\star a_1} \star [\varpi_2]^{\star a_2} \star \dots
  \star [\varpi_{n-1}]^{\star a_{n-1}}$ with a nonzero coefficient
  then $\nu \leq \mu$. Furthermore, $[\mu]$ always occurs in the
  product with coefficient $1$.
\end{enumerate}
In the above, $\star$ can be $\star_{\mc{R}_{\kappa}}$ or
$\star_{\mc{F}_{\kappa}}$.
\end{algo}

As an application of this algorithm, we prove Conjecture \ref{c:mtate}
for $\mf{sl}_n$ in Section \ref{s:int}.
\begin{remark}\label{r:real}
  The number of steps in the algorithm needed to compute
  $[\lambda]\star[\mu]$ is independent of $\kappa$; this allows one to
  define rings $\mf{R}_{\kappa}(\mf{sl}_n)$ for all real $\kappa$. The
  reason is that the formula in Lemma \ref{DelMos}, and therefore also
  Lemma \ref{genp} (at the level of the Hodge polynomial), make sense
  for real $\kappa$ so we can define
  $[\lambda]\star_{\kappa}[\varpi_1]$ for all
  $\kappa \in \R^{\times}$--- we use the subscript $\kappa$ here
    to emphasize the dependence of the product on $\kappa$---and then
  use the algorithm above to define $[\lambda]\star_{\kappa}[\mu]$ for
  all $\mu$. The fact that this leads to a commutative and associative
  $\Z[t]$-algebra follows by approximation: the algorithm shows that
  for any irrational $\kappa$,
  $[\lambda]\star_{\kappa}[\mu] = [\lambda]\star_{\kappa'}[\mu]$ for
  any rational $\kappa'$ sufficiently close to $\kappa$ (depending on
  $\lambda$ and $\mu$). The commutativity and associativity of
  $\star_{\kappa}$ for real $\kappa$ then follows from the case of
  rational $\kappa$.

  It should be possible to define $\Z[t]$-algebras
  $\mf{R}_{\kappa}(\mf{g})$ for all $\mf{g}$ by combining the methods
  of this article with the theory of complex mixed Hodge modules being
  developed by Sabbah and Schnell \cite{ss}. For $\mf{sl}_n$ we expect
  that the resulting algebra is the same as that defined above by
  approximation.
\end{remark}

  Before explaining the details of the algorithm, we give a simple
  informal illustration of how the algorithm works in the case of the
  enriched representation ring.
\begin{example}\label{e:algo}
  Let $\mf{g} = \mf{sl}_4$. Then
  $V_{\vpi_2}\otimes V_{\vpi_2} = V_{2\vpi_2} \oplus V_{\vpi_1 +
    \vpi_3} \oplus V_{0}$ and suppose we want to compute the
  coefficient of $[0]$ in $[\vpi_2]\star[\vpi_2]$. We observe that
  $V_{\vpi_1} \otimes V_{\vpi_1} = V_{2\vpi_1} \oplus V_{\vpi_2}$ and
  $V_{2\vpi} \otimes V_{\vpi_2} = V_{\vpi_1 + \vpi_3} \oplus
  V_{2\vpi_1 + \vpi_2}$, so it follows that the coefficient we want to
  compute is equal to the coefficient of $[0]$ in
  $([\vpi_1]\star[\vpi_1])\star[\vpi_2]$. This can be computed using
  Theorem \ref{genp} and the associativity of the product.
\end{example}

\subsection{The star product with $\varpi_1$}\label{s:vpi1}

\begin{defi}
  Define $\langle\alpha\rangle \in [0,1)$ by
  $\langle \alpha\rangle =\alpha  -\lfloor \alpha\rfloor$ for any
  $\alpha\in \mb{R}$.
\end{defi}
The Hodge data for the Galois conjugates $[a;\kappa]$ are described in
the following. The first part follows from \cite[Corollary 2.21]
{DeligneMostow}. Assume $\kappa>0$ and $1/\kappa$ is not an integer.
\begin{lemma}\label{DelMos}%
  $ $
  \begin{enumerate}
  \item If $a/\kappa$ and $(1+a)/\kappa$ are both in $\mb{Q}-\mb{Z}$,
    then $[a;\kappa]$ is pure of weight $1$ whose Hodge polynomial is
    determined from the integer
    $\langle -1/\kappa\rangle +\langle -a/\kappa\rangle +\langle
    (1+a)/\kappa \rangle$: If this integer is $2$ then the Hodge
    polynomial of $[a;\kappa]$ is $t$, and the Hodge polynomial equals
    $1$ if this integer is $1$.
  \item $a/\kappa\in\mb{Z}$ then as MHS,
    $[a;\kappa]= K(0)$. Therefore the Hodge polynomial of
    $[a;\kappa]$ is $1$.
  \item If $(1+a)/\kappa\in\mb{Z}$ then as MHS,
    $[a;\kappa]=K(-1)$. Therefore the Hodge polynomial of
    $[a;\kappa]$ is $t$.
  \end{enumerate}
\end{lemma}
\begin{remark}\label{r:int}
  If $\kappa>0$, and $1/\kappa$ and $a$ are integers, then
  $[a;\kappa]= H^1(\mb{A}^1,\{0,1\};K) =K(0)$. If
  $\kappa<0$ and $1/\kappa$ is an an integer, then
  $[a;\kappa]= H^1_c(\mb{A}^1 -\{1,0\},K)=K(-1)$.
\end{remark}

\begin{proof}[Proof of Theorem \ref{genp}]        
  Suppose $V_{\mu}$ appears in the tensor product
  $V_{\lambda}\tensor V_{\vpi_1}$.  Then by Lemma \ref{l:pieri},
  $\mu=\lambda +\vpi_1 - \sum n_i\alpha_i$ where for some
  $M\in \{0,\dots,n-1\}$,
  $\sum n_i\alpha_i =L_1-L_{M+1}=\alpha_1+\dots +\alpha_M$. We assume
  $M>0$.  If $M =1$ then the lemma follows from Lemma \ref{DelMos} so
  we may assume $M>1$. If $n=2$ then $M \leq 1$, so we may also assume
  that $n>2$.

Let $\delta_a={(\lambda,\alpha_a)}$ for $a=1,\dots,n-1$. Assuming
$\delta_a\geq 0$ for all $a$,
$\mu=\lambda+\vpi_1 -(\alpha_1+\dots+\alpha_M)$ is dominant if and
only if $\delta_M>0$ (consider the two cases $M=1$ and $M>1$
separately).

We use now use the notation and constructions of Section \ref{s:sv}
and also recall that $z_1 =0$ and $z_2 = 1$.  Using that
$(\vpi_1,\alpha_a) = 0$ for $a > 1$ and $(\alpha_a,\alpha_b) = 0$ if
$|a-b| > 1$, the master function $\ms{R}$ from
\eqref{e:master} is equal to 
\begin{equation}\label{ee:master}
  (t_1-1)^{\frac{1}{\kappa}}\prod_{1\leq a\leq M}t_a^{\frac{\delta_a}{\kappa}}
  \prod_{1\leq a\leq M-1}(t_a-t_{a+1})^{\frac{1}{\kappa}}
  \ . 
\end{equation}
In the construction of $\mathcal{KZ}_{\kappa}(\lambda,\vpi_1,\mu^*)$
we work with the hyperplane arrangement in $\A^M_{\msk}$ (with
coordinates $(t_1,\dots,t_M)$), given by the equations $t_i-t_{j}=0$,
for all $i\neq j$ , and $t_{i}=0$, $t_i =1$ for all $i$. Let
$$U'=\{(t_1,\dots,t_M)\mid t_i\neq 0, t_{i}\neq t_{j},
t_i\neq 1\} \subseteq \A^M_{\msk} $$ and
$${U}=\{(t_1,\dots,t_M)\mid t_i\neq 0,  t_{i}\neq t_{i+1}, t_1\neq
1\}\subseteq \A^M_{\msk} .$$ The form $\eta = d\ms{R}/\ms{R}$ is
regular on ${U}$ which contains $U'$. The mixed local system on
${U}$ is denoted by $\ml$.  So we get two ``Aomoto situations'' as in
Section \ref{s:sub}: the one for $U'$ produces the KZ motive by
definition---note that the group $\Sigma$ is trivial here---and by
Lemma \ref{imageKZ} this is a subquotient of the motive associated to
$U$. We will show that the motive corresponding to $U$ is of rank one
and compute it explicitly. Since the KZ motive is also of rank one
this will give the desired formula.

Let ${P}=P_M$ be a smooth projective compactification of ${U}$
constructed as an iterated blowup of $(\P^1)^M$ as described in
Theorem \ref{t:abnormal}, with
$P \setminus U=\cup_{\alpha} E_{\alpha}$ a divisor with normal
crossings. It is easy to see that in this case the abnormal strata
that need to be blown up iteratively are of the form
$t_{a}=t_{a+1}=\dots=t_{a+m}=0$, and
$t_{a}=t_{a+1}=\dots=t_{a+m}=\infty$ for all $a$ such that
$1\leq a \leq M-2$, and all $m$ such that $m \geq 2$ and $a+m \leq M$.

Let $V=P \setminus \cup'_{\alpha}E_{\alpha}$, where the union is
restricted to $\alpha$ such that $\op{Res}_{E_{\alpha}}\eta$ is not a
strictly positive integer. Similarly let
$V'=P \setminus \cup'_{\alpha}E_{\alpha}$, where the union is
restricted to $\alpha$ such that $\op{Res}_{E_{\alpha}}\eta$ is not a
non-negative integer. Note that $V'\supseteq V$.  Let $j:U\to V$, and
$k:V\to P$; $j:U\to V'$, and $k:V'\to P$ be the inclusions and
consider the map
\begin{equation}\label{KZmap2}
	H^M({P},{k}'_*{j}'_{!}\ml)\lra H^M({P},{k}_*{j}_{!}\ml) .
\end{equation}
To prove the lemma it suffices to show that $H^M(P,k'_*j'_{!}\ml)$ has
rank one and is given by the formula in Lemma \ref{genp}. We will do
this by induction on $M$. The case $M=1$ follows from Lemma
\ref{DelMos} so we assume $M>1$.

Let $f:P\to \Bbb{P}^1$ be given by composing the map
$P \to (\P^1_{\msk})^M$ with the first projection
$t_1:(\P^1_{\msk})^M \to \P^1_{\msk}$. The form of the explicit
blowups used to construct $P$ implies that:
\begin{enumerate}
\item The inverse image $f^{-1}(0)$ is the union of the strict
  transforms of $t_1=\dots=t_m=0$ for $m=1,\dots,M$ in
  $(\Bbb{P}^1)^M$.  The residue of $\eta$ along each of these
  irreducible components is non-negative (by inspection, where we use
  $1/\kappa>0$).
\item The inverse image $f^{-1}(1)$ is the strict transform of $t_1=1$
  (which is irreducible), and the residue of $\eta$ is $1/\kappa$
  which is positive.
\item The inverse image $f^{-1}(\infty)$ is the union of the strict
  transforms of $t_1=\dots=t_m=\infty$ for $m=1,\dots,M$. The residue
  of $\eta$ along each of these irreducible components is negative by
  an easy computation: the functions $t_i$ for $i=1,\dots,m$
  as well as $t_i-t_{i+1}$ for $i=1,\dots,m-1$ (the case $i=m-1$ for
  $t_i-t_{i+1}$ is handled separately) have simple poles along the
  strict transform of $t_1=\dots=t_m=\infty$.
\end{enumerate}

Let $\mc{F} = k_*j_{!}\ml$. From the signs of the residues of $\eta$
along irreducible components of $f^{-1}(0)$, $f^{-1}(1)$ and
$f^{-1}(\infty)$, and the proper base change theorem, $f_*{\mc{F}}$
has zero stalks over $0$ and $1$, and it is the star pushforward of
its restriction to $\Bbb{A}^1_{\msk}$.  Therefore
$H^*(P,k'_*j'_{!}\ml)=H^*(\Bbb{A}^1_{\msk};\iota_!\iota^*f_*{\mc{F}})$,
where $\iota:\Bbb{A}^1_{\msk}-\{0,1\}\to \Bbb{A}^1_{\msk}$ is
the inclusion.

Consider the birational morphism
$\Bbb{A}^1_{\msk}\times \Bbb{A}^{M-1}_{\msk}\to \Bbb{A}^M_{\msk}$ which is an isomorphism
over $t_1\neq 0$ (and $t\neq 0$) given by
$$(t,T_2,\dots,T_M)\mapsto (t, tT_2,\dots, t T_M).$$ The pullback of
the master function \eqref{ee:master} equals by an easy computation
\begin{equation}\label{e:pullback}
  \bigl (t^{(M-1+\sum_{a=1}^M \delta_a)/\kappa} \cdot (t-1)^{1/\kappa}
  \bigr ) \cdot \prod_{2\leq a\leq M}T_a^{\frac{\delta_a}{\kappa}}
  \prod_{2\leq a\leq M-1}(T_a-T_{a+1})^{\frac{1}{\kappa}} (T_2-1)^{\frac{1}{\kappa}} .
\end{equation}

We claim that the birational morphism
$\Bbb{A}^1_{\msk}\times \Bbb{A}_{\msk}^{M-1}\to
\Bbb{A}_{\msk}^M$ extends to an isomorphism
\[
  (\Bbb{A}^1_{\msk}-\{0,1\}) \times P_{M-1}\leto{\simeq}
  f^{-1}(\Bbb{A}^1_{\msk}-\{0,1\}) \subset P .
\]
Here $P_{M-1}$ is the compactification for the arrangement in the
$\A^{M-1}_{\msk}$ factor constructed in the same way as $P_M$ and using the
ordering of the abnormal strata induced by the ordering used to
construct $P_M$. To prove this claim note that over
$ \Bbb{A}^1_{\msk}-\{0,1\}$ the variety $P_M$ is obtained by
blowing up $(\Bbb{P}^1_{\msk})^M$ along the abnormal strata of the
form $t_{a}=t_{a+1}=\dots=t_{a+m}=0$ and
$t_{a}=t_{a+1}=\dots=t_{a+m}=\infty$ with $a>1$.

Let $\ti{P} = \P^1_{\msk}\times P^{M-1}_{\msk}$, let
$\ti{k'}$ be the inclusion of $\A^1_{\msk} \times \P^M$ into $\ti{P}$,
and $\ti{j'}$ the inclusion of
$(\A^1_{\msk}\sm \{0,1\}) \times \P^M$ into
$\A^1_{\msk} \times \P^M$. Let $\ti{\mc{F}}$ be the pullback of $\mc{F}$
to $(\A^1_{\msk}\sm \{0,1\}) \times \P^M$. By the proper base
change theorem we have
\[
  H^*(\ti{P}, \ti{k'}_* \ti{j}'_! \ti{\mc{F}}) =
  H^*(\Bbb{A}^1_{\msk};\iota_!\iota^*f_*{\mc{F}}) =
  H^*(P,k'_*j'_{!}\ml) 
\]
so it suffices to compute the LHS. Now $\ti{P}$ is a normal crossings
compactification of the arrangement corresponding to the pulled back
master function \eqref{e:pullback} and the object $\ti{k'}_* \ti{j}'_!
\ti{\mc{F}} \in \dbc(\ti{P})$ is the corresponding sheaf $\mc{F}'$ in
the notation of Section \ref{s:sub}. It them follows from Lemma
\ref{l:kaom} that
\[
  H^*(\ti{P}, \ti{k'}_* \ti{j}'_! \ti{\mc{F}}) \simeq
  H^*(\Bbb{A}^1_{\msk},\{1,0\};\mc{M}) \otimes_K H^*(P_M,
  \mc{F}_{M-1}) \ .
\]
Here $\mc{M}$ is the mixed local system corresponding to the first
factor in \eqref{e:pullback} and $\mc{F}'_{M-1}$ is the object in
$\dbc(P_{M-1})$ corresponding to the second factor of the function in
\eqref{e:pullback}. Since
$H^0(\Bbb{A}^1_{\msk},\{1,0\};\mc{M}) = 0$ and since
$\delta_M > 1$, the second factor corresponds to data of the same type
for $\mf{sl}_{n-1}$, we are done by induction and Lemma \ref{DelMos}.
\end{proof}

\subsection{Computing all Pieri coefficients}\label{s:pieri}
We now explain our algorithm for computing all products
$\lambda\star\vpi_k$.

Write $\lambda=a_1L_1+\dots a_{n-1}L_{n-1}$. We want to compute the
coefficient of $[{\mu}]$ in $[\lambda]\star[\vpi_k]$.  By Pieri's formula,
$\mu=\lambda +\vpi_k - \sum n_i\alpha_i$ for some $n_i\geq 0$, where
for some $K\subset\{1,\dots,n\}$ with $|K|=k$,
\begin{equation}\label{SumRoots1}
	\sum n_i\alpha_i =L_1+\dots+L_{k}- \sum_{i\in K}L_i \ .
\end{equation}
Let $M=\sum n_i$, the number of variables. We make the following
inductive assumptions (for a fixed $\kappa$):
\begin{itemize}
\item[--] For all $m<n$, we know how to compute all Pieri coefficients for
  $\mf{sl}_m$.
\item[--] We know how to compute all products $[\lambda]\star [\vpi_k]$
  whenever $\sum a_i$ is smaller. When $\sum a_i=1$, $\lambda=\vpi_1$
  and we already know how to compute in this case.
\item[--] For a fixed $\sum_i a_i$, we know how to compute all products
  $[\lambda]\star [\vpi_k]$ whenever the number of variables $M$ is
  smaller (so this is an inner layer of the induction).  When $M=1$,
  we use the formulas in Lemma \ref{DelMos}.
\end{itemize}

\begin{lemma}\label{redux}
  If $L_1$ or $L_n$ does not appear in the RHS of \eqref{SumRoots1}
  (after cancellation), then the corresponding Pieri coefficient can
  be computed by the induction assumptions.
\end{lemma}
\begin{proof}
  If $L_n$ does not appear then $n\not\in K$ and
  $K\subseteq \{1,\dots,n-1\}$. It is easy to check that
  $\alpha_{n-1}$ does not appear in \eqref{SumRoots1}.  Consider the
  following highest weight of $\mf{sl}_{n-1}$
  $$\lambda'= a_1 L_1+\dots a_{n-2}L_{n-2} +a_{n-1}L_{n-1}=  (a_1-a_{n-1}) L_1+\dots (a_{n-2}-a_{n-1})L_{n-2}. $$
  Consider $[\lambda']\star [\vpi_{k}]$ for $\mf{sl}_{n-1}$; since the
  case $k=n-1$ is covered by assumption, we assume $k\leq n-2$.
	
  Then we look at the term corresponding to the same $K$ in
  $[\lambda']\star [\vpi_{k}]$.  When $i<n-1$, $(\lambda',\alpha_i)$
  computed for $\mf{sl}_{n-1}$ is the same as $(\lambda,\alpha_i)$
  computed for $\mf{sl}_{n}$, similarly for $\vpi_k$).  The
  corresponding local systems, the number of variables, etc., for
  $\mf{sl}_{n-1}$ remain the same as the Pieri case for $\mf{sl}_{n}$
  that we started with, i.e., the corresponding motives are the same,
  so the product can be computed by the induction hypothesis. The
  other case when $L_1$ does not appear is similar is dual to this
  case.
\end{proof}

The rest of the algorithm is broken up into two cases.

\subsubsection{}

If $a_1>a_2$, let $\lambda'=\lambda-\vpi_1$. By the Pieri formula,
$[\lambda']\star[\vpi_1]$ has a summand $[\lambda]$ (with coefficient
1), and the other $[\nu]$ with nonzero coefficients are of the form
$[\lambda'+L_i]$. Now consider the $[\mu]$ term in
$[{\lambda'+L_i}]\star [\vpi_k]$. We compute (using \eqref{SumRoots1})
$$\lambda'+L_i+\vpi_k-\mu=\lambda-L_1+L_i -\mu=\sum
n_i\alpha_i-(\alpha_1+\dots +\alpha_{i-1}) .$$ It follows that all
these computations have a smaller number of variables, so can be
carried out by the inductive hypothesis. We can also compute
$[\lambda']\star [\vpi_k]$ by the outer induction (the number of boxes
in $\lambda'$ is smaller). Using the fact that
$([\lambda']\star [\vpi_k]) \star [\vpi_1] = ([\lambda']\star
[\vpi_1]) \star [\vpi_k]$, which we can compute since we can compute
$[\lambda']\star [\vpi_k]$, it follows that we can also compute
$[\lambda]\star[\vpi_k]$.

\subsubsection{}
Now suppose $a_1=\dots=a_s$ and $a_s>a_{s+1}$ with $s>1$. Note that
$a_{n}=0$, so $s=n-1$ is a possibility. Let $\lambda'=\lambda-L_{s}$,
so $[\lambda']\star [\vpi_1]$ has $[\lambda]$ occurring with a nonzero
coefficient which we know how to compute.

The element $[\lambda'+L_1]$ also has a nonzero coefficient in
$[\lambda']\star [\vpi_1]$ and suppose the product
$[{\lambda'+L_1}]\star [\vpi_k]$ has $[\mu]$ ocurring with a nonzero
coefficent.  When this happens
$$\lambda'+L_1+\vpi_k-\mu= L_1+\dots+L_k-\sum_{i\in K}  L_i +L_1-L_s$$
which has $L_1$ appearing in the RHS with multiplicity $2$ unless
$L_1\in K$. Therefore, by the Pieri formula we must have $L_1\in K$,
so we may compute the coefficient of $[\mu]$ in
$[{\lambda'+L_1}]\star [\vpi_k]$ by Lemma \ref{redux}.

The other $\nu$ such that $[\nu]$ occurs in $[\lambda']\star [\vpi_1]$
with nonzero coefficient have the form $\lambda'+L_{a}-L_s$ with $a>s$
(by the Pieri formula). The number of variables in the term
corresponding to $[\mu]$ in $[{\lambda'+L_{a}-L_s}]\star[\vpi_k]$ is
computed from
$\lambda-L_s+L_a -\mu=\sum n_i\alpha_i-(\alpha_s+\dots
+\alpha_{a-1})$, so the number of variables has dropped, hence the
coefficient of $[\mu]$ can be computed by induction.

Combining the above observations, and using the same trick of
switching the order as in the previous case, it follows that the
coefficient of $[\mu]$ in $[\lambda]\star [\varpi_k]$ can be computed
from the induction hypothesis.

\label{subsection}\label{s:allprod}

\subsection{KZ motives when $1/\kappa$ is a positive integer.}\label{s:int}

\begin{lemma}\label{l:zero}
  Let $\msf{V}$ be an effective admissible VMHS over $S - \{s_0\}$, where
  $S$ is a smooth curve and $s_0 \in S$ is any point, and let $t$ be a
  parameter at $s_0$. If $\Psi_t(\msf{V})$ has weight zero, then $\msf{V}$ is
  pure of weight zero.
\end{lemma}
\begin{proof}
  Since $\Psi_t$ is an exact functor we may assume that $\msf{V}$ is a
  polarised VHS and then by summing all its Galois conjugates we may
  also assume that $\msf{V}$ is an effective $\Q$-VHS.  Since the weight
  filtration used to define the mixed Hodge structure on $\Psi_t(\msf{V})$
  is the monodromy weight filtration \cite[Section 1]{saito} and
  $\Psi_t(\msf{V})$ is pure of weight zero, it follows that $\msf{V}$ must also
  be of weight zero.
\end{proof}

\begin{corollary}\label{c:sln}
  If $\mf{g} = \mf{sl}_n$ and $1/\kappa$ is a positive integer then
  all the variations
  $\msf{KZ}_{\kappa}(\vv{\lambda},\nu^*)$ are of weight zero and
  have finite scalar monodromy of order dividing $n$. Furthermore, if
  $1/\kappa$ is a multiple of $n$ then the monodromy is trivial.
\end{corollary}
\begin{proof}
  Since KZ motives are always effective (Remark \ref{r:eff}), using
  Theorem \ref{t:fact} and Lemma \ref{l:zero} we reduce to the case
  where $\vv{\lambda} = (\lambda_1, \lambda_2)$. It then follows from
  the assumption on $\kappa$ and Remark \ref{r:int} that all the
  motives occurring in \eqref{e:ring} are $\Q(0)$, from which and the
  algorithm described in Section \ref{s:pieri}, it follows that for
  $\msf{M} = \msf{KZ}_{\kappa}(\vec{\lambda},\nu^*)_{\vec{z}}$, the
  polynomials $P(\msf{M})$ (which are independent of
  the point $\vec{z}$) are constants, i.e., in $\Z$. By the definition
  of these polynomials in terms of the Hodge filtration, it follows
  that we have $F^0\msf{M}= \msf{M}$. The assumption on $\kappa$
  implies that $\msf{M}$ is a $\Q$-Hodge structure, so it follows from
  Hodge symmetry that $\msf{M}$ must be of weight $0$.
	
  That this implies finiteness of the monodromy is well-known:
  Firstly, by summing over the Galois conjugates of
  $\msf{KZ}_{\kappa}(\vv{\lambda},\nu^*)$ we can assume that we have a
  VHMS $\msf{V}$ defined over $\Q$. Since the associated local system
  comes from geometry, it has a $\Z$-structure. The monodromy also
  preserves a polarisation which, by the effectivity and weight zero
  conditions, must be a positive definite bilinear form. Since a
  discrete subgroup of a compact group is finite, the monodromy is
  finite.

  Finally, by \cite[Lemma 1.4.1 (iii)]{BK} the eigenvalues of the
  residue of the KZ connection along the divisor given by $z_i = z_j$
  are of the form
  \[
    \frac{1}{2\kappa}\bigl ( c(\mu) - c(\lambda_i) - c(\lambda_j) \bigr)
  \]
  where $\mu\in P^+$ is such that $V_{\mu}$ is a summand of
  $V_{\lambda_i} \otimes V_{\lambda_j}$. If $V_{\nu}$ is another such
  summand then using the fact that $(\alpha,2\rho) \in 2\Z$ for all
  roots $\alpha$ and noting that $\lambda_i + \lambda_j -\mu$ is in
  the root lattice (similarly for $\nu$) one sees that
  $c(\mu) - c(\nu) \in 2\Z$. It follows that if $\kappa$ is an integer
  then all eigenvalues of the local monodromy along the divisor are
  equal to a fixed $n$-th root of unity (possibly depending on the
  divisor). As the monodromy is finite, the local monodromy must be
  scalar of order dividing $n$. Since the fundamental group of the
  configuration spaces $\mc{C}_r(\C)$ is generated by the local
  monodromies, it follows that the global monodromy is also scalar of
  order dividing $n$. If $\kappa$ is a multiple of $n$ the same
  argument implies that all the local monodromies are trivial, so the
  global monodromy is also trivial.
\end{proof}

\begin{example}\label{e:mtate}
  Suppose $\mf{g} = \mf{sl}_2$, all the $\lambda_i$ and $\nu$ are
  multiples of the positive root $\alpha$, and $\kappa = 2$. Then the
  master function $\ms{R}$ from \eqref{e:master} is single valued so,
  as observed in Section \ref{s:integral}, the corresponding KZ
  motives are mixed Tate. However, they need not always be pure of
  weight zero. For a specific example, consider
  $\vv{\lambda} = (\alpha,\alpha,\alpha)$. Using factorisation and
  Lemma \ref{DelMos} one easily computes that the Hodge polynomials
  for the sequence $\nu = 0, \alpha, 2\alpha, 3\alpha$ are given by
  the sequence $t,2t+1,2,1$. By adding more points with weight
  $\alpha$, one can see that the length of the weight filtration is
  not bounded as the number of points increases.
\end{example}

\subsection{The map to fusion for $\mf{sl}_2$}\label{s:sl2}

In this section we prove Conjecture \ref{c:maptofusion} for
$\mf{g} = \mf{sl}_2$.

\smallskip

Let $\ell \geq 0$ and recall that in the case of $\mf{sl}_2$ we have
$h^{\vee} = 2$, so $\kappa = \ell + 2$. Let $\vpi = \vpi_1$ be the
fundamental dominant weight so any dominant integral weight $\lambda$
equals $p \vpi$ for some $p \geq 0$.  To simplify notation we let
$S = [\vpi]$ and for any $p>0$ we let $S_p := [p\vpi]$ (so
$S = S_1$).

Recall from Definition \ref{d:pi} that for $\ell \geq 1$ the map
$\pi_{\kappa}: \mf{R}_{\kappa}(\mf{sl_2}) \to
\mf{F}_{\kappa}(\mf{sl}_2)$ is defined to be the unique
$\Z[t]$-algebra homomorphism such that $\pi_{\kappa}(S) = S$. We may
also define a map for $\ell = 0$ by requiring that
$\pi_{\kappa}(S) = 0$.

\begin{proposition}\label{p:pisl2}
  Let $p \geq 1$ be any integer. Then
  \begin{equation}\label{e:sl2}
    \pi_{\kappa}(S_p) =
    \begin{cases}
      0 &\mbox{ if } p + 1 \equiv 0 \mod \ell + 2,
      \\ 
      c_{\kappa}(p)S_r & \mbox{ if } p = 2n(\ell + 2) + r \mbox{ with } n \in \N
      \mbox{ and } 0 \leq r \leq \ell, \\
      -c_{\kappa}(p)S_{\ell -r} & \mbox{ if } p = (2n + 1)(\ell + 2) + r \mbox{
        with } n \in \N \mbox{ and } 0 \leq r \leq
      \ell,
    \end{cases}
  \end{equation}
  where each $c_{\kappa}(p)$ is a monomial in $t$. Furthermore, in the
  second case $c_{\kappa}(p)$ only depends on $n$ and in the third
  case we have the formula
  $c_{\kappa}(p) = c_{\kappa}(p-1)d_{\kappa}(p)$, where
  $d_{\kappa}(p)$, which is either $1$ or $t$, is defined by
  \begin{equation}\label{e:pro}
    S_p \star_{\mf{R}_{\kappa}}S = S_{p+1} + d_{\kappa}(p)S_{p-1}
  \end{equation}
  for all $p>0$.
\end{proposition}

\begin{proof}
  The proof will be by induction on $p$. For $p=1$ we have by
  definition that $\pi_{\kappa}(S) = 0$ if $\ell =0$ and
  $\pi_{\kappa}(S) = S$ if $\ell>0$ and these values agree with the
  values given by \eqref{e:sl2}.
	
  We now assume that the proposition holds for all $p' \leq p$ and we
  will prove it for $p+1$.  Since $\pi_{\kappa}$ is a ring
  homomorphism, we get from \eqref{e:pro} that
  \begin{equation}\label{e:prod}
    \pi_{\kappa}(S_p) \star_{\mf{F}_{\kappa}}\pi_{\kappa}(S)
    = \pi_{\kappa}(S_{p+1}) + d_{\kappa}(p)\pi_{\kappa}(S_{p-1}).
  \end{equation}
  If $\ell = 0$ then the LHS, hence also the RHS, is $0$ from which we
  easily get that the proposition holds in this case. We may therefore
  assume that $\ell > 0$, so $\pi_{\kappa}(S) = S$.
	
  Consider the three cases for $\pi_{\kappa}(S_p)$ as in the statement
  of the proposition.  In the first case the RHS of \eqref{e:prod} is
  again zero so we get that
  $\pi_{\kappa}(S_{p+1}) = -d_{\kappa}(p) \pi_{\kappa}(S_{p-1})$ which
  by induction is of the desired form.
	
  Now suppose we are in the second case, and then consider the
  subcases:
  \begin{enumerate}
  \item $r = 0$: In this case $p-1$ corresponds to the first case and
    so $\pi_{\kappa}(S_{p-1}) = 0$. It then follows from
    \eqref{e:prod} that $\pi_{\kappa}(S_{p+1}) = c_{\kappa}(p) S$ as desired.
  \item $0 < r < \ell$: In this case the LHS of \eqref{e:prod} equals
    $c_{\kappa}(p)S_{r+1} + c_{\kappa}(p) d_rS_{r-1}$ whereas the RHS equals
    $\pi_{\kappa}(S_{p+1}) + d_{\kappa}(p)c_{\kappa}({p-1})S_{r-1}$. Now $c_{\kappa}(p) = c_{\kappa}(p-1)$
    by induction and $d_{\kappa}(r) = d_{\kappa}(p)$ by Lemma \ref{DelMos} since
    $r \equiv p \mod \ell + 2$, so it follows that
    $\pi_{\kappa}(S_{p+1}) = c_{\kappa}(p)S_{r+1}$ as desired.
  \item $r = \ell$: In this case the LHS of \eqref{e:prod} equals
    $c_{\kappa}(p)d_{\kappa}(\ell)S_{\ell -1}$. On the other hand
    $d_{\kappa}(p)\pi_{\kappa}(S_{p-1}) = d_{\kappa}(p) c_{\kappa}({p-1})S_{\ell -1}$. As above,
    $d_{\kappa}({\ell}) = d_{\kappa}(p)$ and $c_{\kappa}({p-1}) = c_{\kappa}(p)$ by induction, so we get that
    $\pi_{\kappa}(S_{p+1}) = 0$ as desired.
  \end{enumerate}
  
  For the third case we shall need the following:
  \begin{claimn}%
    Let $0 \leq r < \ell$ and let $b$ be any integer relatively
    prime to $\ell +2$. Then
    $d_{\kappa}({\ell -r}) = d_{\kappa}({r+1})$ where the Galois
    action $\sigma$ corresponds to $b$.
  \end{claimn}
  
  To prove the claim, we will use Lemma \ref{DelMos} with the $a$ there
  corresponding to the weights, i.e., $\ell - r$ and $r+1$.  The Galois
  action is incorporated into the data by taking $\kappa = (\ell +2)/b$.
	
  First note that none of $\tfrac{(\ell -r)b}{\ell + 2}$,
  $\tfrac{(\ell -r +1)b}{\ell +2}$, $\tfrac{(r+1)b}{\ell +2}$ and
  $\tfrac{(r+2)b}{\ell +2}$ are integers by the assumptions on $r$ and
  $b$.  We then observe that
  $\bigl \langle \tfrac{-(\ell -r)b}{\ell + 2} \bigr \rangle = \bigl
  \langle \tfrac{(r+2)b}{\ell +2} \bigr \rangle$ and
  $\bigl \langle \tfrac{(r+1)b}{\ell +2} \bigr \rangle = \bigl \langle
  \tfrac{-(\ell -r +1)b}{\ell +2} \bigr \rangle$, so we get the
  desired equality from Case (1) of Lemma \ref{DelMos}.
	
  \smallskip
	
  We now consider the third case and again break it up into subcases:
  \begin{enumerate}
  \item $r =0$: In this case $p-1$ corresponds to the first case so
    $\pi_{\kappa}(S_{p-1}) = 0$.  Since
    $\pi_{\kappa}(S_p) = -c_{\kappa}(p)S_{\ell}$ we get from
    \eqref{e:prod} that
    $\pi_{\kappa}(S_{p+1}) = c_{\kappa}(p)d_{\kappa}({\ell})S_{\ell
      -1}$. Using the claim with $r=0$ we have
    $d_{\kappa}({\ell}) = d_{\kappa}(1) = d_{\kappa}({p+1})$, where
    the second equality holds since $p+1 \equiv 1 \mod \ell + 2$.
  \item $0 < r < \ell$: In this case the LHS of \eqref{e:prod} equals
    $-(c_{\kappa}(p)S_{\ell - r + 1} + c_{\kappa}(p)d_{\kappa}({\ell
      -r})S_{\ell -r -1})$ whereas the RHS equals
    $\pi_{\kappa}(S_{p+1}) - d_{\kappa}(p)c_{\kappa}(p-1)S_{\ell -r
      +1}$.  By induction we have
    $c_{\kappa}(p) = d_{\kappa}(p)c_{\kappa}(p-1)$, so it follows that
    $\pi_{\kappa}(S_{p+1}) = -c_{\kappa}(p)d_{\kappa}({\ell -
      r})S_{\ell -r -1}$. Furthermore,
    $d_{\kappa}({\ell - r}) = d_{\kappa}({r+1}) = d_{\kappa}({p+1})$
    by the claim since $p \equiv r$ modulo $\ell +2$, so
    $c_{\kappa}(p+1) = d_{\kappa}({p+1})c_{\kappa}(p)$ proving the
    induction step in this case.
  \item $r = \ell$: In this case the LHS of \eqref{e:prod} equals
    $-c_{\kappa}(p)S$ whereas the RHS equals
    $\pi_{\kappa}(S_{p+1}) - d_{\kappa}(p)c_{\kappa}(p-1)S$. Since
    $c_{\kappa}(p) = d_{\kappa}(p) c_{\kappa}(p-1)$ by induction, it
    follows that $\pi_{\kappa}(S_{p+1}) = 0$ as desired.
	\end{enumerate}
\end{proof}

\begin{remark}\label{r:sl2}
  The proof shows that
  $c_{\kappa}(p) = \prod_{p' \in W_p}d_{\kappa}({p'})$, where
  $W_p = W_p' \cup W_p''$ with
  \begin{align*}
    W_p' = & \{p' \mid 1 \leq p' \leq p \mbox{ and } p' = (2n+1)(\ell +2) + r  \mbox{ for some }
             n \in \N \mbox{ and } 0 \leq r \leq \ell\} , \\
    W_p'' = & \{p' \mid 1 \leq p' \leq p \mbox{ and } p' + 1 \equiv \ell +
              1 \mod \ell + 2\} .
	\end{align*}
\end{remark}

\begin{theorem}\label{t:sl2}
  Conjectures \ref{c:maptofusion} and \ref{c:ec} hold for
  $\mf{g} = \mf{sl}_2$.
\end{theorem}

\begin{proof}
  Conjecture \ref{c:maptofusion} for $\mf{sl}_2$ follows directly from
  Proposition \ref{p:pisl2} since it is easy to check that the first
  case corresponds to the nonexistence of $w$, the second case
  corresponds to when $\ve(w) = 1$ and the third to when
  $\ve(w) = -1$.
	
  To prove Conjecture \ref{c:ec} we will use Remark \ref{r:sl2} to
  obtain explicit formulas for $c_{\kappa}(p)$.
	
  First suppose that $\sigma = e$ and $\kappa = \ell + 2$.  For any
  $p' \equiv r \mod \ell + 2$ with $0 < r \leq \ell + 1$ we apply
  Lemma \ref{DelMos} with $a=p'$ to determine $d_{\kappa}({p'})$.
  \begin{enumerate}
  \item $0 < r \leq \ell$: We are in Case (1) of Lemma \ref{DelMos}
    and
    $ \bigl \langle \tfrac{-1}{\ell + 2} \bigr \rangle + \bigl \langle
    \tfrac{-p'}{\ell + 2} \bigr \rangle + \bigl \langle \tfrac{1 +
      p'}{ \ell + 2} \bigr \rangle $ always equals $2$, so the Hodge
    polynomial is $t$.
  \item $r = \ell +1$: We are in Case (3) of Lemma \ref{DelMos} and
    the Hodge polynomial is $t$.
  \end{enumerate}
  Thus, $d_{\kappa}(p') = t$ for all $p' \in W_p$ so to work out the
  exponent of $c_{\kappa}(p)$ we only need to know $|W_p|$. It is easy
  to see that this is given by the formula in Case (1) of Conjecture
  \ref{c:ec} by separately considering the second and third cases in
  Proposition \ref{p:pisl2}.
	
  Now suppose $\sigma = c$, so we take
  $\kappa = (\ell + 2)/(\ell + 1)$. For $p'$ as above we again apply
  Lemma \ref{DelMos} with $a=p'$ to determine $d_{\kappa}({p'})$.
  \begin{enumerate}
  \item $0 < r \leq \ell$: We are in Case (1) of Lemma \ref{DelMos}
    and
    $ \bigl \langle \tfrac{-(\ell + 1)}{\ell + 2} \bigr \rangle +
    \bigl \langle \tfrac{-p'(\ell + 1)}{\ell + 2} \bigr \rangle +
    \bigl \langle \tfrac{(1 + p')(\ell + 1}{ \ell + 2} \bigr \rangle $
    always equals $1$, so the Hodge polynomial is $1$.
  \item $r = \ell +1$: We are in Case (3) of Lemma \ref{DelMos} and
    the Hodge polynomial is $t$.
  \end{enumerate}
  Thus $d_{\kappa}(p') = 1$ for $p' \in W_p'$ and $d_{\kappa}(p') = t$
  for $p' \in W_p$. Using this one again easily sees that Case (2) of
  Conjecture \ref{c:ec} holds.
\end{proof}

\section{$\mf{sl}_2$ motives at level $0$}\label{s:level0}

Let $\mf{g}=\mf{sl}_2$ and $\kappa=2$.  In this section, which is an
extended example, we relate the motives $\mc{KZ}_2(\vv{\lambda},0)$
for $\vec{\lambda} \in (P^+)^n$ to motives coming from natural
families of (nodal) hyperelliptic curves. This relation allows us to
determine properties of the weight filtration explicitly in many
cases, and shows that general KZ motives (as opposed to CB motives)
can have complicated weight filtrations with non-split extensions, but
they can also be surprisingly simple in some cases.

\subsection{The basic case}\label{sec:nearbycurves}

We begin by recalling some elementary facts about the cohomology of
a family of smooth projetive curves in the setting of mixed sheaves.

Let $S$ be a smooth projective variety over $\msk$ and
$\pi:C \to S$ a family of smooth projective curves of genus $g >
0$. We denote by $\mc{H}^1(C)$ the mixed local system of rank $2g$ on
$S$ given by $H^1(\pi_*\Q_C)$. By duality for the smooth projective
morphism $f$ we have a canonical injection
$\delta:\Q_S(-1) \to \wedge^2 (\mc{H}^1(C))$. For any integer $r$ with
$0 \leq r \leq g$, we have a canonical sub-local system $\mc{P}^r(C)$
of $\wedge^r (\mc{H}^1(C))$ given as the kernel of the map
$\wedge^r (\mc{H}^1(C)) \to (\wedge^{2g-r+2}(\mc{H}^1(C)))(g-r+1)$
induced by ``wedging with $\delta(1)$'' $g-r+1$ times. The rank of
$\mc{P}^r(C)$ is $d(g,r):=\binom{2g}{r}-\binom{2g}{r-2}$. We may also
view $\mc{P}^r(C)$ as the quotient of $\wedge^r (\mc{H}^1(C))$ by
$\im(\delta) \wedge (\wedge^{r-2} (\mc{H}^1(C))$.

Let $X_r = C\times_SC\times \dots \times_SC$, where there are $r$
factors and let $\pi_r:X_r \to S$ be the structure map. By the Kunneth
formula there is a natural injection
$\otimes^r (\mc{H}^1(C)) \to H^r(\pi_{r,*}\Q_{X_r})$ and the image of
$\wedge^r (\mc{H}^1(C))$ lies in
$H^r(\pi_{r,*}\Q_{X_r})^{\Sigma_r}$. Consequently, we have an
injection $\mc{P}^r(C) \to H^r(\pi_{r,*}\Q_{X_r})^{\Sigma_r}$.

\subsubsection{The basic case and hyperelliptic curves}

Let $\msk = \Q$, let $S$ be the configuration space $\mc{C}_{2g+2}$
and let $\pi:C \to \mc{C}_{2g+2}$ be the family of hyperelliptic
curves given by the normalisation of the closure of
$y^2=\prod_{i=1}^{2g+2}(z_i-t) \subset \mc{C}_{2g+2} \times \A^2$ in
$\mc{C}_{2g+2} \times \P_{\msk}^2$.  There is a natural map
$\pi:C \to \mc{C}_{2g+2} \times \Bbb{P}^1$ given by projection to $t$.
Over $\infty\in\P_{\msk}^1$, letting $t=1/T$ the equation becomes
$(yT^{g+1})^2=\prod_{i=1}^{2g+2}(z_iT-1)$. Write this as
$V^2=\prod_{i=1}^{2g+2}(z_iT-1)$ which is smooth. Now setting $T=0$,
$V=\pm 1$, so the map $\pi:C \to \mc{C}_{2g+2}$ has a section mapping
to $\mc{C}_{2g+2} \times \{\infty\}$ in $\mc{C}_{2g+2} \times \P_{\msk}^1$.

It is well-known that the monodromy group of the topological local
system $\rat(\mc{H}^1(C))$ on $\mc{C}_{2g+2}(\C)$ is Zariski dense in
the group $\mr{Sp}_{2g}$ (see, e.g., \cite[Theorem 1 (2)]{ACN} for a
more precise result). Since $\mc{P}^r(C)$ corresponds to the $r$th
fundamental representation of $\mr{Sp}_{2g}$ (see, e.g., \cite[Theorem
17.5]{fulton-harris}) the monodromy representation of
$\rat(\mc{P}^r(C))$ is absolutely irreducible.

Let $\vec{\vpi}^{2g+2} := (\vpi,\dots,\vpi)$ with $\vpi$ appearing
$2g+2$ times, where $\vpi:= \vpi_1$ is the fundamental dominant weight
of $\mf{sl}_2$. The starting point of our computations is:
\begin{proposition}\label{prop:level0basic}
  The mixed local system $\mc{KZ}_{2}(\vv{\vpi},0)$ is isomorphic to
  the mixed local system $(\mc{P}^g(C)\tensor_{\Q} K h^{-1/4})(-1)$, where
  {$h:\mc{C}_{2g+2} \to \A^1_{\msk}$} is the function given by
  $\prod_{1\leq i<j \leq 2g+2}(z_i-z_j)$, and $K=\Q_4$.
\end{proposition}

\begin{proof}
  The master function $\ms{R}$ (see \ref{e:master}) in this case is
  given by
\[
\ms{R} = \prod_{1 \leq i < j \leq 2g+2} (z_i -
z_j)^{-\frac{1}{4}} \prod_{b=1}^{g+1}
\prod_{j=1}^{2g+2} (t_b - z_j)^{\frac{1}{2}}
\prod_{1 \leq b < c \leq M} (t_b - t_c)^{-1} .
\]
We let 
\begin{equation}\label{e:mastersl2} %
  \ms{T} =  \prod_{b=1}^{g+1}
  \prod_{j=1}^{2g+2} (t_b - z_j)^{\frac{1}{2}} 
  \prod_{1 \leq b < c \leq g+1} (t_b - t_c)^{-1}
\end{equation}
and note that $\mc{R} =  h^{-1/4} \mc{T}$.

Let
$\widehat{U}\to U\subseteq \mc{C}_{2g+2} \times
\Bbb{A}_{\msk}^{g+1}$ be the covering given by the tuples
$(\vec{z}, t_1,\dots,t_{g+1}, u)$ such that $u^2=\ms{T}^2$. This
covering has Galois group $\Sigma_{g+1}\times \mu_2$.

Let $C^0\subset C$ be the open subset of points on the affine part
(i.e., $t\neq \infty$) such that $t\neq z_i$ for all $i$ and, for any
$r>0$, let $X^0_r = C^0 \times_SC^0\times_S\dots \times_SC^0$ with $r$
factors and let $\Delta_r$ be the union of the diagonals $t_i = t_j$,
$1 \leq i < j \leq r$ in $X^0_r$.  There is an unramified Galois
covering $X^0_{g+1}\setminus \Delta_{g+1}\to\widehat{U}$ given by
$$\prod_{i=1}^{g+1} (y_i,t_i)\mapsto (t_1,\dots,t_{g+1},y_1\dots
y_{g+1}\prod_{1 \leq b < c \leq g+1} (t_b - t_c)^{-1}).$$ The
composition $X_0^{g+1} \setminus \Delta_{g+1}\to{U}$ is an unramified Galois
covering with Galois group $\Sigma_{g+1}\times \mu_2^{g+1}$. For any $r>0$,
let $\pi_{r}^0: X^0_{r} \setminus \Delta_r \to \mc{C}_{2g+2}$ be the structure map.
Keeping track of isotypical components and considering that
$\ms{R}=h^{-1/4}\ms{T}$, by forgetting the support condition in
  the definition of KZ motives we get a map (not a priori an
injection)
\begin{equation}\label{e:mapfromkz}
  \mc{KZ}_{2}(\vv{\vpi},0)\to  H^{g+1}(\pi^0_{g+1,*}\Q_{X^0_{g+1}\setminus\Delta_{g+1}})^{\tau}\tensor Kh^{-1/4}
\end{equation}
where $\tau: \Sigma_{g+1} \times \mu_2^{g+1}\to\mu_2$ is trivial on
$\Sigma_{g+1}$ and the product on the second (the triviality on the
$\Sigma_{g+1}$ factor is because of the second factor in
\eqref{e:mastersl2}).

 Recall that we have two sections of $\pi$ over $\infty$. Setting
 $t_1=\infty$ over one of these section and using the Gysin
 isomorphism for a smooth divisor in a smooth variety (in the setting
 of mixed sheaves this follows by dualising \cite[Proposition
 7.6]{saito-form}), we get a boundary map
 $H^{g+1}(\pi^0_{g+1,*}\Q_{X^0_{g+1} \setminus \Delta_{g+1}})\to
 H^{g}(\pi^0_{g,*}\Q_{X^0_g\setminus\Delta_g})(-1)$. We therefore get
 a composed map
\begin{equation}\label{e:mapfromkz2}
  \mc{KZ}_{2}(\vv{\vpi},0)\to H^{g}(\pi^0_{r,*}\Q_{X^0_{g}-\Delta})^{\tau'}(-1)\tensor Kh^{-1/4}
\end{equation}
Here $\tau':\Sigma_{g}\times \mu_2^g\to \mu_2$ is trivial on the
$\Sigma_g$ factor and the product on the second. Let $I$ be the image
of this map.

We have also a map of mixed local systems
$\mc{P}^g(C) \otimes Kh^{-1/4} \to
H^{g}(\pi^0_{r,*}\Q_{X^0_{g}-\Delta})^{\tau'} \otimes h^{-1/4}$ given by
tensoring the pullback via the inclusion $i:X^0_{g}-\Delta_g$ in $X_g$
by $Kh^{-1/4}$ and we let its image be $J$; the invariance under
$\tau'$ uses that the hyperelliptic involution acts by multiplication
by $-1$ on $H^1(\pi_*\Q_C)$.
\begin{claimn}
  The local system $\rat(I) \cap \rat(J)$ is nonzero.
\end{claimn}

We first show that the claim implies the proposition and then prove
the claim. Since $J$ is nonzero by the claim and the monodromy
representation of $\rat(\mc{P}^g(C))$ is irreducible, we deduce that
$J \simeq \mc{P}^g(C)$ and then the claim further implies that
$\rat(I) \cap \rat(J) = \rat(J)$. The rank of
$\mc{KZ}_{2}(\vv{\vpi},0)$ is equal to $d(g,g)$\footnote{This is
  well-known, cf. Exercise 6.2.5 a) and b) in \cite{stanley2}.}  so we
deduce that $\rat(I) = \rat(J)$ and the map \eqref{e:mapfromkz2} is
injective. The exactness of $\rat$ then implies that $I = J$ and
therefore proves the proposition.

\smallskip

To prove the claim we may assume that the base field and the
coefficient field are both equal to $\C$. Thus, we may let
$\vec{z} \in \mc{C}_{2g+2}(\C)$ and prove the statement after
restricting to the fibre over $\vec{z}$.

The vector space of conformal blocks
$\rat(\mc{CB}_3(\vec{\vpi}^{2g+2},0)_{\vec{z}})$ is one dimensional
and is naturally a subspace of
$\rat(\mc{KZ}_2(\vec{\vpi}^{2g+2},0)_{\vec{z}})$ (the dual statement
is explained in the introduction). For $\phi$ a generator of
$\rat(\mc{CB}_3(\vec{\vpi}^{2g+2},0)_{\vec{z}})$, the image of $\phi$
under the map in \eqref{e:mapfromkz2} in
$\rat(H^{g}(\pi^0_{g,*}\Q_{X^0_{g}-\Delta_g})_{\vec{z}}^{\mu'}(-1))$
is by \cite[Proposition 3.1.8]{hodge2} (after identifying the target
with de Rham cohomology) the de Rham cohomology class of a holomorphic
form on $(C^0)_{\vec{z}}^g \setminus\Delta_{g,\vec{z}}$ obtained by
taking the residue along
$(C^0)_{\vec{z}}^g \setminus\Delta_{g,\vec{z}}$ of a holomorphic form
on $(C^0)_{\vec{z}}^{g+1} \setminus\Delta_{g+1,\vec{z}}$ having log
poles along $(C^0)_{\vec{z}}^g \setminus\Delta_{g,\vec{z}}$.

To see the log-poles part of the above assertion, note that the image
of the form $\phi$ in the target of \eqref{e:mapfromkz} is a
holomorphic form on
$(C^0)_{\vec{z}}^{g+1} \setminus\Delta_{g+1,\vec{z}}$ times the master
function \eqref{e:mastersl2} along $t_1=\infty$ (along one of the
sections). The master function has a simple pole at $t_1=\infty$ (the
order is $-(g+1) +g=-1$), and the holomorphic form above on
$(C^0)_{\vec{z}}^{g+1} \setminus\Delta_{g+1,\vec{z}}$ is regular along
$t_1=\infty$ This is a general representation theoretic property of
correlation functions; it also follows from $fv_0=0$ in the
computations below.

To prove the claim, it suffices to show that this form (the residue) extends to a holomorphic form on
$C_{\vec{z}}^g$, since by its invariance properties the extended form
would have to be a generator of the one dimensional space
$\wedge^g H^{1,0}(C_{\vec{z}})$. This implies the claim, since this
space is contained in $\rat(\mc{P}^g(C)_{\vec{z}}) \otimes \C$ and
injects into
$\rat(H^{g}(\pi^0_{g,*}\Q_{X^0_g-\Delta_g})_{\vec{z}})\otimes \C$ by
\cite[Corollaire 3.2.3 (ii)]{hodge2}.

To prove this extension property, we use the computations (and
notation) of \cite{b-unitarity} with which we assume the reader is
familiar. The holomorphic form on
$(C^0)_{\vec{z}}^g \setminus\Delta_{g,\vec{z}}$ above is obtained as a
residue of a correlation function with insertion of $f(-1)v_0$ at
infinity and $f(-1)v_0$ (with $v_0$ the generator of the trivial
representation of $\mf{g}$) at the finite points $t_1,\dots,t_g$ times
\begin{equation}\label{modR}
\prod_{b=1}^{g}
	\prod_{j=1}^{2g+2}(t_b - z_j)^{\frac{1}{2}}
	\prod_{1 \leq b < c \leq g} (t_b - t_c)^{-1}.  
\end{equation}
 A look at the computations on the strata shows that the only possible problem is as $t_1=\dots=t_m=\infty$ for some $m$. The order of pole of \eqref{modR} is  $$(2g+2)m/2 -m(m-1)/2 -m(g-m)= m(m+3)/2.$$

The correlation function vanishes to order $\geq (m+1)^2 -m-1=m^2+m$, since at infinity we are looking at terms like
$ u_1^{a_1-1}u_2^{a_2-1}\dots u_m^{a_m-1}f(-a_1)\dots f(-a_m)f(-1)v_0$, and we use \cite[Proposition 7.4]{b-unitarity}. So we need $m^2+m\geq (m^2+3m)/2$  which is true for all $m\geq 1$.

(On the stratum $t=\dots=t_m=z_1$, the order of vanishing turns out to
be (using \cite[Proposition 7.4]{b-unitarity}) at least
$m/2 - m(m-1)/2 +(m^2-m)=m^2/2\geq 0$, and on $t_1=\dots=t_m$, at
least $-m(m-1)/2 + (m^2-m)\geq 0$.)
\end{proof}

\subsubsection{The nearby cycles of  $\mc{P}^r(C)$}\label{s:nbcpr}

We first recall the computation of the nearby cycles of
$\mc{H}^1(C)$. Fix $I = \{i_0,i_1\} \subset [2g+2]$ and let
$W_I \supset \mc{C}_{2g+2}$ be the open subset of $\A^{2g+2}_{\msk}$
as in Section \ref{s:motivicFormula}. The family of curves
$\pi: C \to \mc{C}_{2g+2}$ extends to a family of curves
$\pi_I: C_I \to W_I$; over $Z_I$ the fibres of $\pi_I$ are singular
curves with one nodal singularity. Let $\sigma_I: D_I \to Z_I$ be the
restriction of $\pi_I$ to $Z_I$ (so it is a family of nodal
hyperelliptic curves) and let $\wt{\sigma_I}: \wt{D_I} \to Z_I$ be the
normalisation of the family $\sigma_I$: this is canonically isomorphic
to the pullback via the projection $p: Z_I \to \mc{C}_{2g}$ (given by
forgetting $z_{i_0}$ and $z_{i_1}$) of the family
$\pi':C' \to \mc{C}_{2g}$ of the same type as $C$ with $g$ replaced by
$g-1$. We let $n_I:\wt{D_I} \to D_I$ be the normalisation map.
\smallskip
  
Let $\mc{H}_I = H^1(\sigma_{I,*}\Q_{D_I})$ and let
$\wt{\mc{H}_I} = H^1(\wt{\sigma_I}_* \Q_{\wt{D_I}})$. Let
$f_I:W_I \to \A^1_{\msk}$ be the function $z_{i_1} - z_{i_0}$ and let
$h_I: Z_I \to \A^1_{\msk}$ be the function
$\prod_{i \notin I}(z_i - z_{i_0})$.
\begin{lemma}\label{l:node}
  The mixed local system $\Psi_{f_I}(\mc{H}^1(C))$ on $Z_I$ canonically
  contains $\mc{H}_I$ as a local subsystem with quotient
  $\Q h_I^{1/2}(-1)$. Furthermore, $\mc{H}_I$ contains $\Q h_I^{1/2}$ as a
  local subsytem with quotient $\wt{\mc{H}_I} \simeq p^*(\mc{H}^1(C'))$.
\end{lemma}

\begin{proof}
  The basic properties of the nearby cycles functor \cite[Section
  5]{saito-form} imply that there is a map
  $\Q_{D_I} \to \Psi_{f_I \circ \pi_I} \Q_C$ of mixed sheaves on
  $D_I$.  Applying $H^1(\sigma_{I,*} -)$ to this map gives the map
  $\mc{H}_I \to \Psi_{f_I}(\mc{H}^1(C))$. The fact that this is an
  injection and the quotient is a rank one mixed local system follows
  (by using $\rat$) from the corresponding topological fact, which is
  well-known.

  Let $\wt{N_I}$ be the inverse image $\wt{D_I}$ of the image $N_I$ of
  the nodal section $s_I$ of $\sigma_I$. An elementary computation of
  the normal cone along $N_I$ shows that this is given by the equation
  $y^2 = h_I$. This implies that we have an exact sequence
  \[
    0 \to \Q_{D_I} \to n_{I,*} \Q_{\wt{D_I}} \to s_{I,*}(\Q h_I^{1/2}) \to 0
  \]
  of mixed local systems on $D_I$. Applying $\sigma_{I,*}$ to this
  sequence, the long exact sequence of cohomology
  gives us a sequence
  \[
    0 \to \Q h_I^{1/2} \to \mc{H}_I \to \wt{\mc{H}_I} \to 0 
  \]
  whose exactness follows from the corresponding topological
  statement.  That $\Psi_{f_I}(\mc{H}^1(C))/\mc{H}_I$ is isomorphic to
  $\Q h_I^{1/2}(-1)$ follows from this and the fact that the nearby
  cycles functor is compatible with duality \cite[Proposition
  5.11]{saito-form} or by using the monodromy operator $N$
  \cite[(5.7.3)]{saito-form}.
\end{proof}

\begin{remark}\label{r:wt}
  The lemma implies that the terms of the weight
  filtration\footnote{Here we are using the naive normalisation of the
    weight filtration (which is different from that of
    \cite{saito-form}), so the constant local system $\Q_X$ in degree
    $0$ on any irreducible smooth variety $X$ has weight $0$.} $W$ of
  $\Psi_{f_I}(\mc{H}^1(C))$ are given by $W_{-1} = 0$,
  $W_0 = \Q h_I^{1/2}$, $W_1 = \mc{H}_I$ and
  $W_2 = \Psi_{f_I}(\mc{H}^1(C))$.
\end{remark}

\begin{proposition}\label{p:nbasic}
  For $2 \leq r \leq g$, the nearby cycles of
  $\Psi_{f_I}(\mc{P}^r(C))$ equals
  \[
    \wedge^r (\Psi_{f_I}(\mc{H}^1(C)))/\im(\delta_I) \wedge
    (\wedge^{r-2} (\Psi_{f_I}(\mc{H}^1(C)))) ,
  \]
  where $\delta_I: \Q_{Z_I} \to \wedge^2(\Psi_{f_I}(\mc{H}^1(C)))$ is
  the map given by applying the nearby cycles functor to $\delta$. We
  have
  \begin{enumerate}
  \item $W_{r-2}(\Psi_{f_I}(\mc{P}^r(C))) = 0$.
  \item $W_{r-1}(\Psi_{f_I}(\mc{P}^r(C)))$ is the image of
    $ (\wedge^{r-1} (\mc{H}_I)) \otimes \Q h_I^{1/2} $ in
    $ \Psi_{f_I}(\mc{P}^r(C))$.  This is isomorphic to
    $p^*(\mc{P}^{r-1}(C')) \otimes \Q h_I^{1/2}$.
  \item $W_r (\Psi_{f_I}(\mc{P}^r(C)))$ is the image of
    $\wedge^r (\mc{H}_I)$ in $ \Psi_{f_I}(\mc{P}^r(C))$.  We have
    \[
      W_r/W_{r-1} \simeq p^*(\mc{P}^r(C')) \oplus
      p^*(\mc{P}^{r-2}(C'))(-1) ,
    \]
    where the first summand is taken to be $0$ if $r=g$.
  \item $W_{r+1}(\Psi_{f_I}(\mc{P}^r(C))) =\Psi_{f_I}(\mc{P}^r(C)))$ and
    $W_{r+1}/W_r \simeq p^*(\mc{P}^{r-1}(C')) \otimes \Q h_I^{1/2}(-1)$.
  \end{enumerate}
\end{proposition}

\begin{proof}
  The description of $\Psi_{f_I}(\mc{P}^r(C))$ as
  $ \wedge^r (\Psi_{f_I}(\mc{H}^1(C)))/\im(\delta_I) \wedge
  (\wedge^{r-2} (\Psi_{f_I}(\mc{H}^1(C))))$ is an immediate
  consequence of the exacness of $\Psi_f$ and the description of
  $\mc{P}^r(C)$ as a quotient.

  Now $\delta_I$ corresponds to the self duality of
  $\Psi_{f_I}(\mc{H}^1(C))$ (up to a twist). It induces a perfect
  pairing (up to a twist) of $W_0$ with $W_2/W_1$ and of $W_1/W_0$
  with itself. Using this and simple linear algebra one easily
  determines the weight filtration on $\Psi_{f_I}(\mc{P}^r(C))$ as
  well as its associated graded subquotients. For (3) we use that
  $\Q h_I^{1/2} \otimes \Q h_I^{1/2} \simeq \Q h_I \simeq \Q_{Z_I}$.
\end{proof}

\subsubsection{}

Let $X'$ be a hyperelliptic cuve over $\C$ of genus $a>1$. Let
$x_1 \in X'(\C)$ be a non-Weierstrass point and let $x_2$ be the image
of $x_1$ under the hyperelliptic involution. Let $X$ be the
(seminormal) curve obtained from $X'$ by identifying $x_1$ and
$x_2$. For $1 \leq r \leq a$ we define $P^r(X')$ analogously to the
definition of $\mc{P}^r(C)$ at the beginning of this section, but we
view them simply as Hodge structures.
We have a basic extension of mixed Hodge structures
\begin{equation}\label{e:ext0}
  0 \to \Q \to H^1(X,\Q) \to H^1(X',\Q) \to 0 .
\end{equation}
Tensoring \eqref{e:ext0} with $P^a(X')$ we get an extension
\begin{equation*}
  0 \to P^a(X') \to E_1 \to P^{a-1}(X')\otimes H^1(X',\Q)  \to 0 .
\end{equation*}
There is a natural map
$P^a(X')\otimes H^1(X',\Q) \to P^{a-1}(X')(-1)$, so pushing out
the extension $E_1$ via this map we get an extension
\begin{equation}\label{e:ext2}
  0 \to P^a(X') \to E \to P^{a-1}(X') \to 0 .
\end{equation}

\begin{lemma}\label{l:exte}
  If $x_1$ is general, then the extensions \eqref{e:ext0} and
  \eqref{e:ext2} are non-split.
\end{lemma}

\begin{proof}
  The statement for \eqref{e:ext0} is classical. The extension
  corresponds to the element $[x_1] - [x_2]$ in
  $\mr{Pic}^0(X') \otimes \Q$ (see, e.g., \cite[Section 3]{carlson}),
  so it is nonzero for general $x_1$. The statement for
  \eqref{e:ext2} follows from this since the extension in
  \eqref{e:ext0} can be recovered from \eqref{e:ext2} tensored with
  $P^a(X')$. (It might help the reader to recall that $P^r$
  corresponds to the $r$th fundamental representation of
  $\mr{Sp}_{2a}$.)
\end{proof}

\begin{corollary}\label{c:ext}
  The extension of mixed Hodge structures associated to the
  restriction of $\Psi_{f_I}(\mc{P}^r(C))$ to a general point of
  $Z_I$ is non-split.
\end{corollary}
\begin{proof}
  It suffices to show that the extension
  \[
    0 \to W_{r-1} \to W_r \to W_{r}/W_{r-1} \to 0
  \]
  is non-split. This follows from Lemma \ref{l:exte} and (3) of
  Proposition \ref{p:nbasic}: some elementary linear algebra shows
  that the extension there is of the same type as the one in
  \eqref{e:ext2}.
\end{proof}

\subsection{Higher weight cases}\label{s:higher}

The results of Section \ref{sec:nearbycurves} can be used to
inductively compute KZ motives for $\mf{sl}_2$ when $\kappa =2$ (and
the weight $\nu$ at infinity is $0$). The main idea is to inductively
compare the description of the nearby cycles given in Theorem
\ref{t:fact} with the description of nearby cycles given in
Proposition \ref{p:nbasic}. We illustrate this by an example.

\subsubsection{}

Suppose $\vec{\lambda} = (2\vpi,\vpi,\dots,\vpi)$ with $\vpi$
occurring $2g$ times. We consider the nearby cycles of
$\mc{KZ}_{2}(\vec{\vpi}^{2g+2},0)$ for the case $I = \{1,2\}$. By
Theorem \ref{t:fact} we have that it has a two step filtration with the
sub being
\[
  \mc{KZ}_2(\vec{\lambda},0) \boxtimes \mc{KZ}_2((\vpi,\vpi), 2\vpi)
\]
and the quotient being
\[
  \mc{KZ}_2(\vec{\vpi}^{2g},0) \boxtimes \mc{KZ}_2((\vpi,\vpi), 0) .
\]
Combining Proposition \ref{prop:level0basic} with the case $r=g$ of
Proposition \ref{p:nbasic}, and also using that
$\mc{KZ}_2((\vpi,\vpi), 2\vpi)$ is the motive $\Q(0)$, we get that
\[
  \mc{KZ}_2(\vec{\lambda},0) \simeq
  W_r(\Psi_{f_I}(\mc{P}^r(C)))(-1) \otimes K (h')^{-1/4} ,
\]
where $h'$ is the function $\prod_{3\leq i<j\leq 2g+2}(z_i -
z_j)$. Here we also use that $(h/(z_1-z_2))|_{Z_I} = h' \cdot h_I^2$.

Thus, $\mc{KZ}_2(\vec{\lambda},0)$ is not pure and we have, again by
using Proposition \ref{p:nbasic}, that
\begin{enumerate}
\item $W_r(\mc{KZ}_2(\vec{\lambda},0)) = 0$,
\item $W_{r+1} (\mc{KZ}_2(\vec{\lambda},0)) = p^*(\mc{P}^{g-1}(C'))(-1)
  \otimes K (h' \cdot h_I^{-2})^{-1/4}$,
\item
  $W_{r+2} (\mc{KZ}_2(\vec{\lambda},0)) = \mc{KZ}_2(\vec{\lambda},0)$
  and
  $W_{r+2}(\mc{KZ}_2(\vec{\lambda},0))
  /W_{r+1}(\mc{KZ}_2(\vec{\lambda},0))$ is isomorphic to
  $p^*(\mc{P}^{g-2}(C'))(-2) \otimes K (h' \cdot h_I^{-2})^{-1/4}$.
\end{enumerate}

\begin{remark}\label{r:nsplit}
  It follows from Corollary \ref{c:ext} and the above that for a
  general point $\vec{z}$ of $\mc{C}_{2g+1}(\C)$ the mixed Hodge
  structure $\msf{KZ}_2(\vec{\lambda}, 0)_{\vec{z}}$ is
  non-split. 
\end{remark}

\subsubsection{}

The above method can be extended to compute
$\mc{KZ}_2(\vec{\lambda}, 0)$ for more general $\vec{\lambda}$ by
induction on the weights: we increase the weights by letting points
coincide.  Keeping track of the extensions is quite involved though,
and it is easier to only consider semisimiplifications. To simplify
further, we only consider the associated mixed Hodge structures
$\msf{KZ}_2(\vec{\lambda}, 0)_{\vec{z}}$ for the rest of this
section.  Given
$\vec{\lambda} = (a_1 \vpi, a_2\vpi,\dots, a_n\vpi) \in (P^+)^{n}$ and
$\vec{z} \in \mc{C}_{n}(\C)$, we let $C_{\vec{z}}$ be the double cover
of $\P^1_{\C}$ ramified at all the $z_i$ such that $a_i$ is odd; the
number $m$ of such points must be even in order to have a nonzero
space of invariants, so we assume this from now on. We let
$g = (m/2) -1$ be the genus of $C_{\vec{z}}$.

   We then have the following:
\begin{expect}  
  $\msf{KZ}_2(\vec{\lambda}, 0)_{\vec{z}}^{\mr{ss}}$ is a direct sum of
  Tate twists of terms of the form $P^r(C_{\vec{z}})$, possibly with
  multiplicities, for various $r$ such that $0 \leq r \leq g$.
\end{expect}
  \begin{enumerate}
  \item If all $a_i$ are odd then we have proved a more precise result
    which gives an explicit formula for
    $\msf{KZ}_2(\vec{\lambda}, 0)_{\vec{z}}^{\mr{ss}}$ which implies
    that it is in fact pure of weight $M+1 = (\sum_i a_i)/2 + 1$,
    where $M$ is defined in Section \ref{s:sv}; in particular,
    $\msf{KZ}_2(\vec{\lambda}, 0)_ {\vec{z}} \simeq
    \msf{KZ}_2(\vec{\lambda}, 0)_{\vec{z}}^{\mr{ss}}$. The proof is
    by induction on $n$ and the $a_i$ and reduces, using Theorem
    \ref{t:motivic} and Proposition \ref{p:nbasic}, to careful
    book-keeping; we do not give the details here. (To facilitate the
    induction one needs to prove a more general statement.)
  \item If all $a_i$ are even then
    $\msf{KZ}_2(\vec{\lambda}, 0)_{\vec{z}}$ is mixed Tate---this
    follows from Corollary \ref{c:mtate} as in this case the master
    function \eqref{e:master} is single-valued (ignoring the first
    factor)---but in general not pure (see Example
      \ref{e:mtate}). It would be interesting to know if there are
    any non-trivial extensions.
  \item In general, the length of the weight filtration is not
    bounded, but we expect (and can prove in many cases) that it can
    be bounded in terms of the number of even $a_i$. The MHS is likely
    to be non-split in many cases and we expect that the extensions
    will depend on the points of $C_{\vec{z}}$ above the points $z_i$
    with $a_i$ even (as in Remark \ref{r:nsplit}).
  \end{enumerate}

  \begin{remark}
    It is likely that the methods of this section can be extended to
    $\mc{KZ}_2(\vv{\lambda},\nu)$ for $\nu \neq 0$. Instead of taking
    $I = \{i_0,i_1\}$ in Section \ref{s:nbcpr}, one can let
    $I = [0,2g+2] \setminus \{i_0\}$ and then use Theorem \ref{t:fact}
    instead of Theorem \ref{t:motivic} to determine the nearby
    cycles. A suitable analogue of Lemma \ref{l:node}, which we expect
    to be straightforward, is then all that is needed.
  \end{remark}

\section{Conjectures, problems and examples}\label{s:Cpe}

\subsection{A BGG type conjecture}
Assume $\ell>0$ throughout this section and let
$\kappa=\ell+h^{\vee}$. Recall that we have defined the affine
  Weyl group $W_{\ell}$, affine walls, alcoves, etc., in Section
  \ref{s:affine}.

\subsubsection{Results of Teleman}\label{s:tele}
Let $\lambda_1,\dots,\lambda_n,\nu$ be dominant integral weights of $\mf{g}$ and let
$\vec{\lambda}=(\lambda_1,\dots,\lambda_n)$. We have the KZ mixed local system  $\mathcal{KZ}_{\kappa}(\vec{\lambda},\nu^*)$
on  $\mathcal{C}_n$.

Now assume that $\vec{\lambda}$ is arbitrary, but $\nu$ is at level
$\ell$.  Associated to this data, Teleman \cite{TelemanCMP} defines
certain relative Lie algebra cohomology groups, proves that they
vanish in all but at most one degree, and computes the nonzero
cohomology group.  Together, these results imply some exactness
statements for BGG resolutions. We review these results first.  Let
$\gamma_i=\lambda_i^*$.

Let $\vec{z}=(z_1,\dots,z_n)\in\mathcal{C}_n(\C)$.  The relative Lie algebra
cohomology defined by Teleman is denoted by
$H^*(\mf{g}_{\Bbb{C}[z]};\mf{g}_{\Bbb{C}}, \mathcal{H}(\nu)\tensor_{\C}
V_{\gamma_1}(z_1)\tensor_{\C} \dots\tensor_{\C} V_{\gamma_n}(z_n))$.
Then \cite[Theorem 0]{TelemanCMP} gives the following:
\begin{enumerate}
\item[(a)] If any  $\gamma_i +\rho$ lies on an affine wall, then the
  relative Lie algebra cohomology vanishes in all degrees.
\item[(b1)] Assume none of the $\gamma_i+\rho$ are on affine
  walls. Then the relative Lie algebra cohomology lives in a single
  degree $\sum_{i=1}^n l(w_i)$. Here $w_i$ is the affine Weyl group
  element with smallest length $l(w_i)$ such that
  $\mu_i=w_i(\gamma_i+\rho)-\rho$ is dominant at level $\ell$.
\item[(b2)] In case (b1), for each $i$ let $w_i$ be the affine Weyl
  group element that brings $\gamma_i+\rho$ into the fundamental
  alcove, and let $\mu_i=w_i(\gamma_i+\rho)-\rho$ which is dominant at
  level $\ell$. Then the nonvanishing cohomology group is isomorphic,
  as a vector space, to the fibre of the conformal blocks local system
  $\rat(\mathcal{CB}_{\ell}(\vec{\mu}^*,\nu^*))$ at the point
  $\vec{z}$, where $\vec{\mu}^* = (\mu_1^*,\dots,\mu_n^*)$.
\end{enumerate}

\subsubsection{}

The second set of results proved by Teleman concerns a certain complex
defined using the BGG resolution: There is a BGG complex $(K^q,d)$
where
$$K^q = \bigoplus {}'(V_{\lambda}\tensor V_{\gamma_1}\tensor V_{\gamma_2}\tensor\dots V_{\gamma_n})^{\mf{g}}\otimes_{\C} \mc{O}_{\mc{C}_n} $$
where $\lambda$ runs through dominant integral weights so that
$\lambda+\rho$ is affine Weyl group conjugate to $\nu^*+\rho$ and
$l(w)=q$ where $w\in W_{\ell}$ is of the smallest length such that
$\nu^*=w(\lambda+\rho)-\rho$. Teleman proves \cite[Proposition
3.4.8]{TelemanCMP} that it gives rise to a complex of KZ local systems
on $\mc{C}_n$ and this complex computes the relative Lie algebra
cohomology. This motivates the following:
\begin{conj}\label{BGGtype}
  There is a complex of motivic local systems on $\mathcal{C}_n$ with
  $q$th term given by
  $\mathcal{K}^q=
  \oplus_{\lambda}\mathcal{KZ}(\vec{\gamma};\lambda)\tensor_K
  L_{\lambda}$, where $\lambda$ runs through the same indexing set as
  in the BGG complex. Here $L_{\lambda}$ is a rank one $K$-motive
  whose Hodge polynomial is $t^{\mf{s}(\lambda-\nu^*)}$ as in
  Definition \ref{d:C} and the conjugate Hodge polynomial is
  $t^{l(w)}$ where $w\in W_{\ell}$ is the element of smallest length
  such that $\nu^*:=w(\lambda+\rho)-\rho$ is dominant at level $\ell$
  (see Conjecture \ref{c:ec}).
\end{conj}
One might hope that the motives $L_{\lambda}$ are given as a tensor
product of rank one motives, one for each factor in a minimal length
decomposition of the corresponding element of the affine Weyl
group. This would facilitate a definition of the maps in the motivic
BGG complex: suppose $\lambda$ is a summand in $K^q$ and $\lambda'$ a
summand in $K^{q+1}$. Then we have
$L_{\lambda'}\tensor_K L_{\lambda}^{-1}$ as the twist corresponding to
the corresponding component of $d_q$. We would like this twist to be
understandable in terms of the simple reflection that takes $\lambda$
to $\lambda'$.

\smallskip

The above conjecture implies Conjectures \ref{c:maptofusion} and
\ref{c:ec}. { Our current evidence for the conjecture is that we
  have numerically computed the alternating sum of the Hodge
  polynomials of the terms in the complex in many examples (for
  $\mf{sl}_n$) and these computations agree with what the conjecture
  would imply.}

\subsection{On the Hodge and weight filtrations}

Suppose $\vec{\lambda}$ and $\nu$ are at level $\ell$, and
$\kappa=\ell +h^{\vee}$.  Suppose
$\mathcal{KZ}_{\kappa}(\vec{\lambda},\nu^*)\neq 0$ and let $M \geq 0$
be the number of simple roots, counted with multiplicity, in an
expression of $\sum\lambda_i-\nu$ as a sum of simple roots.
\begin{conj}\label{c:HodgeFiltration}
  $ $
\begin{enumerate}
\item
  $\mathcal{KZ}_{\kappa}(\vec{\lambda},\nu^*)/\mathcal{CB}_{\kappa}(\vec{\lambda},\nu^*)$
  is mixed  of weights $> M$.
\item For $k\geq 0$,
  $F^{M-k}(\msf{KZ}_{\kappa}(\vec{\lambda},\nu^*)/\msf{CB}_{\kappa}(\vec{\lambda},\nu^*))$
   is a subspace of
  $\mb{V}^*_{\mathfrak{g},\vec{\lambda},\nu^*,\ell+k+1}/
  \mb{V}^*_{\mathfrak{g},\vec{\lambda},\nu^*,\ell}$, with equality for $k=0$.
\end{enumerate}
\end{conj}
We have verified the inequality (and equality) of ranks in the second
part in many examples. For $k=0$ the opposite inclusion is known to
hold, hence equality of ranks implies equality. This is because
extending the argument in \cite{b-unitarity} which shows that
conformal blocks at level $\ell$ extend holomorphically to any smooth
projective compactification of $\wh{U}$, the ramified cover of the
complement of the Schechtman--Varchenko arrangement (see Section
\ref{s:ao2} for the definition of $\wh{U}$), one can show that
conformal blocks at level $\ell+1$ give log forms on smooth projective
compatifications of $\wh{U}$, hence inject into
  $F^MH^M(\wh{U},\C)$ by \cite[Corollaire (3.2.13) (ii)]{hodge2}.

Part (1) of the above conjecture appears in \cite[Question 7.5]{BBM}
(where $\nu$ was zero). Because of our computations with Hodge
numbers, we have collected more evidence for this part and state it as
a conjecture here. Part (2) is asked in \cite[Question 7.5]{BBM} in
the stronger form of equality. This stronger assertion may fail, see
Example \ref{e:ceg}. Note that it was verified in the arxiv version of
\cite{BBM} that the filtration
$\mb{V}^*_{\mathfrak{g},\vec{\lambda},\nu^*,\ell+k+1}/\mb{V}^*_{\mathfrak{g},\vec{\lambda},\nu^*,\ell}$
of the VMHS
$\msf{KZ}_{\kappa}(\vec{\lambda};\gamma)/\msf{CB}_{\kappa}(\vec{\lambda};\gamma)$
satisfies Griffiths transversality.}

Recall that if $(a,N)=1$, Then
$\mathcal{KZ}_{\kappa/a}(\vec{\lambda},\nu^*)$, is a Galois conjugate
of $\mathcal{KZ}_{\kappa}(\vec{\lambda},\nu^*)$. Suppose $1<a<N$ and
$\kappa/a = h^{\vee} +\ell' -\alpha$ where $0<\alpha<1$ and
$\ell'<\ell$ is a positive integer. Assume $\lambda_1,\dots,\lambda_n$
and $\nu$ are at level $\ell'$. Then analogously to part (2) of the
above conjecture, we may ask if for $k\geq 0$,
$F^{M-k}\msf{KZ}_{\kappa/a}(\vec{\lambda},\nu^*)$  is a subspace
of $\mb{V}^*_{\mathfrak{g},\vec{\lambda},\nu^*,\ell'+k}$ with
  equality for $k=0$.   We have verified the numerical inequality (and equality) in many examples.

\smallskip

It would be very interesting to describe the weight filtration---or
even just the ranks thereof---of
$\mathcal{KZ}_{\kappa}(\vec{\lambda},\nu^*)/
\mathcal{CB}_{\kappa}(\vec{\lambda},\nu^*)$. These motivic local
systems are not pure in general, see Example \ref{e:notpure}.

More generally, we have the following:
\begin{problem}\label{weight}
  Compute the ranks of the terms in the weight filtration of all KZ
  motives.
\end{problem}
While the main results of this paper are about the Hodge filtration,
they can be used to give non-trivial bounds on the weights using Lemma
\ref{l:weights}. Note that the results of Section \ref{s:higher} show
that the weight filtration can be quite complicated in the general KZ
setting.

More ambitiously, one could also ask for an explicit description of
the weight filtration.  This question already appears in \cite[Remark
3.8]{L2}.

\subsection{General Lie algebras}\label{s:genG}

\subsubsection{The Lie algebra $\mf{so}_5$}Let $\varepsilon_1$ and
$\varepsilon_2$ be an orthogonal basis of $\mathbb{R}^2$. The simple
roots are $\alpha_1=\varepsilon_1-\varepsilon_2$ and
$\alpha_2=\varepsilon_2$. The other positive roots are
$\alpha_1+\alpha_2$ and $\alpha_1+2\alpha_2$ and the fundamental
weights $\varpi _1=\varepsilon_1=\alpha_1+\alpha_2$ and
$\varpi _2=\frac{1}{2}(\varepsilon_1+\varepsilon_2)$.  Any dominant
integral weight of $\mf{so}_5$ is of the form
$\lambda=a\varpi _1+b\varpi _2$, where $a, b$ are non-negative
integers.

As in the case of $\mf{sl}_n$, to compute all products in
$\mf{R}_{\kappa}(\mf{so}_5)$ it is enough to compute all products
$[\lambda]\star [\vpi_{k}]$ for $k=1,2$ for any dominant integral
weight $\lambda$ (and similarly for $\mf{F}_{\kappa}(\mf{so}_5)$).

The two lemmas below follow from \cite[\S B]{weyman}.
\begin{lemma}\label{prop:LRruleso5}
  The product $[a\vpi_1+b\vpi_2]\otimes [\vpi_1]$ in $\mf{R}(\mf{so}_5)$
  decomposes
  into the following factors with multiplicity one:
    \begin{enumerate}
  \item $[(a-1)\vpi_1+b\vpi_2]$ and the difference is
    $2(\alpha_1+\alpha_2)$.
  \item $[(a+1)\vpi_1+(b-2)\vpi_2]$ provided $b\geq 2$ and the
    difference is $\alpha_1+2\alpha_2$ case.
  \item $[a\vpi_1+b\vpi_2]$ provided $b>0$ and the difference is
    $\alpha_1+\alpha_2$.
  \item $[(a-1)\vpi_1+(b+2)\vpi_2]$ and the difference is the simple
    root $\alpha_1$.
  \item $[(a+1)\vpi_1+b\vpi_2]$ and the difference is zero.
  \end{enumerate}
\end{lemma}
\begin{lemma}\label{prop:tensor2}
  The product $[a\vpi_1+b\vpi_2]\otimes [\vpi_2]$ in
  $\mf{R}(\mf{so}_5)$ decomposes into the following factors with
  multiplicity one
  \begin{enumerate}
  \item $[a\vpi_1+(b-1)\vpi_2]$ provided $b>0$ and the difference is
    $\alpha_1+2\alpha_2$.
  \item $[(a+1)\vpi_1+(b-1)\vpi_2]$ provided $b> 0$ and the difference
    is $\alpha_2$.
  \item $[(a-1)\vpi_1+(b+1)\vpi_2]$ and the difference is
    $\alpha_1+\alpha_2$.
  \item $[a\vpi_1+(b+1)\vpi_2]$ and the difference is zero.
  \end{enumerate}
\end{lemma}
In the above lemmas we ignore any case with a negative coefficient.

\begin{proposition}
  If we can compute all products of the form
  $[a\vpi_1+b\vpi_2]\star[\vpi_2]$ and $[b\vpi_2]\star [\vpi_1]$ in
  $\mf{R}_{\kappa}(\mf{so}_5)$, then we can compute
  $[a\vpi_1+b\vpi_2]\star [\vpi_1]$ in the ring
  $\mf{R}_{\kappa}(\mf{so}_5)$.
\end{proposition}
\begin{proof}
We use induction on $a$ to show how all products of the form
$[a\vpi_1 + b \vpi_2]\star[\vpi_1]$ can be computed. The case $a=0$ is
part of the assumptions, so we suppose $a>0$ and assume that all
products of the form $[(a-1)\vpi_1+c\vpi_2]\star [\vpi_1]$ are known
for any $c \geq 0$. We have
\begin{multline*}
  [(a-1)\vpi_1+(b+1)\vpi_2]\star [\vpi_2] = *[(a-1)\vpi_1+b\vpi_2]+*
  [a\vpi_1+b\vpi_2]
  \\
  +
  *[(a-2)\vpi_1+(b+2)\vpi_2] + *[(a-1)\vpi_1+(b+2)\vpi_2],
\end{multline*} (the third term is ignored if $a=1$), where $*$
represents a known coefficient (by induction).  By induction, we 
know how to multiply the first, third and fourth terms on the RHS by
$[\vpi_1]$. It follows that it suffices to compute
\[
  ([(a-1)\vpi_1+ (b+1)\vpi_2]\star [\vpi_2])\star[\vpi_1] =
  ([(a-1)\vpi_1+ (b+1)\vpi_2]\star [\vpi_1])\star[\vpi_2] ,
\]
where the equality holds by associativity and commutativity of the
$\star$ product.  By the inductive assumption we can compute the
product $[(a-1)\vpi_1+ (b+1)\vpi_2]\star [\vpi_1]$ and by the initial
assumptions we know how to multiply all dominant integral weights by
$[\vpi_2]$, so the proof is complete.
\end{proof}
The point of the proposition is that it allows one to avoid having to
explicitly compute the coefficient in case (1) of Lemma
\ref{prop:LRruleso5}. In general, it seems that the more multiple
roots there are in the expression of $ \lambda + \mu -\nu$ as a sum of
simple roots, the harder the coefficient of $[\nu]$ in
$[\lambda]\star[\mu]$ is to compute (and we do not know how to do this
in most cases). In the case of $\mf{sl}_n$ we are able to avoid all
such cases, but this does not appear to be possible in any other case.

We may consider the products that appear in the above proposition as
the ``minimal cases'' and expect that they are amenable to an analysis
by curve fibrations, e.g., in the case (1) in Lemma
\ref{prop:tensor2}, we are looking at cohomology of a local system
$\mc{L}$ (with support and images of such) over an open subset
$U\subseteq \A^3_{\msf{k}}$. This open subset maps to $\A^1_{\msf{k}}$
by the variable corresponding to $\alpha_1$. The fibres of
$U\to \A^1_{\msf{k}}$ are open subsets of $\A^2_{\msf{k}}$, and the
local system $\mc{L}$ on the fibres of $U\to \A^1$ looks like the
local system coming from a suitable master function for $\mf{sl}_2$. A
suitable adaptation/generalisation of \cite[Proposition 2.3.2]{DS}
could then help in the computation of the Hodge numbers.

  \begin{problem}\label{general}
  Give an algorithm for computing all products in the rings
  $\mf{R}_{\kappa}(\mf{g})$ and $\mf{F}_{\kappa}(\mf{g})$ for an
  arbitrary simple Lie algebra $\mf{g}$.
\end{problem}
We do not have a strategy for attacking this problem in general, but
it seems reasonable to hope that the computation of products for all
$\mf{sp}_{2n}$ could be reduced to the products for all $\mf{sl}_n$
for all $n$ (which we already know) and for $\mf{sp}_4$
($\simeq \mf{so}_5$) which we expect can be done along the lines
discussed above.

\begin{remark}
  It would be interesting even to just to have a formula for the Hodge
  polynomial of the rank $1$ motives
  $\mathcal{KZ}_{\kappa}((\lambda,\lambda^*);0)$ for all $\mf{g}$,
  $\lambda \in P^+$, and rational $\kappa$; we do not know how to
  determine even the weight of these motives.
\end{remark}

\subsection{Examples}\label{s:someexamples}

In this section we give a few numerical examples (worked out using a
computer implementation\footnote{Available on request from the
  authors} of our $\mf{sl}_n$ algorithm from Section \ref{s:sln})
which illustrate some of our results and conjectures. We use the usual
``partition notation'' to represent dominant integral weights for
$\mf{sl}_n$, e.g., $(4,2,1)$ denotes the weight
$2\vpi_1 + \vpi_2 + \vpi_3$ for $\mf{sl}_4$.

\begin{example}\label{e:introduction}
  For $\mf{sl}_3$, the motive
  $\mathcal{KZ}_{13}((7,5), (9,5)), (8,6))_{(0,1)}= M_{13}((7,5),
  (9,5), (8,6))$ is neither pure not mixed Tate because it has
  Hodge polynomial $2t^8+t^7$ and complex conjugate Hodge polynomial
  $t^2+t+1$. All representations are at level $13-3=10$, and the rank
  of conformal blocks is one (which is pure of weight $8$ with Hodge
  polynomial $t^8$ and conjugate Hodge polynomial $1$). At levels $11$
  (resp.~$12$), the ranks of the conformal blocks are $2$ and $3$,
  consistent with Conjecture \ref{c:HodgeFiltration}.
  
\end{example}

\begin{example}\label{e:notpure}
For the Lie algebra $\mf{sl}_4$, let $\lambda_1=(3,3,1)$,
$\lambda_2=(3,2,1)$, $\lambda_3=(3,2,2)$, $\lambda_4=(1,1,0)$, and
$\nu=(3,3,0)$, at level $3$, and $\kappa=7$. The value of $M$ is
11. The rank of conformal blocks is $2$ at level $3$.

The Hodge polynomial of $\msf{KZ}_{\kappa}(\vec{\lambda},\nu^*)$ is
$24t^{11}+28t^{10}+6t^9$ and the complex conjugate Hodge polynomial is
$7t^{3}+27t^{2}+22t+2$. $\msf{CB}$ is pure of weight $11$. The Hodge
polynomial of $\msf{KZ}/\msf{CB}$ is therefore
$22t^{11}+28t^{10}+6t^9$ and the complex conjugate Hodge polynomial is
$7t^3+27t^2+22t$. This is not pure of any weight.  (Note that if the
Hodge polynomial of a pure motive of weight $M$ is $P$, then the Hodge
polynomial of the complex conjugate is $t^MP(1/t)$.) 

In fact it appears that the weight filtration of $\msf{KZ}/\msf{CB}$
has two steps. The weight $12$ part is likely of rank $55$ with Hodge
polynomial $22 t^{11}+27t^{10}+6t^9$ and conjugate Hodge polynomial
$6t^3+27t^2+22t$, and there is likely a one dimensional weight $13$
quotient with Hodge polynomial $t^{10}$ and conjugate Hodge polynomial
$t^3$.

At levels $4$, $5$, and $6$, the ranks of conformal blocks are
$24, 24+28=52$ and $24+28+6=58$ respectively. This illustrates Conjecture 
\ref{c:HodgeFiltration}.
\end{example}

\begin{example}\label{e:ceg}
  Consider $\mf{sl}_2$ variations of the form
  $\msf{KZ}_{\kappa}(\vec{\lambda},0)$ where $\lambda_1=\dots=\lambda_{12}=(1)$
  with $\kappa=3$. The Hodge polynomial is
  $32t^6 + 56t^5 + 34t^4 + 9t^3 + t^2$, $M=6$, and the ranks of
  conformal blocks at levels $1,2,3,4,5,6$ are $1, 32, 89, 122,131$
  and $132$ respectively. Note that $32+ 56=88< 89$. This shows that
  the conjectured containment in the second part of Conjecture
  \ref{c:HodgeFiltration} cannot be an equality for $k=1$.

  With the same $\vec{\lambda}$, the Hodge polynomial of the Galois
  conjugate $\msf{KZ}_{5/2}(\vec{\lambda},0)$ of
  $\msf{KZ}_{5}(\vec{\lambda},0)$ is computed to be
  $t^6 + 31t^5 + 56t^4 + 34t^3 + 9t^2 + t$. Here $\ell'=1$, and the
  ranks of the conformal blocks at levels $1,2,3,4,5$ are as listed in
  the previous paragraph. But equality cannot hold for $k=2$ since
  $1+31+56<89$.
\end{example}

\begin{example}
  The local monodromy of a variation of pure Hodge structure over a
  punctured disc is determined by the limiting mixed Hodge
  structure. We use this method to give examples of KZ variations
  which are pure but with non-semisimple local monodromy using Theorem
  \ref{t:fact} and Remark \ref{r:topological}. The basic $\mf{sl}_2$
  level $0$ case in Section \ref{s:level0} already gives such
  examples, but the following example shows how one can do the same
  based on general principles, without using any explicit geometry.
  
  Let $\ell \geq 1$ be such that $\kappa=\ell+2$ is a prime. We claim
  that $\mf{sl}_2$ variations of the form
  $\msf{KZ}_{\kappa}(((\ell+1),(\ell+1),(\ell+1)),(\ell+1))$ (three
  finite points and one infinite point) are pure of weight
  $M = \ell+1$.  This follows from the observation that the residues
  of the logarithmic derivative of the master function $\eta$ are
  non-integral for the abnormal strata (see Section \ref{s:resol})
  using formulas in \cite[Section 5]{loo-sl2}: The residue when $k>1$
  of the $M=\ell+1$ points come together is $k(k-1)/\kappa$, when
  $k\geq 1$ of the points come together at $z_1$ is
  $k(1-k+\ell+1)/\kappa$, and when $k\geq 1$ of the points come
  together at infinity is $-k(k+\ell+2)/\kappa$ (the computation at
  infinity is different from the one in \cite[Section 5]{loo-sl2}
  since we have a nonzero weight at infinity, but can be carried out
  in a similar way).

  The VMHS
  $\msf{KZ}_{\kappa}(((\ell+1),(\ell+1),(\ell+1)),(\ell+1))$ appears to
  have rank $\ell +2$ and the local monodromy as two of the finite
  points come together appears to have $(\ell+1)/2$ Jordan blocks of
  size two and a single block of size one. This can be seen by
  examining the weights of the graded pieces in the factorisation, but
  since we do this by explicit computations in the enriched fusion
  ring, we have only verified these statements for small values of
  $\ell$. We consider a specific case below.

  Let $\ell=5$ and $\kappa=\ell+2 = 7$. The multiplicity motives
  $M_7((6),(6),(2a))$ are of rank one for $a=0,1,\dots, 6$, so we can
  determine their weights by computing their Hodge polynomials
  together with those of their complex conjugates.  They are
  $7,6,5,3,2,1,0$ for $a = 0,1,\dots,6$.  Similarly, the weights for
  $\msf{KZ}_{7}(((6),(2a)),(6))$ (i.e., with weight $(6)$ at $\infty$)
  are given by $0,1,1,2,3,3,4,5$ for $a=0,1,\dots,6$. So the weights
  of the tensor product in the factorisation for
  $\msf{KZ}_7(((6),(6),(6)),(6))$ of Theorem \ref{t:motivic}
  are $7,7,7,6,5,5,5$.

  Therefore, using Theorem \ref{t:fact} and the fact that the weight
  filtration on nearby cycles of a pure VHS is given by the monodromy
  filtration \cite[\S 2.3]{saito}, it follows that the monodromy
  operator $\mr{N}$ when two of the finite points come together has
  three Jordan blocks of size two and one of size one (corresponding
  to the weight $(6)$).
\end{example}

\begin{example}\label{e:FiniteGoursat3}
  For the Lie algebra $\mf{sl}_4$, let $\lambda_1=\lambda_2=(5,2,2)$,
  $\lambda_2=(6,3,0)$ and $\lambda_4=(1,0,0)$, at level $8$, and
  $\kappa=12$.  The Hodge polynomials of the Galois conjugates of
  $\msf{CB}_{\kappa}(\vec{\lambda},0)$ are $4t^{27}$, $4t^{15}$,
  $4t^{12}$ and $4$. The global monodromy of
  $\msf{CB}_{\kappa}(\vec{\lambda},0)$ over $\mc{C}_4$ is therefore
  finite. (In this example we have
  $\msf{CB}_{\kappa}(\vec{\lambda},0) =
  \msf{KB}_{\kappa}(\vec{\lambda},0)$.)

  Fix $z_1,z_2,z_3$ and let $z=z_4$ vary.  We get a rank four
  conformal blocks VMHS on $\P_{\C}^1-\{z_1,z_2,z_3,\infty\}$ by
  restricting $\msf{CB}_{\kappa}(\vec{\lambda},0)$. One knows that the
  local system is rigid and irreducible (see Remark \ref{r:rigid}
  below).
  The local system has
  central monodromy at $\infty$, and at each of $z_1$, $z_2$ and $z_3$
  there are three eigenvalues, with two of multiplicity one and one of
  multiplicity two, see below. This implies that this local system
  belongs to the Goursat III family \cite{goursat, RR-V} of rigid
  local systems of rank $4$, the Katz rigidity condition holds since
  $3(1^2+1^2+2^2)=(3-2)4^2 +2$.  We do not know of any finite
  monodromy Goursat III type examples in the literature, but
  \cite{RR-V} gives rank four rigid local systems with finite
  monodromy in the Goursat II family.

  All the local monodromies (ignoring the one at infinity which can be
  absorbed in the others) have three distinct eigenvalues, so this is
  not a hypergeometric system. Note that if we ignore the $(z_i-z_j)$
  factors in the master function, the eigenvalues of local monodromies
  are $\kappa$th roots of unity, but we keep these factors in the
  computation below.

  By Remark \ref{r:topological}, the local exponents at $z_1$ (or
  $z_2$) are $\frac{1}{2\kappa} (-c(\mu)+ c(\lambda_1) +c(\lambda_4))$
  where $\mu=(6,2,2), (5,3,2)$ and $(4,1,1)$, with the corresponding
  multiplicities being $1$, $2$ and $1$. The eigenvalues corresponding
  to these exponents are easily seen to be different, i.e., $2\kappa$
  does not divide the differences of $c(\mu)$ for different values of
  $\mu$.

  The local exponents at $z_3$ are
  $\frac{1}{2\kappa} (-c(\mu)+ c(\lambda_3)+c(\lambda_4))$ where
  $\mu=(7,3,0)$, $(6,4,0)$ and $(6,3,1)$, with the corresponding
  multiplicities being $1$, $1$ and $2$.  The eigenvalues
  corresponding to these exponents are also different different.

The central monodromy at infinity has exponent $-c(\lambda_4)/\kappa$.  
\end{example}

\begin{remark}\label{r:rigid}
  The following method from \cite{BRigid} constructs unitary rigid
  local systems. We have recast the quantum cohomology formalism used
  there in the equivalent fusion ring version.

  Let $\mu_1,\dots,\mu_s$ be highest weights for $\mf{sl}_n$ at level
  $\ell$, so that the rank of $\mathcal{CB}_{\kappa}(\vec{\mu},0)$ is
  one. Assume that $\mu_i$ is attached to the point
  $z_i\in \mb{A}^1_{\C}$. Let $\mu'_1$ be a summand that appears in
  the classical fusion product $[\mu_1]\cdot[\vpi_{n-1}]$ with
  coefficient $1$ (i.e., nonzero, since this is a case of the Pieri
  formula). Then we consider the conformal blocks for $\mf{sl}_n$ for
  $s+1$ points at the same level $\ell$ with representations
  $\lambda_1=\mu'_1$, and $\lambda_2=\mu_2,\dots,\lambda_s=\mu_s$ and
  $\lambda_{s+1}=\vpi_1$ (at a new point $p=z_{s+1}$), and the trivial
  representation at $\infty$.  This local system with the point $p$
  moving gives an irreducible unitary rigid local system on
  $\mb{P}_{\C}^1-\{z_1,\dots,z_s,\infty\}$ \cite[Section
  5.7]{BRigid}. The eigenvalues of the local monodromy (after
  tensoring the local system by an explicit one dimensional local
  system) are $\kappa=(n+\ell)$th roots of unity (they have central
  monodromy at infinity). This unitary local system has an eigenvalue
  of the monodromy at $z_1$ of multiplicity one. All unitary rigids
  with this property (some local monodromy having an eigenvalue of
  multiplicity one) arise in the above manner, see \cite[Section
  5.8]{BRigid}. (Unitary rigids without this property also come from
  conformal blocks, but the description is a bit more complicated. No
  examples are known where this property fails for unitary rigids.)
  
  To obtain Example \ref{e:FiniteGoursat3}, let $\mu_1=(7,3,0)$,
  $\mu_2=\mu_3=(5,2,2)$, and verify that $[\mu_1]\cdot[\vpi_{3}]$ has
  the term $[\lambda_1]$ with multiplicity one.
 
\end{remark}

\appendix
\section{KZ local systems and cohomology}\label{s:promises}

\subsection{Changes for arbitrary weights at infinity.}

In  \cite{L2,BBM}, Proposition \ref{p:kz} is proved with $\nu=0$. Here we indicate the changes needed for arbitrary dominant integral $\nu$ following the outline given in \cite{BF}. If $\vec{\lambda}$ is an $n$-tuple of weights, by the notation $V(\vec{\lambda})^*_{\nu}$, we mean $(V(\vec{\lambda})^*)_{\nu}$-the $\nu$-th weight space of $V(\vec{\lambda})^*$ 

As in the case $\nu=0$, the first step is to define a map
$\mb{A}(\vv{\lambda},\nu^*)^*\to H^M(A^{\starr}(U), \eta)$ as the
composite
\begin{equation}\label{compositum}
	\mb{A}(\vv{\lambda},\nu^*)^*\lra V(\vv{\lambda})^*_{-\nu}\lra
        H^M(A^{\starr}(U), \eta) .
\end{equation}
The first map is injective and arises from the
map
$$V(\vv{\lambda})\lra V(\vv{\lambda})\tensor V_{\nu}^*, \ \ v_1\tensor
v_2\tensor\dots\tensor v_n \mapsto v_1\tensor v_2\tensor\dots\tensor
v_n\tensor v,$$ where $v$ is the lowest weight vector of $V_{\nu}^*$,
of weight $-\nu$ (as in \cite[(12)]{FSV2}). The
map
$$\Omega^{SV}_{\vv{\lambda}, \nu^*}:V(\vv{\lambda})^*_{-\nu}\lra
A^M(U)$$ is defined in \cite[(6.5.2)]{SV}. The spaces
$H^M(A^{\starr}(U), \eta)= A^M(U)/\eta\wedge A^{M-1}(U)$ do not change
under scaling of the map, hence the composite \eqref{compositum} is constant under scaling of $\kappa$.\\

\paragraph{\bf Injectivity of the map} The next step is to prove that
the map \eqref{compositum} is injective. It suffices to prove this
when $\kappa=\ell+h^{\vee}$ with $\ell\in \mb{Z}$ sufficiently large.
Then,
$\mb{A}(\vv{\lambda},\nu^*)^*=\mb{V}_{\mathfrak{g},\vec{\lambda},\nu^*,\ell}(\mb{P}^1,z_1,\dots,z_n)$,
the space of conformal blocks.  Note that \cite[Prop 8.5]{b-unitarity}
can be extended to show that correlation functions with $v$ placed at
infinity agree with the differential forms obtained as the image of
the map in \cite{SV}, i.e., $\mb{A}(\vv{\lambda},\nu^*)^*\to A^M(U)$
above. This is because we are using the lowest weight vector
$v\in V_{\nu}^*$ and the $f$ operators act by zero on $v$ (this
becomes relevant when we use the gauge condition).

Thus, by \cite{b-unitarity} (also \cite[Remark 4.5, Claim 4.6]{BF} for
changes for arbitrary $\nu$), the composite
$$\mb{V}_{\mathfrak{g},\vec{\lambda},\nu^*,\ell}(\mb{P}^1,z_1,\dots,z_n)\lra H^M(A^{\starr}(U), \eta)\lra H^M(U,\ml(\eta))$$ 
is injective, so we have injectivity in \eqref{compositum}.\\

\paragraph{\bf Determining the image}
Using Lemma \ref{l:topo}, we write \eqref{compositum} as
\begin{equation}\label{eqn:Loo1}
	\begin{tikzcd}
\mb{A}(\vv{\lambda},\nu^*)^* \ar[r, hook] &  H^M(A^{\starr}(U), \eta)^{\chi} \ar[r, "\cong"] &H^M(V,j_{!}\mathcal{L}(a))^{\chi}.
\end{tikzcd}
\end{equation}
To determine the image of
$\mb{A}(\vv{\lambda},\nu^*)^*\to H^M(V,j_{!}\mathcal{L}(a))^{\chi}$,
following Looijenga \cite{L2} we consider the same map for $-\kappa$
and dual weights and dualizing the resulting map. Note that
$$\sum\lambda^*-\nu^*= \sum_{b=1}^M (-w_0\circ \beta(b)),$$
where $w_0$ is the longest element in the Weyl group of $\mf{g}$ (so that $\lambda^*=-w_0\lambda$).
We obtain an injective map (the vertical left one) making the
following diagram commute:
\begin{equation}
\begin{tikzcd}
\mb{A}(\vv{\lambda}^*,\nu)^*\ar[d, hook] \ar[r, hook] & V(\vv{\lambda}^*)^*_{-\nu^*} \ar[d, "\Omega_{\vv{\lambda}^*,\nu}^{SV}"]\\
H^M(A^{\starr}(U), -\eta)^{\chi} & (A^{M}(U))^{\chi}\ar[l, two heads]
\end{tikzcd}
\end{equation}
Here we note that since $\beta(b)$ are simple roots, so are $-w_0\circ \beta(b)$. Therefore we get the same $U$ for this dual situation. The master function
is the inverse of the old master function, and we get a surjection (after dualizing)
\begin{equation}\label{eqn:Loo2}
H^M(V',q_{!}\mathcal{L}(a))^{\chi}\cong (H^M(A^{\starr}(U),-\eta)^{\chi})^*\twoheadrightarrow \mb{A}(\vv{\lambda}^*,\nu),
\end{equation}
where $V'=P-\bigcup'_{\alpha}E_{\alpha}$ and the union is restricted to
$\alpha$ such that $a_{\alpha}=\op{Res}_{E_{\alpha}}\eta$ is not a
non-negative integer. Note that $V'\supseteq V$, and $q:U\to
V'$. 

At the level of forms we  dualize $$V(\vv{\lambda}^*)^*_{-\nu^*}\to A^M(U)^{\chi}\to H^M(A^{\starr}(U), -\eta)^{\chi}$$ to get that the composite below is surjective:
\begin{equation}\label{duall}
	H^M(A^{\starr}(U), -\eta)^*\lra A^M(U)^*\lra
        V(\vv{\lambda^*})_{\nu^*}= V(\vv{\lambda})^*_{\nu^*}\lra
        \mb{A}(\vv{\lambda}^*,\nu) .
\end{equation}

The isomorphism
\begin{equation*}
	\begin{tikzcd}[column sep=small]
\mb{A}(\vv{\lambda},\nu^*)^*\ar[r,equal]& \left((V(\vv{\lambda})^*\otimes V_{\nu^*}^*)_{\mf{g}}\right)^* \ar[r,equal] & \left((V(\vv{\lambda})\otimes V_{\nu^*}\right)^{\mf{g}}\ar[r,equal] &\left((V(\vv{\lambda}^*)^*\otimes V_{\nu^*}\right)^{\mf{g}}\ar[r,"\cong"]&  \mb{A}(\vv{\lambda}^*,\nu),
\end{tikzcd}
\end{equation*} which we take to be the natural inclusion of  $\mb{A}(\vv{\lambda},\nu^*)^*$ to $\mb{A}(\vv{\lambda}^*,\nu)$ times the scalar $\frac{1}{\dim V_{\nu}}$, 
implies that 
$\mb{A}(\vv{\lambda},\nu^*)^*\hookrightarrow
H^{M}(V,j_{!}\mathcal{L}(a))^{\chi}$ is the image of a map 
\begin{equation}\label{rep}
	H^M(V',q_{!}\mathcal{L}(a))^{\chi}\lra H^M(V,j_{!}\mathcal{L}(a))^{\chi}.
\end{equation}
obtained by  the surjection in  \eqref{eqn:Loo2} and the injection in  \eqref{eqn:Loo1}. \\

\paragraph{\bf  Identification of maps}To prove  Proposition \ref{p:kz}, we need to show  that  \eqref{rep} defined here  coincides with the map between the source and target of \eqref{rep} induced from topology by the natural inclusion $V\subseteq V'$. This is
proved in \cite[Theorem 7.3]{BBM} for $\nu=0$. The same argument extends for arbitrary $\nu$ as we explain below. 

The main point is that the set-up in \cite[(6.5.2)]{SV} is valid for
arbitrary weight spaces. The commutative diagram replacing \cite[
(8.3)]{BBM} is the following (obtained from \cite[ (6.5.3) and Theorem
6.6]{SV}) combined with \cite[(2.3.2) and (3.3.1)]{SV}.
\begin{equation}\label{bigS}
\begin{tikzcd}
A^M(U)^*\ar[r, "\Omega_{\vv{\lambda},\nu^*}^{SV*}"]\ar[d] & V(\vv{\lambda})_{\nu}\ar[d]\\
A^M(U) & V(\vv{\lambda})^*_{-\nu}\ar[l, "\Omega_{\vv{\lambda},\nu^*}^{SV}"']
\end{tikzcd}
	\end{equation}
The bottom horizontal map is the map $\Omega_{\vv{\lambda},\nu^*}^{SV}:V(\vv{\lambda})^*_{-\nu}\to A^M(U)$ above  defined in \cite{SV}, and the horizontal map on the top is the dual.

Diagram \eqref{bigS}  is obtained from a larger diagram (where the objects in the middle column are weight spaces of tensor product of Verma modules as in \cite[Section 6.4]{SV}) and their duals.
\begin{equation}\label{eqn:gettingridofvarchenko}
\begin{tikzcd}
A^M(U)^*\ar[d,"S"']\ar[r] &  M(\vec{\lambda})_{\nu}\ar[d,"\operatorname{Sh}"']\ar[r] & V(\vv{\lambda})_{\nu}\ar[d,"\cong","\operatorname{Sh}"']\ar[from=ll, bend left=25,"\Omega_{\vv{\lambda},\nu^*}^{SV*}" description]\\
A^M(U)\ar[from=rr, bend left=25,"\Omega_{\vv{\lambda},\nu^*}^{SV}" description]& M(\vec{\lambda})_{-\nu}^*\ar[l]& V(\vv{\lambda})^*_{-\nu}\ar[l]
\end{tikzcd}
\end{equation}

The vertical map $\operatorname{Sh}$ on the right is obtained from the
Shapovalov form (which is an isomorphism on finite dimensional
representations) from representation theory \cite[Section 6.4]{SV}.
The vertical map on the left in \eqref{bigS} is the form $S$ from
\cite[(1.3)]{BBM} (up to identification is the quasi-classical
contravariant form from \cite[(3.3.1)]{SV}) is the one that induces
the natural topological map
$H^M(V',q_{!}\mathcal{L}(a))^{\chi}\to
H^M(V,j_{!}\mathcal{L}(a))^{\chi}$ obtained from the inclusion
$V\subseteq V'$ as shown in \cite{BBM}.  All maps in the diagram
\eqref{eqn:gettingridofvarchenko} except for the vertical map on the
left are representation theoretic.

The action of the longest element $w_0$ of the Weyl group identifies
$w_0: V(\vv{\lambda})^*_{-\nu}\to V(\vv{\lambda}^*)_{\nu^*}$.
Incorporating this in the diagram \eqref{bigS}, we write a bigger
diagram

\begin{equation}\label{bigSS}
	\begin{tikzcd}[column sep=large]
	A^M(U)^*\ar[r, "\Omega_{\vv{\lambda},\nu^*}^{SV*}" description]\ar[d] & V(\vv{\lambda})_{\nu}\ar[r, "\operatorname{Sh}"] & V(\vv{\lambda})^*_{-\nu}\ar[r, "w_0"]\ar[d,equal] &  V(\vv{\lambda}^*)_{\nu^*}\ar[r]  \ar[from=lll, bend left=25, "\Omega_{\vv{\lambda}^*,\nu}^{SV*}" description]& \mb{A}(\vv{\lambda}^*,\nu)\\
		A^M(U) &  & V(\vv{\lambda})^*_{-\nu}\ar[ll,"\Omega_{\vv{\lambda},\nu^*}^{SV}" ] & &   \mb{A}(\vv{\lambda},\nu^*)^*\ar[u, "\cong"]\ar[ll]
		\end{tikzcd}
\end{equation}
Here the rightmost horizontal maps of the top and bottom rows are
natural surjections and inclusions respectively.  The right hand
square in the above diagram does not commute. We rewrite it as
follows:
\begin{equation}\label{bigone}
	\begin{tikzcd}
		 V(\vv{\lambda}^*)_{\nu^*}\ar[r]    & \mb{A}(\vv{\lambda}^*,\nu)\\
	          V(\vv{\lambda})^*_{-\nu}\ar[u,"w_0"]&\mb{A}(\vv{\lambda},\nu^*)^*\ar[l]\ar[u,"\cong"']
	\end{tikzcd}
\end{equation}

The arguments in \cite[Lemma 8.3 and Lemma 8.4]{BBM} show that the map
$\Omega^{SV*}_{\vv{\lambda}^*,\nu}: A^M(U)^*\to
V(\vv{\lambda}^*)_{\nu^*}$ in \eqref{duall} constructed in \cite{SV}
for the pair $(\vv{\lambda}^*,\nu^*)$ (also the top curved arrow in
diagram \eqref{bigSS}) is the same as the the map
$A^M(U)^*\to V(\vv{\lambda})^*_{-\nu}$ from \eqref{bigS} composed
with the action of $w_0$ given by
$V(\vv{\lambda})^*_{-\nu}\to V(\vv{\lambda})^*_{\nu^*}$ since
$\nu^*=-w_0\nu$.  In particular,
\begin{equation}
\Omega^{SV*}_{\vv{\lambda}^*, \nu}=w_0\circ \operatorname{Sh} \circ
\,\Omega^{SV*}_{\vv{\lambda}, \nu^*} .
\end{equation}

The  map from \eqref{bigS}  is written as a composition $A^M(U)^*\to V(\vv{\lambda})_{\nu}\to V(\vv{\lambda})^*_{-\nu}$, and the identity is verified after taking duals, noting that the top horizontal arrow in \eqref{bigS} is the dual of the lower horizontal arrow.

The composite obtained by chasing arrows in the right half of
\eqref{bigSS}
$$ 
\begin{tikzcd}
V(\vv{\lambda})^*_{-\nu}\ar[r,"w_0"] &  V(\vv{\lambda}^*)_{\nu^*}\ar[r] & \mb{A}(\vv{\lambda}^*,\nu)\ar[r,"\cong"] & \mb{A}(\vv{\lambda},\nu^*)^*\ar[r]&   V(\vv{\lambda})^*_{-\nu}
\end{tikzcd}
$$
is not the identity, but differs from it by a map to $\mathfrak{n}_{+}V(\vv{\lambda})^*$. In fact we claim the following:
\begin{lemma}\label{r:factor}
  The  composite $$
	\begin{tikzcd}
          \mb{A}(\vv{\lambda},\nu^*)^* \ar[r, hook] &
          V(\vv{\lambda})^*_{-\nu} \ar[r]&
          \mb{A}(\vv{\lambda}^*,\nu)\end{tikzcd} $$is same as the map
        $\mb{A}(\vv{\lambda},\nu^*)^*\lra \mb{A}(\vv{\lambda}^*,\nu)$
        (which, as before, is the natural map divided by
        $\dim V_{\nu}$).
\end{lemma}
Assuming this lemma, to complete the proof we have to use the fact that
the map $V(\vv{\lambda})_{\nu}\to A^M(U)$ in \eqref{bigS} sends the
image of $\mathfrak{n}_{-}$ to exact forms in the Aomoto complex
\cite[Theorem 4.2.2] {FSV2}.

\begin{proof}[Proof of Lemma \ref{r:factor}]
  This reduces easily to the case $n=1$ and $\lambda=\nu$, so we
  assume this.
  
  We have an invariant element $x$ in $ (V_{\nu}\tensor V_{\nu}^*)^*$
  which is the image of $1$ under a natural map
  $\Bbb{C}\to (V_{\nu}\tensor V_{\nu}^*)^*$. Now as before
  $(V_{\nu}\tensor V_{\nu}^*)^*$ is mapped into $(V_{\nu}^*)_{-\nu}$.
  We first show the property that the image of $1$ under the composite
  $\Bbb{C}\to (V_{\nu}\tensor V_{\nu}^*)^*\to (V_{\nu}^*)_{-\nu}$ is
  the lowest weight vector of $V_{\nu}^*$.

  The composite arises by dualizing
  $V_{\nu}\to (V_{\nu}\tensor V_{\nu}^*)\to \C$ where the second map
  is evaluation, and the first the tensor product with the lowest
  weight $v$ of $V_{\nu}^*$.
 
  Now $V_{\nu}\to (V_{\nu}\tensor V_{\nu}^*)\to \C$ sends
  $y\in V_{\nu}$ to $y(v)$. Therefore, it sends the highest weight
  vector of $V_{\nu}$ to $1$ and is zero on other weight vectors,
  which proves the desired property.

  Next, $v$ is acted on by $w_0$ which turns it into the highest
  weight of $V_{\nu}^*$. We tensor this with the lowest weight vector
  of $V_{\nu}$ and thus get a pure tensor in
  $V_{\nu}^*\tensor V_{\nu}$ which evaluates to $1$ under the natural
  mapping of coinvariants (a one dimensional space) to $\C$. The
  invariant $x$ (which is not a pure tensor if $\dim V_{\nu}>1$) maps
  to $\dim V_{\nu}$, since the natural composite
  $\C\to V_{\nu}\tensor V_{\nu}^*\to \C$ is multiplication by
  $\dim V_{\nu}$. This completes the proof.
\end{proof}


\vspace{0.1 in}

\noindent
P.B.: Department of Mathematics, University of North Carolina, Chapel
Hill, NC 27599, USA\\
\vspace{0.02 cm}

\noindent
N.F. and S.M.: School of Mathematics, Tata Institute of Fundamental Research, Homi Bhabha Road, Mumbai 400005, India\\
\vspace{0.08 cm}

\noindent
{email: P.B. belkale@email.unc.edu, N.F. naf@math.tifr.res.in, S.M. swarnava@math.tifr.res.in}

\end{document}